\documentclass[10pt,a4paper]{amsart}

\usepackage[all]{xy}
\usepackage{amsthm}
\usepackage{amssymb} 
\usepackage{amsmath,amscd}
\usepackage{mathrsfs}
\usepackage{enumerate}

\renewcommand{\theequation}{\arabic{section}-\arabic{equation}}

\newcommand{\agl}[1]{\langle #1 \rangle}

\newcommand{\isagl}[1]{\langle #1 \rangle_{\mathsf{S^{-1}}}}

\newcommand{\eragl}[1]{\langle #1 \rangle_{\mathsf{ER}}}

\newcommand{\Dbmod}[1]{\sfD^{\mrb}( #1 \operatorname{mod})}

\def\pldeg{{\operatorname{\textup{pl}\!\deg \ } }}

\def\bideg{{\textup{bi}\!\deg \ }}

\def\Euch{{\chi}}
\def\Euvect{{\underline{\chi}}}

\def\brho{{\breve{\varrho}}}

\def\frkm{{\mathfrak{m}}}

\def\We{{\mathcal{W}}}

\def\sfcan{{\mathsf{can}}}

\def\sfind{{\mathsf{ind}}}

\def\sfiso{{\mathsf{iso}}}

\def\turn!{\textup{!`}}

\def\chara{\operatorname{char}}

\def\sfAR{\mathsf{AR}}

\def\op{\textup{op}}

\def\Tr{\operatorname{Tr}}

\def\ind{\mathop{\mathrm{ind}}\nolimits}

\def\DG{\textup{DG}}

\def\Gr{\operatorname{Gr}}

\def\mrb{\mathrm{b}}

\def\mre{\mathrm{e}}

\def\Spec{\operatorname{Spec}}
\def\maxSpec{\operatorname{MaxSpec}}

\def\Tor{\operatorname{Tor}}

\def\lvvee{\vartriangleleft} 
\def\rvvee{\vartriangleright}

\def\kk{{\mathbf k}}

\def\CC{{\Bbb C}}

\def\NN{{\Bbb N}}

\def\PP{{\Bbb P}}
\def\QQ{{\Bbb Q}}

\def\ZZ{{\Bbb Z}}

\def\cC{{\mathcal C}}

\def\cG{{\mathcal G}}
\def\cH{{\mathcal H}}
\def\cI{{\mathcal I}}

\def\cL{{\mathcal L}}
\def\cM{{\mathcal M}}
\def\cN{{\mathcal N}}

\def\cP{{\mathcal P}}

\def\cR{{\mathcal R}}

\def\cT{{\mathcal T}}
\def\cU{{\mathcal U}}

\def\vars{{\mathfrak s}}

\def\vart{{\mathfrak t}}

\def\sfa{{\mathsf{a}}}
\def\sfb{{\mathsf{b}}}
\def\sfc{{\mathsf{c}}}
\def\sfd{{\mathsf{d}}}
\def\sfe{{\mathsf{e}}}

\def\sfg{{\mathsf{g}}}
\def\sfh{{\mathsf{h}}}

\def\sfl{{\mathsf{l}}}
\def\sfm{{\mathsf{m}}}

\def\sfr{{\mathsf{r}}}

\def\sfC{{\mathsf{C}}}
\def\sfD{{\mathsf{D}}}

\def\sfG{{\mathsf{G}}}
\def\sfH{{\mathsf{H}}}

\def\sfJ{{\mathsf{J}}}
\def\sfK{{\mathsf{K}}}
\def\sfL{{\mathsf{L}}}
\def\sfM{{\mathsf{M}}}
\def\sfN{{\mathsf{N}}}

\def\sfO{{\mathsf{O}}}
\def\sfP{{\mathsf{P}}}

\def\sfS{{\mathsf{S}}}
\def\sfT{{\mathsf{T}}}
\def\sfU{{\mathsf{U}}}
\def\sfV{{\mathsf{V}}}

\def\tuD{{\textup{D}}}

\def\tuH{{\textup{H}}}
\def\tuHH{{\textup{HH}}}

\def\tuZ{{\textup{Z}}}

\def\id{\operatorname{id}}
\def\op{\operatorname{op}}

\def\pr{\operatorname{\textup{pr}}}

\def\mod{\operatorname{mod}}
\def\Mod{\operatorname{Mod}}

\def\Ker{\mathop{\mathrm{Ker}}\nolimits}

\def\proj{\operatorname{proj}}

\def\add{\operatorname{add}}

\def\Coker{\operatorname{Cok}}

\def\Hom{\operatorname{Hom}}

\def\cpxHom{\operatorname{Hom}^{\bullet}}

\def\End{\operatorname{End}}
\def\ResEnd{\operatorname{ResEnd}}

\def\Ext{\operatorname{Ext}}

\def\gldim{\operatorname{gldim}}

\newcommand{\RHom}{\operatorname{\Bbb{R}Hom}}

\newcommand{\lotimes}{\otimes^{\Bbb{L}}}

\newcommand{\cone}{\operatorname{\mathsf{cn}}}

\def\RHom{\operatorname{\mathbb{R}Hom}}

\newtheorem{lemma}{Lemma}[section]
\newtheorem{proposition}[lemma]{Proposition}
\newtheorem{theorem}[lemma]{Theorem}
\newtheorem{corollary}[lemma]{Corollary}

\newtheorem{remark}[lemma]{Remark}
\newtheorem{problem}[lemma]{Problem}

\newtheorem{example}[lemma]{Example}
 
\newtheorem{claim}[lemma]{Claim}

\newtheorem{observation}[lemma]{Observation}

\newtheorem{definition}[lemma]{Definition}

\newtheorem{question}[lemma]{Question}

\theoremstyle{remark}

\def\vvee{\Bbb{V}}

\def\eepsilon{\tilde{\epsilon}}

\def\oomega{\tilde{\omega}}

\def\varppi{\tilde{\varpi}}
\def\xxi{\tilde{\xi}}
\def\eeta{\tilde{\eta}}
\def\zzeta{\tilde{\zeta}}

\def\ppi{\tilde{\pi}}
\def\rrho{\tilde{\varrho}}

\def\ttheta{\tilde{\theta}}

\def\SL{\operatorname{SL}}

\def\Pa{A}

\def\PPa{\widetilde{A}}
\def\Qa{B}

\def\YM{\Lambda}

\def\YYM{\widetilde{\Lambda}}

\def\rad{\operatorname{rad}}
\def\mlap{\operatorname{lap}}

\def\irr{\operatorname{irr}}

\def\PPi{\widetilde{\Pi}}

\title[Quiver Heisenberg algebras]
{
Quiver Heisenberg algebras 
: a cubic analogue of preprojective algebras
}

\date{}

\author{Martin Herschend and  Hiroyuki Minamoto}

\date{}



\address{Department of Mathematics, Uppsala University, Uppsala, Sweden}
\email{martin.herschend@math.uu.se}

\address{Department of Mathematics Osaka Metropolitan University, Sakai City, Japan}
\email{minamoto@omu.ac.jp}

\begin{document}

\begin{abstract}
In this paper we study a certain class of central extensions of  preprojective algebras of quivers 
under the name quiver Heisenberg algebras (QHA).  
There are several classes of algebras introduced before by different researchers from different view points, which have the QHA as a special case. While these have mainly been studied in characteristic zero, we also study the case of positive characteristic.
Our results show that the QHA is closely related to the representation theory of the corresponding path algebra 
in a similar way to the preprojective algebra.
 
Among other things, one of our main results is that 
the QHA provides an exact sequence of bimodules over the path algebra of a quiver, which can be called the universal Auslander-Reiten sequence. 
Moreover, we show that the QHA provides minimal left and right approximations with respect to the powers of the radical functor. 
Consequently, we obtain  a description of the QHA as a module over the path algebra, 
which in the Dynkin case, 
gives a categorification (as well as a generalization to the positive characteristic case) of the dimension formula by Etingof-Rains. 
\end{abstract}

\maketitle

\tableofcontents

\section{Introduction} 

\subsection{Introduction}
The notion of a quiver, an alias for an oriented graph was introduced to representation theory by Gabriel \cite{Gabriel}. 
A representation of a quiver attaches to each vertex a vector space and to each arrow a linear map.

From a quiver $Q$ and a field $\kk$, an algebra $\kk Q$ called the \emph{path algebra} of  $Q$, is constructed. 
It is a kind of free algebra construction and hence a lot of algebras are obtained as residue algebras of the path algebras. 
Another important feature is that a representation of $Q$ is the same as a module over $\kk Q$ and 
that consequently, the category of representations of $Q$ is equivalent to the module category of $\kk Q$. 

The study of these equivalent categories, which is at the heart of quiver representation theory, has uncovered rich structures in them and lead to  deep connections to many other important subjects.

The path algebra $\kk Q$ of $Q$ is finite dimensional precisely when $Q$ is finite and acyclic. 
Finite dimensional path algebras are one of the central objects of study in representation theory of finite dimensional algebras and the structure of their modules categories have been extensively investigated. 
Auslander-Reiten theory which provides the module category of a finite dimensional algebra with an orderly structure, serves as a principal tool to investigate the category of representations of a finite acyclic quiver. One of the reasons path algebras are so important is that they are prototypical among hereditary algebras. 

From a quiver $Q$, another important  algebra $\Pi( Q)$ called  the \emph{preprojective  algebra} of  $Q$, is constructed. 
It was first introduced by Gelfand-Ponomarev \cite{Gelfand-Ponomarev}. 
Soon after that Dlab-Rignel \cite{Dlab-Ringel: the preprojective algebra} gave 
 a description by generators and relations  that is currently accepted as the definition of it   
(as well as a generalization  in the context of  modulated graphs). 
A way to introduce the preprojective algebra from Auslander-Reiten theory of $\kk Q$, 
was found by Baer-Geigle-Lenzing \cite{BGL}  
and was confirmed by Crawley-Boevey \cite{Crawley-Boevey: preprojective algebras} and Ringel \cite{Ringel: preprojective algebras}.

The preprojective algebra $\Pi(Q)$ of $Q$ is also one of the central objects of interest in representation theory of algebras and has been extensively studied. 
Moreover the path algebras $\kk Q$ and the preprojective algebras $\Pi(Q)$ have been shown to have wide range of applications:
 cluster algebras and related combinatorics, McKay correspondence, 
canonical basis, Nakajima quiver varieties, Kashiwara-Saito's realization of the crystal basis.

In this paper, we construct an algebra ${}^{v}\!\YM(Q)$ from a quiver, which we call the \emph{quiver Heisenberg algebra}.  
As is explained below, this algebra turns out to be isomorphic to a special case of algebras previously introduced by several other researchers. 
However, making use of our definition, 
we prove that quiver Heisenberg algebra is closely related to representation theory of the path algebra, as is the case with the preprojective algebra. 
 
To explain our results in detail, we first need to recall the preprojective algebras $\Pi(Q)$ and their relationship to 
Auslander-Reiten theory of $\kk Q$.

\subsection{Preprojective algebras and Auslander-Reiten theory of path algebras}

We recall definitions and basic facts about the preprojective algebra $\Pi(Q)$ of a quiver $Q$. 

Let $Q$ be a finite acyclic quiver and $A = \kk Q$ its path algebra. 
We denote by  $\overline{Q}$ the double of $Q$. Namely, $\overline{Q}$ is obtained from $Q$ by formally adding 
an opposite arrow $\alpha^{*}: j \to i$ for each arrow $\alpha: i \to j$ of the original quiver $Q$. 
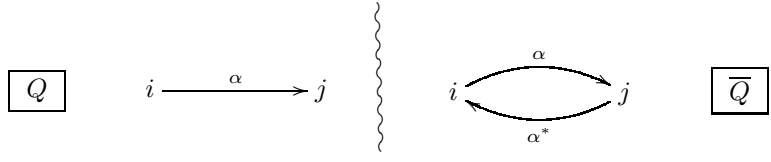
\[
\begin{xymatrix}{ & & \\ \boxed{ \ Q \ } &   i \ar[rr]^{\alpha} &&j }\end{xymatrix} \ \ \ \  \begin{xymatrix}@R=15pt@C=5pt{
\ar@{~}[ddd] \\  \\  \\ \\ 
}\end{xymatrix}  \ \ \ \ \ \ 
\begin{xymatrix}{ & \\ 
i \ar@/^10pt/[rr]^{\alpha} &&
 j  \ar@/^10pt/[ll]^{\alpha^{*}} &\boxed{ \ \overline{Q} \ } 
}\end{xymatrix} 
\]

Recall that for a vertex $i \in Q_{0}$, the \emph{mesh relation  $\rho_{i}$ at $i$} is the element of $\kk \overline{Q}$ 
that is given by 
\begin{equation}\label{mesh relation}
\rho_{i} := \sum_{\alpha \in Q_{1}: t(\alpha) = i} \alpha\alpha^{*} - \sum_{\alpha \in Q_{1}: h(\alpha) = i} \alpha^{*} \alpha. 
\end{equation}
The total sum  $\rho:= \sum_{ i \in Q_{0}} \rho_{i}$ is also  called  the \emph{mesh relation}.

The \emph{preprojective algebra} is defined to be the path of $\overline{Q}$ with mesh relations: 
\[
\Pi = \Pi(Q) =\frac{ \  \kk \overline{Q} \ }{ \   (\rho) \  } = \frac{ \kk \overline{Q} }{ (\rho_{i} \mid i \in Q_{0}) }.  
\]
We equip $\overline{Q}$ with a grading  which we call the $*$-grading
 \begin{equation}\label{202111191709}
 \deg^{*} \alpha := 0, \ \deg^{*} \alpha^{*} :=1 \textup{ for } \alpha \in Q_{1}. 
  \end{equation}
The mesh relations $\rho_{i}$ are homogeneous of degree $1$ and consequently $\Pi(Q)$ is a $*$-graded algebra. 
We denote the $*$-degree $n$ part of $\Pi(Q)$ by $\Pi(Q)_{n}$. 

The $*$-degree $0$ part $\Pi(Q)$ coincides with the path algebra $\Pa$ and
we may regard $\Pi(Q)_{n}$ as a bimodule over $\Pa$.

An important fact due to Baer-Geigle-Lenzing \cite{BGL} (see also \cite{Crawley-Boevey: preprojective algebras, Ringel: preprojective algebras}) 
is that there is an isomorphism 
$\Pi(Q)_{1} \cong  \Ext_{\Pa}^{1}(\tuD(\Pa), \Pa)$
of bimodules over $\Pa$ and  that the endofunctor $\Pi_{1} \otimes_{\Pa}-$ of the module category 
$\Pa\mod$ is the inverse of the Auslander-Reiten translation $\tau^{-1}_{1}$
\begin{equation}\label{202111191333}
\tau_{1}^{-1} = \Pi(Q)_{1} \otimes_{\Pa}- .
\end{equation}
Moreover, the above isomorphism of bimodules extends to  an  isomorphism below of graded algebras  
\begin{equation}\label{preprojective algebra isomorphism}
\Pi(Q) \cong \sfT_{\Pa} \Ext_{\Pa}^{1} (\tuD(\Pa), \Pa)
\end{equation}
where the right hand side is the tensor algebra  over $\Pa$ and the grading is given by the tensor degree.  
Thus  it can be said that the preprojective algebra $\Pi(Q)$ originates from Auslander-Reiten theory of the module category $\Pa \mod$.

One nice consequence is the following description of $\Pi(Q)$ as a module over $\Pa$. 
Let $\cP(Q) = \add \{\tau_{1}^{-i}A \mid i \ge 0 \}\subset A\mod$ be the category of the preprojective modules.

\begin{theorem}[{Gelfand-Ponomarev \cite{Gelfand-Ponomarev}, Dlab-Ringel \cite{Dlab-Ringel: the preprojective algebra}}]\label{202011062314}
The following assertions hold. 
\begin{enumerate}[(1)]
\item 
Let $i \in Q_{0}$ be a vertex. 
We set $P_{i} := \Pa e_{i}$. 
We have the following isomorphism of $\Pa$-modules: 
\[
\Pi(Q) e_{i} \cong \bigoplus_{n \geq 0} \tau_{1}^{-n} P_{i}. 
\] 

\item 
We have the following isomorphism of $\Pa$-modules: 
\[
\Pi(Q) \cong \bigoplus_{N  \in \ind \cP(Q) } N.  
\]
\end{enumerate}
\end{theorem}

In case $Q$ is Dynkin, it is well-known that $\cP(Q) $ coincides with the whole module category $\Pa \mod$. 
Therefore we obtain the following corollary, in which $\ind Q$ denotes the set of isomorphism class of indecomposable $\Pa$-modules. 

\begin{corollary} 
If $Q$ is a Dynkin quiver, 
then we have the following isomorphism of $\Pa$-modules.
\[
\Pi(Q) \cong \bigoplus_{N  \in \ind Q } N.
\]
\end{corollary}

\subsubsection{}\label{202111261159} 

Depending on whether $Q$ is Dynkin or non-Dynkin, properties of $\Pi(Q)$ change. 
But in both cases, $\Pi(Q)$ has salient  properties.

\begin{theorem}\label{202111191530}
For a Dynkin quiver $Q$, the following assertions hold. 
\begin{enumerate}[(1)]

\item 
The preprojective algebra $\Pi(Q)$ is a finite dimensional Frobenius algebra of dimension 
\[
\dim \Pi(Q) = \sum_{N \in \ind Q} \dim N =\frac{rh(h+1)}{6}
\]
where $h$ is the Coxeter number of $Q$ and $r := \# Q_{0}$. 
(The second equality is given by Etingof-Rains \cite{Etingof-Rains}.)

\item $\Pi(Q)$ is stably $2$-Calabi-Yau (i.e., the stable category $\underline{\mod} \Pi(Q)$ is a $2$-Calabi-Yau triangulated category).

\end{enumerate} 
\end{theorem}

\begin{theorem}\label{202111261158}
Let $Q$ be a non-Dynkin quiver. 
Then the preprojective algebra $\Pi(Q)$ is an infinite dimensional  $2$-Calabi-Yau algebra. 
\end{theorem}

\subsection{The  quiver Heisenberg algebras}

Now we introduce the quiver Heisenberg algebra (QHA) ${}^{v}\!\YM(Q)$ of a quiver $Q$ which is the main object of this paper. 
This  algebra is defined from a quiver $Q$ with explicit relations and 
has a parameter $v \in \kk^{\times} Q_{0}$ (a collection  of elements of $\kk^{\times} $ indexed by the vertices $i \in Q_{0}$).  
We remark that using the isomorphism given in Lemma \ref{202011242056} below, we can define the algebra ${}^{v}\!\YM(Q)$ for 
any element $v \in \kk Q_{0}$. 

We call an element $v$ of  $\kk Q_{0}$ (resp. $\kk^{\times} Q_{0}$)  \emph{weight} (resp. \emph{sincere weight}).

\begin{definition}\label{quiver Heisenberg algebras}
Let $v\in \kk^{\times} Q_{0}$ be a sincere weight.

\begin{enumerate}[(1)]
\item For $i \in Q_{0}$, we set 
${}^{v}\!\varrho_{i} := v_{i}^{-1} \rho_{i}$ and  
 ${}^{v}\!\varrho := \sum_{i \in Q_{0}} {}^{v}\!\varrho_{i}$. 
They are collectively  called  the \emph{weighted mesh relations}.

\item 
For $a \in \overline{Q}_{1}$, the \emph{quiver Heisenberg relation} ${}^{v}\!\eta_{a}$ is defined to be 
the commutator of $a$ with ${}^{v}\!\varrho$. Namely, 
for  an arrow  $a \in Q_{1}$ with $i =h(a), j =t(a)$, we set
\[
{}^{v}\!\eta_{a} := [a, {}^{v}\!\varrho] = a {}^{v}\!\varrho_{i} - {}^{v}\!\varrho_{j} a
= v_{i}^{-1} a \rho_{i} - v_{j}^{-1}\rho_{j} a.  
\]
Explicitly, 
for an arrow  $\alpha \in Q_{1}$ with $i =h(\alpha), j =t(\alpha)$, we have 
\[
\begin{split}
{}^{v}\!\eta_{\alpha} &= 
\sum_{\beta: t(\beta) = i} v_{i}^{-1}\alpha \beta\beta^{*}
- \sum_{\beta : h(\beta) = i} v_{i}^{-1} \alpha \beta^{*}\beta
- \sum_{\beta: t(\beta) = j} v_{j}^{-1} \beta\beta^{*}\alpha 
+\sum_{\beta: h(\beta) = j} v_{j}^{-1}\beta^{*}\beta \alpha, \\ 
{}^{v}\! \eta_{\alpha^{*}} &= 
\sum_{\beta: t(\beta) = j} v_{j}^{-1}\alpha^{*} \beta\beta^{*}
- \sum_{\beta : h(\beta) = j} v_{j}^{-1} \alpha^{*} \beta^{*}\beta
- \sum_{\beta : t(\beta) = i}  v_{i}^{-1}\beta\beta^{*}\alpha^{*} 
+\sum_{\beta : h(\beta) = i} v_{i}^{-1} \beta^{*}\beta \alpha^{*}. 
\end{split}
\]

\item 
We define the \emph{quiver Heisenberg algebra (QHA)} ${}^{v}\!\YM = {}^{v}\!\YM(Q)$ to be 
the path algebra of the double quiver $\overline{Q}$ with the quiver Heisenberg relations: 
\[
{}^{v}\!\YM :={}^{v}\!\YM(Q) := \frac{  \kk \overline{Q}}{ \ ({}^{v}\!\eta_{a} | a \in \overline{Q}_{1}) \ }. 
\]
\end{enumerate}
\end{definition}

\subsubsection{Remark about the naming }

The authors originally studied the case that $v_{i} = 1$ for all $i \in Q_{0}$. 
In this case, if   $Q= \ \circlearrowleft$ is  a loop, 
then $\Pi(\circlearrowleft )$ is isomorphic to the polynomial algebra $S=\kk[x,y]$ in two variables 
and  ${}^{v}\!\YM(\circlearrowleft )$ is isomorphic to the usual Heisenberg algebra $H$ in variables $x,y$:
\[
{}^{v}\!\YM(  \circlearrowleft  ) \cong H := \frac{ \ \kk \langle x, y \rangle \ }{ ([x,[x,y]], \ [y,[x,y]])  }.
\]
We remark that in the sequel throughout the paper $Q$  denotes a finite \emph{acyclic} quiver.

Let  $Q$ be an extended Dynkin quiver.  
A fundamental fact in algebraic McKay correspondence is that  $\Pi(Q)$ is Morita equivalent  to the skew group algebra $S* G$ 
where $G$ is a finite subgroup of $\SL(2;\kk)$ corresponding to $Q$. 
If we assume that  $v_{i} = 1$ for all $i \in Q_{0}$, then  
${}^{v}\!\YM(Q)$ can be shown to be Morita equivalent to the skew group algebra $H* G$. 
Hence we gave the algebra ${}^{v}\!\YM(Q)$ the name ``quiver Heisenberg algebra''.
In a sense, we consider $\Pi(Q)$ as a quiver version of the polynomial algebra $S= \kk[x,y]$ in two variables and 
${}^{v}\!\YM(Q)$ as a quiver version of the usual Heisenberg algebra $H$. 
In this comparison, we are looking at arrows $\alpha$ of the original quiver $Q$ as the variable $x$ and 
the opposite arrows $\alpha^{*}$ as the variable $y$. 
Since $S$ and $H$ are basic and important examples of Artin-Schelter (AS) algebras \cite{Artin-Schelter} in two variables, 
it might be worth pursing  quiver versions of other AS-regular algebras in two variables. 

\subsection{Related algebras and preceding results}

In this section we explain that the algebras ${}^{v}\!\YM(Q)$ form a very special case of other classes of algebras which were introduced before by several researchers from different view points.

For this purpose, we describe the quiver Heisenberg algebra as a quotient of the path algebra $\kk[z]\overline{Q}$ with polynomial coefficients.

\subsubsection{The quiver Heisenberg algebra as a path algebra with polynomial coefficients}

\begin{lemma}\label{202011242056}
Let $v \in \kk^{\times} Q_{0}$ be a sincere weight.
We have the following isomorphism of algebras:
\[
{}^{v}\!\YM(Q) \cong \frac{\kk[z]\overline{Q}}{( \rho_{i} -v_{i}z e_{i} \mid i \in Q_{0} )}.
\]
\end{lemma}

\begin{proof}
For simplicity, we set the algebra in the right hand side to be   ${}^{v}\!\YM'(Q)$.  
Let $f'': \kk\overline{Q} \to {}^{v}\!\YM(Q)$ be the canonical surjection and  $f': \kk[z] \to {}^{v}\!\YM(Q)$ the homomorphism of algebras which sends $z$ to ${}^{v}\!\varrho$.
Since ${}^{v}\!\varrho$ is central in ${}^{v}\!\YM(Q)$, the linear map $\kk[z] \overline{Q} = \kk[z] \otimes_{\kk} \kk \overline{Q} \to {}^{v}\!\YM(Q)$, which sends $p \otimes a$ to $f'(p)f''(a)$ is a homomorphism of algebras, which induces a homomorphism $f: {}^{v}\!\YM'(Q) \to {}^{v}\!\YM(Q)$. 

On the other hand, the canonical homomorphism $\kk \overline{Q} \to \kk [z] \overline{Q} \to {}^{v}\!\YM'(Q)$ of algebras 
induces a homomorphism $g: {}^{v}\!\YM(Q) \to {}^{v}\!\YM'(Q)$ of algebras. 
It is easy to see that the maps $f$ and $g$ are inverse to each other. 
\end{proof}

\begin{remark}
For a weight $v \in \kk Q_{0}$ which is not sincere, we interpret the symbol ${}^{v}\!\YM(Q)$ as the algebra in the right hand side of the above lemma.
\end{remark}

\subsubsection{Related algebras}

In view of Lemma \ref{202011242056}, 
it is clear that the algebra ${}^{v}\!\YM(Q)$ is a special case of the \emph{central extension of the preprojective algebras}
introduced by Etingof-Rains \cite{Etingof-Rains}, which is defined  to be 
\[
\Pi(Q)_{\lambda, \mu} := \frac{\kk[z] \overline{Q}}{ ( \rho_{i}- (\lambda_{i} z + \mu_{i})e_{i}\mid i \in Q_{0})}
\]
where  
$\lambda_{i}, \mu_{i} \in \kk$ for each $i \in Q_{0}$. 

Replacing the non-homogeneous linear polynomials $\lambda_{i}z + \mu_{i}$ with 
general polynomials $P_{i}(z)$,  
 we obtain the  \emph{$N=1$-quiver algebra} by Cachazo-Katz-Vafa \cite{CKV}, which is given as 
 \[
\Pi(Q)_{P} := \frac{ \kk [z] \overline{Q}}{ (\rho_{i}- P_{i}(z) e_{i}\mid i \in Q_{0}) } 
\]
where $P_{i}(z) \in \kk[z]$ for each $i \in Q_{0}$. 

Finally, in their influential work 
Crawley-Boevey-Holland \cite{Crawley-Boevey-Holland} introduced
the  \emph{deformation family of the preprojective algebra} 
which is defined to be 
\[
\Pi(Q)_{\bullet} := \frac{\kk [x_{1}, \ldots, x_{r}]\overline{Q}}{ (\rho_{i} -x_{i}e_{i} \mid i \in Q_{0}) } 
\] 
where $r = \# Q_{0}$. 
We may regard $\Pi(Q)_{\bullet}$ as a family of algebras over the $r$-dimensional affine space $\kk^{r}$
and  the algebra $\Pi(Q)_{P}$ is obtained from $\Pi(Q)_{\bullet}$ as the pull-back 
by the polynomial map $\kk \to \kk^{r}, \ z \mapsto (P_{1}(z), \ldots, P_{r}(z) )$. 
Thus in particular, the QHA ${}^{v}\!\YM(Q)$ is obtained as the restriction of $\Pi(Q)_{\bullet}$ to the line $\langle v \rangle \subset \kk^{r}$ 
connecting $v\in \kk^{r}$ and the origin. 
  
We note that in the previous studies of these algebras, the case $\chara \kk = 0$ was mainly considered.

\subsubsection{Preceding results}

By specializing the results obtained for the algebras $\Pi(Q)_{\lambda, \mu}$ of Dynkin type, 
we can deduce some results about ${}^{v}\!\YM(Q)$. 
To state them, we need to introduce one condition on weights. 

A weight $v \in \kk Q_{0}$ is called \emph{regular} if $ \sum_{i \in Q_{0}} v_{i}  \dim (e_{i} M) \neq 0$ for any indecomposable $\Pa$-module $M$ 
(see Definition \ref{202111192031} where the dimension vector is denoted by $\Euvect$).
We note that a regular weight is sincere. 

In the case $Q$ is  Dynkin and $\chara \kk = 0$, the vector space $\kk Q_{0}$ may be identified with the Cartan subalgebra $\mathfrak{h}$ of the semi-simple Lie algebra $\mathfrak{g}$ corresponding to $Q$. 
By Gabriel's theorem the dimension vectors of indecomposable $\Pa$-modules are precisely the roots of $\mathfrak{g}$, so
the regularity given here coincides with that are used by Etingof-Rains.

\begin{theorem}[{(1) Etingof-Rains \cite{Etingof-Rains}, (2) Etingof-Latour-Rains \cite{ELR}, (3) Eu-Schedler \cite{Eu-Schedler} }]\label{202111191454}

Assume that $\chara \kk =0$. 
Let $Q$ be a Dynkin quiver and $v \in \kk Q_{0}$ be a regular weight. 
Then the following assertions hold. 
\begin{enumerate}[(1)]

\item The algebra  ${}^{v}\!\YM(Q)$ is a finite dimensional  Frobenius algebra   
of dimension 
\[
\dim {}^{v}\!\YM(Q) = \sum_{N \in  \ind Q} (\dim N)^{2} =\frac{rh^{2}(h+1)}{12} 
\]
where $h$ is the Coxeter number of $Q$ and $r := \# Q_{0}$. 

\item If $v \in \kk Q_{0}$ is generic, then ${}^{v}\!\YM(Q)$ is symmetric. 

\item 
The algebra ${}^{v}\!\YM(Q)$ is stably $3$-Calabi-Yau. 

\end{enumerate}

\end{theorem}

Comparing these results with  the results of the preprojective algebra $\Pi(Q)$ of Dynkin type given in Theorem \ref{202111191530}, 
it is maybe too optimistic but, 
we expect that the algebra ${}^{v}\!\YM(Q)$ may have nice analogous properties with that of the preprojective algebras.
Our results prove  that this is indeed the case.

\subsection{Our results $1/2$: the quiver Heisenberg algebras and Auslander-Reiten theory of the path algebras}\label{subsection: QHA and AR-theory}

We start explaining our result.

\subsubsection{} 
Recall that the quiver Heisenberg relations ${}^{v}\!\eta_{a} = [a, {}^{v}\!\varrho]$ are commutators of ${}^{v}\!\varrho$ with the generators $a \in \overline{Q}_{1}$ of the algebra $\kk \overline{Q}$.
It follows that  ${}^{v}\!\varrho $ becomes a central element of ${}^{v}\!\YM$. 
It is easy to see that ${}^{v}\!\YM/({}^{v}\!\varrho) = \Pi$. 
Putting these observations differently, 
we have a canonical surjective homomorphism ${}^{v}\!\pi: {}^{v}\!\YM \to \Pi$ of algebras 
and  an exact sequence 
\begin{equation}\label{202009212139}
 {}^{v}\!\YM \xrightarrow{ \ {}^{v}\!\varrho \ }  {}^{v}\!\YM \xrightarrow{ \ {}^{v}\!\pi \ } \Pi \to 0
\end{equation}
of ${}^{v}\!\YM$-bimodules 
where the first arrow is the multiplication by ${}^{v}\!\varrho$.

The quiver Heisenberg relations are homogeneous with respect to the $*$-grading \eqref{202111191709}:
 $\deg^{*} {}^{v}\!\eta_{\alpha} = 1, \ \deg^{*} {}^{v}\!\eta_{\alpha^{*}} = 2$ for $\alpha \in Q_{1}$. 
Therefore, ${}^{v}\!\YM$ is a $*$-graded algebra and the map ${}^{v}\!\pi$ preserves the $*$-grading.
 Since $\deg^{*} {}^{v}\!\varrho =1$, we get that by taking the $*$-grading into account, the exact sequence 
 \eqref{202009212139} becomes 
\begin{equation}\label{202009212139gr}
 {}^{v}\!\YM(-1) \xrightarrow{ \ {}^{v}\!\varrho \ }  {}^{v}\!\YM \xrightarrow{ \ {}^{v}\!\pi \ } \Pi \to 0
\end{equation}
 where $(-1)$ denotes the $*$-degree shift by $-1$, i.e., $({}^{v}\!\YM(-1))_{n} = {}^{v}\!\YM_{n -1}$. 
Looking at the $*$-degree $1$ part of this exact sequence 
we obtain an exact sequence of  $\Pa$-bimodules
\begin{equation}\label{universal AR-sequence}
\Pa \xrightarrow{ \ {}^{v}\!\varrho \ } {}^{v}\!\YM_{1} \xrightarrow{ \ {}^{v}\!\pi_{1} \ } \Pi_{1}  \to 0.
\end{equation}
Let $M$ be an indecomposable $\Pa$-module. 
By \eqref{202111191333},  taking the tensor product $- \otimes_{\Pa} M$ with the above exact sequence, 
we obtain an exact sequence of $\Pa$-modules of the following form 
\begin{equation}\label{202111181637}
M \xrightarrow{ \ {}^{v}\!\varrho_{M} \ } {}^{v}\!\YM_{1} \otimes_{\Pa} M \xrightarrow{ \ {}^{v}\!\pi_{1, M} \  } \tau^{-1}_{1}M  \to 0
\end{equation}
where we set ${}^{v}\!\varrho_{M} := {}^{v}\!\varrho \otimes_{\Pa} M$ and ${}^{v}\! \pi_{1, M} := {}^{v}\!\pi_{1} \otimes_{\Pa} M$. 
This exact sequence looks like an Auslander-Reiten sequence starting from $M$. 
The next theorem says that this is the case under certain conditions.

\begin{theorem}[Universal Auslander-Reiten sequence]\label{universal AR-sequence theorem}
Assume that the weight  $v\in \kk Q_{0}$ is regular.  
Let $M$ be an indecomposable non-injective $\Pa$-module. 
Then the morphism ${}^{v}\!\varrho_{M} $ is injective and the exact sequence \eqref{202111181637} 
is an AR-sequence starting from $M$. 
\[
0\to M \xrightarrow{ \ {}^{v}\!\varrho_{M} \ } {}^{v}\!\YM_{1} \otimes_{\Pa} M \xrightarrow{ \ {}^{v}\!\pi_{1, M} \  } \tau^{-1}_{1}M  \to 0.
\]
\end{theorem}

In view of this theorem, we may call the exact sequence \eqref{universal AR-sequence} the \emph{universal Auslander-Reiten sequence}. 

\subsubsection{}

The QHA ${}^{v}\!\YM$ is generated by   ${}^{v}\!\YM_{0} = \Pa$ and ${}^{v}\!\YM_{1}$. 
Hence the multiplication map ${}^{v}\!\zeta_{2}: {}^{v}\! \YM_{1} \otimes_{\Pa} {}^{v}\!\YM_{1} \to {}^{v}\!\YM_{2}$ 
is surjective.  
We will show that there is  a morphism 
${}^{v}\!\eta^{*}_{2}: \Pi_{1} \to {}^{v}\!\YM_{1} \otimes_{\Pa} {}^{v}\!\YM_{1}$ 
of bimodules over $\Pa$ 
whose image coincides with the kernel of ${}^{v}\!\zeta_{2}$. 
In other words, we have an exact sequence 
\[
\Pi_{1} \xrightarrow{ \ {}^{v}\!\eta^{*}_{2} \ } {}^{v}\!\YM_{1} \otimes_{\Pa} {}^{v}\!\YM_{1} \xrightarrow{ \ {}^{v}\!\zeta_{2} \ } {}^{v}\!\YM_{2} \to 0
\]
of bimodules over $\Pa$. 
A rough explanation of the first map ${}^{v}\!\eta^{*}_{2}$ is that 
the $\Pa$-bimodule $\Pi_{1}$ is generated by the arrows $\alpha^{*} (\alpha \in Q_{1})$ 
and the map ${}^{v}\!\eta^{*}_{2}$ sends $\alpha^{*}$ to ${}^{v}\!\eta_{\alpha^{*}}$ regarded as  elements of ${}^{v}\!\YM_{1} \otimes_{\Pa} {}^{v}\!\YM_{1}$. 

The map ${}^{v}\!\eta^{*}_{2}$ also has an AR-theoretic meaning.
Let $M$ be an indecomposable non-injective $\Pa$-module and 
${}^{v}\!\YM_{1} \otimes_{\Pa} M =\bigoplus_{i =1}^{r} N_{i}$ an indecomposable decomposition.   
By Theorem \ref{universal AR-sequence theorem},  an AR-sequence starting from $M$ is of the following form.
\[
\begin{xymatrix}@R=0.5mm@C=5mm{
 &M \ar[rr] && {}^{v}\!\YM_{1} \otimes_{\Pa} M  \ar[rr] && \tau_{1}^{-1}(M) \hspace{40pt} & &&\\
\ar@{-}[rrrrrrr]&&&&&&&&\\ 
&&& N_{1}\ar[ddrr]&& & &&   \\
&&& N_{2}\ar[drr] &&&&&\\
&M  \ar[uurr] \ar[urr]
\ar[ddrr] & & \vdots && \tau^{-1}_{1}( M )&  && \\ 
&&&&& &&&\\
&&& N_{r}\ar[uurr]&&  &&& \\
}\end{xymatrix}
\hspace{70pt}
\]

Observe that there is an irreducible morphism $N_{i} \to \tau^{-1}_{1}M$ for each $i=1,2,\ldots, r$. 
If we assume that $N_{1}, \ldots, N_{r}$ are not injective, 
then by Theorem \ref{universal AR-sequence theorem}, 
the module ${}^{v}\!\YM_{1} \otimes_{\Pa} {}^{v}\!\YM_{1} \otimes_{\Pa} M$ is the middle term of the direct sum of AR-sequences starting from $N_{i}$.
\[
0 \to {}^{v}\! \YM_{1} \otimes_{\Pa} M 
\xrightarrow{ {}^{v}\!\varrho_{ {}^{v}\!\YM_{1} \otimes M}}   
{}^{v}\! \YM_{1} \otimes_{\Pa} {}^{v}\! \YM_{1} \otimes_{\Pa} M 
\xrightarrow{{}^{v}\!\pi_{1, {}^{v}\!\YM_{1}\otimes M} } \tau^{-1}_{1}({}^{v}\!\YM_{1} \otimes_{\Pa} M) \to 0
\] 
The point is that the module 
${}^{v}\!\YM_{1} \otimes_{\Pa} {}^{v}\!\YM_{1} \otimes_{\Pa} M$ contains  $\tau^{-1}_{1}(M)$ as a direct summand. 
The next theorem says that under a certain assumption, 
the morphism 
\[
{}^{v}\! \eta^{*}_{2, M}= {}^{v}\!\eta^{*}_{2}\otimes_{\Pa} M:  \tau^{-1}_{1}(M) \to {}^{v}\!\YM_{1} \otimes_{\Pa} {}^{v}\!\YM_{1} \otimes_{\Pa} M
\] 
provides this direct summand. 
\[
\begin{xymatrix}@R=0.5mm@C=5mm{ &&&&&&\Pi_{1} \otimes_{\Pa} M \ar@/_10pt/[ddl]_-{{}^{v}\!\eta^{*}_{2, M} } & \\  
&&&&&&& \\
&M  && {}^{v}\!\YM_{1} \otimes_{\Pa} M  \ar[rr]
&& {}^{v}\!\YM_{1} \otimes_{\Pa}{}^{v}\!\YM_{1}  \otimes_{ \Pa} M \ar[rr] \ &&
\tau^{-1}_{1}({}^{v}\!\YM_{1} \otimes_{\Pa} M )  \ \ \ \ \ 
 &\\
\ar@{-}[rrrrrrr]&&&&&&&&\\ 
&&& N_{1}\ar[ddrr] \ar[drr] \ar[rr] && L_{1}& &&   \\
&&& N_{2}\ar[drr] \ar[urr] && L_{2} &\tau^{-1}_{1}(M) \ar@/_10pt/[dl]_{\cong}& &\\
&M  
& & \vdots && \tau^{-1}_{1}( M )&  &  & \\ 
&&&&& &&&\\
&&& N_{r}\ar[uurr]\ar[rr]&& L_{s} &&& \\
}\end{xymatrix}
\]
The precise statement is the following.

\begin{theorem}\label{202111181854}
Let $v \in \kk Q_{0}$ be a regular weight. 
Let $M \in \Pa \mod$ be a non-injective indecomposable module 
such that ${}^{v}\!\YM_{1} \otimes_{\Pa} M$ does not contain an injective module as a direct summand. 
Assume that $ \sum_{i \in Q_{0}} v_{i}\dim (e_{i} {}^{v}\!\YM_{1} \otimes_{\Pa} M) \neq 0$.

Then 
there exists a morphism $\xi_{2, M}: {}^{v}\!\YM_{1} \otimes_{\Pa} {}^{v}\!\YM_{1} \otimes_{ \Pa} M \to \tau_{1}^{-1}(M)$ 
which satisfies   the following equation. 

\begin{enumerate}[(1)] 

\item 
$
\xi_{2, M} {}^{v}\!\eta^{*}_{2, M}  = x \id_{\tau_{1}^{-1} (M)}$ 
where we set 
\[
x := -\frac{ \sum_{i \in Q_{0}} v_{i}\dim (e_{i} {}^{v}\!\YM_{1} \otimes_{\Pa} M)}{ \sum_{i\in Q_{0}} v_{i}\dim(e_{i} M)}.
\]

\item $\xi_{2, M}  {}^{v}\!\rho_{ \YM_{1}\otimes M}= {}^{v}\!\pi_{1, M}$. 

\end{enumerate}

Namely we have the following commutative  diagram. 
\[
\begin{xymatrix}@C=40pt{ 
& \tau_{1}^{-1} (M) \ar[d]^{{}^{v}\!\eta^{*}_{2,M} }  
\ar@/^20pt/[ddr]_{\cong}^{  x\id} 
& 
\\ 
{}^{v}\!\YM_{1}\otimes_{\Pa} M \ar@/_20pt/[rrd]_{{}^{v}\!\pi_{1, M} } \ar[r]^-{{}^{v}\!\rho_{ {}^{v}\!\YM_{1} \otimes M}}
 &
 {}^{v}\!\YM_{1} \otimes_{\Pa} {}^{v}\!\YM_{1} \otimes_{\Pa} M \ar[dr]^{\xi_{2, M}}
 & \\
&& \tau_{1}^{-1}(M),
}
\end{xymatrix}\]

\end{theorem}

We note that this theorem is an immediate consequence of Theorem  \ref{2020071920551}.

\subsubsection{} 
The following theorem shows that the QHA  ${}^{v}\!\YM(Q)$ originates from AR-theory, in a similar way to the preprojective algebra $\Pi$.

\begin{theorem}
We have an isomorphism of graded algebras 
\[
{}^{v}\!\YM(Q) \cong \frac{ \ \sfT_{\Pa} {}^{v}\!\Lambda_{1} \ }{ ( {}^{v}\!\eta^{*}_{2}(\Pi_{1}) ) }. 
\]
\end{theorem}

We observe  that, as an algebra over $\Pa$, the algebra ${}^{v}\!\YM$ is generated by 
the bimodule ${}^{v}\!\YM_{1}$ that provides the middle terms of AR-sequences.  
Moreover, the relations  come only from  the embedding  ${}^{v}\!\eta^{*}_{2}$
of $\tau^{-1}_{1}M$ in the middle terms of the AR-sequences starting from the middle terms of the AR-sequence starting from $M$.

\subsubsection{}

Next we establish an analogous result with Theorem \ref{202011062314}, which describes  the structure of ${}^{v}\!\YM(Q)$ as a module of the path algebra of $Q$.

\begin{theorem}\label{description of QHA as kQ-modules}
Assume that the weight $v \in \kk Q_{0}$ is regular. 
Then  the following assertions hold. 
\begin{enumerate}[(1)]
\item Let $i \in Q_{0}$. 
We have the following isomorphism of $\Pa$-modules: 
\[
{}^{v}\!\YM(Q)e_{i} \cong \bigoplus_{ N \in \ind \cP(Q)} N^{\oplus \dim e_{i}N}.
\]
\item 
We have the following isomorphism of $\Pa$-modules: 
\[
{}^{v}\!\YM(Q) \cong \bigoplus_{ N \in \ind \cP(Q)} N^{\oplus \dim N}.
\]
\end{enumerate}
\end{theorem}

As a consequence we obtain the following result for Dynkin quivers. 

\begin{corollary}\label{202111191825}
Let $Q$ be a Dynkin quiver and $v \in \kk Q_{0}$ a regular weight. 
We have the following isomorphism of $\kk Q$-modules: 
\[
{}^{v}\!\YM(Q) \cong \bigoplus_{ N  \in \ind Q} N^{\oplus \dim N}.
\]
\end{corollary} 

Comparing dimensions, we obtain the dimension formula by Etingof-Rains  for ${}^{v}\!\YM(Q)$ of Dynkin type 
even in the case $\chara \kk > 0$. 

\begin{corollary}\label{202111191828}
Let $Q$ be a Dynkin quiver and $v \in \kk Q_{0}$ a regular weight. 
Then the following equality holds:
\[
\dim {}^{v}\!\YM(Q) = \sum_{ N \in \ind Q} (\dim N)^{2}=\frac{rh^{2}(h +1)}{12}.
\]
\end{corollary}

We note that Theorem \ref{202109131544} proves that ${}^{v}\!\YM(Q)$ of Dynkin type is finite dimensional if and only if $v$ is regular. 

We also prove the following result for non-Dynkin quivers, which corresponds to Theorem \ref{202111261158} about $\Pi(Q)$.

\begin{theorem}[{Theorem \ref{202111261228}}]\label{202111261227}
Let $Q$ be a non-Dynkin quiver and $v\in \kk Q_{0}$ a sincere weight. 
Then the QHA ${}^{v}\!\YM(Q)$ is $3$-Calabi-Yau. 
\end{theorem}

In subsequent work \cite{Minamoto Smith}, we give a generalization of the Calabi-Yau completion by Keller \cite{Keller: Calabi-Yau completion} 
(see also Section \ref{202111261238}).

\subsection{Our results $2/2$: the derived  quiver Heisenberg algebras and $\rad^{n}$-approximation theory of the path algebras}\label{subsection: derived QHA and rad^n-approximation theory}

Until now, we have dealt with ordinary algebras  and worked in the categories of ordinary modules. 
But to prove our results, 
we need to introduce the dg-algebras ${}^{v}\!\YYM(Q)$ which we call the \emph{derived quiver Heisenberg algebras (derived QHA)} and  work with derived categories.
The derived QHA ${}^{v}\!\YYM = {}^{v}\!\YYM(Q)$ is a dg-algebra explicitly constructed from a quiver $Q$ (see Definition \ref{202111201445}).
In addition to the cohomological grading, 
it also acquires a $*$-grading and the $0$-th cohomology algebra $\tuH^{0}({}^{v}\!\YYM(Q))$ is canonically isomorphic to ${}^{v}\!\YM(Q)$ 
 as $*$-graded algebras. 

A crucial role is played by 
\emph{approximations with respect to the $n$-th power  $\rad^{n}$ of the radical functor $\rad$} which we call $\rad^{n}$-approximations for short.  
Recall that an ideal $\cI$ of a $\kk$-linear category $\sfC$ is a $\kk$-linear sub bifunctor of the bifunctor $\Hom_{\sfC}(-,+)$. A left or right approximation of an object $M \in \sfC$ with respect to an ideal $\cI$ is, roughly speaking, a morphism 
belonging to $\cI$ starting from or ending at $M$ that is as close as possible. 
It is well-known for experts that $\rad$-approximation theory in the module categories is nothing but Auslander-Reiten theory. 
Theory of
$\rad^{n}$-approximations  in the module categories was initiated by Igusa-Todorov \cite{Igusa-Todorov}. They investigated constructions of $\rad^{n}$-approximations by using the notion of \emph{ladder} which was introduced by them in the same paper. 
Later Iyama \cite{Iyama: ladder}  established a criterion for  existence of ladders in  
$\tau$-categories which is an abstraction of modules category with AR-translations $\tau_{1}^{\pm 1}$ introduced by him.

The universal Auslander-Reiten sequence (Theorem \ref{universal AR-sequence theorem}) tells us that 
the $*$-degree $1$ part ${}^{v}\!\YM_{1}$ of the QHA provides left and right minimal $\rad$-approximations in the category $\Pa\mod$ of $\Pa$-modules. 
Our main results can be roughly summarized that
the $*$-degree $n$ part ${}^{v}\!\YYM_{n}$ of the derived QHA provides minimal left and right  $\rad^{n}$-approximations in the derived category $\Dbmod{\Pa}$ and moreover  ${}^{n}\YYM$ provides a special type of a left ladder. 

\subsubsection{} 

We set $\PPi_{1} := \RHom_{\Pa}(\tuD(\Pa), \Pa)[1]$ considered as a complex of bimodules over $\Pa$.
Recall that the endofunctor $\PPi_{1} \lotimes_{\Pa}- $ of $\Dbmod{\Pa}$ coincides 
the inverse of the Auslander-Reiten translation  $\nu_{1}^{-1}$ of $\Dbmod{\Pa}$. 
\[
\nu_{1}^{-1} \cong \PPi_{1}\lotimes_{\Pa}-.  
\]
The derived preprojective algebra $\PPi = \PPi(Q)$ of $Q$ is, by definition, the tensor algebra of the complex $\PPi_{1}$ over $\Pa$ 
and its $0$-th cohomology algebra $\tuH^{0}(\PPi)$ is isomorphic to the preprojective algebra $\Pi$ 
as $*$-graded algebras where we equip $\PPi=\sfT_{\Pa} \PPi_{1}$ with  $*$-grading given by the tensor degree.

There exists a $*$-graded dg-algebra morphism ${}^{v}\!\ppi: {}^{v}\!\YYM \to \PPi$ 
whose $0$-th cohomology morphism coincides with the canonical morphism ${}^{v}\!\pi: {}^{v}\!\YM \to \Pi$. 
Moreover,  this morphism is part of an exact triangle 
\[
{}^{v}\sfU: {}^{v}\!\YYM(-1) \xrightarrow{ \ {}^{v}\!\rrho \ } {}^{v}\!\YYM \xrightarrow{ \ {}^{v}\!\ppi \ } \PPi \to {}^{v}\!\YYM(-1)[1]
\]
where ${}^{v}\!\rrho$ denotes the right multiplication by ${}^{v}\!\varrho \in {}^{v}\!\YYM$. 
Taking the $*$-degree $1$ part of this exact triangle we obtain an exact triangle below which we call ${}^{v}\!\sfAR$ 
\[
{}^{v}\!\sfAR: \Pa \xrightarrow{ \ {}^{v}\!\rrho \ } {}^{v}\!\YYM_{1} \xrightarrow{ \ {}^{v}\!\ppi_{1} \ } \PPi_{1} \to A[1].  
\]
As in the case of the universal AR-sequence,
taking tensor product ${}^{v}\!\sfAR_{M} := {}^{v}\!\sfAR \lotimes_{\Pa} M$ with an indecomposable object 
$M \in \ind \Dbmod{\Pa}$, gives an exact triangle looking like an AR-triangle starting from $M$. 

\begin{theorem}[{Universal Auslander-Reiten triangle (Theorem \ref{universal Auslander-Reiten triangle})}]\label{202111201538}

\

Let $M\in \ind \Dbmod{\Pa}$. 
Assume that the weight $v \in \kk Q_{0}$ is regular.
Then 
the exact triangle ${}^{v}\!\sfAR_{M}$ is an Auslander-Reiten triangle 
starting from $M$. 
\[
{}^{v}\!\sfAR_{M}: M \xrightarrow{ \ {}^{v}\!\rrho_{M} \ } {}^{v}\!\YYM_{1}\lotimes_{\Pa} M  \xrightarrow{ \ {}^{v}\!\ppi_{1,M} \ } \nu_{1}^{-1}M \xrightarrow{} M[1].  
\]
\end{theorem}

Note that we first prove this theorem and then Theorem \ref{universal AR-sequence theorem} is obtained as an immediate consequence. 
Also note that this theorem is proved in Section \ref{section: the universal Auslander-Reiten triangle} before the introduction of the derived QHA and later in Section \ref{section: the derived quiver Heisenberg algebras} 
the exact triangle ${}^{v}\!\sfAR$ is obtained from ${}^{v}\!\YYM$ and $\PPi$.

The above theorem says that the $*$-degree $1$ parts of ${}^{v}\!\YYM$ and $\PPi$ provide 
minimal left and right $\rad$-approximations in $\Dbmod{\Pa}$. 

\subsubsection{}

We collect cohomological features of derived quiver Heisenberg algebras. 

Recall that the derived preprojective algebra $\PPi(Q)$ is $2$-Calabi-Yau algebra (of Gorenstein parameter $1$).
Compared to this we have the following result for the derived quiver Heisenberg algebras.
\begin{theorem}[{Theorem \ref{202112122233}}]
The derived quiver Heisenberg algbebra ${}^{v}\!\YYM$ is a $3$-Calabi-Yau 
algebra (of Gorenstein parameter $2$). 
\end{theorem}

There is a canonical DG-algebra homomorphism $\PPi(Q) \to \Pi(Q)$ which induces an isomorphism 
$\tuH^{0}(\PPi(Q)) \cong \Pi(Q)$. 
It is well known that 
if $Q$ is non-Dynkin, then  $\PPi(Q)$ is concentrated in $0$-th cohomological degree and the above map is  quasi-isomorphism 

Similarly, there is a canonical DG-algebra homomorphism ${}^{v}\!\YYM(Q) \to {}^{v}\!\YM(Q)$ 
which induces an isomorphism $\tuH^{0}({}^{v}\!\YYM(Q)) \cong {}^{v}\!\YM(Q)$. 
We have the following result for non-Dynkin quivers.

\begin{theorem}[{Proposition \ref{202207111600}, Theorem \ref{202111261228}}]
Assume that $Q$ be a non-Dynkin quiver. 
Let $v \in \kk^{\times}Q_{0}$ be a sincere weight. 
Then the derived QHA ${}^{v}\!\YYM(Q)$ is concentrated in the $0$-th cohomological grading and 
 the canonical morphism ${}^{v}\!\YYM(Q) \to {}^{v}\!\YM(Q)$ is a quasi-isomorphism. 
Consequently, the QHA ${}^{v}\!\YM(Q)$ is $3$-Calabi-Yau. 
\end{theorem}

In the case where $Q$ is Dynkin, it is also well-known that $\Pi(Q)$ is a finite dimensional Frobenius algebra.

As was  mentioned  in Corollary \ref{202111191825}, if $v\in \kk Q_{0}$ is regular, then $\dim {}^{v}\!\YM(Q) < \infty$. 
We prove the converse holds true. Namely we have 

\begin{theorem}[{Theorem \ref{202109131544}}]
Let $Q$ be a Dynkin quiver. 
Then ${}^{v}\!\YM(Q)$ is of finite dimension if and only if $v$ is regular. 
\end{theorem}

As was mentioned in Theorem \ref{202111191454}(2),
it was proved by Etingof-Latour-Rains \cite{ELR} 
that if $\chara \kk = 0$ and $v \in \kk Q_{0}$ is generic, then  the QHA ${}^{v}\!\YM(Q)$ is symmetric. 
In our subsequent work, 
we prove the algebra is symmetric  
for a regular weight $v \in \kk Q_{0}$ and an arbitrary base field $\kk$. 

\begin{theorem}[{\cite{Herschend-Minamoto: tilting theory of QHA}}]\label{2021113009471}
Let $Q$ be a Dynkin quiver. Assume that  the weight $v \in \kk Q_{0}$ is regular. 
Then ${}^{v}\!\YM(Q)$ is a symmetric algebra.
\end{theorem}

A key of the proof  of this theorem is the following description of the cohomology algebra $\tuH({}^{v}\!\YYM(Q))$ of 
the derived QHA ${}^{v}\!\YYM(Q)$. 

\begin{theorem}[{Theorem \ref{cohomology algebra structure theorem}}]\label{202209191949}
Let $Q$ be a Dynkin quiver and assume that the weight  $v \in \kk Q_{0}$ is regular.

We identify $\tuH^{0}({}^{v}\!\YYM(Q)) $ with $ {}^{v}\!\YM(Q)$. 
Then the cohomology algebra $\tuH({}^{v}\!\YYM(Q))$ has a central element $u \in \tuH^{-2}({}^{v}\!\YYM(Q)_{h})$ of 
cohomological degree $-2$ and of $*$-degree $h$ 
which induces an isomorphism 
\[
\tuH({}^{v}\!\YYM(Q)) \cong {}^{v}\!\YM(Q)[u]
\]
of algebras with cohomological degrees and $*$-gradings 
where the right hand side denotes the polynomial algebra in a single variable $u$. 
\end{theorem}

Another key ingredient for the proof is the algebra
${}^{v}\!B(Q)$, which we give an explanation of in Section \ref{202209192000}. 

\subsubsection{}

Assume that $\chara \kk  \neq 2$. In section \ref{202102071454} we explain that the derived quiver Heisenberg algebra ${}^{v}\!\YYM(Q)$ is isomorphic to the Ginzburg dg-algebra $\cG\left(\overline{Q}, W \right)$, of the double quiver $\overline{Q}$ with potential 
\[
W: = -\frac{1}{2} {}^{v}\!\varrho \rho.
\]
Moreover, the ordinary quiver Heisenberg algebra ${}^{v}\!\YM(Q)$ is isomorphic to the Jacobi algebra $\cP\left (\overline{Q}, W \right)$.

If $Q$ is Dynkin quiver and the weight $v \in \kk Q_{0}$ is regular, then Theorem~\ref{2021113009471} says that $\cP\left (\overline{Q}, W \right)$ is symmetric. In particular, $\left (\overline{Q}, W \right)$ is a selfinjective quiver with potential, in the sense of \cite{HI}. As explained in \cite[Section 6.3]{JKM}, this means that $\cP\left (\overline{Q}, W \right)$ is the endomorphism algebra of a $2\ZZ$-cluster tilting object in the associated Amiot cluster category $\mathcal{C}\left (\overline{Q}, W \right)$. Moreover, by \cite[Theorem 6.3.1]{JKM}, any algebraic $\Hom$-finite $\kk$-linear triangulated category, which is Krull-Schmidt and has a $2\ZZ$-cluster tilting object with endomorphism algebra isomorphic to the quiver Heisenberg algebra ${}^{v}\!\YM(Q)$ must be equivalent to $\mathcal{C}\left (\overline{Q}, W \right)$.

\subsubsection{}

Next we explain that 
for $n \geq 2$ 
the $*$-degree $n$ part of ${}^{v}\!\YYM$ and $\PPi$ provides 
minimal left and right $\rad^{n}$-approximations in $\Dbmod{\Pa}$. 
To state the results, we need to introduce a subset $N_{Q} \subset \NN$ depending on $Q$. 
We define subsets $N_{Q} \subset \NN$ to be 
\[
\begin{split}
&N_{Q} := 
\begin{cases}
\{n \in \NN \mid 0 \leq n \leq h -2\} & 
( Q \textup{ is a Dynkin quiver  with the Coxeter number } h )\\
\NN & ( Q \textup{ is a non-Dynkin quiver}). 
\end{cases}\\
\end{split}
\]

\begin{theorem}[{$\rad^{n}$-approximation theorem (Theorem \ref{right approximation theorem}, Theorem \ref{pre left approximation theorem})}]\label{202111201606}

\ 

Let $M \in \Dbmod{\Pa}$ and $n \in N_{Q}$. 
Assume that the weight $v \in \kk Q_{0}$ is regular. 
Then the following statements hold. 

\begin{enumerate}[(1)]

\item The morphism \[
{}^{v}\!\ppi_{n, M} := {}^{v}\!\ppi_{n} \lotimes_{\Pa} M: {}^{v}\!\YYM_{n} \lotimes_{\Pa} M \to \PPi_{n} \lotimes_{\Pa} M
\]
is a minimal right $\rad^{n}$-approximation of $\PPi_{n} \lotimes_{\Pa} M$.

\item 
The object ${}^{v}\!\YYM_{n} \lotimes_{\Pa} M$ provides a minimal left $\rad^{n}$-approximation of $M$. 

More precisely, 
there exists a minimal left $\rad^{n}$-approximation  $\beta^{(n)}:  M \to {}^{v}\!\YYM_{n} \lotimes_{\Pa} M$ of $M$.
\end{enumerate}
\end{theorem}

On the other hand, 
as a straightforward generalization of the well-known description of 
the middle term of an AR-triangle,  
we have a description of the object which provides a minimal left $\rad^{n}$-approximation 
(Theorem \ref{202008172145}). 
Combining this and the above $\rad^{n}$-approximation theorem, we obtain 

\begin{theorem}[{Theorem \ref{202111110809}}]\label{202111201641}
Let $M \in \ind \Dbmod{\Pa}$ and $\cC_{M} \subset \Dbmod{\Pa}$ be the full subcategory that consists of objects belonging to the 
same component as $M$ in the AR-quiver. 
Assume that the weight  $v \in \kk Q_{0}$ is regular 
and that the following condition does not  hold: $Q$ is wild and $M$ is a shift of a regular module. 
Then we have an isomorphism 
\[
\bigoplus_{n \in N_{Q}} {}^{v}\!\YYM_{n} \lotimes_{\Pa} M  \cong \bigoplus_{N \in \ind \cC_{M} } N^{\oplus \dim\Hom(M, N)}
\]
in $\sfD(\Pa)$. 
\end{theorem}

\subsubsection{}

In Theorem \ref{202111201606}(1) it is proved that the morphism ${}^{v}\!\ppi_{n, M}: {}^{v}\!\YYM_{n} \lotimes_{\Pa} M \to \PPi\lotimes_{\Pa} M$ 
is a minimal right $\rad^{n}$-approximation of $\PPi_{n} \lotimes_{\Pa} M$. 
Compared to this, (2) of the same theorem only establishes the existence of 
a minimal left $\rad^{n}$-approximation morphism $\beta^{(n)}: M \to {}^{v}\!\YYM_{n} \lotimes_{\Pa} M$. 
Since the multiplication ${}^{v}\!\varrho_{M} : M \to {}^{v}\!\YYM_{1} \lotimes_{\Pa} M$ by the weighted mesh relation 
is a minimal left $\rad$-approximation of $M$ by Theorem \ref{202111201538}, 
it is natural to consider the following problem.  

\begin{problem}\label{introduction:left approximation problem}
Let $n \in N_{Q}$ and $M \in \ind\Dbmod{\Pa}$. 
Assume that $v \in \kk Q_{0}$ is regular. 
Then,  
is the multiplication ${}^{v}\!\varrho_{M}^{n}: M \to {}^{v}\!\YYM_{n} \lotimes_{\Pa} M$ a minimal left $\rad^{n}$-approximation?
\end{problem}

It turns out the answer is no in general. 
As is shown below in Proposition \ref{202111201702}, even in the case $n =2$, 
we need to put more conditions on the weight $v$ than regularity.
Theorem \ref{202103021651} provides a linear combination $f = \sum_{i \in Q_{0}} f_{i} v_{i}$ of the coordinates of $v$ 
depending on $M$ 
such that  the condition $f \neq 0$ is sufficient for the answer to be yes in case  $n=2$. 
We are not able to show that it is a necessary condition. 
Also we have not succeeded to obtain  explicit equations of the coordinates of $v$ that guarantees the answer is yes for $n \geq 3$. 

We are able to give a partial solution to  the case $\chara \kk = 0$. 
Namely, we prove that the answer is yes for a generic weight $v \in \kk Q_{0}$.

\begin{theorem}[{Theorem \ref{202105241547}}]\label{202206221845}
Assume $\chara \kk = 0$. 
Let $Q$ be a  quiver, 
$M \in \ind \Dbmod{\Pa}$ and $n\in N_{Q}$.  

Then for a generic weight $v \in \kk Q_{0}$ 
the morphism 
\[
{}^{v}\!\varrho_{M}^{n}: M \to {}^{v}\! \YYM_{n} \lotimes_{\Pa} M
\]
is a minimal left $\rad^{n}$-approximation. 
\end{theorem}

In the case $Q$ is Dynkin, since $\Pa$ is of finite representation type, we immediately deduce the following corollary. 

\begin{corollary}[{Theorem \ref{202105241546}}]
Assume $\chara \kk = 0$. 
Let $Q$ be a Dynkin quiver with the Coxeter number $h$. 
Then for a generic weight $v \in \kk Q_{0}$
the following assertion holds: 

For $M \in \ind\Dbmod{\Pa}$ and $n = 1,2, \ldots, h -2$, 
the morphism 
\[
{}^{v}\!\varrho_{M}^{n}: M \to {}^{v}\! \YYM_{n} \lotimes_{\Pa} M
\]
is a minimal left $\rad^{n}$-approximation. 
\end{corollary}

In Section \ref{section: check the AN case} 
we study the $A_{N}$-quiver case and prove that  
for a generic weight Problem \ref{introduction:left approximation problem} has the answer yes provided that 
the base field $\kk$ has a primitive $(N+1)$-th root of unity.

The strategy to prove these results is the following. 
First we establish a criterion on the weight $v\in \kk Q_{0}$ for which Problem \ref{introduction:left approximation problem} has the answer yes. Next we show this criterion is a kind of open condition on $v\in \kk Q_{0}$.
Finally we check there is a weight $v$ that satisfies the criterion  in each case. 

Let $\Phi$ be the Coxeter matrix of $Q$ and $\Psi:= \Phi^{-t}$. 
Roughly speaking, the criterion is that a weight $v$ \emph{behaves like} an eigenvector of $\Psi$ with a good eigenvalue $\lambda$.  
We verify that this condition guarantees that Problem \ref{introduction:left approximation problem} has the answer yes, 
by showing that the derived QHA provides a special type of a left ladder.  
In the final  Section \ref{section: the universal ladder}, 
we prove that if a weight $v$ \emph{is} an eigenvector of $\Psi$ with a good eigenvalue $\lambda$, 
then ${}^{v}\!\YYM(Q)$ provides a diagram of complexes of $\Pa$-bimodules that can be called a \emph{universal left ladder}.

\subsection{Subsequent work} 

\subsubsection{The Hilbert series of preprojective algebras of Dynkin type}

The Hilbert series of preprojective algebras of Dynkin type 
was first computed in \cite{MOV} in characteristic zero, 
using module categories over representations of quantum $\mathrm{SL}(2)$. 
In \cite{Minamoto Hilbert}, the second author establishes a formula in arbitrary characteristic 
using the Hilbert series of differential graded algebras with Adams grading. 
In that proof, the Hilbert series of QHAs of Dynkin type is computed. Then the argument is completed using the exact sequence relating 
the preprojective algebra and the QHA, as given in Proposition \ref{202006221820} and Remark \ref{202603182140}.

\subsubsection{The algebra ${}^{v}\!\Qa(Q)$}\label{202209192000}

In subsequent work \cite{Herschend-Minamoto: tilting theory of QHA},  
we introduce   a finite dimensional algebra ${}^{v}\! \Qa(Q)$ 
for a quiver $Q$, 
which behaves like a $2$-dimensional version of the path algebra $\Pa = \kk Q$ of $Q$. 
The algebra ${}^{v}\!\Qa(Q)$ is to the QHA ${}^{v}\!\YM(Q)$ what $\kk Q$ is to the preprojective algebra $\Pi(Q)$.

The definition of ${}^{v}\!\Qa(Q)$ is  the following upper triangular matrix algebra: 
\[
{}^{v}\!B(Q) 
:= \begin{pmatrix} {}^{v}\!\YM(Q)_{0} & {}^{v}\!\YM(Q)_{1} \\ 0 & {}^{v}\!\YM(Q)_{0}  \end{pmatrix}
= \begin{pmatrix} \kk Q & {}^{v}\!\YM(Q)_{1} \\ 0 & \kk Q \end{pmatrix}. 
\]
We prove that ${}^{v}\!B(Q)$ is $2$-hereditary in the sense of \cite{HIO}. 
Moreover ${}^{v}\!B(Q)$ is $2$-representation finite (resp. infinite) if and only if $Q$ is Dynkin (resp. non-Dynkin). 
This result is an analogue of Gabriel's theorem that asserts that 
the path algebra $A = \kk Q$ is representation finite (resp. infinite) 
if and only if $Q$ is Dynkin (resp. non-Dynkin). 
Moreover we prove that in many other aspects,  
the algebra ${}^{v}\!B(Q)$ behaves like the path algebra $\Pa = \kk Q$. 

A relationship betweem the QHA ${}^{v}\!\YM(Q)$ and the algebra ${}^{v}\!\Qa(Q)$ is the following: 
the $3$-preprojective algebra $\Pi_{3}({}^{v}\!\Qa(Q))$ is 
isomorphic to $2$-nd quasi-Veronese algebra ${}^{v}\!\YM(Q)^{[2]}$ with respect to $*$-grading. 
Thus it can be said that QHA ${}^{v}\!\YM(Q)$ originates from ($1$-dimensional) AR-theory of $\Pa\mod$ and lives in $2$-dimensional AR-theory of ${}^{v}\!\Qa(Q)\mod$.

\subsubsection{Relationships to compound Du Val singularities}

Assume that $\chara \kk =0$. 
Let $Q$ be an extended Dynkin quiver. 
As we mentioned before, 
$\Pi(Q)$ is Morita equivalent to the skew group algebra $S *G$ where 
$G < \SL(2;\kk)$ is the corresponding finite group and $S=\kk[x, y]$. 
Moreover the algebra $e_{0} \Pi(Q) e_{0}$  is isomorphic to the fixed subalgebra $S^{G}$, which is the coordinate algebra of 
the Kleinian singularity (= Du Val singularity) of corresponding Dynkin type. 
Thus the deformation family $\Pi(Q)_{\bullet}$ by Crawley-Boevey-Holland gives non-commutative deformation of the Kleinian singularity. 

Recall that a compound Du Val (cDV) singularity is, 
by definition, a three dimensional singularity whose hyperplane section is Du Val singularity. 
Since the algebra $\Pi(Q)_{P}$ is obtained as the pull-back of $\Pi(Q)_{\bullet}$ to a one-dimensional space $\kk$ along a polynomial map $P: \kk \to \kk^{r}$, it is related to a cDV-singularity.

In their study of (commutative and non-commutative) algebraic geometry of cDV-singularity, Donovan-Wemyss introduced
the notion of \emph{contraction algebras} \cite{Donovan-Wemyss}. 
Now cDV-singularity and contraction algebras have been  studied extensively by many researchers 
and have grown as a fascinating subject. 
In subsequent work, we discuss a relationship between the QHA and cDV-singularity. 
We prove that QHA of Dynkin type are contraction algebras and as a consequence we see that 
the dimension formula by Etingof-Rains (Theorem \ref{202111191454}(1)) 
is an analogue of Toda's dimension formula of contraction algebras \cite{Toda}. 
Moreover, we see that one of our main result Corollary \ref{202111191825}  gives a categorification of these formulas.

\subsection{Example: $A_{3}$-quiver}\label{202103221959}

As a guide to understand the QHA for Dynkin quivers, consider the following example.
Let $Q$ be a directed $A_{3}$-quiver. 
\[
Q: \begin{xymatrix}{ 
1 \ar[rr]^{\alpha} && 2\ar[rr]^{\beta} && 3}\end{xymatrix}, 
\ \ 
\overline{Q}: 
\begin{xymatrix}{ 
1 \ar@<2pt>[rr]^{\alpha} && 2\ar@<2pt>[rr]^{\beta} \ar@<2pt>[ll]^{\alpha^{*}} && 3 \ar@<2pt>[ll]^{\beta^{*}} }\end{xymatrix}. 
\]
The mesh relations are 
\[
\rho_{1} = \alpha\alpha^{*}, \ \rho_{2} = -\alpha^{*} \alpha + \beta \beta^{*}, \ \rho_{3} = -\beta^{*}\beta.
\]

Let $v = (v_{1}, v_{2}, v_{3} )^{t} \in \kk^{\times} Q_{0}$ be a sincere weight. 
The quiver Heisenberg relations are 
\[
\begin{split}
&{}^{v}\!\eta_{\alpha} = [\alpha, {}^{v}\!\varrho]
 =v_{2}^{-1} \alpha \rho_{2} - v^{-1}_{1} \rho_{1} \alpha 
 = - (v_{2}^{-1} + v_{1}^{-1})\alpha\alpha^{*} \alpha 
 + v_{2}^{-1}\alpha\beta\beta^{*},\\ 
& {}^{v}\!\eta_{\beta} = -(v^{-1}_{2} + v_{3}^{-1}) \beta \beta^{*} \beta +  v_{2}^{-1}\alpha^{*} \alpha \beta, \\
&{}^{v}\!\eta_{\alpha^{*}} =  (v_{1}^{-1} + v_{2}^{-1}) \alpha^{*} \alpha \alpha^{*} - v_{2}^{-1} \beta\beta^{*} \alpha^{*}, \\
& {}^{v}\!\eta_{\beta^{*}} = (v_{2}^{-1} +v_{3}^{-1})  \beta^{*}\beta\beta^{*} - v_{2}^{-1} \beta^{*} \alpha^{*} \alpha. 
\end{split}
\]
The weight $v \in \kk Q_{0}$ is regular 
if and only if it satisfies the following
\[
v_{1} \neq 0, v_{2} \neq 0, v_{3} \neq 0, v_{1} +v_{2} \neq 0, v_{2} + v_{3} \neq 0, v_{1} +v_{2} + v_{3} \neq 0. 
\]

If $v \in \kk Q_{0}$ is regular, the composition series of indecomposable projective ${}^{v}\!\YM$-modules are given by,  
\[
{}^{v}\!\YM e_{1}= 
\begin{smallmatrix}
1 & & \\
  &2 & \\
  1 && 3 \\
  &2& \\
  1&&
\end{smallmatrix}, \ 
{}^{v}\!\YM e_{2}= 
\begin{smallmatrix}
 & 2 & \\
 1 &  & 3 \\
   &  \hspace{-5pt} \ \ 2^{\oplus 2} \hspace{-5pt} &  \\
 1 & & 3 \\
  & 2&
\end{smallmatrix}, \ 
{}^{v}\!\YM e_{3}= 
\begin{smallmatrix}
 & & 3 \\
  &2 & \\
  1 && 3 \\
  &2& \\
  && 3
\end{smallmatrix}.
\]
Theorem \ref{description of QHA as kQ-modules}(1) says that 
these composition series of  the indecomposable projective modules ${}^{v}\!\YM e_{1}, {}^{v}\!\YM e_{2}, {}^{v}\!\YM e_{3}$ can be read off from the Auslander-Reiten quiver of $\Pa \mod$. 

We explain this by taking  ${}^{v}\!\YM e_{2}$ as an example.  
The slant lines in the composition series express the layers induced by the $*$-grading. 
We can see these layers correspond to the columns  in the Auslander-Reiten quiver. 
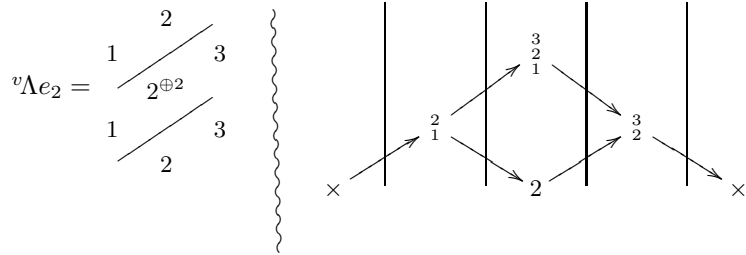
\[
\begin{xymatrix}@R=10pt{\\ \\ 
{}^{v}\!\YM e_{2}= 
}\end{xymatrix}
{\small
\begin{xymatrix}@R=2pt@C=5pt{\\
&  2 & \ar@{-}[ddll] \\
1 & &  3 \\ 
& 2^{\oplus 2} & \ar@{-}[ddll] \\
1 & & 3 \\ 
 & 2& 
}\end{xymatrix}} 
\ \ \ \
\small{\begin{xymatrix}@R=15pt@C=5pt{
\ar@{~}[ddddd] \\  \\  \\  \\ \\ \\  
}\end{xymatrix}} \ \ \ \
\begin{xymatrix}@C=10pt@R=10pt{
& \ar@{-}[ddd] && \ar@{-}[ddd] & & \ar@{-}[ddd] &&
 \ar@{-}[ddd]
&  \\
&&&& {\begin{smallmatrix} 3 \\ 2 \\ 1 \end{smallmatrix}} \ar[drr] &&&& \\
&&  {\begin{smallmatrix} 2 \\ 1 \end{smallmatrix}} \ar[drr]\ar[urr] &&&& {\begin{smallmatrix}3 \\ 2 \end{smallmatrix}} \ar[drr] && 
\\
\times \ar[urr] &&&& 2 \ar[urr] &&&& \times 
}\end{xymatrix}
\]

The $*$-degree $1$ part 
$\YM_{1}e_{2}=\begin{smallmatrix}
 &  & \\
  &  &3  \\
   &  \hspace{-5pt} \ \ 2^{\oplus 2} \hspace{-5pt} &  \\
 1& & 
\end{smallmatrix}$ 
is isomorphic  as a $\kk Q$-module, 
to 
the direct sum $M$ of $ P_{3} = \begin{smallmatrix} 1\\2\\3 \end{smallmatrix}$ and $S_{2} = 2$, 
which constitute the middle column. 
Observe that $P_{3} \oplus S_{2}$ is the middle term of the Auslander-Reiten sequence starting from $P_{2}$. 
\[
0 \to P_{2} \xrightarrow{f}  P_{3} \oplus S_{2} \to \tau_{1}^{-1} P_{2} \to 0. 
\]
Theorem \ref{universal AR-sequence theorem}  says that the minimal left almost split morphism $f: P_{2} \to P_{3} \oplus S_{2}$ is given by 
the multiplication by the mesh relation ${}^{v}\!\varrho$. Namely, we have the following commutative diagram
\[
\begin{xymatrix}@C=60pt{
P_{2} \ar@{=}[d] \ar[r]^{f}  & P_{3} \oplus S_{2} \ar[d]^{\cong} \\ 
{}^{v}\!\YM_{0} e_{2} \ar[r]_{{}^{v}\!\varrho} & {}^{v}\!\YM_{1} e_{2}.
}\end{xymatrix} 
\] 

Next we look at the $*$-degree $2$ part. 
It follows from Theorem \ref{202111201606} that $\YM_{2} e_{2}$ provides a minimal left $\rad^{2}$-approximation of $P_{2}= {}^{v}\!\YM_{0} e_{2}$. 
On the other hand, 
a morphism  $g: P_{2} \to I_{2}$  which sends the top of $P_{2}$ isomorphically to the socle of $I_{2}$ is a minimal left $\rad^{2}$-approximation 
of $P_{2}$ (see Theorem \ref{kQ approximations theorem add}(2)). 
Therefore we conclude that ${}^{v}\!\YM_{2} e_{2}$ is isomorphic to $I_{2}$ as $\Pa$-modules.

However even if the weight $v$ is regular,  
the multiplication by the square ${}^{v}\!\varrho^{2}$  does not always give a minimal left $\rad^{2}$-approximation. 
Indeed, 
in this case we can determine the locus of $v$ where 
the morphism 
${}^{v}\!\varrho^{2}: P_{2} \to {}^{v}\!\YM_{2} e_{2}$ is a minimal left $\rad^{2}$-approximation. 

\begin{proposition}\label{202111201702}
Assume that the weight $v\in \kk Q_{0}$ is regular.
Then,  
the morphism 
${}^{v}\!\varrho^{2}: {}^{v}\!\YM_{0} e_{2}  \to {}^{v}\!\YM_{2} e_{2}$ is a minimal left $\rad^{2}$-approximation 
if and only if 
$v_{1} + 2v_{2} + v_{3} \neq 0$. 
\[
\begin{xymatrix}@C=60pt{
P_{2} \ar@{=}[d] \ar[r]^{g}  & I_{2} \ar[d]^{\cong} \\ 
{}^{v}\!\YM_{0} e_{2} \ar[r]_{{}^{v}\!\varrho^{2}} & {}^{v}\!\YM_{2} e_{2}.
}\end{xymatrix} 
\] 
\end{proposition} 

\begin{proof}
The ``if'' part follows Theorem \ref{202103021651}. 
Conversely, 
by direct calculation below, we  show that if $v_{1} + 2v_{2} + v_{3} = 0$, then $({}^{v}\!\varrho_{2})^{2}=0$  
(see also  Example \ref{202103051509}). 
Since ${}^{v}\!\varrho$ is central in ${}^{v}\!\YM$, it follows that the multiplication map ${}^{v}\!\varrho^{2} : {}^{v}\!\YM_{0} e_{0} \to  {}^{v}\!\YM_{2} e_{2}$ 
is zero. 

The quiver Heisenberg relations turn to 
\[
\begin{split}
&v_{1} \alpha\beta\beta^{*} = (v_{1} +v_{2}) \alpha \alpha^{*} \alpha, \ v_{3} \alpha^{*} \alpha\beta =(v_{2} + v_{3}) \beta\beta^{*} \beta, \\ 
&v_{1} \beta\beta^{*} \alpha^{*} = (v_{1} +v_{2}) \alpha^{*}\alpha \alpha^{*},\  v_{3} \beta^{*} \alpha^{*} \alpha = (v_{2} +v_{3}) \beta^{*} \beta \beta^{*}. 
\end{split}
\]
Using these equations, we have 
\[
\begin{split}
&\beta\beta^{*}\alpha^{*} \alpha = \frac{v_{2} + v_{3} }{v_{3}} \beta\beta^{*} \beta \beta^{*}, 
\ \alpha^{*} \alpha\beta\beta^{*} = \frac{v_{2} + v_{3} }{v_{3}} \beta\beta^{*} \beta\beta^{*}, \\
&\alpha^{*} \alpha\alpha^{*} \alpha = \frac{v_{1}}{v_{1} +v_{2}} \alpha^{*} \alpha\beta\beta^{*} = \frac{v_{1}( v_{2} + v_{3}) }{v_{3}(v_{1} +v_{2})} \beta\beta^{*} \beta\beta^{*}. 
\end{split}
\]
Consequently, 
\[
\begin{split} 
{}^{v}\!\varrho_{2}^{2} & = v_{2}^{-2} (\alpha^{*} \alpha\alpha^{*} \alpha -\alpha^{*} \alpha\beta\beta^{*} -\beta\beta^{*}\alpha^{*} \alpha + \beta\beta^{*} \beta \beta^{*})\\ 
&= v_{2}^{-2} \left( \frac{v_{1}( v_{2} + v_{3}) }{v_{3}(v_{1} +v_{2})} - 2\frac{v_{2} + v_{3} }{v_{3}}  +1 \right) \beta\beta^{*} \beta \beta^{*}\\
& = -\frac{v_{1} + 2v_{2} + v_{3}}{v_{2}v_{3}(v_{1} + v_{2})}\beta\beta^{*} \beta \beta^{*} = 0.
\end{split}
\]
\end{proof}

\subsection{$2$-Kronecker quiver }

As a guide to understand the QHA for non-Dynkin quivers, more specifically tame quivers, consider the following example.
Let $Q: 1 \rightrightarrows 2$ be the $2$-Kronecker quiver with arrows $\alpha, \beta$. 
\[
\begin{xymatrix}{
Q: 1 \ar@<3pt>[r]^{\alpha} \ar@<-3pt>[r]_{\beta} & 2
}\end{xymatrix}, \ \ \ \ \ \ \ \
\overline{Q} : 
\begin{xymatrix}@C=40pt{
1 \ar@/^10pt/@<8pt>[r]^{\alpha} \ar@/^10pt/@<2pt>[r]_{\beta} & 
2 \ar@/^10pt/@<8pt>[l]^{\alpha^{*}} \ar@/^10pt/@<2pt>[l]_{\beta^{*}}.
}\end{xymatrix}
\]

Then the mesh relations are 
\[
\rho_{1} = \alpha \alpha^{*} + \beta \beta^{*}, \ \rho_{2} =- \alpha^{*} \alpha - \beta^{*} \beta. 
\]
The weight  $v=(v_{1}, v_{2})^{t} \in \kk Q_{0}$ is regular 
if and only if
\[
n v_{1} +(n -1)v_{2} \neq 0 \ (\forall n \in \ZZ), \ \ v_{1} +v_{2} \neq 0. 
\]
The quiver Heisenberg relations are 
\[
\begin{split}
{}^{v}\!\eta_{\alpha} & = v_{2}^{-1}\alpha \rho_{2} -  v_{1}^{-1}\rho_{1} \alpha 
= - (v^{-1}_{2}\alpha \beta^{*} \beta + (v_{1}^{-1} +v_{2}^{-1}) \alpha \alpha^{*} \alpha + v_{1}^{-1} \beta \beta^{*} \alpha),\\
{}^{v}\!\eta_{\beta} & =  v_{2}^{-1} \beta \rho_{2} - v_{1}^{-1} \rho_{1} \beta 
= -(v_{2}^{-1}  \beta \alpha^{*} \alpha + (v_{1}^{-1} + v_{2}^{-1} )\beta\beta^{*}\beta +v_{1}^{-1}  \alpha \alpha^{*} \beta),\\
{}^{v}\!\eta_{\alpha^{*}} & =v_{1}^{-1}  \alpha^{*} \rho_{1} - v_{2}^{-1}  \rho_{2} \alpha^{*}
 = v_{1}^{-1}  \alpha^{*} \beta \beta^{*} + (v_{1}^{-1}  +v_{2}^{-1} )\alpha^{*} \alpha \alpha^{*} + v_{2}^{-1}  \beta^{*} \beta \alpha^{*},\\
{}^{v}\!\eta_{\beta^{*}} & = v_{1}^{-1} \beta^{*} \rho_{1} -v_{2}^{-1}  \rho_{2} \beta^{*} 
=  v_{1}^{-1} \beta^{*} \alpha  \alpha^{*} + (v_{1}^{-1} + v_{2}^{-1}) \beta^{*}\beta \beta^{*} + v_{2}^{-1}  \alpha^{*} \alpha \beta^{*}. \\
\end{split}
\]

Assume that the weight $v$ is regular. Then, by Theorem \ref{description of QHA as kQ-modules}, for $m \geq 0$ we have isomorphisms of $\Pa$-modules 
\begin{equation}\label{202111302005}
\begin{split}
\YM_{2m} e_{1 } & \cong  (\tau^{-m }_{1} P_{1})^{ \oplus 2m +1}, \ \  
\YM_{2m +1 } e_{1 }  \cong  (\tau^{-m }_{1} P_{2})^{ \oplus 2m +2}, \\
\YM_{2m} e_{2} &\cong  (\tau^{-m }_{1} P_{2})^{ \oplus 2m +1}, \ \ 
\YM_{2m +1 } e_{2 } \cong  (\tau^{-m-1 }_{1} P_{1})^{ \oplus 2m +2}.
\end{split}
\end{equation}

\subsubsection{Regular modules}

Assume for simplicity $\kk$ is algebraically closed. 
Then the regular components $\{ \cT_{\lambda} \}_{ \lambda \in \PP^{1}}$ are classified by points of $\PP^{1} = \kk \sqcup \{ \infty\}$. 
For $\lambda \in \kk \subset \PP^{1}$, the indecomposable modules of $\cT_{\lambda}$ are of the form 
$
R^{(m)}_{\lambda}: k[x]/(x^{m+1}) \rightrightarrows k[x]/(x^{m+1})
$ where the actions of $\alpha$ and $\beta$ are the multiplications by $x+\lambda$ and $1$ for some $m \geq 0$. 
In  the case $\lambda = \infty$, 
then the indecomposable modules of $\cT_{\lambda}$ are of the form 
$
R^{(m)}_{\lambda}: k[x]/(x^{m+1}) \rightrightarrows k[x]/(x^{m+1})
$ where the actions of $\alpha$ and $\beta$ are the multiplications by $1$ and $x$ for some $m \geq 0$. 

Applying Theorem \ref{202111201606}  to  $M = R^{(m)}_{\lambda}$ where  $n \geq 0$ and $\lambda \in \PP^{1}$, 
we obtain the following isomorphism of $\Pa$-modules for $n \geq 0$ 
\[
\YM_{n} \otimes_{\Pa} R^{(m)}_{\lambda} \cong
\begin{cases}
 \bigoplus_{k =  0}^{n} R^{(m -n + 2k)}_{\lambda} & ( n \leq m), \\
\bigoplus_{k =0}^{m} R^{(n -m+ 2k)}_{\lambda} & (n > m). 
\end{cases} 
\]

\subsubsection{The case $v= (1,1)^{t}$}

Recall that
for a $\ZZ$-graded algebra $R =\bigoplus_{n \in \ZZ} R_{n}$, its second quasi-Veronese algebra $R^{[2]}$ is defined to be 
\[
R^{[2]} := \bigoplus_{ n \in \ZZ} \begin{pmatrix} R_{ 2n} & R_{2n +1} \\ R_{2n -1} & R_{2n} \end{pmatrix}
\]
with the matrix multiplication. 

Let $T:= \kk \agl{x, y}$ be the non-commutative polynomial algebra in $x,y$  with the grading $\deg x : = 1, \deg y:= 1$. 
Then we have an isomorphism $T^{[2]} \to \kk \overline{Q}$ of graded algebras given by 
\[
\left( \begin{smallmatrix} 1 _{T} & 0 \\ 0 & 0 \end{smallmatrix} \right)\mapsto e_{1}, 
\left( \begin{smallmatrix} 0 & 0 \\ 0 & 1 _{T}  \end{smallmatrix} \right) \mapsto e_{2}, 
\left( \begin{smallmatrix} 0 & x \\ 0 & 0 \end{smallmatrix} \right) \mapsto \alpha, 
\left( \begin{smallmatrix} 0 & y \\ 0 & 0 \end{smallmatrix} \right) \mapsto \beta, 
\left( \begin{smallmatrix}0 & 0 \\ x & 0 \end{smallmatrix} \right) \mapsto -\beta^{*}, 
\left( \begin{smallmatrix}0 & 0 \\ y & 0 \end{smallmatrix} \right) \mapsto \alpha^{*}.   
\] 
Let $S := \kk [x,y] = T/([x,y])$ be the commutative polynomial algebra in $x, y$.  
Then the above isomorphism descents to isomorphisms $S^{[2]} \to \Pi(Q)$ of graded algebras.
Using this isomorphism,  we obtain the following description of the preprojective modules. 
\[
\tau^{-n}_{1} P_{1} \cong \Pi(Q)_{n} e_{1} \cong \begin{pmatrix} S_{ 2n} \\ S_{ 2n -1} \end{pmatrix}, 
\tau^{-n}_{1} P_{2} \cong \Pi(Q)_{n} e_{2} \cong \begin{pmatrix} S_{ 2n +1} \\ S_{ 2n} \end{pmatrix}.
\]

Assume that $v= (1,1)^{t}$. 
 Let  $H := T/([x,[x,y]], [y,[x,y]])$ be the Heisenberg algebra in $x, y$. 
Then the above isomorphism descents to isomorphisms  $H^{[2]} \to {}^{v}\!\YM(Q)$ 
of graded algebras. 
Using this isomorphism,  we obtain the following isomorphisms of $\Pa$-modules. 
\[
\YM_{n} e_{1} \cong \begin{pmatrix} H_{ 2n} \\ H_{ 2n -1} \end{pmatrix}, 
\YM_{n} e_{2} \cong \begin{pmatrix} H_{ 2n +1} \\ H_{ 2n} \end{pmatrix}.
\]

Using \eqref{202111302005} we obtain the following isomorphism of $\Pa$-modules.  
\[
\begin{split}
\begin{pmatrix} H_{ 2n} \\ H_{ 2n -1} \end{pmatrix} \cong \YM_{n} e_{1}   \cong \begin{pmatrix} S_{n} \\ S_{n -1} \end{pmatrix}^{\oplus n +1},
\begin{pmatrix} H_{ 2n +1 } \\ H_{ 2n} \end{pmatrix} \cong \YM_{n} e_{2}  \cong \begin{pmatrix} S_{ n +1 } \\ S_{n } \end{pmatrix}^{\oplus n +1}.
\end{split}
\]

\subsection{$\hat{A}_{N}$-quiver}

Let $N \geq 2$ and  $Q$ be an $\hat{A}_{N}$-quiver of the following orientation: 
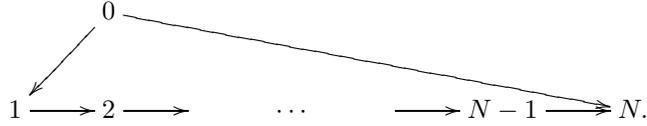
\[
\begin{xymatrix}{ 
&0 \ar[dl] \ar[drrrrr] \\
1 \ar[r] & 2 \ar[r] & &  \cdots & \ar[r] & N -1 \ar[r] & N. 
}\end{xymatrix}\]
Using this labeling, 
we identify the set $Q_{0}$ of vertices with a subset of $\ZZ$. 

Using 
Theorem \ref{202111201606}(2)  we can obtain an indecomposable decomposition of ${}^{v}\!\YM(Q)_{n} e_{i}$ for $n \geq 0$ and $i\in Q_{0}$ 
in the following way. 
Applying the knitting algorithm to $\ZZ Q$ (or the universal covering of it) and Theorem \ref{202008172145}, 
we can compute an indecomposable decomposition of $\rad^{n}$-approximation object of $\Pa e_{i}$ 
which is isomorphic to ${}^{v}\!\YM(Q)_{n} e_{i}$. 
To write down the results, 
 we define an $A$-module $M_{(a,b)}$ for $(a,b) \in \NN^{2}$, to be 
\[
M_{(a,b)} := \tau_{1}^{-q-b}P_{r}
\]
where $q\in \ZZ, r\in \{ 0,1, \ldots, N\}$ are determined from the equation 
\[
a-2b= q(N+1)+r.
\] 
Then, the isomorphism of $\Pa$-modules given in Theorem \ref{description of QHA as kQ-modules} is the following 
\[
\YM(Q)_{n}e_{i} \cong \bigoplus_{b= 0}^{n} M_{(i +n, b)}. 
\]

\subsection{Question: a representation theoretic understanding of the quiver Heisenberg relation }

To end the introduction we propose one question. 

Until  now we have seen that the quiver Heisenberg algebra ${}^{v}\!\YM(Q)$ can be looked as a three dimensional version 
of the preprojetive algebra $\Pi(Q)$.  

The reason why the defining relation of the preprojective algebra $\Pi(Q)$ is 
the mesh relation $\rho$  may be understood from the fact that it comes from Auslander-Reiten quiver of the preprojective component $\cP(Q)$. 
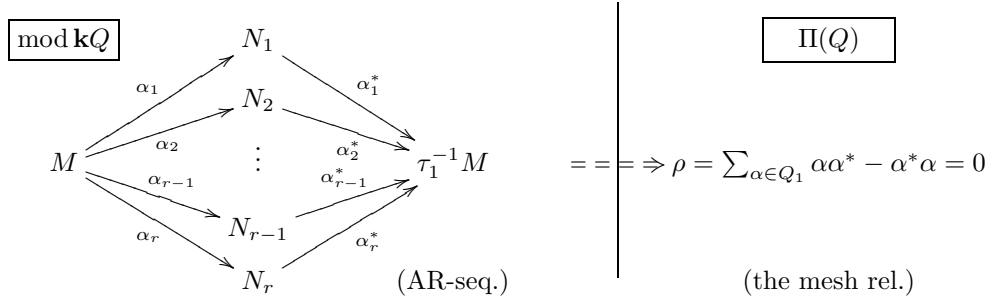
\[
\begin{xymatrix}@R=1.5mm@C=5mm{
&&&&& &\ar@{-}[ddddd] &\\
\textup{\fbox{$\mod \kk Q$}}&& N_{1}\ar[ddrr]^{\alpha_{1}^{*}} && & &&  \textup{ \fbox{ \ \ $\Pi(Q)$ \ \ } } \\
&& N_{2}\ar[drr]_{\alpha_{2}^{*}} &&&&&\\
M  \ar[uurr]^{\alpha_{1} } \ar[urr]_{\alpha_{2}} \ar[drr]^{\alpha_{r -1}} 
\ar[ddrr]_{\alpha_{r}} & & \vdots && \tau^{-1}_{1} M &  \ar@{==>}[rr]&& \rho = \sum_{\alpha \in Q_{1}} \alpha \alpha^{*} - \alpha^{*} \alpha = 0 \\ 
&& N_{r-1}\ar[urr]^{\alpha_{r-1}^{*}} &&
&&&& &&&\\
&& N_{r}\ar[uurr]_{\alpha_{r}^{*}}  && \textup{(AR-seq.)}  &&& \textup{(the mesh rel.)}\\
}\end{xymatrix}
\]
It is reasonable to ask if there is a categorical understanding of the quiver Heisenberg relations ${}^{v}\!\eta_{\alpha}, \ {}^{v}\!\eta_{\alpha^{*}}$?
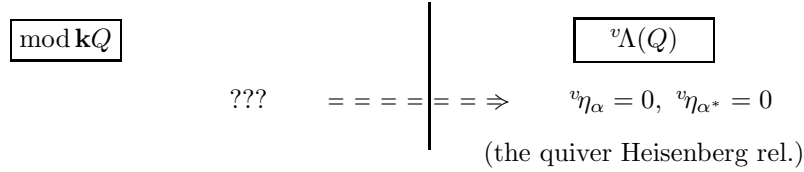
\[
\begin{xymatrix}@R=1.5mm@C=5mm{
&&& && \ar@{-}[ddd]& \\
\textup{\fbox{$\mod \kk Q$}}&& & & & &  \textup{ \fbox{ \ \ ${}^{v}\!\YM(Q)$ \ \ } } \\
&& ??? & \ar@{==>}[rrr] &&&  \ \ \ \ \ \ {}^{v}\!\eta_{\alpha}=0, \ {}^{v}\!\eta_{\alpha^{*}} = 0 \\
&&&& && \textup{(the quiver Heisenberg  rel.)}\\
}\end{xymatrix}
\]

\subsection{Organization of the paper}\label{subsection: organization}

Dependence of sections (except appendixes)  is given in the diagram below:
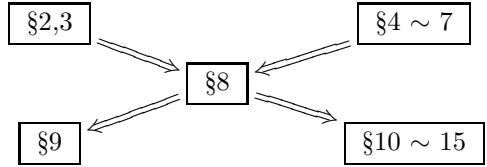
\[
\begin{xymatrix}@C=30pt@R=1pt{
\fbox{ \S 2,3 } \ar@{=>}[dr]  && \fbox{ \S 4 $\sim$ 7 } \ar@{=>}[dl] \\
&\fbox{ \S 8 } \ar@{=>}[dl]\ar@{=>}[dr] & \\
\fbox{ \S 9 } && \fbox{ \S 10  $\sim $ 15 } 
}\end{xymatrix} 
\]

Section \ref{section: rad^{n} approximation} collects basics of approximations by the functors $\rad^{n}$.   
Section \ref{section: kQ approximations} studies $\rad^{n}$-approximation theory of the path algebra $\Pa$ of a quiver $Q$. 

Section \ref{section: dg-modules} prepares notations of dg-algebras and dg-modules, 
recalls the octahedral axiom and provides a technical lemma. 
Section \ref{section: the universal Auslander-Reiten triangle} establishes universal Auslander-Reiten triangles for $\Pa$. 
A key point in the proof is the trace formula (Theorem \ref{trace formula}) which says that the pairing of Serre duality computes the weighted trace of an endomorphism of an object $M$ of the derived category.  
Section \ref{section: the derived quiver Heisenberg algebras} introduces derived quiver Heisenberg algebras ${}^{v}\!\YYM(Q)$ and verify their basic properties. 
It turns out that the morphisms ${}^{v}\!\ttheta, {}^{v}\!\rrho$ and ${}^{v}\!\ppi$ which appear in the universal AR triangles are obtained from morphisms involving derived preprojective algebras and derived  quiver Heisenberg algebras. 
Section \ref{202207111353} deals with QHA of non-Dynkin type for arbitrary weights $v \in \kk Q_{0}$. 

Section \ref{section: right rad-n approximation} 
proves relationships between the derived quiver Heisenberg algebras and $\rad^{n}$-approximations in $\Dbmod{\Pa}$. 
Theorem \ref{pre left approximation theorem} proves that if the weight $v\in \kk Q_{0}$ is regular, then ${}^{v}\!\YYM(Q)_{n} \lotimes_{\Pa}M$ 
provides a minimal left $\rad^{n}$-approximation of $M \in \Dbmod{\Pa}$. 
As a consequence we obtain a description of ${}^{v}\!\YYM(Q) \lotimes_{\Pa} M$ for $M\in \Dbmod{\Pa}$ as an object of $\Dbmod{\Pa}$. 

Section \ref{section: QHA Dynkin case} studies the (derived)  quiver Heisenberg algebra of Dynkin type 
and the cohomology algebra of the derived quiver Heisenberg algebras. 

The main result of 
Section \ref{section: the middle term of the middle terms}  is  Theorem \ref{2020071920551} (Theorem \ref{202111181854}). 
As a corollary, we obtain a condition on $v$ and $M$ for which the multiplication  by  the square  ${}^{v}\!\varrho^{2}$ of the weighted mesh relation 
gives a minimal left $\rad^{2}$-approximation of $M$. 
 
In Section \ref{section: minimal left rad^n-approximations} to \ref{section: check the AN case} we investigate  
when multiplication by the $n$-th power  ${}^{v}\!\varrho_{M}^{n}$ of the mesh relation gives a minimal left $\rad^{n}$-approximation of an object $M \in \Dbmod{\Pa}$.  
In Section \ref{section: minimal left rad^n-approximations} we establish a sufficient condition on the weight $v\in \kk^{\times} Q_{0}$ for this to be true. In Section \ref{section: deformation} we prove that the locus of $v\in \kk^{\times } Q_{0}$ such that the multiplication by ${}^{v}\!\varrho_{M}^n$ is a minimal left $\rad^{n}$-approximation contains a (possibly empty) Zariski open set. In Section \ref{section: check the case char k is 0} we check the condition of Section \ref{section: minimal left rad^n-approximations} in the case $\chara \kk =0$. 
Consequently, in this case we conclude  that for a generic weight $v \in \kk Q_{0}$, 
the multiplication by ${}^{v}\!\varrho_{M}^{n}$ is a minimal left $\rad^{n}$-approximation. 
In Section \ref{section: check the AN case} we prove that  the condition of Section \ref{section: minimal left rad^n-approximations} is satisfied if $Q=A_{N}$, provided that 
 the  base field has a primitive $(N+1)$-th root of unity. 
 
 Section \ref{section: the universal ladder} 
 studies the bimodule structure of the derived quiver Heisenberg algebras and shows that if the weight $v \in \kk Q_{0}$ is 
 an eigenvector of the transpose $\Psi:= \Phi^{t}$  of the Coxeter matrix with some extra conditions, then the derived quiver Heisenberg algebras 
 provides a diagram of bimodule complexes that can be called universal left ladder. 
 
In Section \ref{section: homotopy Cartesian square} we recall the notion of homotopy Cartesian squares.  
In Section \ref{section: Happel's criterion} we fix notation about Serre functors and reviews Happel's criterion for   a co-connecting morphism of an  Auslander-Reiten triangle. 
In Section \ref{section: natural isomorphisms} we establish several natural isomorphisms of complexes.

\subsection{Notations and Conventions}

Throughout this paper,  the symbol $\kk$ denotes a field.

 ``(dg-)algebra" means (dg-)$\kk$-algebra. 
The symbol $\tuD$ denotes the $\kk$-dual functor $\tuD := \Hom_{\kk}(-,\kk)$. 

\subsubsection{}
Let $R$ be an algebra. 
Unless otherwise stated, the word  ``$R$-modules" means  left $R$-modules.  
We denote the opposite algebra  by $R^{\op}$. 
We identify right $R$-modules with (left) $R^{\op}$-modules.

For an $R$-module $M$, we set $\ResEnd_{R}(M) := \End_{R}(M)/\rad \End_{R}(M)$.

An  $R$-$R$-bimodule $D$ is always assumed to be $\kk$-central, 
i.e., $ad = da $ for $d \in D, \ a \in \kk$. 
Therefore we may identify  $R$-$R$-bimodules  
with  modules over the enveloping algebra $R^{\mre}:= R \otimes_{\kk}  R^{\op}$. 

Using  the isomorphism $\sfc: R^{\mre} \xrightarrow{\cong }  (R^{\mre})^{\op}, \ \sfc(a\otimes b) = b\otimes a$,  
we identify right $R^{\mre}$-modules with left $R^{\mre}$-modules.
In other words, we identify the category $R^{\mre}\Mod$ of left $R^{\mre}$-modules with
the category $\Mod R^{\mre}$ of right $R^{\mre}$-modules via the restriction functor $\sfc_{*}$ along $\sfc$
\[
\sfc_{*}: \Mod R^{\mre} \cong (R^{\mre})^{\op}\Mod  \xrightarrow{\cong}  R^{\mre}\Mod. 
\]

We denote $R^{\mre}$-duality by $(-)^{\vee} := \sfc_{*} \Hom_{R^{\mre}}(-, R^{\mre})$. 
Since we are identifying $\Mod R^{\mre}$ with $R^{\mre} \Mod$, 
 we may  denote $(-)^{\vee} = \Hom_{R^{\mre}}(- , R^{\mre})$
\[
(-)^{\vee} := \sfc_{*} \Hom_{R^{\mre}}(-, R^{\mre}): R^{\mre}\Mod \xrightarrow{ \ \Hom_{R^{\mre}}(-, R^{\mre}) \  }
 \Mod R^{\mre} \xrightarrow{ \sfc_{*} } R^{\mre}\Mod. 
\]
 
We denote by $\sfC(R), \sfC_{\DG}(R)$ 
the category of complexes of $R$-modules (i.e., dg-$R$-modules) and cochain morphisms 
and the dg-category of dg-$R$-modules. 
The symbols $\sfK(R)$ and $\sfD(R)$ denote the homotopy category and the derived category respectively. 
The shift functor of the homotopy category $\sfK(R)$ and  that of the derived category $\sfD(R)$ triangulated category are  denoted by $[1]$. 
The shift functor of an abstract  triangulated category $\sfD$ is denoted by $\Sigma$.

We denote the $R$-duality  $\RHom_{R}(M, R)$ of $M \in \sfD(R) $ by $M^{\lvvee}$. 
We denote the $R$-duality  $\RHom_{R^{\op}}(N, R)$ of $N \in \sfD(R^{\op}) $ by $N^{\rvvee}$. 
Abusing notations, we set  $X^{\lvvee} := \RHom_{R}(X, R), X^{\rvvee} := \RHom_{R^{\op}}(X, R)$ for $X \in \sfD(R^{\mre})$. 
We denote the derived functor of $(-)^{\vee}$ by $(-)^{\vvee} := \RHom_{R^{\mre}}(-, R^{\mre})$.

\bigskip 
We use  the same terminology and notation for a dg-algebra $R$ and dg-modules over $R$ 
as for $R$ an algebra and for $R$-modules.

\subsubsection{}

Let $\tilde{\phi}: X \to Y$ be a morphism in $\sfD(R^{\mre})$. 
For $M \in \sfD(R)$, we use the following  abbreviation
\[
\tilde{\phi}_{M}:= \tilde{\phi} \lotimes_{R} M: X \lotimes_{R} M \to Y \lotimes_{R} M.
\]
Similarly if $M \in \sfD(R)$ and $N \in \sfD(R^{\op})$ are given,  
we set ${}_{N} \tilde{\phi}:= N \lotimes_{R} \tilde{\phi}$ and $ {}_{N} \tilde{\phi}_{M} := N \lotimes_{R} \tilde{\phi} \lotimes_{R} M$. 

We point out that if $\tilde{\phi}: X \to Y$ is a morphism in $\sfD(R^{\mre})$ 
and $f: M \to M'$ is a morphism in $\sfD(R)$, then we have 
the equality $({}_{Y} f) (\tilde{\phi}_{M}) = (\tilde{\phi}_{M'}) ( {}_{X} f)$. 
\begin{equation}\label{202111281249}
\begin{xymatrix}{ 
X \lotimes_{R} M \ar[r]^{\tilde{\phi}_{M}} \ar[d]_{{}_{X} f} & Y \lotimes_{R} M \ar[d]^{{}_{Y} f} \\
X \lotimes_{R} M' \ar[r]^{\tilde{\phi}_{M'}}  & Y \lotimes_{R} M' 
}\end{xymatrix}
\end{equation}

\subsubsection{The path algebras of a  quiver $Q$ }

Unless otherwise stated, a quiver $Q$ is finite acyclic and connected. 

Let $Q=(Q_{0}, Q_{1}, h, t)$ be a quiver. 
For an arrow $\alpha \in Q_{1}$, we denote by $t(\alpha), h(\alpha)$ its tail and head respectively: $t(\alpha) \xrightarrow{\alpha} h(\alpha)$. 
The composition $i \xrightarrow{ \alpha } j \xrightarrow{ \beta} k $  of  two arrows is denoted as $ \alpha\beta$. 

We set $\Pa := \kk Q$. 
For a vertex $i\in Q_{0}$, we denote by $P_{i} := \Pa e_{i}$, $I_{i} := \tuD(e_{i}\Pa)$ and $S_{i}$,
the indecomposable projective module
the indecomposable  injective module 
and the simple module  corresponding to $i$.  

\section*{Acknowledgment}

The authors thank Osamu Iyama for his numerous comments and suggestions on the subjects of the paper.  
The authors thank Takahide Adachi for pointing out a short proof of Lemma \ref{202103301920}. 
The  authors thank Yuya Mizuno and Kota Yamaura for  discussions on path algebras and preprojective algebras. 

The authors thank 
 Aaron Chan and Rene Marczinzik 
 for informing to them their observation about dimension of the quiver Heisenberg algebra of Dynkin type, 
 which leads to Theorem \ref{202109131544}.

The  authors thank Osaka Prefecture University (OPU), 
since the research began when the first author visited the second author belonging to OPU
by using  Special Guest Professor Program run by OPU.

At the beginning stage of  this research, the authors used QPA to calculate the dimensions of quiver Heisenberg algebras of Dynkin quivers 
and got convinced of their results. 
They are grateful to  the QPA-team, QPA - Quivers, path algebras and representations, Version 1.31; 2018.

The second author was partially supported by JSPS KAKENHI Grant Number JP21K03210.

\section{$\rad^{n}$-approximations in $R \mod$ and in $\Dbmod{R}$}\label{section: rad^{n} approximation}

It is well-known to experts that $\rad$-approximation theory in the module categories is nothing but Auslander-Reiten theory. 
$\rad^{n}$-approximation theory in the module categories was initiated by Igusa-Todorov \cite{Igusa-Todorov}. They investigated constructions of $\rad^{n}$-approximations by using the notion of \emph{ladder} which was introduced by them in the same paper. 
Later Iyama \cite{Iyama: ladder}  established a criterion for  existence of ladders in  
$\tau$-categories which is an abstraction of modules category with AR-translations $\tau_{1}^{\pm 1}$ introduced by him. 

In this section we collect basic properties of $\rad^{n}$-approximations in the module category $R\mod$ 
and in the derived category $\Dbmod{R}$  of a finite dimensional algebra $R$.  
Although many results in this section are given in \cite{Igusa-Todorov, Iyama: ladder}, 
we reproduce them for the convenience of the readers.

\subsection{$\rad^{n}$-approximations in a $\kk$-linear category $\sfD$.}

Let $\sfD$ be a $\Hom$-finite $\kk$-linear category that is Krull-Schmidt. In particular, indecomposable objects have local endomorphism algebras. 
\subsubsection{The radical $\rad$}

Recall that
for $X, Y \in \sfD$, the \emph{radical} $\rad(X, Y)$ is defined to be 
a subspace of $\Hom_{\sfD}(X,Y)$ consisting of all elements $f$ that satisfy the following property: 
for any $Z\in \ind \sfD$ and any morphisms $s:Z  \to X, t:Y \to Z$, 
the composition $tfs: Z \to Z$ is not an isomorphism.  
 The radicals $\rad(X,Y)$ form an $\kk$-linear additive subfunctor 
of $\Hom_{\sfD}$ (see e.g. \cite[Proposition V 7.1]{ARS}) . 

For $n \geq 2$, we define $\rad^{n}(X, Y)$ to be the subspace of $\Hom_{\sfD}(X,Y)$ consisting of the elements $f$ which are obtained as 
$n$-times compositions of morphisms in $\rad$.

\subsubsection{$\rad^{n}$-approximations}

\begin{definition}
Let $n \geq 1$. 
\begin{enumerate}[(1)] 

\item 
 A morphism $f: X \to Y$ is called  a \emph{left $\rad^{n}$-approximation} 
if $f$ belongs to $\rad^{n}(X, Y)$ and 
any morphism $g: X \to Z $ belonging to $\rad^{n}(X, Z)$ factors through $f$, i.e., 
there exists $h :X \to Z$ such that $g= hf$. 

\[
\begin{xymatrix}@C=60pt{
X \ar[r]^{f} \ar[dr]_{g} & Y\ar@{-->}[d]^{h} \\
& Z
}\end{xymatrix}
\]

\item A morphism $f: X \to Y$ is called  a \emph{minimal left $\rad^{n}$-approximation}
if it is both left minimal and a left $\rad^{n}$-approximation.

\end{enumerate}

\end{definition}

We define (minimal) right $\rad^{n}$-approximations $Y \to X$ of $X$ in the dual way.
For simplicity we only  discuss minimal left $\rad^{n}$-approximations. 
We note that  right versions of  all the results of this section hold true.  

\begin{remark}
Fu et al \cite{FGHT} studied approximation theory with respect to an ideal $\cI(-, +) \subset \Hom(-, +)$ 
in general context. 
In their terminology, left and right $\rad^{n}$-approximations are called $\rad^{n}$-preenvelopes and $\rad^{n}$-precovers respectively.
\end{remark}

We leave it to the readers to prove the following lemma. 
\begin{lemma}\label{202105171427}

Assume that $X$ is indecomposable. 
Then $f: X\to Y$ is a minimal left almost split morphism  if and only if it is a minimal left $\rad$-approximation. 

\end{lemma}

\subsection{$\rad^{n}$-approximations in $R\mod$.} 

Let $R$ be a finite dimensional algebra. 
By Auslander-Reiten theory, an indecomposable module $M \in \ind R$ has minimal left $\rad$-approximation and minimal right $\rad$-approximation. 
Our starting point is the following observation.

\begin{lemma}\label{202007132105}

For $M \in R\mod $ and $n \geq 1$, 
a minimal left $\rad^{n}$-approximation $M \to N$ exists, and is unique up to composition by an isomorphism from $N$. 

\end{lemma}

\begin{proof}
We use an induction on $n\geq 1$. 
Let 
 $M = \bigoplus_{i =1}^{r} M_{i}$ be an indecomposable decomposition. 
Let $f_{i}:M_{i} \to N_{i}$ be  a minimal left $\rad$-approximation. 
Then it is clear that the direct sum $f:= \bigoplus_{i=1}^{r} f_{i}: M \to \bigoplus_{i=1}^{r} N_{i}$ is a  minimal $\rad$-approximation.

We assume that the statement holds for $n-1$. 
Let $f: M \to N$ be a minimal left $\rad^{n-1}$-approximation. 
We take a minimal left $\rad$-approximation $g: N \to K$. 
Then the composition $fg: M \to K$ is a left $\rad^{n}$-approximation. 
Taking a minimal part of it, we obtain a minimal left $\rad^{n}$-approximation. 

Uniqueness follows from minimality.
\end{proof}

\subsubsection{A description of $\rad^{n}$-approximations}

We give  $\rad^{n}$-versions of well-know description of approximations. 

For $n \geq 1$, we set 
\[
\irr^{n}(M,N):=\frac{ \rad^{n}(M,N)}{\rad^{n +1}(M,N) }
\]

\begin{theorem}\label{202008172145}
Let $n\geq 1$. 

For a morphism $f: M \to N$ the following statements are equivalent. 

\begin{enumerate}[(1)] 
\item $f$ is a minimal left $\rad^{n}$-approximation of $M$. 

\item Any indecomposable object $K$ that has a morphism $g: M \to K$ belonging to $\rad^{n}(M, K) \setminus \rad^{n +1}(M, K)$ 
is a direct summand of $N$. 

Let $N = \bigoplus_{i = 1}^{r} K_{i}^{d_{i}}$ be an indecomposable decomposition where the $K_{i}$'s are pairwise non-isomorphic. 
We exhibit  $f= ( f_{1}, \ldots, f_{r})^{t}$ according to the decomposition 
where 
$f_{i}: M \to K^{d_{i}}_{i}$. 
We exhibit $f_{i}= (f_{i1} , \ldots, f_{i d_{i}})^{t}  :M \to  K^{d_{i}}_{i} $ where $f_{ij}: M \to  K_{i}$. 

Then for each $i = 1,2, \ldots, r$, the set $\{f_{i1}, \ldots, f_{id_{i}}\}$ becomes  a basis of $\irr^{n}(M, K_{i})$ over 
$\ResEnd(K_{i})$. 

\end{enumerate}

\end{theorem}

We leave it to the reader to prove Theorem \ref{202008172145}, 
since it is done in the same of way of the classical $n = 1$ case,
which can be found in, for example,  \cite[Chapter V]{ARS}.

\subsubsection{}

We use the following lemma in the proof of Proposition \ref{202012171832}.

\begin{lemma}\label{202105252216}
Let $n \geq 1$. Given  a minimal left $\rad^{n}$-approximation $f: M \to N$ and  
a morphism $ g: N \to L$ belonging to $\Hom_{R}(N ,L) \setminus \rad(N, L)$, 
then the composition $gf$ belongs to $\rad^{n}(M, L ) \setminus \rad^{n +1}(M, L)$. 
\end{lemma}

\begin{proof}
We use the notation of Theorem \ref{202008172145}(2). 
Since 
$ g \in \Hom_{R}(N ,L) \setminus \rad(N, L)$, there exists  $i=1,2, \ldots, r$  
 such that  there exists a decomposition $L = K_{i} \oplus L'$ and 
the composition  $h: K_{i}^{\oplus d_{i}} \hookrightarrow N \xrightarrow{g}  L \to K_{i}$, 
where the first arrow is the  canonical embedding and the third is the canonical projection, 
is a split epimorphism. 
 If we  write $h = (h_{1}, h_{2}, \ldots,h_{d_{i}}): K^{\oplus d_{i}}_{i} \to K_{i}$ with $h_{p} \in \End_{R}(K_{i})$, 
then  not all the components $h_{1}, \ldots, h_{d_{i}}$ become zero in $\ResEnd_{R}(K_{i})$. 
It follows from Theorem \ref{202008172145}(2) that the sum $\sum_{p=1}^{d_{i}}h_{p} f_{i,p}$ belongs to 
$\rad^{n}(M, K_{i}) \setminus \rad^{n+1}(M, K_{i})$.  
Thus we conclude $gf \in \rad^{n}(M, L ) \setminus \rad^{n +1}(M, L)$. 
\end{proof}

\subsection{$\rad^{n}$-approximations in $R\mod$ and ladders}

Let $M \in R \mod$. If $M$ is indecomposable and injective, then there is a minimal left almost split morphism $M \to L$, which is surjective. If $M$ is indecomposable and not injective, then there is an AR-sequence 
\[
0 \to M \to L \to \tau_{1}^{-1} M \to 0.
\]
Note that in both these situations we obtain a right exact sequence
\[
M \to L \to \tau_{1}^{-1} M \to 0,
\]
where the first morphism is minimal left almost split. By taking direct sums the same holds for any $M \in R \mod$. We refer to this as a direct sum of AR-sequences starting from $M$.

We note that by Lemma \ref{202105171427}, the first morphism $M \to L$ is a minimal left $\rad$-approximation of $M$ 
in $R\mod$. 
By convention, in the case where $M = 0$, 
we refer the sequence $0 \to 0 \to 0 \to 0$ as a direct sum of AR-sequences. 

We note that the morphism $L \to \tau_{1}^{-1} M$ belongs to $\rad$. We also note 
that  the morphism $M \to L$ is injective if and only if $M$ does not have an indecomposable  injective module  as its direct summand.

To discuss more properties of $\rad^{n}$-approximations, 
we introduce the following terminology.

\begin{definition}\label{202007141926}

A morphism $f: M\to N$ said to satisfy the \emph{left $\rad$-fitting} condition, 
if it is a direct summand of a minimal left $\rad$-approximation $g: M \to L$ of $M$. 
This means that there is an isomorphism $L \cong N \oplus N'$ 
under which $g$ corresponds to $(f, f')^{t}$ for some $f': M \to N'$. 
In other words, there exists a split epimorphism $s: L \to N$ such that $sg =f$. 

\end{definition}

We note that in the case where $M = 0$, a minimal left $\rad$-approximation is the morphism $ 0 \to 0$. 
Thus, a morphism $0 \to N$ satisfies the left $\rad$-fitting condition if and only if $N = 0$. 

In a dual way we define the \emph{right $\rad$-fitting condition}.

The main result of this section is the following theorem that gives a way to construct a minimal left $\rad^{n}$-approximation.

\begin{theorem}\label{202105171618}

Let $M \in R\mod$. 
For $n \geq 1$, we denote a minimal left $\rad^{n}$-approximation by $\lambda_{n}: M \to L_{n}$. 
By convention we set $L_{0} := M$. 

Then the following assertions hold. 

\begin{enumerate}[(1)]

\item The cokernel $\Coker \lambda_{n}$ is isomorphic to $\tau_{1}^{-1} L_{n -1}$. 

\item The cokernel morphism $ \lambda'_{n}: L_{n } \to \tau_{1}^{-1} L_{n -1}$ of $\lambda_{n}$  satisfies the left and the right $\rad$-fitting conditions. 

\item Let $\breve{\lambda}: L_{n } \to L$ be a morphism. 
Then the composition $\breve{\lambda} \lambda_{n}: M \to L$ is a minimal left $\rad^{n + 1}$-approximation 
if and only if the morphism $(-\lambda'_{n}, \breve{\lambda})^{t}: L_{n } \to \tau_{1}^{-1} L_{n-1} \oplus L$ is a minimal left $\rad$-approximation. 

Note that the latter condition means that 
we have 
a direct sum of AR-sequences
\[
L_{n } \xrightarrow{ (-\lambda'_{n}, \breve{\lambda})^{t} }  \tau_{1}^{-1} L_{n-1} \oplus L \to \tau_{1}^{-1} L_{n } \to 0.
\]
\end{enumerate}
\end{theorem} 

First we provide the following lemma. 

\begin{lemma}\label{2020121919190}
Let $n \geq 1$, $M \in R\mod$ and $\lambda_{n}: M \to L_{n}$ be a minimal left $\rad^{n}$-approximation. 
We take $L_{n} \overset{f}{\to} N \overset{g}{\to} \tau_{1}^{-1}L_{n}  \to 0 $ to be a direct sum of AR-sequences starting from $L_{n}$. 
We consider  a morphism $\lambda: M \to L$. 
Assume that  there exists a split monomorphism $s: L \to N$ that fits the commutative diagram 
\begin{equation}\label{2020121919340}
\begin{xymatrix}{
M \ar[r]^{\lambda} \ar[d]_{\lambda_{n }} & L \ar[r]^{\rho} \ar[d]^{s} & \tau^{-1}_{1} L_{n}  \ar@{=}[d]  \ar[r]& 0\\
L_{n} \ar[r]_{f} & N \ar[r]_{g} & \tau_{1}^{-1} L_{n} \ar[r]& 0
}\end{xymatrix} 
\end{equation}
where the upper row is exact.
Then $\lambda$ is a minimal left $\rad^{n+1}$-approximation. 
\end{lemma}

\begin{proof}
There exists an isomorphism $ N \cong L \oplus L'$ induced from the split monomorphism $s: L \to  N$. 
Modifying  
the diagram \eqref{2020121919340} by this isomorphism, we obtain 
\[
\begin{xymatrix}{
M \ar[r]^{\lambda} \ar[d]_{\lambda_{n } } & L \ar[r]^{\rho} \ar[d]^{\tiny \begin{pmatrix} \id \\0 \end{pmatrix} } & \tau_{1}^{-1} L_{n} \ar@{=}[d]  \ar[r]& 0\\
L_{n} \ar[r] & L \oplus L' \ar[r]_{ (\rho, *) } &  \tau_{1}^{-1} L_{n} \ar[r]& 0
}\end{xymatrix} 
\]
It follows that $\lambda$ is a left $\rad^{n}$-approximation of $M$. It also follows that  $\rho$ belongs to $\rad$ and hence 
 $\lambda$ is a left minimal morphism. 
 \end{proof}

\begin{lemma}\label{2021050918010}
Let $n \geq 1$ be a positive integer, $M \in R\mod$ 
and $\lambda_{n}: M \to L_{n}$ a minimal left $\rad^{n}$-approximation, which fits an exact sequence  
$M \xrightarrow{ \lambda_{n}} L_{n} \xrightarrow{ \lambda'_{n} } C_{n} \to 0$. 

Assume that  $\lambda'_{n}: L_{n} \to C_{n}$ satisfies  the left $\rad$-fitting condition.
Let $\breve{\lambda}: L_{n} \to L$ 
be a morphism such that the morphism $f: =(-\lambda'_{n }, \breve{\lambda})^{t}: L_{n} \to C_{n } \oplus L$ 
is a minimal left $\rad$-approximation. 
We write a direct sum of AR-sequence obtained from $f$ as below:  
\[
L_{n } \xrightarrow{ (-\lambda'_{n}, \breve{\lambda})^{t}} C_{n} \oplus L \xrightarrow{(\alpha, \beta)}  \tau_{1}^{-1} L_{n} \to 0.  
\]

Then the following holds. 

\begin{enumerate}[(1)] 
\item
The composition $
\breve{\lambda} \lambda_{n } : M \to L
$ 
is a minimal left $\rad^{n +1}$-approximation 
and  the morphism $\beta: L \to \tau_{1}^{-1} L_{n} $ is a cokernel morphism of $\breve{\lambda}\lambda_{n }$.

\item 
If $\breve{\lambda'}: L_{n} \to L'$ is  a morphism such that 
the composition $\breve{\lambda'} \lambda_{n}: M \to L'$ is a minimal $\rad^{n+1}$-approximation, 
then the morphism 
\[
\begin{pmatrix} -\lambda'_{n} \\  \breve{\lambda'}\end{pmatrix} : L_{n } \to C_{n }  \oplus L' 
\]
is a minimal left $\rad$-approximation of $L_{n}$. 
\end{enumerate}

\end{lemma}

\begin{proof}
(1) 
We have the following commutative diagram both rows of which are exact. 
\[
\begin{xymatrix}@C=60pt{
 M \ar[r]^{\breve{\lambda}\lambda_{n} } \ar[d]_{\lambda_{n}} 
& L \ar[r] \ar[d]^{( 0,\id )^{t}} & \Coker \breve{\lambda} \lambda_{n}  \ar@{-->}[d]^{\ind} \ar[r]& 0 \\
 L_{n}  \ar[r]_{( - \lambda'_{n} , \breve{\lambda} )^{t}} & C_{n} \oplus L \ar[r]_{(\alpha,\beta)} & \tau_{1}^{-1} L_{n} \ar[r]& 0.  
}\end{xymatrix}
\]
It is easy to check that the induced morphism $\ind$ is an isomorphism. 
Thus by Lemma \ref{2020121919190} the composition $\breve{\lambda}\lambda_{n}$ is a minimal left $\rad^{n +1}$-approximation. 
Moreover, it also  follows  that the morphism $\beta$ is a cokernel morphism of $\breve{\lambda}\lambda_{n}$.

(2) 
Let $\breve{\lambda'}: L_{n } \to L'$ be a morphism such that the composition 
$\breve{\lambda'}\lambda_{n }: M \to L'$ is a minimal left $\rad^{n+1}$-approximation. 
Then there exists an isomorphism $\gamma: L \to L'$ such that $\breve{\lambda'} \lambda_{n } = \gamma \breve{\lambda} \lambda_{n}$. 
Since $\lambda'_{n}$ is a cokernel  morphism of $\lambda_{n }$, there exists a morphism $\delta: C_{n} \to L'$ such that $\breve{\lambda'} = -\delta \lambda'_{n } + \gamma \breve{\lambda}$.  
It follows that the morphism $\epsilon: = \begin{pmatrix} \id & 0 \\ \delta& \gamma \end{pmatrix}: C_{n} \oplus L \to C_{n} \oplus L'$ 
is an isomorphism 
such that $(-\lambda'_{n}, \breve{\lambda'})^{t} = \epsilon (- \lambda'_{n }, \breve{\lambda})^{t}$. 
Thus we conclude that the morphism $(-\lambda'_{n}, \breve{\lambda'})^{t}$ is a minimal left $\rad$-approximation of $L_{n}$. 
\end{proof}

\begin{proof}[Proof of Theorem \ref{202105171618}]

Let $\lambda_{1}: M \to L_{1}$ be a minimal left $\rad$-approximation of $M$. 
Since it is a direct sum of left almost split morphisms, its cokernel $\Coker \lambda_{1}$ is isomorphic to $\tau_{1}^{-1}M$. 
It is clear that the cokernel morphism $\lambda'_{1}: L_{1} \to \tau_{1}^{-1} M$ satisfies the right $\rad$-fitting condition. 

We use the theory of ladders which was introduced by Igusa-Todorov \cite{Igusa-Todorov} in the case where  the base field $\kk$ is algebraically closed or $R$ is of finite representation type. 
We also refer  a generalization  due to  Iyama \cite{Iyama: ladder}. 
By \cite[Theorem 2.15]{Igusa-Todorov}, \cite[Theorem 3.3]{Iyama: ladder}, the morphism $\lambda'_{1}$ extends to 
a left ladder
\[
\begin{xymatrix}{ 
L_{1} \ar[r]^{\breve{\lambda}_{1}} \ar[d]_{\lambda'_{1}} &
L_{2} \ar[r]^{\breve{\lambda}_{1} }\ar[d]_{\lambda'_{2}} &
\cdots & 
L_{n} \ar[r]^{\breve{\lambda}_{n} }\ar[d]_{\lambda'_{n}} & 
L_{n +1} \ar[r]^{\breve{\lambda}_{n+1}} \ar[d]_{\lambda'_{n +1}} & \cdots \\ 
\tau_{1}^{-1} L_{0} \ar[r]_{\tilde{\lambda}_{1}} &
\tau_{1}^{-1} L_{1} \ar[r]_{\tilde{\lambda}_{2}} &
\cdots&
\tau_{1}^{-1} L_{n-1} \ar[r]_{\tilde{\lambda}_{n}} &\tau_{1}^{-1} L_{n} \ar[r]_{\tilde{\lambda}_{n+1}} &\cdots 
}\end{xymatrix} 
\]
More precisely, we have  commutative squares that yield 
 direct sums of AR-sequences 
\[
L_{n} \xrightarrow{( -\lambda'_{n}, \breve{\lambda}_{n})^{t}} \tau_{1}^{-1} L_{n -1} \oplus L_{n +1}
 \xrightarrow{ (\tilde{\lambda}_{n}, \lambda'_{n +1} )} \tau_{1}^{-1} L_{n } \to 0. 
\]
This in particular tells us that $\lambda'_{1}$ satisfies the left $\rad$-fitting condition. 

It follows from Lemma \ref{2021050918010} that 
the composition $\breve{\lambda}_{n -1} \breve{\lambda}_{n -2} \cdots \breve{\lambda}_{1} \lambda_{1} : M \to L_{n}$ 
is a minimal left $\rad^{n}$-approximation. 
It is straightforward to verify the statements by induction on $n\geq 1$ using Lemma \ref{2021050918010}. 
\end{proof}

\subsection{$\rad^{n}$-approximation in $\Dbmod{R}$} 

To discuss $\rad^{n}$-approximations in the derived category $\Dbmod{R}$, we need to assume that 
$R$ has finite global dimension. 
We set $\nu_{1}:= \nu \circ[-1]$, where $\nu := \tuD(R) \lotimes_{R}-$ is the Nakayama functor of $\Dbmod{R}$.
By Happel \cite{Happel Book}, $\Dbmod{R}$ has AR-triangles. 
For $M \in \Dbmod{R}$, by taking the direct sum of AR-triangles starting from indecomposable direct summands of $M$, we obtain
\[
M \to L \to \nu_{1}^{-1} M \to M[1]
\]
which we refer to as a direct sum of AR-triangles starting from $M$. 
We note that by Lemma \ref{202105171427}, the first morphism $M \to L$ is a minimal left $\rad$-approximation of $M$ 
in $\Dbmod{R}$. 
 
It is clear that an analogue of  Lemma \ref{202007132105} holds. 
Namely, for $n \geq 1$, an object $M \in \Dbmod{R}$ has a minimal left $\rad^{n}$-approximation.
We can also establish an analogue statement with Theorem \ref{202008172145}. 
However we remark that the complete analogue of Theorem \ref{202105171618} does not hold. 
The cone morphism $\lambda'_{n} : L_{n} \to C_{n}$ of a minimal left $\rad^{n}$-approximation does not necessarily satisfy 
the left $\rad$-fitting condition.

\begin{example} 
Let $Q$ be a Dynkin quiver with the Coxeter number $h$. We set $R:= \kk Q$. 
We show in Theorem \ref{kQ approximations theorem add} 
that if $n \geq h -1$, then a minimal left $\rad^{n}$-approximation is 
given by the zero morphism $\lambda_{n}: M \to 0$. 
Thus if $M \neq 0$, the cone morphism $\lambda'_{n}: 0 \to M[1]$ does not 
satisfy the left $\rad$-fitting condition. 
\end{example} 

In terms of Iyama's theory of $\tau$-categories, 
the difference comes from the fact that $R\mod$ is a strict $\tau$-category but $\Dbmod{R}$ is a $\tau$-category which is not strict.  

\subsubsection{} 

We use the following construction of minimal left $\rad^{n+1}$-approximations
from minimal left $\rad^{n}$-approximations, 
which is an analogue of Lemma \ref{2020121919190}.  
Since the proof is similar, we omit it.

\begin{lemma}\label{202012191919}
Let $n \geq 2$, $M \in \Dbmod{R}$ and $\lambda_{n}: M \to L_{n}$ be a minimal left $\rad^{n}$-approximation. 
Assume that $L_{n} \neq 0$. 
We take $L_{n} \overset{f}{\to} N \overset{g}{\to} \nu^{-1}_{1}(L_{n}) \to L_n[1]$ to be a direct sum of AR-triangles starting from $L_{n}$. 
Then a morphism $\lambda: M \to L$ is a minimal left $\rad^{n+1}$-approximation 
if there exists a split monomorphism $s: L \to N$ that fits the following commutative diagram 
\begin{equation}\label{202012191934}
\begin{xymatrix}{
M \ar[r]^{\lambda} \ar[d]_{\lambda_{n}} & L \ar[r]^{\rho} \ar[d]^{s} & \nu^{-1}_{1}(L_{n}) \ar@{=}[d]  \ar[r]& M[1] \ar[d]^{\lambda_{n}[1]}\\
L_{n} \ar[r]_{f} & N \ar[r]_{g} & \nu^{-1}_{1}(L_{n})  \ar[r]& L_{n}[1]
}\end{xymatrix} 
\end{equation}
whose upper row is an exact triangle.
\end{lemma}

An analogue of Lemma \ref{2021050918010} also holds.

\begin{lemma}\label{202105091801}
Let $n \geq 1$ be a positive integer, $M \in \Dbmod{R}$ 
and $\lambda_{n}: M \to L_{n}$ a minimal left $\rad^{n }$-approximation, which fits an exact triangle 
$M \xrightarrow{ \lambda_{n}} L_{n} \xrightarrow{ \lambda'_{n} } C \xrightarrow{\lambda''_{n}} M[1]$. 

Assume that  $\lambda'_{n}: L_{n} \to C_{n}$ satisfies  the left $\rad$-fitting condition.
Let $\breve{\lambda}: L_{n} \to L$ 
be a morphism such that the morphism $f: =(-\lambda'_{n}, \breve{\lambda})^{t}: L_{n} \to C_{n} \oplus L$ 
is a minimal left $\rad$-approximation. 
We write the direct sum of AR-triangles obtained from $f$ as below:  
\[
L_{n} \xrightarrow{ (-\lambda'_{n}, \breve{\lambda})^{t}} C_{n} \oplus L \xrightarrow{(\alpha, \beta)} C \to L_{n}[1].  
\]

Then the following holds. 

\begin{enumerate}[(1)] 
\item
The composition $
\breve{\lambda} \lambda_{n} : M \to L
$ 
is a minimal left $\rad^{n+1}$-approximation 
and  the morphism $\beta: L \to C$ is a cone morphism of $\breve{\lambda}\lambda_{n}$.

\item 
If $\breve{\lambda'}: L_{n} \to L'$ is  a morphism such that 
the composition $\breve{\lambda'} \lambda_{n}: M \to L'$ is a minimal $\rad^{n+1}$-approximation, 
then the morphism 
\[
\begin{pmatrix} -\lambda'_{n} \\  \breve{\lambda'}\end{pmatrix} : L_{n} \to C_{n}  \oplus L' 
\]
is a minimal left $\rad$-approximation of $L_{n}$. 
\end{enumerate}

\end{lemma}

\begin{proof}
(1) 
Observe that there is the following commutative diagram whose vertical arrows are isomorphisms 
\[
\begin{xymatrix}@C=60pt{ 
M \ar@{=}[d] \ar[r]^{( -\lambda_{n}, \breve{\lambda}\lambda_{n})^{t}} &
 L_{n} \oplus L \ar[d]^{\tiny \begin{pmatrix} -\id & 0 \\ \breve{\lambda} & \id \end{pmatrix} }
  \ar[r]^{\tiny \begin{pmatrix} \lambda'_{n}  & 0 \\   \breve{\lambda}& \id \end{pmatrix} }&
  C_{n } \oplus L \ar[r]^{(-\lambda''_{n }, 0)} \ar[d]^{\tiny \begin{pmatrix} -\id  & 0 \\   0& \id \end{pmatrix} } & M[1] \ar@{=}[d] \\ 
M \ar[r]_{( \lambda_{n }, 0)^{t}} &
 L_{n} \oplus L \ar[r]_{\tiny \begin{pmatrix} \lambda'_{n}  & 0 \\   0& \id \end{pmatrix} }&
  C_{n} \oplus L \ar[r]_{(\lambda''_{n}, 0)}  & M[1] 
  }\end{xymatrix}
  \] 
  Since the bottom row is an exact triangle, so is the upper row. 
It follows that the square below is a homotopy Cartesian square (see Section \ref{section: homotopy Cartesian square}) 
\[
\begin{xymatrix}@C=60pt{
M \ar[r]^{\breve{\lambda}\lambda_{n} } \ar[d]_{\lambda_{n}} & L   \ar[d]^{( 0,\id )^{t}}  \\
L_{n }  \ar[r]_{( -\lambda'_{n} , \breve{\lambda} )^{t}} & C_{n} \oplus L. 
}\end{xymatrix}
\]

By  \cite[Lemma 1.4.4]{Neeman}, the morphism $\beta: N \to K$ is a cone morphism of the composition $\breve{\lambda}\lambda$. 
Therefore we have the following diagram 
\[
\begin{xymatrix}@C=60pt{
C [-1] \ar@{=}[d] \ar[r] & M \ar[r]^{\breve{\lambda}\lambda_{n} } \ar[d]_{\lambda_{n}} & L \ar[r]^{\beta} \ar[d]^{( 0,\id )^{t}} & C \ar@{=}[d] \\
C[-1] \ar[r] & L_{n}  \ar[r]_{( -\lambda'_{n} , \breve{\lambda} )^{t}} & C \oplus L \ar[r]_{(\alpha,\beta)} & C,
}\end{xymatrix}
\]
both rows of which are exact.
It follows from Lemma \ref{202012191919} that 
the composition $\breve{\lambda}\lambda_{n}$ is a minimal left $\rad^{n+1 }$-approximation of $M$.

(2) is proved in the same way as Lemma \ref{2021050918010}(2).
\end{proof}

As a corollary we deduce the following statement.

\begin{corollary}\label{2021050917551}
Let $n \geq 1$ be a positive integer, $M \in \Dbmod{R}$ and 
$\lambda_{n}: M \to L_{n}$ be a minimal  left $\rad^{n}$-approximation, which fits into an exact triangle
\[
M \xrightarrow{ \lambda_{n}} L_{n} \xrightarrow{ \lambda'_{n}} C_{n} \xrightarrow{ \lambda''_{n}} M[1]. 
\]
Assume that the morphism $\lambda'_{n}$ satisfies the left $\rad$-fitting condition. 
Then for a morphism  $\breve{\lambda}; L_{n} \to L$ the following conditions are equivalent. 
\begin{enumerate}[(1)] 

\item 
The composition $\breve{\lambda}\lambda_{n}: M \to L$ is a minimal left $\rad^{n}$-approximation 

\item The morphism $(-\lambda'_{n}, \breve{\lambda})^{t}: L_{n} \to C_{n} \oplus L$ is a minimal left $\rad$-approximation.

\item 
There exists the following commutative diagram 
\[
\begin{xymatrix}{ 
& M \ar@{=}[r] \ar[d]_{\lambda_{n}} & M \ar[d]^{\breve{\lambda} \lambda_{ n}}& \\
K[-1] \ar@{=}[d] \ar[r] & L_{n} \ar[d]_{\lambda'_{n}} \ar[r]^{\breve{\lambda}} &
L \ar[r] \ar[d]^{\beta}  & K\ar@{=}[d] \\
K [-1] \ar[r] & C_{n} \ar[d]_{\lambda''_{n}} \ar[r]_{\alpha} & C \ar[r] \ar[d] & K \\ 
& M [1] \ar@{=}[r]  & M  & 
}\end{xymatrix} 
\]
where the middle columns and the middle rows are exact and 
the middle square is homotopy Cartesian 
that is folded to a direct sum of AR-triangles sating from $L_{n}$ 
\[
L_{n} \xrightarrow{ (-\lambda'_{n}, \breve{\lambda})^{t}} C_{n} \oplus L \xrightarrow{(\alpha, \beta)} C \to L_{n}[1].  
\]

\end{enumerate}

\end{corollary}

Consequently, 
we obtain an analogue of Theorem \ref{202105171618} for $\Dbmod{R}$. 
We remark that compared to the previous theorem, we need to assume the left $\rad$-fitting conditions 
in the next theorem.

\begin{theorem}\label{202105181644} 

Let $M \in \Dbmod{R}$.  
For $n \geq 1$, we denote a minimal left $\rad^{n}$-approximation by $\lambda_{n}: M \to L_{n}$. 
By convention we set $L_{0} := M$ and $\lambda_{0}: = \id_{M}$. 

Assume there exists $n$ such that the cone morphism $\lambda'_{m}: L_{m} \to C_{m}$ satisfies the left $\rad$-fitting conditions 
for $m = 0,1 \cdots, n -1$.
Then the cone  $C_{m} $ of $\lambda_{m}$ is isomorphic to $\nu_{1}^{-1} L_{m -1}$ for $m= 1,2, \ldots, n$.  
Moreover, the cone morphism $\lambda'_{n}: L_{n} \to C_{n}$ satisfies the right $\rad$-fitting condition. 
\end{theorem}

\subsubsection{} 

We use the following easy observation later.

\begin{lemma}\label{202105171508} 
Let $f: M \to N$ be a morphism that satisfies left $\rad$-fitting condition 
and $g: M \to L$ a minimal left $\rad$-approximation. 
If a  morphism $t: L \to N$ satisfies the equality $tg = f$, then it is a split epimorphism. 
\end{lemma} 

\begin{proof}
Let $s: L \to N$ be a split epimorphism such that $sg = f$. Then 
we have  $(t- s) g= 0$. 

Let $h: L \to \nu_{1}^{-1} M$ a cone morphism of $g$. 
It follows that there exists a morphism $r: \nu_{1}^{-1} M \to  N$ such that 
$ t= s+ rh$. Since $rh$ belongs to $\rad$, we conclude that $t$ is a split epimorphism. 
\end{proof}

\section{$\rad^{n}$-approximation theory of   $\Dbmod{\kk Q}$}\label{section: kQ approximations}

In this section, we study $\rad^{n}$-approximations of $M \in \Dbmod{\kk Q}$, the derived category of the path algebra $\kk Q$ of a quiver $Q$. 

First we introduce certain subsets $N_{Q}, N_{Q}^{\geq i}$ of $\NN$.

\begin{definition}\label{definition: NQ}
We define the subset $N_{Q} \subset \NN$ to be 
\[
N_{Q} := 
\begin{cases}
\{n \in \NN \mid 0 \leq n \leq h -2\} & 
( Q \textup{ is a Dynkin quiver  with the Coxeter number } h )\\
\NN & ( Q \textup{ is a non-Dynkin quiver}). 
\end{cases}\\
\]
For $i \in \NN$, we set $N_{Q}^{\geq i} := N_{Q} \cap \{ n \in \NN \mid n \geq i\}$. 
\end{definition}

We give the table of the Coxeter nunbers:
\begin{center}
\begin{tabular}{c|c|c|c|c|c}
& $A_{n}$  & $D_{n}$  & $E_{6} $ & $E_{7}$ & $E_{8}$ \\ \hline
$h$ & $n +1$ & $2 n- 2$ & $12$ & $18$ & $30$
\end{tabular}
\end{center}

The main goal of this section is to prove the following two theorems. 
We set $\Pa := \kk Q$. 

\begin{theorem}\label{kQ approximations theorem} 

Let $M$ be a non-zero object of $\Dbmod{\Pa}$. 

For  $n \in \NN$ we let $\lambda_{n}: M \to L_{n}$ be a minimal left $\rad^{n}$-approximation in $\Dbmod{\Pa}$ 
and $\lambda'_{n}: L_{n} \to C_{n}$ a cone morphism of $\lambda_{n}$. 
Thus, for $n \geq 1$,  we have the following exact triangles 
\[
M \xrightarrow{ \lambda_{n} } L_{n} \xrightarrow{ \lambda'_{n} } C_{n} \to M[1]
\]
Then for all $n \in N_{Q}$, we have $L_{n}\neq 0$.
Moreover, for all $n \in N_{Q}^{\geq 1}$, the morphisms $\lambda'_{n}$ satisfy the left and the right $\rad$-fitting condition.
\end{theorem}

\begin{theorem}\label{kQ approximations theorem add}
Let $Q$ be a Dynkin quiver with the Coxeter number $h$. Then the following assertions hold. 
\begin{enumerate}[(1)] 
\item For $n \in \NN$, we have $L_{n} \neq 0$ if and only if $n \in N_{Q}$. 

\item $L_{h -2} \cong \nu(M)$.
\end{enumerate}
\end{theorem}

We note that since $\Dbmod{\Pa}^{\op}$ is equivalent to $\Dbmod{\kk Q^{\op}}$ 
and the Coxeter number $h_{Q^{\op}}$ of the opposite quiver $Q^{\op}$ coincides with that of $Q$, 
right versions of Theorem \ref{kQ approximations theorem} and Theorem \ref{kQ approximations theorem add} also hold true.

\begin{question}
Theorem \ref{kQ approximations theorem} says that for the path algebra $\Pa$ of a quiver $Q$ there exists $\ell$ a natural number or $\infty$ 
such that for each $M \in \ind \Dbmod{\Pa}$, an $\rad^{n}$-approximation object $L_{n}$ of $M$ is non-zero if and only if $n < \ell$. 

We do not know that the same statement holds true for any finite dimensional algebra $R$ of finite global dimension.
\end{question}

\subsection{Proof of Theorem \ref{kQ approximations theorem} and Theorem \ref{kQ approximations theorem add}}

We note that $L_{0} \cong  M$ and $\lambda_{0}$ is an isomorphism. 
Thus in the proof we may assume that $n \geq 1$. 

\subsubsection{An observation}

We identify $\Pa \mod$ with the essential image of the canonical embedding functor $\Pa\mod \hookrightarrow \Dbmod{\Pa}$. 
Let $M$ be an indecomposable  $\Pa$-module which is not injective (resp. projective).  
Then by \cite[I.4.7]{Happel Book}, 
a morphism $M \to L$ ( resp. $ R \to M$) is a minimal  left (resp. right ) $\rad$-approximation in $\Pa \mod$ 
if and only if it has the same property in $\Dbmod{\Pa}$.
In the case where  $M$ is indecomposable injective (resp. projective), 
a minimal left (resp. right) $\rad$-approximation $ M \to L$ (resp. $R \to M$) in $\Pa\mod$ 
is completed to 
a minimal left (resp. right) $\rad$-approximation $ M \to L\oplus L' $ (resp. $R\oplus R' \to M$) in $\Dbmod{\Pa}$. 
Thus we deduce the following lemma.

\begin{lemma}\label{202105190925} 
A morphism $f: M \to N$ in $\Pa \mod$ satisfies 
the left and the right $\rad$-fitting conditions in $\Pa\mod$ 
if and only if it satisfies the same conditions in $\Dbmod{\Pa}$. 
\end{lemma}

\subsubsection{The non-Dynkin case} 
We assume that $Q$ is non-Dynkin. 
Since an indecomposable object of $\Dbmod{\Pa}$ is a shift of an indecomposable $\Pa$-module, 
we may assume that $M$ is an indecomposable $\Pa$-module. 
In the case where $M$ is a preinjective module, 
applying $\nu_{1}$ and $[-1]$ (if it is necessary), 
we may reduce this case to the case where  $M $ is a prepeojective module. 
Thus, we may assume that $M$ is either a regular module or a preprojective module. 
We note 
that for those classes of modules, the functors $\nu_{1}^{-1} $ and $\tau_{1}^{-1}$ coincide. 

We denote by $\cP$ and $\cR$  the classes of preprojective and  regular modules respectively.
To deal with the cases $M \in \ind \cP$ and $M \in \ind \cR$ at the same time, 
we denote the class to which  $M$ belongs by $\cC \in \{\cP, \cR\}$.

We prove the statement by induction on $n \geq 1$. 
First we deal with the case $n =1$. 
Let $0 \to M \xrightarrow{\lambda_{1}} L_{1} \xrightarrow{ \lambda'_{1}} \tau_{1}^{-1} M \to 0$ 
be the AR-sequence starting from $M$. 
Then by Theorem \ref{202105171618} the morphism $\lambda'_{1}$ satisfies both the right  and the left $\rad$-fitting condition $\Pa\mod$. 

By \cite[I.4.7]{Happel Book}, 
the above AR-sequence become the AR-triangle 
$M \xrightarrow{\lambda_{1}} L_{1} \xrightarrow{ \lambda'_{1}} \nu_{1}^{-1} M \to M[1]$. 
Thus $\lambda_{1}: M \to L_{1}$ is a minimal left $\rad$-approximation. 
It is clear $L_{1}$ is a non-zero $\Pa$-module belonging to $\cC$. 
It follows from Lemma \ref{202105190925} that $\lambda'_{1}$ satisfies the left and the right $\rad$-fitting conditions.

Next we deal with the case $n \geq 2$. 
Assume that the case $n -1$ is proved. 
Since it is established in the induction procedure as shown below, we may also assume that $L_{n -1}$ belongs to $\cC$. 
Let $f: L_{n -1} \to L$ be a minimal $\rad$-approximation. We note that $f$ is a monomorphism in $\Pa\mod$. 
Then 
since a minimal left $\rad^{n}$-approximation $\lambda_{n} : M \to L_{n}$ is a minimal part of the composition $f \lambda_{n -1}$, 
$L_{n}$ belongs to $\cC$ and $\lambda_{n}: M \to L_{n}$ is a monomorphism in $\Pa\mod$. 
It follows that $L_{n} \neq 0$ and that the cokernel $\Coker \lambda_{n}$ in $\Pa\mod$ coincides with a cone $C_{n}$ of $\lambda_{n}$ in $\Dbmod{\Pa}$. 
Finally combining Theorem \ref{202105171618} and Lemma \ref{202105190925} 
we conclude that $\lambda'_{n}$ satisfies the left and the right $\rad$-fitting conditions.

\subsubsection{The Dynkin case}\label{202105252310}
We assume that $Q$ is Dynkin.
Since an indecomposable object of $\Dbmod{\Pa}$ is a shift of an indecomposable $\Pa$-module, 
we may assume that $M$ is an indecomposable $\Pa$-modules. 
Applying $\nu_{1}^{\pm 1}$ (if it is necessary), 
we may assume that $M $ is an indecomposable projective module $M = P_{i_{0}}$ corresponds to a vertex $i_{0} \in Q_{0}$. 

Since the underlying graph $|Q|$ of $Q$ is a tree, $Q$ is catenary, which means that 
for any pair $i, j \in Q_{0}$ of vertices, any paths from $i $ to $j$ have the same length. 
Thus there exists a unique map $p : Q_{0} \to \ZZ$ which satisfies the following conditions: 
\begin{enumerate}[(1)] 
\item if there exists an arrow $i \to j$ in $Q$ 
then $p(j) = p(i) +1$,
\item $p(i_{0}) = 0$
\end{enumerate}
We extend the map $p$ to $p: \ZZ(Q)_{0} \to \ZZ$ by $p(i, m) := p(i) + 2m$. 

Recall that $ \add \{ \nu_{1}^{-m} P_{i} \mid m \in \ZZ, i \in Q_{0}\} = \sfD^{\mrb}(\Pa \mod)$. 
We may regard $p$ as a map $\ind \sfD^{\mrb}(\Pa \mod) \to \ZZ$ by using a canonical  bijection 
$\ind \sfD^{\mrb}(\Pa \mod) \cong \ZZ (Q)_0$. 
Explicitly, for an indecomposable object $M = \nu^{-m}_{1} P_{i} \in \ind \cU_{\Pa}$, we set    
\[
p(M) := p(i,m) = p(i) + 2m.  
\]
We set 
\[
\cU_{n}:= \add \{ M \in \ind \sfD^{\mrb}(\Pa \mod) \mid p(M) = n\}. 
\]
We note that $\sfD^{\mrb}(\Pa \mod) = \bigvee_{n \in \ZZ} \cU_{n}$. 

From the shape of the AR-quiver, we deduce the following lemmas. 

\begin{lemma}\label{202012141917}
The following assertions hold. 
\begin{enumerate}
\item 
An object $N \in \sfD^{\mrb}(\Pa \mod)$ belongs to $\cU_{n}$ if and only if $\nu_{1}^{-1}(N)$ belongs to $\cU_{n +2}$.

\item 
Let $N \in \ind \cU_{n}$ and $N \to N' \to \nu_{1}^{-1}(N) \to N[1]$ be an AR-triangle. 
Then the middle term $N'$ belongs to $\cU_{n +1}$. 
\end{enumerate}
\end{lemma}

\begin{lemma}\label{202105082234}
Let $n \in \ZZ, \ m \in \NN$ and $M \in \cU_{n}, \ N \in \cU_{n +m}$. 
Then any  morphism $f: M \to N$ belongs to $\rad^{m}$. 
\end{lemma}

\begin{lemma}\label{basics of place number preliminary 2}
Let $M ,N\in \ind \cU_{\Pa}$ such that $M \not\cong N$. 
If  $\Hom_{\Pa} (N, M) \neq 0$, then $p(N) < p(M)$. 
\end{lemma}

\begin{proof}
Assume $\Hom_{\Pa}(N ,M) \neq 0$. 
Then there exists a direct path from $N$ to $M$ in the AR-quiver of $\Dbmod{\Pa}$. 
From the assumption $M \not\cong N$, we conclude that $p(N) < p(M)$. 
\end{proof}

\begin{corollary}\label{202105082233}
Let $n \in \ZZ$ and $M \in \cU_{n}, \ N \in \cU_{n +1}$. 
Then 
$\Hom_{\Pa}(M,N) = \rad(M, N)$ and 
$\rad^{2}(M, N) =0$. 
\end{corollary}

The above lemmas and corollary hold for all catenary quivers.
For the following lemma, we need to use the assumption that 
$Q$ is a Dynkin quiver with the Coxeter number $h$.

\begin{lemma}\label{basics of place number Dynkin case}
Let $M, N\in    \ind \sfD^{\mrb}( \Pa \mod )$. 
Then the following assertions hold. 
\begin{enumerate}[(1)]

\item 
For a triangulated autoequivalence $F$ of $\sfD^{\mrb}( \Pa\mod )$, 
we have 
\[
p(FM) - p(FN) = p(M) -p(N). 
\]

\item
\[
p(M[1]) = p(M) + h. 
\]
\item 
\[
p(\nu (M)) = p(M)+h-2.
\]
In particular, for $i \in Q_{0}$ we have 
\[
p(I_{i}) = p(P_{i}) + h -2. 
\]
\end{enumerate}
\end{lemma}

To prove this lemma, we need to use the following fundamental result. 
\begin{theorem}[{\cite[p359, Table 1]{Miyachi-Yekutieli}}]\label{Miyachi-Yekutieli lemma}
We have the following natural  isomorphism of autoequivalences of $\sfD^{\mrb}( \Pa\mod )$ 
\[
\nu^{-h}_{1}  \cong [2]. 
\]
\end{theorem}

\begin{proof}[Proof of Lemma \ref{basics of place number Dynkin case}]
(1) Since $Q$ is Dynkin, there is a path in the AR-quiver of $\sfD^{\mrb}( \Pa\mod )$ from $M$ to $\nu_{1}^{-m}(N)$ for some $m \ge 0$. Let $l$ be the length of such a path. Then $p(M) - p(N) = l -2m$. Since there is a path of the same length $l$ from $FM$ to $F(\nu_{1}^{-m}N) \cong \nu_{1}^{-m}(FN)$ we get $p(FM) - p(FN) = l -2m = p(M) - p(N)$.

(2) By (1), there exists an integer $h'$ such that $p(M[1]) = p(M) + h'$ 
for all $M \in \ind \sfD^{\mrb}( \Pa\mod )$. 
Consequently, $p(M[2]) = p(M[1]) + h' = p(M) + 2h'$. 
On the other hand, $p(M[2]) = p(M)+ 2h$ by 
Theorem \ref{Miyachi-Yekutieli lemma}. 
Thus, we conclude $h' = h$.

(3) 
By (2), $p(\nu(M)) = p(\nu_1(M)[1]) = p(\nu_1(M))+h = p(M)+h-2$.

\end{proof}

\begin{proof}[Proof of Theorem \ref{kQ approximations theorem}]

Let $M = P_{i_{0}}$. 
  By Lemma \ref{basics of place number Dynkin case},  $I_{i_{0}} \in \ind \cU_{h -2}$. There is a non-zero morphism $M \to I_{i_{0}}$, which by  Lemma  \ref{202105082234}
belongs to $\rad^{h -2}$. 
It follows that  $L_{n} \neq 0$ for $1 \leq n \leq h-2$. 

By induction on $n \in N_{Q}^{\geq 1}$, we prove  the statement  that $\lambda'_{n}$ satisfies the left and the right $\rad$-fitting condition 
and additionally the statement that $L_{n} \in \cU_{n }$ and $C_{n} \in \cU_{n+1}$.

We deal with the case $n =1$. 
The exact triangle $M \to L_{1} \to C_{1} \to M[1]$ is nothing but 
the AR-triangle $M \to L_{1} \to \nu_{1}^{-1}M \to M[1]$ starting from $M$.
Thus it is clear that $\lambda'_{1}: L_{1} \to C_{1}$ satisfies the right $\rad$-fitting condition. 
It follows from Lemma \ref{202012141917} that $L_{1} \in \cU_{ 1} $  and $C_{1} \in \cU_{2}$. 
We can prove that  $\lambda'_{1}$ satisfies the left $\rad$-fitting condition as in the general case shown below. 

We deal with the case $n \geq 2$. 
Assume that the statements are proved in the case $n -1$. 
Then by Theorem \ref{202105181644}, $\lambda'_{n}$ satisfies the right $\rad$-fitting condition and  $C_{n} \cong \nu_{1}^{-1} L_{n -1}$. 
Thus it only remains to show that $\lambda'_{n}$ satisfies the left $\rad$-fitting condition. 

We claim that $L_{n} \in \cU_{n}$. 
Indeed, it follows from Lemma \ref{202012141917} and induction hypothesis that $L_{n-1} \in \cU_{n -1}$. 
Let $f: L_{n -1} \to L$ be a minimal left $\rad$-approximation. 
By Lemma \ref{202012141917} that $L \in \cU_{n }$.
Since  $L_{n}$ is a direct summand of $L$, it also belongs to $\cU_{n}$. 

Assume towards contradiction, that $\lambda'_{n}$ does not satisfies the left $\rad$-fitting condition. 
Let $f: L_{n} \to L'$ be a minimal left $\rad$-approximation of $L_{n}$.  
\[
\begin{xymatrix}{
M \ar[r]^{\lambda_{n}} & L_{n} \ar[r]^{\lambda'_{n}} \ar[dr]_{f'} & C_{n} \ar[r]  & M[1] \\
&& L'\ar@{-->}[u]_{g} &
}\end{xymatrix}
\]

Since $\lambda'_{n}$ belongs to $\rad$, there exists a morphism  $g: L' \to C_{n}$ such that $\lambda'_{n} =gf$.  
By assumption, $g$ is not a split-epimorphism. 
We take an indecomposable decomposition $C_{n} = \bigoplus_{i =1}^{a} K_{i}$ and 
decompose $g$ as 
$g= (g_{1}, g_{2}, \ldots, g_{a})^{t}: L' \to  \bigoplus_{i =1}^{a} K_{i} $. 
Then there exists $j = 1,2, \ldots, a$ such that $g_{j} $ is not a split-epimorphism. 
Since $K_{j}$ is indecomposable, we have $g_{j} \in \rad(L, K_{i})$. 
It follows that $g_{j} f' \in \rad^{2}(L_{n }, C_{n})$. 
However, 
by Corollary \ref{202105082233}, $\rad^{2}(L_{n}, C_{n}) =0$. 
Therefore $g_{j} f' =0$. 

It follows from the description
$\lambda'_{n} = gf' =  (g_{1}f', g_{2}f', \ldots, g_{a}f')^{t}: L_{n} \to C_{n}$ 
that $K_{j}$ is embedded in $M[1]$ as a direct summand via 
the connecting morphism $ C_{n} \to M[1] $.  
However $K_{i} \in \ind \cU_{n +1}$ and $M[1] \in \ind \cU_{h}$. 
Therefore, we have $ n+1 = h$ a contradiction. 
This shows that $\lambda'_{n}$ satisfies the left $\rad$-fitting condition. 
\end{proof}

\begin{proof}[Proof of Theorem \ref{kQ approximations theorem add}]

We may assume $M = P_{i_{0}}$. 
Let $N \in \ind \Dbmod{\Pa}$ be such that $\Hom_{\Pa}(M, N) \neq 0$. 
Then $N$ belongs to $\Pa \mod$ and has the simple module $S_{i_{0}}$ as a composition factor. 
Therefore, $\Hom_{\Pa}(N, I_{i_{0}}) \neq 0$. 
It follows from Lemma \ref{basics of place number preliminary 2}  and Lemma \ref{basics of place number Dynkin case} that 
$p(N) \leq p(I_{i_{0}}) =  h-2$. 
Thus  $L_{n} = 0$ for $n \geq h -1$.

It also follows from Lemma \ref{basics of place number preliminary 2} that $L_{h -2}$ is a direct sum of $I_{i_{0}}$. 
Since $\Hom_{\Pa}(M, I_{i_{0}}) \cong \kk$, 
$L_{h -2}$ contains a single copy of $I_{i_{0}}$ as a direct summand. 
Thus $L_{h -2}   = I_{i_{0}} = \nu(M)$. 
This finishes proof of Theorem \ref{kQ approximations theorem add}. 
\end{proof}

\subsection{A description of the direct sum of left $\rad^{n}$-approximation objects}

Let $M$ be an indecomposable object of $\Dbmod{\Pa}$ and 
$L_{n}$ a minimal left $\rad^{n}$-approximation object of $M$. 
In this section we give a description of the total sum $\bigoplus_{n \geq 0} L_{n}$. 

\begin{theorem}\label{20211111751} 
Let $M \in \ind \Dbmod{\Pa}$ and $\cC_{M} \subset \Dbmod{\Pa}$ be the full subcategory that consists of objects belonging to the 
same component as $M$ in the AR-quiver. 
For $n\geq 0$, we denote by $\lambda_{n}: M \to L_{n}$ a minimal left $\rad^{n}$-approximation of $M$.
Assume that the following condition does not  hold: $Q$ is wild and $M$ is a chomological degree shift of a regular module. 

Then we have an isomorphism 
\[
\bigoplus_{n \geq 0} L_{n} \cong \bigoplus_{N \in \ind \cC_{M} } N^{\oplus \dim\Hom(M, N)}
\]
in $\sfD(\Pa)$. 
\end{theorem}

\begin{proof}
We may assume that $M$ is an indecomposable $\Pa$-module. 

Recall (e.g., \cite[X Definition 3.1]{Simson-Skowronski: II})
 that a connected component $\cC$ of the AR-quiver of a finite dimensional algebra $R$ is called \emph{generalized standard} 
if for any pairs $(X, Y)$ of indecomposable objects in $\cC$, we have $\rad^{n}(X, Y) = 0$ for $n \gg 0$. 

By \cite[X. Proposition 3.2]{Simson-Skowronski: II} a preprojective component and a preinjective component are generalized standard. 
By \cite[X. Theorem 4.5, XI. Theorem 2.8]{Simson-Skowronski: II}, 
if $Q$ is extended Dynkin and $M$ is a regular module, then the component $\cC_{M}$ is generalized standard.  
Therefore 
it follows from the assumption on $Q$ and $M$
that, for  $N \in\ind\cC_{M}$, we have $\dim \Hom(M,N) = \sum_{n\geq 0} \dim \irr^{n}(M,N)$. 
Thus by Theorem \ref{202008172145} we deduce the desired conclusion. 
\end{proof}

\section{dg-modules over a dg-algebra}\label{section: dg-modules}
In this Section \ref{section: dg-modules} we prepare and fix  notion and terminology  of dg-modules over dg-algebras. 
For  basics of dg-algebras and  dg-modules, we refer \cite{Keller: DG-categories, Keller: ICM}.

\subsection{Dg-modules over a dg-algebra}

\subsubsection{}
Let $R$ be a dg-algebra. 
We denote by $|x|$ the cohomological degree of a homogeneous element $x \in R$. 

Commutators $[-,+]$ are taken in the cohomological graded sense.  
Namely for homogeneous elements $x, y\in R$, we set $[x,y] := xy - (-1)^{|x||y|}yx$. 
Commutators satisfy 
\[
\begin{split}
[x, y]  &= (-1)^{|x||y| +1} [y ,x]  \ \textup{(skew commutativity) and }\\
[x, [y,z] ] &= [[x,y] ,z] + (-1)^{|x||y|}[y, [x,z] ] \ \textup{(graded Leibniz rule}).
\end{split}
\]

A homogeneous morphism $D: R\to R$ is said to be a  \emph{derivation} if it satisfies 
$D(xy) =D(x) y + (-1)^{|D||x|}xD(y)$. 
Note that we have $D[x,y] = [D(x), y]+ ( -1)^{|D||x|}[x,D(y)]$. 
We also note that if $D, D'$ are derivations on $R$, then the commutator $[D,D']_{\Hom}: = DD' -(-1)^{|D||D'|}D'D$ taken in $\cpxHom_{\kk}(R, R)$ is a derivation on $R$.

Let $x \in R$ be a homogeneous element. 
We denote by $\sfl_{x}$ the left multiplication morphism $\sfl_{x}: R \to R, \ \sfl_{x}(y):= xy$ 
and by
 $\sfr_{x}$ the right multiplication morphism $\sfr_{x}: R \to R, \ \sfr_{x}(y) := (-1)^{|x||y|}yx$. 
We set  $\sfb_{x}:= \sfl_{x} - \sfr_{x}$. In other words, the morphism $\sfb_{x} $ is defined to be $\sfb_{x} : R \to R, \ \sfb_{x}(y) := [x,y]$. 
We note that $\sfb_{x}$ is a derivation of degree $|x|$.

\subsubsection{Morphisms and cochain maps of dg-modules}

Let $X, Y$ be dg-$R$-modules. 
A \emph{morphism} $f: X \to Y$ means a homogeneous morphism. 
A \emph{cochain map} $f: X \to Y$ is a homogeneous map which compatible with the differentials.

We denote by $\uparrow$ the canonical morphisms $X \to X[-1], X[1] \to X$ induced from the identity map of underlying 
ungraded module, and
by $\downarrow$ 
the canonical morphisms $X \to X[1], X[-1] \to X$ induced from the identity map of underlying 
ungraded module. 
For example, for a homogeneous  morphism $f: X \to Y$ of $\sfC_{\DG}(R)$,  
we have $f[1] = ( -1)^{|f|}\downarrow f \uparrow$ as morphisms from $X[1]$ to $Y[1]$. 
Note in particular that if we denote by $d_{X}$ the differential of $X$, then $d_{X[1]} = d_{X} [1] = - \downarrow d_{X} \uparrow$. 
We also note that 
for a cochain map $f: X \to Y$ of dg-$R$-modules, 
the cone $\cone(f)$ is a dg-module whose underlying graded module is $Y\oplus X[1]$ and 
the differential is given by $d_{\cone(f)} := \begin{pmatrix} d_{Y} & f\uparrow \\ 0  & d_{X[1]} \end{pmatrix}$. 
In other words, the cone $\cone(f)$ and the co-cone $\cone(f)[-1]$ of $f$ are given as below.
\[
\cone (f) := \left( Y \oplus X[1], \begin{pmatrix} d_{Y} & f\uparrow \\ 0  & d_{X[1]} \end{pmatrix} \right), \ 
\cone (f)[-1] = \left( Y[-1] \oplus X, \begin{pmatrix} d_{Y[-1]} & - \uparrow f\\ 0  & d_{X} \end{pmatrix} \right). 
\]
Recall that the cone $\cone(f)$  fits into the exact triangle below  in $\sfK(R)$ 
\[
Y \xrightarrow{ i_{1}^{f} } \cone(f) \xrightarrow{ p_{2}^{f} } X[1] \xrightarrow{ -f[1]} Y[1].
\]
where $i^{f}_{1}: = {\small \begin{pmatrix} \id_{Y} \\0 \end{pmatrix}}$ is the canonical inclusion and 
$p^{f}_{2} :=( 0, \id_{X})$ is the canonical projection.

Let $a \in \kk\setminus \{0 \}$. 
Then  we may identify the cone $\cone(af)$ with $\cone(f)$ via the isomorphism $ {\tiny \begin{pmatrix} \id_{Y} & 0 \\ 0 & a \id_{X[1]} \end{pmatrix} } : \cone(af) \to \cone(f)$, 
which  provides  the following isomorphism of exact triangles: 
\begin{equation}\label{202102222229}
\begin{xymatrix}@C=60pt{ 
X \ar[r]^{af} \ar@{=}[d] &Y \ar[r]^{i^{af}_{1}} \ar@{=}[d]  &\cone(af) \ar[d]^{{\tiny \begin{pmatrix} \id_{Y} & 0 \\ 0 & a \id_{X[1]} \end{pmatrix} }} \ar[r]^{p_{2}^{af}} & X[1]  \ar@{=}[d] \\
X \ar[r]_{af} & Y \ar[r]_{i^{f}_{1}}  &\cone(f) \ar[r]_{ a^{-1} p_{2}^{f}} & X[1].
}\end{xymatrix}
\end{equation}

\subsubsection{Homotopies}

Let $f, g: X \to Y$ be morphisms in $\sfC(R)$. 
A homotopy $H$ from $f$ to $g$ is a morphism $H: X \to Y$ of degree $-1$ of $\sfC_{\DG}(R)$ such that 
$f- g = d_{Y}H + Hd_{X}$. We often exhibit the situation as below. 
\[
\begin{xymatrix}{
X \ar@/^1pc/[rr]^{f} \ar@/_1pc/[rr]_{g} \ar@{}[rr]|{\Downarrow H} &&Y
}\end{xymatrix}
\]
It is straightforward to check that a homotopy $H$ from $f$ to $g$ gives an isomorphism 
$\begin{pmatrix} \id_{Y} & H\uparrow \\ 0 & \id_{X[1]} \end{pmatrix} : \cone(f) \to \cone(g)$ in $\sfC(R)$.

Let $f: X \to Y$ be a morphism in $\sfC(R)$. 
Then the morphism $\sfh_{f}=(\downarrow, 0) : \cone(f)[-1] \to Y$ is a homotopy from $f ( -p_{2}^{f})$ to $0$. 
The morphism $\sfg_{f}=(0, \downarrow)^{t} : X \to  \cone(f)$ is a homotopy from $i_{1}^{f} f$ to $0$. 
\begin{equation}\label{202007312045}
\begin{xymatrix}@C=10mm@R=4mm{
&&&\\
\cone(f)[-1] \ar[r]^-{- p_{2}^{f} } \ar@/_2pc/[rr]_{0}   &
X\ar[r]^{f} \ar@{=>}[d]^{ \sfh_{f}  }\ar@/^2pc/[rr]^{0}   
& Y \ar[r]^{i_{1}^{f} } \ar@{=>}[u]^{ \sfg_{f}  }  &\cone(f) \\
&& 
}\end{xymatrix}
\end{equation}

Let $f: X \to Y, \ g : Y \to Z$ be cochain maps and $H$ be a homotpy from $gf$ to $0$. 
Then we have  the induced cochain maps $q_{f,g,H}:=(g, H\hspace{-4pt}\uparrow): \cone(f) \to Z$ and 
$j_{f,g,H}:={\tiny \begin{pmatrix} \uparrow H \\ -f \end{pmatrix} }: X \to \cone(g)[-1]$ 
that fit into the following commutative diagram
\begin{equation}\label{202012082037}
\begin{xymatrix}@R=4mm{
X \ar[drr]_{f}  \ar[rrrr]^{0}  \ar[dd]_{j_{f,g,H} =\tiny \begin{pmatrix} \uparrow \hspace{-2pt}H  \\ -f \end{pmatrix} }
& & && Z\\
&& Y \ar@{=>}[u]^{H} \ar[urr]_{g} \ar[drr]_{i_{1}^{f} }\\
\cone(g)[-1] \ar[urr]_{-p^{g}_{2}[-1]} &&&& \cone(f)  \ar[uu]_{(g, H\uparrow)= q_{f,g,H}}
}\end{xymatrix}
\end{equation}

We leave  the verification of the following lemma to the readers. 

\begin{lemma}\label{202008021445} 
Let $f: X \to Y, g: Y \to Z, h: Z \to W$ be morphisms in $\sfC(R)$ 
and $H: X \to Z$ a homotopy from $gf $ to $0$. 

Then, $h H$ is a homotopy from $hgf$ to $0$ 
and we have an equality $q_{f,hg, hH} = h q_{f, g, H}$ of morphisms 
from $\cone(f) \to  W$. 
\end{lemma}

\subsection{The octahedral axiom}\label{section: octahedral} 

The results of this section is used in Section \ref{section: the derived quiver Heisenberg algebras}. 
We recall the construction of the diagram that appears in the proof that the homotopy category $\sfK(R)$ 
satisfies the octahedral axiom.

\begin{lemma}\label{octahedral axiom}
Let $f: X \to Y, g: Y \to Z$ and $h: X \to Z$ be cochain maps between dg-$R$-modules  
and $H: X\to Z$  a homotopy from $gf$ to $h$.

Then the following statements hold.

\begin{enumerate}[(1)]
\item 
  The following diagram is  commutative in $\sfK(R)$ 
\begin{equation}\label{20191128}
\begin{xymatrix}{
&& X \ar@{=}[rr] \ar[d]_{f}  &&  X \ar[d]^{h}  && && \\
\cone(g)[-1] \ar[rr]^{-p^{g}_{2}[1]} \ar@{=}[d] &&
Y \ar@{=>}[urr]^{H} \ar[d]_{i^{f}_{1}} \ar[rr]^{g} && 
Z \ar[d]_{i^{h}_{1}} \ar[rr]^{i^{g}_{1}} && \cone(g) \ar@{=}[d] 
\\
\cone(g)[-1] \ar[rr]_{\Phi} &&\cone(f) \ar[rr]_{\Psi} \ar[d]_{p^{f}_{2}} && \cone(h) \ar[d]^{p^{h}_{2}} \ar[rr]_{\Upsilon} &&
\cone(g)  \\ 
&& X[1] \ar@{=}[rr] && X[1] && && 
}\end{xymatrix}
\end{equation}
where $\Phi = \begin{pmatrix} 0 & - \id_{Y} \\ 0 & 0\end{pmatrix}, \ \Psi = \begin{pmatrix} g & H\uparrow \\ 0 & \id_{X[1]} \end{pmatrix}$ and $
\Upsilon= \begin{pmatrix} \id_{Z} & - H\uparrow \\ 0 & f[1] \end{pmatrix}$.

\item 
The third horizontal line becomes an exact triangle in $\sfK(R)$.

\item 
\begin{enumerate}[(i)]
\item 
The morphism $K := \begin{pmatrix}  \downarrow & 0 \\ 0& 0 \end{pmatrix}: \cone(g)[-1] \to \cone(h)$ 
of degree $-1$ is a homotopy from $\Psi\Phi$ to $0$.

\item  $L :=  \begin{pmatrix} 0 & 0 \\ \downarrow & 0\end{pmatrix} :\cone(f) \to \cone(g)$ 
is a homotopy from $\Upsilon\Psi$ to $0$.

\item 
The morphism $M := \begin{pmatrix}  0 & 0 \\ 0 & -\downarrow \end{pmatrix}: \cone(h) \to \cone(f)[1]$ of degree $-1$ is 
a homotopy from $-(\Phi[1])\Upsilon$ to $0$. 

\end{enumerate}
\item

The following statements hold.

\begin{enumerate}

\item 
The induced  morphisms $q_{\Phi,\Psi, K} = (\Psi, K\hspace{-3pt}\uparrow): \cone(\Phi) \to \cone (h)$ and 
$- j_{\Upsilon, -\Phi[1], M} = \begin{pmatrix} - \uparrow M \\ \Upsilon \end{pmatrix}: \cone(h) \to  \cone ( -\Phi[1])[-1] = \cone (\Phi)$ 
are homotopy inverse to each other.

\item 

The induced morphisms 
$q_{\Psi,\Upsilon, L}= (\Upsilon, L \hspace{-3pt} \uparrow ): \cone(\Psi) \to \cone(g)$ 
and 
$j_{\Phi,\Psi,K} [1]= \begin{pmatrix}( \uparrow K) [1] \\ -\Phi[1] \end{pmatrix}: \cone(g) \to \cone(\Psi) $ 
are homotopy inverse to each other.

\item

The induced morphisms 
$q_{\Upsilon, -\Phi[1], M}= (-\Phi[1], M \hspace{-3pt} \uparrow ): \cone(\Upsilon) \to \cone(f)[1]$ 
and 
$j_{\Psi,\Upsilon,L} [1]= \begin{pmatrix} (\uparrow L) [1] \\ -\Psi[1] \end{pmatrix}: \cone(f)[1] \to \cone(\Upsilon)$ 
are homotopy inverse to each other.

\end{enumerate}

\item 
We have $\sfh_{\Upsilon} j_{\Psi,\Upsilon, L} = L$ as morphisms from $\cone(f) \to \cone(g)$ of degree $-1$ 
where $\sfh_{\Upsilon} : \cone(\Upsilon)[-1] \to \cone(g)$ is the morphism of \eqref{202007312045}.

\end{enumerate} 

\end{lemma}

\begin{proof}
(1) can be checked by direct calculation. 

(2) 
The case where $h = fg$ and $H = 0$ is proved in  \cite[p.318]{Zimmermann}.
Modifying this case by the cochain isomorphism 
${\tiny \begin{pmatrix} \id_{Z} & H \hspace{-3pt}\uparrow \\ 0 & \id_{X[1]} \end{pmatrix}}: \cone(fg) \to \cone(h)$, 
we verify the  statement for the general case.

(3) is checked by direct calculation. 

(4) 

(a) 
Since $ q= \begin{pmatrix} g & H\hspace{-3pt} \uparrow & \id_{Z} & 0 \\ 0 & \id_{X[1]} & 0 &0 \end{pmatrix}$ 
and $j = \begin{pmatrix} 0 & 0 \\ 0 & \id_{X[1]} \\ \id_{Z} & - H\hspace{-3pt} \uparrow \\ 0 & f[1] \end{pmatrix}$, 
it is straight forward to check $qj = \id_{\cone(h)}$. 
We can check that the morphism 
$\begin{pmatrix} 0  & 0& 0 & 0 \\ 0 & 0 & 0 & 0 \\ 0 & 0& 0 & 0 \\ -\downarrow  & 0& 0& 0 \end{pmatrix} $ 
is a homotopy from $\id_{\cone(\Phi)}$ to $qj$.

We have 
\[
\id_{\cone(\Phi)} - jq
= 
\begin{pmatrix} \id_{Y} & 0& 0 & 0 \\ 0 & \id_{X[1]} & 0 & 0 \\ 0 & 0& \id_{Z} & 0 \\ 0 & 0& 0& \id_{Y[1]} \end{pmatrix} 
- 
\begin{pmatrix} 0  & 0& 0 & 0 \\ 0 & \id_{X[1]} & 0 & 0 \\ g & 0& \id_{Z} & 0 \\ 0 & f[1]& 0& 0 \end{pmatrix} 
=
\begin{pmatrix} \id_{Y}   & 0& 0 & 0 \\ 0 & 0& 0 & 0 \\ - g & 0& 0 & 0 \\ 0 & -f[1]& 0& \id_{Y[1]} \end{pmatrix}.
\]
On the other hand, 
\[
\begin{split}
&\begin{pmatrix} d_{Y}  & f\uparrow & 0 & -\uparrow  \\ 0 & d_{X[1]} & 0 & 0 \\ 0 & 0& d_{Z} & g\uparrow \\ 0  & 0& 0& d_{Y[1]} \end{pmatrix} 
\begin{pmatrix} 0  & 0& 0 & 0 \\ 0 & 0 & 0 & 0 \\ 0 & 0& 0 & 0 \\ -\downarrow  & 0& 0& 0 \end{pmatrix} 
+
\begin{pmatrix} 0  & 0& 0 & 0 \\ 0 & 0 & 0 & 0 \\ 0 & 0& 0 & 0 \\ -\downarrow  & 0& 0& 0 \end{pmatrix} 
\begin{pmatrix} d_{Y}  & f\uparrow & 0 & -\uparrow  \\ 0 & d_{X[1]} & 0 & 0 \\ 0 & 0& d_{Z} & g\uparrow \\ 0  & 0& 0& d_{Y[1]} \end{pmatrix} 
\\ 
&= 
\begin{pmatrix} \id_{Y}  & 0& 0 & 0 \\ 0 & 0 & 0 & 0 \\ -g & 0& 0 & 0 \\ -d_{Y[1]}\downarrow  & 0& 0& 0 \end{pmatrix} 
+ 
\begin{pmatrix} 0  & 0& 0 & 0 \\ 0 & 0 & 0 & 0 \\ 0 & 0& 0 & 0 \\ -\downarrow  d_{Y} & - \downarrow f \uparrow & 0& \id_{Y[1]} \end{pmatrix} \\
& 
=
\begin{pmatrix} \id_{Y}   & 0& 0 & 0 \\ 0 & 0& 0 & 0 \\ - g & 0& 0 & 0 \\ 0 & -f[1]& 0& \id_{Y[1]} \end{pmatrix} 
\end{split}
\]

(b) 
\[
q:=q_{\Psi,\Upsilon,L} = \begin{pmatrix} \id_{Z} & - H \uparrow & 0& 0 \\ 0 & f[1]& \id_{Y[1]} & 0 \end{pmatrix}, 
\ 
j:=j_{\Phi,\Psi, K}[ 1] =
\begin{pmatrix} \id_{Z} & 0 \\ 0 & 0 \\ 0 & \id_{Y[1]} \\ 0& 0 \end{pmatrix}.  
\]
It is clear that $qj = \id_{\cone(g)}$. 
We can check that 
the morphism
$\begin{pmatrix} 0  & 0& 0 & 0 \\ 0 & 0 & 0 & 0 \\ 0 & 0& 0 & 0 \\ 0 & \downarrow & 0& 0 \end{pmatrix} 
$ 
gives a homotopy from $\id_{\cone(\Psi)}$ to $jq$.

We have 
\[
\id_{\cone(\Psi)} - jq
= 
\begin{pmatrix} 0   &  H \uparrow & 0 & 0 \\ 0 & \id_{X[1]} & 0 & 0 \\ 0 & -f[1] & 0 & 0 \\ 0 & 0& 0& \id_{X[2]} \end{pmatrix} 
\]

\[
\begin{split}
&\begin{pmatrix} d_{Z}  & h\uparrow & g \uparrow & H \uparrow \uparrow  \\
 0 & d_{X[1]} & 0 & \uparrow  \\
 0 & 0& d_{Y[1]} & -\downarrow f \uparrow \uparrow \\ 0  & 0& 0& d_{X[2]} \end{pmatrix} 
\begin{pmatrix} 0  & 0& 0 & 0 \\ 0 & 0 & 0 & 0 \\ 0 & 0& 0 & 0 \\ 0 & \downarrow  & 0& 0 \end{pmatrix} 
+
\begin{pmatrix} 0  & 0& 0 & 0 \\ 0 & 0 & 0 & 0 \\ 0 & 0& 0 & 0 \\ 0 & \downarrow  & 0& 0 \end{pmatrix} 
\begin{pmatrix} d_{Z}  & h\uparrow & g \uparrow & H \uparrow \uparrow  \\
 0 & d_{X[1]} & 0 & \uparrow  \\
 0 & 0& d_{Y[1]} & -\downarrow f \uparrow \uparrow \\ 0  & 0& 0& d_{X[2]} \end{pmatrix} 
\\ 
&= 
\begin{pmatrix} 0  & H\uparrow & 0 & 0 \\ 
0 & \id_{X[1]} & 0 & 0 \\
0&  -\downarrow f \uparrow & 0& 0  \\
0&  d_{X[2]}\downarrow  & 0& 0& 0 \end{pmatrix} 
+ 
\begin{pmatrix} 0  & 0& 0 & 0 \\ 0 & 0 & 0 & 0 \\ 
0 & 0& 0 & 0  \\ 0&  \downarrow  d_{X[1] } &  0& \id_{X[2]} \end{pmatrix} \\
& 
=
\begin{pmatrix} 0   &  H \uparrow & 0 & 0 \\ 0 & \id_{X[1]} & 0 & 0 \\ 0 & -f[1] & 0 & 0 \\ 0 & 0& 0& \id_{X[2]} \end{pmatrix} 
\end{split}\]

(c) 
\[
q := q_{\Upsilon, -\Phi[1], M} = \begin{pmatrix} 0 & \id_{Y} & 0& 0 \\ 0 & 0& 0& -\id_{X[2]} \end{pmatrix}, 
\ 
j:= j_{\Psi,\Upsilon,L}[1] \begin{pmatrix} 0 & 0 \\ \id_{Y[1]} & 0 \\ -g[1] & -(H\uparrow)[1] \\ 0 &-\id_{X[2]} \end{pmatrix}.
\]
It is clear that $qj = \id_{\cone(f)[1]}$. 

We can check that 
the morphism
$\begin{pmatrix} 0  & 0& 0 & 0 \\ 0 & 0 & 0 & 0 \\ \downarrow & 0& 0 & 0 \\ 0 & 0 & 0& 0 \end{pmatrix} 
$ 
gives a homotopy from $\id_{\cone(\Upsilon)}$ to $jq$. We have 
\[
\id_{\cone(\Upsilon)} - jq
=
\begin{pmatrix} \id_{Z}   & 0& 0 & 0 \\ 0 & 0& 0 & 0 \\ 0 & g[1]& \id_{Z[1]} & -(H\uparrow)[1] \\
 0 & 0& 0& 0 \end{pmatrix} 
= 
\begin{pmatrix} \id_{Z}   & 0& 0 & 0 \\ 0 & 0& 0 & 0 \\ 
0 & \downarrow g\uparrow & \id_{Z[1]} & -\downarrow H\uparrow\uparrow \\
 0 & 0& 0& 0 \end{pmatrix} 
\]
Next we calculate

\[
\begin{split}
&\begin{pmatrix} d_{Z}  & g\uparrow & \uparrow & - H \uparrow \uparrow  \\
 0 & d_{Y[1]} & 0 & f[1]\uparrow  \\ 
0 & 0& d_{Z[1]} & -\downarrow h\uparrow \uparrow \\
 0  & 0& 0& d_{X[1]} \end{pmatrix} 
\begin{pmatrix} 0  & 0& 0 & 0 \\ 0 & 0 & 0 & 0 \\ \downarrow & 0& 0 & 0 \\ 0 & 0& 0& 0 \end{pmatrix} 
+
\begin{pmatrix} 0  & 0& 0 & 0 \\ 0 & 0 & 0 & 0 \\ \downarrow & 0& 0 & 0 \\ 0 & 0& 0& 0 \end{pmatrix} 
\begin{pmatrix} d_{Z}  & g\uparrow & \uparrow & - H \uparrow \uparrow  \\
 0 & d_{Y[1]} & 0 & f[1]\uparrow  \\ 
0 & 0& d_{Z[1]} & -\downarrow h\uparrow \uparrow \\
 0  & 0& 0& d_{X[1]} \end{pmatrix} \\ 
&= 
\begin{pmatrix} 
\id_{Z}  & 0& 0 & 0 \\ 0 & 0 & 0 & 0 \\ d_{Z[1]}\downarrow & 0& 0 & 0 
\\ 0& 0& 0& 0 \end{pmatrix} 
+ 
\begin{pmatrix} 0  & 0& 0 & 0 \\ 0 & 0 & 0 & 0 \\ \downarrow d_{Z} & \downarrow g \uparrow & \id_{Z[1]} & - \downarrow H \uparrow \uparrow  \\ 0& 0& 0 & 0 \end{pmatrix} \\
&=
\begin{pmatrix} \id_{Z}   & 0& 0 & 0 \\ 0 & 0& 0 & 0 \\ 
0 & \downarrow g\uparrow & \id_{Z[1]} & -\downarrow H\uparrow\uparrow \\
 0 & 0& 0& 0 \end{pmatrix} .
\end{split}
\]
\end{proof}

\subsubsection{A preparation for the proof of Theorem \ref{exact triangle U2}}\label{202008041341}

We provide a technical lemma that is used in the proof of Theorem \ref{exact triangle U2}. 

We continue to use the notations of Lemma \ref{octahedral axiom}. 
We remark that the morphism $\Phi: \cone(g)[-1] \to \cone(f)$ 
and its cone $\cone (\Phi)$ only depends on $f,g$ and does not depend on $h, H$. 
We denote $\Phi$ by $\Phi_{f,g}$ in the case where we need to refer $f, g$. 
Similarly,  we denote  $K: \cone(g)[-1] \to \cone(h)$ by $K_{f,g,h}$.

We denote the morphism $q_{\Phi_{f,g}, \Psi, K_{f,g,h}}: \cone(\Phi_{f,g}) \to \cone(h)$ by $\dot{q}_{f,g, h, H}$.
According to the decompositions 
$\cone(\Phi_{f,g}) = Y \oplus X[1] \oplus Z \oplus Y[1], \ 
\cone(h) = Z \oplus X[1]$, 
this cochain map is exhibited as 
\[
\dot{q}_{f,g,h, H}  = 
 \begin{pmatrix} g & H\hspace{-3pt} \uparrow & \id_{Z} & 0 \\ 0 & \id_{X[1]} & 0 &0 \end{pmatrix}: 
 \cone(\Phi_{f,g}) \to \cone(h).
\]

\begin{definition}\label{202008021600}
Let $f: X \to  Y, g: Y \to Z, k: Z \to W$ be cochain maps in $\sfC(R)$ and $H: X \to W$ a
 homotopy from $kgf$ to $0$. 
 
 We define $\acute{q}_{f,g,k, H}: \cone(\Phi_{f,g}) \to W$ to be the composition 
 \[
 \acute{q}_{f,g,k, H} : \cone(\Phi_{f,g}) \xrightarrow{ q_{f,g,0}} \cone(gf) \xrightarrow{ q_{gf,k ,H}} W. 
 \]
 \end{definition}
 According decomposition $\cone(\Phi_{f,g}) = Y \oplus X[1] \oplus Z \oplus Y[1]$ 
 we have 
 \[
 \acute{q}_{f,g,k,H} = (kg, H\uparrow, k,0 ) : \cone(\Phi_{f,g}) \to W. 
 \]

\begin{lemma}\label{202008021437}
Let $l:X \to X', f: X' \to Y, g: Y \to Z, h: X \to Z$ and $k: Z \to W$ be cochain maps between dg-$R$-modules, 
$H_{1}: X\to Z$  a homotopy from $gfl$ to $h$, 
$H_{2}: X \to W$ a homotopy from $kh$ to $0$ 
and $H_{3}: X' \to W$ be a homotopy from $kgf$ to $0$. 
\[
\begin{xymatrix}{
X  \ar[d]_{l}\ar@{=}[rrrr] &&&& X \ar[d]^{h} \ar@/^40pt/[dd]^{0} &&\\
X' \ar@{=>}[urrrr]^{H_{1}}  \ar[rr]_{f} \ar@/_20pt/[drrrr]_{0} && Y\ar[rr]_{g} & &Z \ar[d]^{k} \ar@{=>}[dll]^{H_{3}} \ar@{=>}[r]^{H_{2}}&&\\
&&&& W&&
}\end{xymatrix} 
\]

Assume we have $H_{3}l = kH_{1} + H_{2}$. 
For simplicity we set $H := H_{3}l$. 

Then $H$ is a homotopy from $kgfl$ to $0$ 
and 
the following diagram is commutative 
\[
\begin{xymatrix}@C=60pt{
\cone(\Phi_{f,g} )  \ar[d]_{\acute{q}_{f,g,k, H_{3}}} 
& 
 \cone(\Phi_{fl,g}) \ar[d]_{\acute{q}_{fl,g,k, H} } \ar[l]_{\sfind} \ar[r]^{\dot{q}_{fl, g,h, H_{1}} }
 & 
\cone(h) \ar[d]^{ q_{h,k, H_{2} } }
  \\
W & W \ar@{=}[r] \ar@{=}[l] & W
}\end{xymatrix}
\]
where $\sfind$ are the morphism induced from $l$.
\end{lemma} 

We note that if $l:X \to X'$ is a quasi-isomorphism, then so is  the induced morphisms $\sfind$. 

\begin{proof}
It is straightforward to check the desired commutativity by using the descriptions of the involved morphisms given below:
\[
\begin{split}
\acute{q}_{f,g,k, H_{3}} & = (kg, H_{3}\uparrow, k,0), \ 
\acute{q}_{fl,g,k, H} = (kg, H \uparrow, k,0),  \ 
q_{h,k, H_{2} } = (k, H_{2}\uparrow), \\
\dot{q}_{fl, g,h, H_{1}} & = \begin{pmatrix} g & H_{1}\uparrow & \id_{Z} & 0 \\ 0 & \id_{X[1]} & 0 & 0 \end{pmatrix}, \ 
\sfind = \begin{pmatrix} \id_{Y} & 0 & 0& 0 \\ 0 & l[1] & 0 & 0 \\ 0 & 0 & \id_{Z} & 0 \\ 0 & 0 & 0 & \id_{Y[1]} \end{pmatrix}.
\end{split}
\]
\end{proof}

\section{Universal Auslander-Reiten triangle for $\kk Q$}\label{section: the universal Auslander-Reiten triangle}

This section \ref{section: the universal Auslander-Reiten triangle} has two aims. 
The first is to fix notation and conventions for path algebras of quivers used throughout the paper. 
The second is to establish universal Auslander-Reiten triangles for path algebras 
and to investigate their basic properties. 
We note that universal Auslander-Reiten triangles have been established for smooth proper dg-algebras in \cite{Minamoto Mukai} via a formal argument using the dg-Morita $2$-category of dg-algebras. 
For the reader's convenience, we give direct proofs here.

\subsection{Path algebras }

\subsubsection{The path algebra of a quiver as a tensor algebra}

Let $Q=(Q_{0}, Q_{1}, h, t)$ be a quiver. 
We recall the construction of the path algebra $A =  \kk Q$ as the tensor algebra 
$\sfT_{\kk Q_{0}}\kk Q_{1}$.

First, we set $\Pa_{0} := \kk Q_{0} = \prod_{i \in Q_{0}} \kk e_{i}$, i.e., 
the direct product of $\kk$ indexed by the set $Q_{0}$ of vertices. 

From now on, by convention, we omit $\otimes_{\Pa_{0}}$ and 
write $MN= M \otimes_{\Pa_{0} } N$ for a right $\Pa_{0}$-module $M$ and a left $\Pa_{0}$-module $N$.

For an arrow $\alpha \in Q_{1}$, 
we denote by $\kk \alpha$ an $\Pa_{0}$-$\Pa_{0}$-bimodule  of $\kk$-dimension $1$ with a $\kk$-basis $\alpha$ 
whose bimodule structure is given by the formulas 
\[
e_{i} \alpha := 
\begin{cases}
\alpha & i = t(\alpha), \\
0 & i \neq t(\alpha), 
\end{cases} \ \ 
 \alpha e_{i} := 
\begin{cases}
\alpha & i = h(\alpha), \\
0 & i \neq h(\alpha). 
\end{cases}
\]
We set $V: = \kk Q_{1} := \bigoplus_{  \alpha \in Q_{1} } \kk \alpha$. 
Then we may identify  the path algebra $\kk Q$ with the tensor algebra $\sfT_{\Pa_{0}} V$ of $V$ over $\Pa_{0}$. 
\[
\kk Q = \sfT_{\Pa_{0}} V = \Pa_{0} \oplus V \oplus VV \oplus VVV \oplus \cdots. 
\]

\subsubsection{The inverse of the dualizing complex}

Recall that the $\Pa^{\mre}$-module $\Pa$ has the following projective resolution 
\[
0 \to \Pa V \Pa  \xrightarrow{\hat{\mu}} \Pa \Pa \xrightarrow{\mu} \Pa \to 0,  
\]
where 
$
\hat{\mu}(p\otimes \alpha \otimes  q) := p\alpha \otimes q -p \otimes\alpha  q, \ 
\mu(p \otimes q) := pq. 
$
We define the complex $\widetilde{\Pa}$ of $\Pa$-$\Pa$-bimodules to be 
the cone $\cone (\hat{\mu})$ of $\hat{\mu}$. 
In the usual way of  exhibiting a complex, $\PPa$ is shown as 
\[
\cdots \to 0 \to \Pa V \Pa \xrightarrow{ \hat{\mu}} \Pa\Pa \to 0 \to \cdots
\]
where $\Pa \Pa$ is placed in the $0$-th cohomological degree. 
But in the sequel, we  exhibit  the complex $\widetilde{\Pa}$ in the following form  
\[
\widetilde{\Pa} := 
\left( \Pa \Pa \oplus  (\Pa V \Pa)[1], \begin{pmatrix} 0 & \hat{\mu}\uparrow \\ 0 & 0 \end{pmatrix} \right).
\]
The morphism $\mu: \Pa \Pa \to \Pa$ induces a quasi-isomorphism $\PPa \to \Pa $ in $\sfC(\Pa^{\mre})$, which we denote by the same symbol by abusing notation, i.e.,
\[
\mu:= (\mu, 0): \PPa \to \Pa. 
\]

We set the $\Pa^{\mre}$-duality to be $(-)^{\vee} := \Hom_{\Pa^{\mre}}(-, \Pa^{\mre})$. 
Note that by our convention, it is an endofunctor of $\Pa^{\mre}\Mod$. 
Since $\PPa$ is a projective resolution of $\Pa$ over $\Pa^{\mre}$, 
the complex $\PPa^{\vee}$ represents  the inverse $\RHom_{\Pa^{\mre}}(\Pa, \Pa^{\mre})$ of the dualizing complex 
\cite{Keller:Calabi-Yau, Keller: Calabi-Yau completion}.

To compute $\PPa^{\vee}$, we introduce 
 a morphism $\hat{\rho} : \Pa\Pa \to \Pa V^{*} \Pa$ 
of $\Pa$-$\Pa$-bimodules. 
Firstly, we set $V^{*} := \tuD(V)$ and let $\{\alpha^{*}\}_{\alpha \in Q_{1}}$ be the dual basis of the basis $\{\alpha \in Q_{1}\}$ 
of the arrow module $V$
 of $Q$. 
Secondly, we note  that there is an isomorphism  
$\Pa\Pa = \Pa \otimes_{\Pa_{0}} \Pa_{0} \otimes_{\Pa_{0}} \Pa \cong  \bigoplus_{i \in Q_{0}} \Pa e_{i} \Pa$ of 
 $\Pa$-$\Pa$-bimodules. 
Therefore to define the morphism $\hat{\rho}$ it is enough to specify  $\hat{\rho}(e_{i}) \in  \Pa V^{*} \Pa$ 
for each $i \in Q_{0}$. 
We define of $\hat{\rho}$ by    
\[ 
\hat{\rho} (e_{i}) :=   \sum_{\alpha: t(\alpha) = i} \alpha \otimes \alpha^{*} \otimes e_{i} 
- \sum_{\alpha: h(\alpha) =i } e_{i} \otimes \alpha^{*} \otimes \alpha. 
\]

\begin{lemma}[{\cite[Lemma 2.3 and the subsequent remark]{Grant-Iyama}}]\label{Grant-Iyama lemma}
Let $U$ be a $\Pa_{0}$-$\Pa_{0}$-bimodule. 
Then the map 
\[F_{U}: \Pa \tuD(U) \Pa \to ( \Pa U \Pa)^{\vee}, F(x \otimes \phi \otimes y) ( z \otimes u \otimes w) 
:= \phi(u) zy \otimes xw 
\]
is an isomorphism of $\Pa$-$\Pa$-bimodules. 
\end{lemma}

In particular we have the following  identifications 
\begin{equation}\label{202006121900}
(\Pa\Pa)^{\vee} \cong \Pa\Pa, \ (\Pa V \Pa)^{\vee} \cong \Pa V^{*} \Pa.
\end{equation}
Moreover, as $(\Pa^{\mre})^{\vee} \cong \Pa^{\mre}$, we may identify the $\Pa^{\mre}$-dual $m^{\vee}$ of the canonical map $m: \Pa^{\mre} \to \Pa\Pa$ 
with the embedding 
\begin{equation}\label{202012091717} 
\Pa\Pa = \bigoplus_{ i \in Q_{0}} \Pa e_{i} \otimes_{\kk} e_{i} \Pa \to \bigoplus_{ i,j \in Q_{0} } \Pa e_{i} \otimes_{\kk} e_{j} \Pa = A^{\mre}. 
\end{equation}

The following lemma gives a description of the complex $\PPa^{\vee}$.

\begin{lemma}\label{202006121909}
The identifications \eqref{202006121900} 
provides the following isomorphism of complexes 
\[
  \PPa^{\vee} \cong 
 \left( \Pa V^{*} \Pa[-1] \oplus \Pa\Pa, \begin{pmatrix} 0 & \uparrow  \hat{\rho}\\  0 & 0 \end{pmatrix} \right). 
\]
\end{lemma}

\begin{proof}
It is easy to show that 
for $f \in \Hom_{\Pa^{\mre}}(\Pa\Pa, \Pa^{\mre})$, 
the composition $f \circ \hat{\mu}: \Pa V\Pa \to \Pa^{\mre}$ corresponds to $- \hat{\rho}(f)$ 
via the isomorphism  $F_{V}$  given in Lemma \ref{Grant-Iyama lemma}. 

Since the $0$-th differential $d_{\Hom}^{0}$ of the complex $\cpxHom_{\Pa^{\mre}}(\PPa, \Pa^{\mre})$ 
is $d_{\Hom}^{0}(f) = d_{\Pa^{\mre}}^{0} \circ f - f\circ d_{\PPa}^{-1} =- f\circ \hat{\mu}$, the assertion. follows. 
\end{proof}

Since $\PPa$ is a projective resolution of $\Pa$ over $\Pa^{\mre}$, the complex $\PPa^{\vee}$ over $\Pa^{\mre}$ is a representative of 
$\Pa^{\vvee} = \RHom_{\Pa}(\Pa, \Pa^{\mre})$. 
We set $\PPi_{1} := \PPa^{\vee}[1]$, since  it is the $*$-degree $1$ part of the derived preprojective algebra $\PPi$ (see Section \ref{section: the derived quiver Heisenberg algebras}). 
Note that the functor $\PPi_{1} \lotimes_{\Pa}- : \Dbmod{\Pa} \to \Dbmod{\Pa}$ is the inverse $\nu_{1}^{-1}$ of the $-1$-shifted 
Nakayama functor $\nu_{1} ( -) := \tuD(\Pa)[-1] \lotimes_{\Pa} -$.

\subsection{Weighted trace and weighted Euler characteristic}

\subsubsection{Trace of an endomorphism of a   complex over $\kk$}

Let $V \in \sfC^{\mrb}(\kk)$ and $f: V \to V$, we define the trace $\Tr(f)$ to be 
\[
\Tr(f) := \sum_{i \in \ZZ} ( -1)^{i}\Tr(f^{i})
\]
where $\Tr(f^{i})$ in the right hand side is the trace of the linear map $f^{i}: V^{i} \to V^{i}$. 
We note that the Euler characteristic  $\Euch(V) := \sum_{i \in \ZZ} ( -1)^{i}\dim V^{i}$ of $V$ is obtained as 
$\Euch (V) = \Tr(\id_{V} )$.

\subsubsection{Weighted trace and  weighted Euler characteristic}

Let $A = \kk Q$ and  $ v= (v_{i}) \in \kk Q_{0}$. 
Let $M \in \sfC^{\mrb}(A)$ and $f: M \to M$. 
Observe that a  vertex $i \in Q_{0}$ induces a morphism  $e_{i}f: e_{i} M \to e_{i} M$ in $\sfC^{\mrb}(\kk)$. 
We define the \emph{weighted trace} ${}^{v}\!\Tr(f)$ of $f$ to be the weighted sum of traces of $e_{i} f$. 
\[
{}^{v}\!\Tr(f) := \sum_{i \in Q_{0}} v_{i} \Tr(e_{i} f).  
\]
We define the \emph{weighted Euler characteristic} ${}^{v}\!\Euch(M)$ to be $ {}^{v}\!\Euch(M) := \sum_{i \in Q_{0} } v_{i} \Euch(e_{i}M)$ 
so that we have 
 ${}^{v}\!\Tr(\id_{M} ) = {}^{v}\!\Euch(M)$.

We denote the dimension vector of $M$ by $\Euvect(M)$.
\[ 
\Euvect(M) := (\Euch(e_{1} M), \Euch(e_{2}M), \ldots, \Euch(e_{r}M) )^{t}  \in \ZZ Q_{0}.
\]
Then the weighted Euler characteristic ${}^{v}\!\Euch(M)$ is obtained as 
\begin{equation}\label{202102051939}
{}^{v}\!\Euch(M)  = v^{t}\Euvect(M),
\end{equation}
i.e., the product of the row vector $v^{t}$ and the column vector $\Euvect(M)$. 

It is straightforward to check that these notions descend to objects $M \in \Dbmod{A}$ and endomorphisms $f: M \to M$. 
We leave the verification of the following lemma to the reader. 
\begin{lemma}\label{202012091517}
\begin{enumerate}[(1)]
\item 
${}^{v}\!\Tr(f) = \sum_{ i \in \ZZ}(-1)^{i}\ {}^{v}\!\Tr(\tuH^{i}(f))$. 

\item 
${}^{v}\!\Tr(\id_{M}) = \sum_{i \in \ZZ} (-1)^{i} \ {}^{v}\!\Euch(\tuH^{i}(M))$. 
\item If $f$ is nilpotent in $\sfD(\Pa)$, then ${}^{v}\!\Tr(f) = 0$. 
\end{enumerate}
\end{lemma}

\subsubsection{}\label{202105081902} 
Recall that we may identify the Grothendieck group $K_{0}(A)$ with $\ZZ Q_{0}$ via the map $[M] \mapsto \Euvect(M)$. 
Let  $F$ be an autoequivalence  of $\Dbmod{A}$. 
We denote by $\underline{F}: \ZZ Q_{0} \to \ZZ Q_{0}$ the automorphism  induced from $F$.  
In other words, $\underline{F}$ is the square matrix that satisfies 
\[
\underline{F}\Euvect(M) = \Euvect(F(M))  
\]
for all $M \in \Dbmod{A}$.
We note the following equation that 
follows from \eqref{202102051939}. 
\begin{equation}\label{202102061837}
{}^{v}\!\Euch(F(M)) = {}^{\underline{F}^{t}(v)}\Euch(M). 
\end{equation}

\subsubsection{The Cartan matrix and the Coxeter matrix}

We recall that  the Cartan matrix $C$  is a matrix defined by 
\[
C := (\Euch(e_{i} \Pa e_{j}) )_{i,j}.
\]
Then the Coxeter matrix $\Phi$ (for left modules) is defined by 
\[
\Phi := - C^{t} C^{-1}. 
\]
Note that in the notation of  Section \ref{202105081902}, we have 
$\Phi = \underline{\nu_{1}}$.
In other words, we have 
\[
\Euvect(\nu_{1}(M)) = \Phi \Euvect(M)
\]
for $M \in \Dbmod{\Pa}$.
Therefore setting $\Psi := \Phi^{-t}$, we have 
\begin{equation}\label{202111101530}
{}^{v}\!\Euch(\PPi_{1} \lotimes_{\Pa} M ) =\Euvect(\nu_{1}^{-1}(M)) = {}^{\Psi(v)} \Euch(M). 
\end{equation}

\subsection{Trace formula}\label{Section:trace formula}

Let $i \in Q_{0}$. 
We define a morphism $\hat{e}_{i}: AA \to AA$ of $A^{\mre}$-modules by the formula 
\[
\hat{e}_{i}(e_{j} ) := \delta_{i,j} e_{i}
\]
for $j \in Q_{0}$. 
We define  a morphism 
$\tilde{e}_{i} :  \PPa^{\vee}\to   \PPa$ of complexes of $\Pa$-$\Pa$-bimodules to be 
\begin{equation}\label{2020081418071}
\tilde{e}_{i}: \small
 \PPa^{\vee}  = \left( \Pa V^{*} \Pa[-1] \oplus \Pa \Pa,  \begin{pmatrix} 0 & \uparrow \hat{\rho} \\ 0 & 0\end{pmatrix} \right)
\xrightarrow{ \tiny \begin{pmatrix} 0 & \hat{e}_{i} \\ 0 & 0 \end{pmatrix}}
\left( \Pa \Pa \oplus \Pa V\Pa [1], \begin{pmatrix} 0 & \hat{\mu} \uparrow \\ 0 & 0\end{pmatrix} \right) =
 \PPa. 
\end{equation}

Recall that  $\PPa^{\vee} \cong  \Pa^{\vvee}$  in $\sfD(\Pa^{\mre})$ and 
there are canonical isomorphisms
\[
\Hom_{\Pa^{\mre}}(\Pa^{\vvee}, \Pa) \cong \tuHH_{0}(\Pa) \cong \Pa/[\Pa, \Pa] \cong \kk Q_{0}.
\]
Note that $e_i \in \kk Q_{0}$ in the right hand side  corresponds to $\tilde{e}_{i}$ in the left hand side.  
Thus the set $\{ \tilde{e}_{i}\}_{i \in Q_{0}}$ forms a basis of $\Hom_{\Pa^{\mre}}(\Pa^{\vvee}, \Pa)$. 

\begin{definition}\label{202102061701} 
For $v \in \kk Q_{0}$, we denote by ${}^{v}\!\ttheta$ the corresponding  element of $\Hom_{\Pa^{\mre}}(\Pa^{\vvee}, \Pa)$. 
In other words, 
for $v \in \kk Q_{0}$, we set 
\[
{}^{v}\!\ttheta := \sum_{ i \in Q_{0}} v_{i} \tilde{e}_{i}. 
\]
\end{definition}

Since $(A^{\vvee})^{\vvee} \cong A$ in $\sfD(\Pa^{\mre})$, 
the $A^{\mre}$-duality $(-)^{\vvee}$ induces an endomorphism  
$(-)^{\vvee}: \Hom_{\Pa^{\mre}}(A^{\vvee}, A) \to \Hom_{\Pa^{\mre}}(\Pa^{\vvee}, \Pa)$.
By a general result  due to Van den Bergh \cite[Proposition 14.1]{Van den Bergh: Calabi-Yau}, this is the identity map (which in our case can be proved by an easy  computation). 
Thus  we have 
\begin{equation}\label{202105111259}
({}^{v}\!\ttheta)^{\vvee} = {}^{v}\!\ttheta. 
\end{equation}

Recall from \cite{Keller:Calabi-Yau} that the functor $\Pa^{\vvee} \lotimes -: \Dbmod{\Pa} \to \Dbmod{\Pa}$ is the inverse of a Serre functor. 
Let 
$\isagl{-,+}: \Hom_{\Pa}(M, N)  \otimes\Hom_{\Pa}(\Pa^{\vvee} \lotimes_{\Pa} N, M) \to \kk$ be the pairing of Serre duality.

Let $M \in \Dbmod{\Pa}$. We recall our convention that  
\[{}^{v}\!\ttheta_{M} = {}^{v}\!\ttheta \lotimes_{\Pa} M: \PPa^{\vvee} \lotimes_{\Pa} M \to M.\]
The following theorem gives a formula that computes the pairing $\isagl{f, {}^{v}\!\ttheta_{M}}$ for $f \in \End_{\Pa}(M)$.

\begin{theorem}\label{trace formula}
Let $M \in \Dbmod{\Pa}$ and $f \in \End_{\Pa}(M)$. 
Then the following equality holds
\[
\isagl{f, {}^{v}\!\ttheta_{M}}  = {}^{v}\!\Tr(f). 
\]
\end{theorem}

\begin{remark}
In subsequent work \cite{Minamoto Mukai},  
we will prove Theorem \ref{trace formula} for smooth proper dg-algebras.
\end{remark}

\begin{proof}
By linearity, it is enough to show that 
$\isagl{f,  \tilde{e}_{i,M}} = \Tr(e_{i} f)$ for each $i \in Q_{0}$.

We recall the isomorphisms in $\sfD(\kk)$ which proves Serre duality from \cite{Keller:Calabi-Yau} (see also Section \ref{section: natural isomorphisms}). 
\[
\RHom_{\Pa}(M,N) 
\xrightarrow{ \cong \  F }\tuD(M) \lotimes_{\Pa} \Pa^{\vvee} \lotimes_{\Pa} N
\xrightarrow{ \cong \ G} \tuD \RHom_{\Pa}( \Pa^{\vvee} \lotimes_{\Pa}N, M). 
\]

First
we explain the construction of $G$. 
Let $L,M \in \sfC(\Pa)$. 
We define a cochain map 
$\tilde{\Phi}_{L,M}: \tuD(M) \otimes_{\Pa} L \to \tuD\cpxHom_{\Pa}(L,M)$ by the formula 
\[
\tilde{\Phi}_{L,M}(\phi \otimes l) (f) := (-1)^{|l||f|}\phi(f(l)) 
\] 
for homogeneous elements $l, \phi, f$ of $L, \tuD(M)$ and $\cpxHom_{\Pa}(L,M)$. 
This map is natural in $L,M$ and descents to a morphism $\Phi_{L,M}:  \tuD(M) \lotimes_{\Pa} L \to \tuD \RHom_{\Pa}(L,M)$ 
in $\sfD(\kk)$. 
The morphism $G$ is set to be $G := \Phi_{\Pa^{\vvee} \lotimes N, M}$.

Next, 
we explain the construction of $F$ in the case $N = M$.  
Let $P \in \sfC^{\mrb}( \Pa \proj)$ be a complex of projective $\Pa$-module quasi-isomorphic to $M$.
Then $F$ is given by the following quasi-isomorphism in $\sfC(\kk)$ which is denoted by the same symbol
\[
\begin{split} 
F: \cpxHom_{\Pa}(P,P) 
& \xrightarrow{ \cong } \cpxHom_{\Pa^{\mre}}(\Pa, \cpxHom_{\kk}(P,P) ) \\
&\xrightarrow{\sim} \cpxHom_{\Pa^{\mre}}( \PPa, \cpxHom_{\kk}(P,P) ) \cong 
\tuD(P) \otimes_{\Pa} \PPa^{\vee} \otimes_{\Pa} P
\end{split}
\]
where the second morphism is induced from $\mu: \PPa \to \Pa$.

To provide an explicit formula of $F$, first we fix a way to exhibit elements of $\Hom^{0}_{\Pa}(P,P)$. 
For this we recall that $\Hom_{\Pa}^{0}(P, P)$ is a subspace $\Hom_{\kk}^{0}(P, P)$ of $\kk$-linear morphisms of degree $0$. 
The underlying cohomologically graded module of $P$ is 
$\bigoplus_{n \in \ZZ} P^{n}[-n]$ and 
hence $\Hom_{\kk}^{0}(P, P) =\prod_{n \in \ZZ} ( P^{n}[-n]) \otimes_{\kk} (\tuD(P^{n})[n])$.
For $j \in Q_{0} $ and $ n \in \ZZ$, 
we fix a $\kk$-basis  $\{\phi_{s}^{(j,n)}\}_{ s= 1}^{d_{j,n}}$ of $e_{j} P^{n}$ and denote by $\{\phi_{s}^{*(j,n)}\}_{s=1}^{d_{j,n}}$ 
the dual $\kk$-basis. 
For simplicity we set $\tilde{\phi}_{s}^{(j,n)} := 
\phi_{s}^{(j, n)}[-n]$ and $\tilde{\phi}_{t}^{*(j, n)} := \phi_{t}^{*(j,n)}[n]$. 
Then 
regarding $\Hom_{\Pa}^{0}(P, P)$ as a subspace of $\Hom_{\kk}^{0}(P, P)$, 
we may write  an element  $f$ of $\Hom_{\Pa}^{0}(P,P)$ as 
\begin{equation}\label{202012091744}
f = \sum_{s,t,n,j,k} f_{st}^{(j,k, n)} \tilde{\phi}_{s}^{(j, n)}\otimes \tilde{\phi}_{t}^{*(k, n)}
\end{equation}
for some $f_{s,t}^{(j,k, n)} \in \kk$.

Recall that as a cohomologically graded modules 
$\PPa = \Pa\Pa \oplus \Pa V\Pa [1]$ and the morphism $\mu: \PPa \to \Pa$ factors through the morphism 
$\Pa\Pa \to \Pa$. 
It follows that the image of the induced morphism $$\cpxHom_{\Pa^{\mre}}(\Pa, \cpxHom_{\kk}(P, P)) \to 
\cpxHom_{\Pa^{\mre}}(\PPa, \cpxHom_{\kk}(P, P))$$
is contained in the direct summand $\cpxHom_{\Pa^{\mre}}(\Pa\Pa, \cpxHom_{\kk}(P, P))$. 
Let $\tilde{m}: \Pa^{\mre} \to \PPa$ be the morphism in $\sfC(A^{\mre})$ induced from the canonical map $m: \Pa^{\mre} \to \Pa\Pa$. 
Then $\tilde{m}$ and $\mu$ induce the following commutative diagram, 
the left column of which is $F$.
\[
\begin{xymatrix}{ 
\cpxHom_{\Pa}(P,P) \ \  \ar@{=}[r] \ar[dd]_{\simeq}  & 
\cpxHom_{\Pa}(P,P) \ \  \ar@{^(->}[r] \ar[d]^{\cong} & \cpxHom_{\kk}(P, P) \ar[dd]^{\cong}  \\ 
& \cpxHom_{\Pa^{\mre}}(\Pa, \cpxHom_{\kk}(P,P)) \ar@{^(->}[d] & \\
\cpxHom_{\Pa^{\mre}}(\PPa, \cpxHom_{\kk}(P,P))  \ar[r]  \ar[d]_{\cong} & 
\cpxHom_{\Pa^{\mre}}(\Pa\Pa, \cpxHom_{\kk}(P,P)) \ar@{^(->}[r] \ar[d]_{\cong} & 
\cpxHom_{\Pa^{\mre}}(\Pa^{\mre}, \cpxHom_{\kk}(P,P)) \ar[d]^{\cong}  \\
\tuD(P) \otimes_{\Pa} \PPa^{\vee} \otimes_{\Pa} P  \ar@{->>}[r]  &
\tuD(P) \otimes_{\Pa} (\Pa\Pa)^{\vee} \otimes_{\Pa} P \ar@{^(->}[r] &  
 \tuD(P) \otimes_{\Pa} (\Pa^{\mre})^{\vee} \otimes_{\Pa} P 
}\end{xymatrix}
\] 
Using this  diagram and \eqref{202012091717}, we see that if $f$ is of the form of \eqref{202012091744}, then $F(f)$ is 
an element of $\tuD(P) \otimes_{\Pa} (\Pa\Pa)^{\vee} \otimes_{\Pa} P \subset \tuD(P) \otimes_{\Pa} \PPa^{\vee} \otimes_{\Pa} P$ 
given as 
\begin{equation}\label{202006151617}
F(f) =\sum_{s,t,n,j} ( -1)^{n} f_{st}^{(j,j,n)} \tilde{\phi}_{t}^{*(j,n)}  \otimes e_{j} \otimes \tilde{\phi}_{s}^{(j, n)}
\end{equation}
where 
we use the identification
\[
\tuD(P) \otimes_{\Pa} (\Pa\Pa)^{\vee} \otimes_{\Pa} P \cong 
\tuD(P) \otimes_{\Pa} \Pa\Pa \otimes_{\Pa} P \cong 
\tuD(P) \otimes_{\Pa_{0}}\Pa_{0} \otimes_{\Pa_{0}} P.
\]

Let $\mu_{P}:  \PPa  \otimes_{\Pa}  P  \to P$ be the canonical quasi-isomorphism, which induces the identification 
$A \lotimes_{A} P \cong P$.  
Now we can check that the desired result holds as follows 
\[
\begin{split}
\isagl{f,   \hat{e}_{i, M}} 
& = G ( F(f) ) (\mu_{P}\tilde{e}_{i, M})\\
& =  \sum_{s,t, n,j,k} (-1)^{n} f_{s,t}^{(j,j, n)} 
\tilde{\phi}_{t}^{*(j,n)}  \left(  (\mu_{P} \hat{e}_{i ,M} )( e_{j} \otimes \tilde{\phi}_{s}^{(j, n)}) \right) \\ 
& = \sum_{s,t,n,j }  (-1)^{n} f_{s,t}^{(i,i,n)}\tilde{\phi}_{t}^{*(i,n)} \left( \tilde{\phi}_{s}^{(i, n)}\right)\\ 
&= \sum_{s,n}  (-1)^{n} f_{s,s}^{(i,i,n)}  \\
  & = \Tr(e_{i} f).
\end{split}
\]
\end{proof}

Let $i\in Q_{0}$ be a vertex.
Applying Theorem \ref{trace formula} to the simple module $S_{i}$ 
and the identity map $\id_{S_{i}}$, we obtain  
\[
\isagl{ \id_{S_{i}}, {}^{v}\!\ttheta_{S_{i}} } = {}^{v}\!\Euch(S_{i}) = v_{i}. 
\]
We point out the following consequence.

\begin{corollary}\label{202102061807}
Let $\phi \in \Hom_{\Pa^{\mre}}(\Pa^{\vvee}, \Pa)$. 
Then we have 
\[
\phi = \sum_{i \in Q_{0}} \isagl{ \id_{S_{i}}, \phi_{S_{i}} } \tilde{e}_{i}.
\]
\end{corollary}

\subsection{The universal Auslander-Reiten triangle}\label{subsection The universal Auslander-Reiten triangle}

Let $v \in \kk Q_{0}$. 
We set  ${}^{v}\!\YYM_{1} := \cone({}^{v}\!\ttheta)$, 
since, as is explained in Section \ref{section: the derived quiver Heisenberg algebras}, 
it is the $*$-degree $1$ part of the derived quiver Heisenberg  algebra ${}^{v}\!\YYM$ 
provided that $v_{i} \neq 0$ for all $i \in Q_{0}$. . 
 
We set  notation  for the exact triangle ${}^{v}\!\sfAR$ obtained from the morphism ${}^{v}\!\ttheta: A^{\vvee} \to A$ 
as below.  

\[
{}^{v}\!\sfAR: \Pa \xrightarrow{ \ {}^{v}\!\rrho \ }  {}^{v}\!\YYM_{1} \xrightarrow{{}^{v}\!\ppi_{1}} \PPi_{1}
\xrightarrow{-{}^{v}\!\ttheta[1]} \Pa[1].     
\]

A morphism is called AR-coconnecting if it is the co-connecting morphism for an AR-triangle.
Combining Theorem \ref{trace formula} and the Happel criterion for AR-coconnecting morphisms (reviewed in Section \ref{subsection: Happel's criterion}), 
we can immediately 
prove that if ${}^{v}\!\Euch(M) \neq 0$, then  ${}^{v}\!\theta_{M}$ is AR-coconnecting and the triangle ${}^{v}\!\sfAR_{M} := {}^{v}\!\sfAR\lotimes M$ is an AR-triangle. 
The precise statement is the following. 

\begin{theorem}\label{semi-universal Auslander-Reiten triangle}
Let $v \in \kk Q_{0}$ and $M \in \sfD^{\mrb}( \Pa \mod )$ an indecomposable object. 
We denote by ${}^{v}\!\sfAR_{M}$ the following exact triangle in $\sfD(\Pa)$. 
\[
{}^{v}\!\sfAR_{M} :
 M  \xrightarrow{ \ {}^{v}\!\rrho_{M} \ } {}^{v}\!\YYM_{1} \lotimes_{\Pa} M \xrightarrow{ \ {}^{v}\!\ppi_{1, M} \ } \PPi_{1} \lotimes_{\Pa} M \xrightarrow{ -{}^{v}\!\ttheta_{M}[1]} M[1]. 
\]
Then the following statements hold. 

\begin{enumerate}[(1)] 

\item If ${}^{v}\!\Euch(M) \neq 0$ in $\kk$, then 
the morphism ${}^{v}\!\ttheta_{M}$ is AR-coconnecting to $M$ and 
 the exact triangle ${}^{v}\!\sfAR_{M}$ is an Auslander-Reiten triangle starting from $M$.

\item 
Assume moreover that $\dim \ResEnd_{\Pa}(M) = 1$. 
Then
 the exact triangle ${}^{v}\!\sfAR_{M}$ is an Auslander-Reiten triangle starting from $M$ 
 if and only if ${}^{v}\!\Euch (M) \neq 0$ in $\kk$. 
 In the case ${}^{v}\!\Euch(M) =0$, the exact triangle ${}^{v}\!\sfAR_{M}$ splits.

\end{enumerate}

\end{theorem}

\begin{proof}

Let $M \in \sfD^{\mrb}( \Pa \mod )$ be an indecomposable object.  
From  Lemma \ref{202012091517} and  Theorem \ref{trace formula}, we have
 \[
\isagl{\id_{M}, \ttheta_{M}}  ={}^{v}\!\Euch(M), \ 
 \isagl{f, \ttheta_{M}} = 0 \textup{ for } f \in \rad\End_{\Pa}(M). 
\] 

(1) Assume that ${}^{v}\!\Euch(M) \neq 0$ in $\kk$. 
It follows from Happel's criterion (Theorem \ref{Happel's criterion}) 
that  ${}^{v}\!\sfAR_{M}$ is an Auslander-Reiten triangle.

(2) follows from Lemma \ref{202102071323}. 
\end{proof}

\subsubsection{Regularity of weights}

We recall from the introduction  the definitions of a sincere weight and  a regular weight 
and 
we add one more notion for weights.

\begin{definition}\label{202111192031}
\begin{enumerate}[(1)]
\item 
A weight $v \in \kk Q_{0}$ is called \emph{sincere} if $v_{i} \neq 0$ for all $i \in Q_{0}$. 

\item 
A weight  $v \in \kk Q_{0}$  is called \emph{regular} 
if we have ${}^{v}\!\Euch(M) \neq 0$ in $\kk$ for all 
$M \in \ind \kk Q$. 

\item 
A weight  $v \in \kk Q_{0}$  is called \emph{semi-regular}  
if we have ${}^{v}\!\Euch(M) \neq 0$ in $\kk$ for all indecomposable 
prepeojective modules $M$ and indecomposable preinjective modules $M$.  
\end{enumerate}
\end{definition}

\begin{remark}
Assume that the base field $\kk$ is algebraically closed and of characteristic $0$. 
In the case $Q$ is  Dynkin, 
the vector space $\kk Q_{0}$ may be identified with the Cartan subalgebra $\mathfrak{h}$ of the semi-simple Lie algebra $\mathfrak{g}$ corresponding to $Q$. 
By Gabriel's theorem \cite{Gabriel} 
that says that the dimension vectors of indecomposable $\Pa$-modules are precisely the positive roots of $\mathfrak{g}$, 
the regularity defined above   coincides with usual notion used about an element  of the Cartan subalgebra $\mathfrak{h}$. 
\end{remark}

Let $\cU_{\Pa} := \add \{ \nu_{1}^{-p}P_{i} \mid p \in \ZZ, i \in Q_{0}\}$. 
Then a  weight  $v \in \kk Q_{0}$ is semi-regular 
 if and only if ${}^{v}\!\Euch(M) \neq 0$ for all $M \in \ind\cU_{\Pa}$. 
We note that a regular weight is sincere. 

We give examples of regular weight  $v \in \kk Q_{0}$.
\begin{example}\label{202102071418} 

\begin{enumerate}[(1)] 
\item Assume that $\QQ \subset \kk$. If $v \in \kk^{\times} Q_{0}$ is strictly positive, 
i.e., $v_{i}\in \QQ_{>0}$ for all $i \in Q_{0}$, then it is regular.

\item 
Let $\PP$ be the prime field contained in $\kk$. 
Assume that $\dim_{\PP} \kk \geq r $. 
 Taking linearly independent element $v_{1}, v_{2}, \ldots, v_{r} \in \kk$ over $\PP$, 
 we obtain a regular weight $v \in \kk Q_{0}$. 
 \end{enumerate}
 \end{example}

Recall that for an indecomposable object  $M \in \ind \Dbmod{A}$
there exist $M' \in \ind \kk Q$ and  a unique integer $p \in \ZZ $ such that  $M \cong M'[p]$ 
and 
hence \[
{}^{v}\!\Euch(M)= ( -1)^{p}{}^{v}\!\Euch(M'). 
\]
Thus if $v \in \kk Q_{0}$ is regular, then  ${}^{v}\!\Euch(M) \neq 0$ for all $M \in \ind \Dbmod{A}$.

\subsubsection{Universal Auslander-Reiten triangle} 

From Theorem \ref{semi-universal Auslander-Reiten triangle}, we conclude the following result.

\begin{theorem}[Universal Auslander-Reiten triangle]\label{universal Auslander-Reiten triangle}
Assume that $v \in \kk Q_{0}$ is regular  (resp. semi-regular). 
Let $M $ be an object of $\sfD^{\mrb}( \Pa \mod )$ (resp. $\cU_{\Pa}[\ZZ]$). 
Then the exact triangle ${}^{v}\!\sfAR_{M}$ is a direct sum of  Auslander-Reiten triangles. 
\[
{}^{v}\!\sfAR_{M} :
 M  \xrightarrow{ \ {}^{v}\!\rrho_{M} \ } {}^{v}\!\YYM_{1} \lotimes_{\Pa} M \xrightarrow{ \ {}^{v}\!\ppi_{1, M} \ } \PPi_{1} \lotimes_{\Pa} M \xrightarrow{ -{}^{v}\!\ttheta_{M}[1]} M[1]. 
\]
In other words, 
the morphism ${}^{v}\!\rrho_{M}: M \to {}^{v}\!\YYM_{1} \lotimes_{A} M$ is a minimal left $\rad$-approximation  of $M$ 
and 
the morphism ${}^{v}\!\ppi_{1, M}: {}^{v}\!\YYM_{1} \lotimes_{A} M \to \PPi_{1} \lotimes_{A} M$ 
is a minimal right $\rad$-approximation of $\PPi_{1} \lotimes_{A} M$. 
\end{theorem}

\subsubsection{} 

We point out the following corollary. 

\begin{corollary}\label{universal Auslander-Reiten triangle corollary} 
Assume that $v \in \kk Q_{0}$ is semi-regular. 
Then the object ${}^{v}\!\YYM_{1} \in \sfD(\Pa^{\mre})$ is concentrated in the $0$-th cohomological degree. 

\end{corollary}

\begin{proof}
It is enough to show that the object ${}^{v}\!\YYM_{1} e_{i} \cong {}^{v}\!\YYM_{1} \lotimes_{\Pa} P_{i}$ of $\sfD(\Pa)$ is concentrated in the $0$-th cohomological degree for $i \in Q_{0}$. 

Let $M \in \ind \Dbmod{\Pa}$. Recall that  $M$ is concentrated in a single cohomological degree and $\Hom_{\Pa}(P_{i}, M) = e_{i} M$. 
It follows that if $\Hom_{\Pa}(P_{i}, M) \neq 0$, then $M$ is concentrated in the $0$-the cohomological degree. 
Since the morphism  ${}^{v}\!\rrho_{P_{i}} : P_{i} \to {}^{v}\!\YYM_{1}e_{i}$ is minimal by Theorem \ref{universal Auslander-Reiten triangle}, 
we conclude that ${}^{v}\!\YYM_{1}e_{i}$ is concentrated in the $0$-th cohomological degree. 
\end{proof}

\subsection{The action of an exact autoequivalence $F$ on ${}^{v}\!\ttheta$}

Let  $\phi: A^{\vvee} \to A$ be a morphism in $\Dbmod{A^{\mre}}$. 
We denote  the induced natural transformation  by the same symbol 
\[
\phi: \sfS^{-1}= \Pa^{\vvee} \lotimes _{\Pa} - \xrightarrow{ \ \phi \lotimes -\ }  \Pa \lotimes_{\Pa} - = \id_{\sfD}
\] 
 where $\id_{\sfD}$ is the identity functor of $\Dbmod{\Pa}$. 

Let $T \in \Dbmod{\Pa}$ be a two-sided tilting complex over $\Pa$ and 
 $F:= T \lotimes_{\Pa}-$  the associated  exact autoequivalence of $\Dbmod{\Pa}$. 
 We note that there is a natural isomorphism $\gamma_{F}: \sfS^{-1} F \to F\sfS^{-1}$ induced from an 
 isomorphism $\gamma_{T}: \Pa^{\vvee} \lotimes_{\Pa} T \to T \lotimes_{\Pa} \Pa^{\vvee}$ of $\Dbmod{\Pa^{\mre}}$ 
 (see Lemma \ref{202106091106}).
 
We define  a natural transformation $F \cdot \phi: \sfS^{-1} \to \id$ to be the following composition 
\[
F\cdot \phi : \sfS^{-1}\xrightarrow{\cong} FF^{-1} \sfS^{-1}  \xrightarrow{ F((\gamma_{F^{-1}})^{-1}) } F\sfS^{-1} F^{-1} \xrightarrow{ \ F \phi F^{-1} \ } F F^{-1}\xrightarrow{\cong} \id_{\sfD}.
\] 
It follows from the above remark  
that the natural transformation $F\cdot \phi$ is induced from 
a morphism $F\cdot \phi: \Pa^{\vvee} \to \Pa$ of $\Dbmod{\Pa^{\mre}}$, which we denote by the same symbol. 

We obtain the following lemma immediately from the definition of $F\cdot \phi$. 

\begin{lemma}\label{202101131703}
Let $T \in \Dbmod{\Pa}$ be a two-sided tilting complex over $\Pa$ and 
 $F:= T \lotimes_{\Pa}-$  the associated  exact autoequivalence of $\Dbmod{\Pa}$. 
Then we have 
\[
F( ( F^{-1}\cdot\phi)_{M}) = \phi_{F(M)} \gamma_{F, M}. 
\]
In other words, the following diagram is commutative
\[
\begin{xymatrix}@C=60pt{
\sfS^{-1}F(M) \ar[r]^{\phi_{F(M)}}  \ar[d]_{\gamma_{F,M}} 
& F(M) \\
F\sfS^{-1}(M) \ar[ur]_{ \  F( (F^{-1} \cdot \phi)_{M}) } &.
}\end{xymatrix}
\]
\end{lemma}

The next theorem provides a way to compute $F\cdot ({}^{v}\!\ttheta)$.

\begin{theorem}\label{202101081603}
Let $T \in \Dbmod{\Pa}$ be a two-sided tilting complex over $\Pa$ and 
 $F:= T \lotimes_{\Pa}-$  the associated  exact autoequivalence of $\Dbmod{\Pa}$. 
Then, for $v \in \kk Q_{0}$, 
we have the following equality in $\Hom_{A^{\mre}}(A^{\vvee}, A)$ 
\[
F^{-1}\cdot ( {}^{v}\!\ttheta) = {}^{\underline{F}^{t}(v)} \ttheta. 
\]
\end{theorem}

 \begin{proof}
 For $M \in \Dbmod{\Pa}$ and $f\in \End_{\Pa}(M)$, 
 we have 
 \[
 \isagl{f, (F^{-1}\cdot {}^{v}\!\ttheta )_{M}}  = \isagl{F(f), F(F^{-1}\cdot ({}^{v}\!\ttheta ))\gamma_{F} }
  = \isagl{F(f), ({}^{v}\!\ttheta)_{F(M)}}.
 \]
Therefore, for $i \in Q_{0}$, we have 
\[
\begin{split}
\isagl{ \id_{S_{i}}, (F^{-1}\cdot {}^{v}\!\ttheta)_{ S_{i} } } 
 = \isagl{ \id_{F(S_{i}) } , {}^{v}\!\ttheta_{F(S_{i})}}  = {}^{v}\!\Euch(F(S_{i}))
=   {}^{\underline{F}^{t}(v)} \Euch(S_{i}) = (\underline{F}^{t} (v))_{i}
\end{split}
\]
where for the second equality we use Theorem \ref{trace formula} 
and for the third the equation \eqref{202102061837}. 
Thus we deduce the desired equality from Corollary \ref{202102061807}. 
\end{proof}

\begin{proposition}\label{202101131644}
Let $v \in \kk Q_{0}$ and $T \in \Dbmod{\Pa}$ be a two-sided tilting complex over $\Pa$. 
We denote by 
 $F:= T \lotimes_{\Pa}-$  the associated  exact autoequivalence of $\Dbmod{\Pa}$. 
Then for $M \in \ind \Dbmod{\Pa}$ such that $\dim \ResEnd_{\Pa}(M) =1$ and ${}^{v}\!\Euch(M) \neq 0$, the following holds. 

\begin{enumerate}[(1)] 
\item 
We have the following equality in $\Hom_{\Pa}(\sfS^{-1}(M), M)$ 
\[
(F^{-1}\cdot{}^{v}\!\ttheta)_{M} = \frac{{}^{v}\!\Euch (F(M))}{{}^{v}\!\Euch (M)} {}^{v}\!\ttheta_{M}.
\]

\item We have the following equality in $\Hom_{A}(F\sfS^{-1}(M), F(M))$ 
\[
\frac{{}^{v}\!\Euch (F(M))}{{}^{v}\!\Euch (M)} F( {}^{v}\!\ttheta_{M}) \gamma_{F} = {}^{v}\!\ttheta_{F(M)}. 
\]
In other words the following diagram is commutative: 
\[
\begin{xymatrix}@C=60pt{
\sfS^{-1}F(M) \ar[r]^{{}^{v}\!\ttheta_{F(M)}}  \ar[d]_{\gamma_{F,M}} & 
F(M) \\
F\sfS^{-1}(M) \ar[ur]_{x F({}^{v}\!\ttheta_{M}) }
& 
}\end{xymatrix}
\]
where \[x: =\frac{{}^{v}\!\Euch (F(M))}{{}^{v}\!\Euch (M)}. 
\]
\end{enumerate}
\end{proposition}

\begin{proof}
(1) 
We have the following equation 
\[
\begin{split}
\isagl{ \id_{M},  ( F^{-1}\cdot {}^{v}\!\ttheta)_{M} } = \isagl{\id_{M}, {}^{\underline{F}^{t}(v)}\ttheta_{M}} = {}^{\underline{F}^{t}(v)} \Euch(M) = {}^{v}\!\Euch(F(M)).
\end{split}
\]
where for the first equality we use Theorem \ref{202101081603} 
and for the third the equation \eqref{202102061837}. 
We deduce the desired equation by using Lemma \ref{202102071323}. 

(2) follows from (1) together with Lemma \ref{202101131703}. 
\end{proof}

From uniqueness of AR-triangles, we obtain the following corollary. 

\begin{corollary}\label{202102071627}
Let $v \in \kk Q_{0}$ and $F$ be an exact autoequivalence of $\Dbmod{A}$. 
Then for $M \in \ind \Dbmod{\Pa}$ such that $\dim\ResEnd_{\Pa}(M) =1$ and  ${}^{v}\!\Euch(M) \neq 0,  {}^{v}\!\Euch(F(M)) \neq 0$, 
there exists an isomorphism ${}^{v}\!\sfa_{F, M}: {}^{v}\!\YYM_{1} \lotimes_{A} F(M) \to F({}^{v}\!\YYM_{1} \lotimes_{A} M)$ 
that gives the following isomorphism of exact triangles:
\[
\begin{xymatrix}@C=60pt{
\sfS^{-1}F(M) \ar[r]^{{}^{v}\!\ttheta_{F(M)}}  \ar[d]_{\gamma_{F}} & 
F(M) \ar[r]^{{}^{v}\!\rrho_{F(M)}} \ar@{=}[d]  &
{}^{v}\!\YYM_{1} \lotimes_{A}F(M) \ar[d]^{{}^{v}\!\sfa_{F, M} } \ar[r]^{{}^{v}\!\ppi_{1, F(M)} } &
\sfS^{-1}F(M) [1] \ar[d]^{\gamma_{F}[1]}    \\
F\sfS^{-1}(M) \ar[r]_{x F({}^{v}\!\ttheta_{M}) } & 
F(M) \ar[r]_{F({}^{v}\!\rrho_{M})}  &
F({}^{v}\!\YYM_{1} \lotimes_{A} M) \ar[r]_{x^{-1}F({}^{v}\!\ppi_{1, M}) } &
F\sfS^{-1}(M) [1]   \\
}\end{xymatrix}
\]
where \[
x: =\frac{{}^{v}\!\Euch (F(M))}{{}^{v}\!\Euch (M)}.
\] 
\end{corollary}

We remark that the morphism ${}^{v}\!\sfa_{F, M}$ is not uniquely determined, but it is determined modulo $\rad$.

  \subsection{Right modules and dualities}\label{Remarks on the right versions}

  Applying the same argument for $A^{\op}$ we can obtain  right versions of above results. 
 In this subsection, we discuss relationships between the previous results and their right versions 
 through dualities. 
 For this purpose first we fix notations and terminologies for right modules.

\subsubsection{The weighted Euler characteristic for right modules}

 Let $N \in \Dbmod{\Pa^{\op}}$.  
  We set \[
  \Euvect(N) : = (\Euch(N e_{1} ), \ldots, \Euch(N e_{r}))\]
   and regard it as a row vector. 
    The weighted Euler characteristic for   $v \in \kk Q_{0}$ is defined by the formula 
  \[
  {}^{v}\!\Euch(N) := \Euvect(N) v = \sum_{i \in Q_{0} }v_{i} \Euch(N e_{i}). 
  \]

  \subsubsection{The properties (I) and  (I') for right modules } 
 
The $\kk$-duality $\tuD(-)$ descents to the transpose of dimension vectors i.e, $\Euvect(\tuD(M)) = \Euvect(M)^{t}$. 
It follows that  ${}^{v}\!\Euch(\tuD(M)) = {}^{v}\!\Euch(M)$. Since the functor $\tuD(-): \Pa\mod \to \Pa^{\op}\mod$ gives a contravariant equivalence of categories 
which sends the preprojective modules to the preinjective modules and the preinjective modules to the preprojective modules, we obtain the following lemma.

\begin{lemma}\label{202103021753}
\begin{enumerate}[(1)]
\item 
An element  $v \in \kk Q_{0}$  is regular (for left modules)
if and only if it satisfies regularity for right modules, 
i.e., 
we have ${}^{v}\!\Euch(N) \neq 0$ in $\kk$ for all 
$N \in \ind \kk Q^{\op}$. 
\item 
An element  $v \in \kk Q_{0}$ is semi-regular (for left modules) 
if and only if it satisfies semi-regularity  for right modules, 
i.e., 
if we have ${}^{v}\!\Euch(N) \neq 0$ in $\kk$ for all 
prepeojective  $\Pa^{\op}$-modules $N$ and preinjective $\Pa^{\op}$ modules $N$.  
\end{enumerate}
\end{lemma}

  \subsubsection{The Coxeter matrix $\Phi_{\textup{right}}$ for right modules}
 We remark that the Coxeter matrix $\Phi_{\textup{right}}$ for right modules is different from that for left modules which we wrote $\Phi$. 
 Later we deal with left modules and right modules at the same time, for this we need to know  relationship between $\Phi$ and $\Phi_{\textup{right}}$.

  Let $C$ be the Cartan matrix of $\Pa = \kk Q$. 
  Then the Coxeter matrix for right modules is given as 
   $\Phi_{\textup{right}}: = - C^{-1} C^{t}$. Then, we have the equality below for all $N \in \Dbmod{\Pa^{\op}}$. 
  \[
  \Euvect(\nu_{1}(N)) = \Euvect(N) \Phi_{\textup{right}}. 
  \]
Observe that we have $\Psi= \Phi^{-t} = - C^{-1}C^{t} = \Phi_{\textup{right}}$. 
Consequently we have
\[
\begin{split}
&\Euvect(N \lotimes_{\Pa} \PPi_{1} ) = \Euvect(N) \Phi_{\textup{right}}^{-1} =\Euvect(N)\Psi^{-1}, \\  {}^{v}& \Euvect(N \lotimes_{\Pa} \PPi_{1}) = {}^{\Psi^{-1}(v)}\Euvect(N).
\end{split}
\]

In particular we deduce the following assertion. 
\begin{lemma}\label{202103021744}
Assume that $v \in \kk Q_{0}$ is an eigenvector of $\Psi$ with the eigenvalue $\lambda$. 
Then for $N \in \Dbmod{\Pa^{\op}}$, we have 
\[
{}^{v}\!\Euch(N \lotimes_{\Pa} \PPi_{1}) = \frac{1}{\lambda} {}^{v}\!\Euch(N). 
\]
\end{lemma}

  \subsubsection{}

  We give  relationships of the weighted Euler characteristic ${}^{v}\!\Euch(M)$ for $M \in \Dbmod{\Pa}$ and that of $\Pa$-dual 
  $M^{\lvvee} := \RHom_{\Pa}(M, \Pa) $ of $M$, which is an object of $\Dbmod{\Pa^{\op}}$. 
 
 We have an isomorphism $M^{\lvvee} \cong \tuD(\nu_{1}(M))[-1]$. 
  Consequently we have 
\begin{equation}\label{202104021807} 
\begin{split} 
&\Euvect(M^{\lvvee} ) = (-\Phi \Euvect(M))^{t}, \ \ 
 {}^{v}\!\Euch(M^{\lvvee}) =- {}^{v}\!\Euch(\nu_{1}(M)) = - {}^{\Psi^{-1}(v)}\Euch(M). 
  \end{split}
  \end{equation}

\subsubsection{Duality}

\begin{proposition}\label{202103051610}
Let $N \in \ind \Dbmod{\Pa^{\op}}$. 
Assume that ${}^{v}\!\Euch(N^{\rvvee} )\neq 0, {}^{v}\!\Euch( \tuD(\Pa) [-1] \lotimes_{\Pa} N^{\rvvee} ) \neq 0$. 
Then,  
there exists an isomorphism 
\[
{}^{v}\!\sfb_{N}: \RHom_{A^{\op}}(N \lotimes_{A} {}^{v}\!\YYM_{1}, \PPi_{1}) \to {}^{v}\!\YYM_{1} \lotimes_{A} N^{\rvvee}
\] 
that completes the following commutative diagram 
\begin{equation}\label{202103051621}
\begin{xymatrix}{
 (N[1], \PPi_{1}) \ar[r] \ar[d]_{\cong}  &
  (N\lotimes_{A} \PPi_{1}, \PPi_{1}) \ar[d]_{\cong} \ar[r] &
  (N\lotimes_{A} {}^{v}\!\YYM_{1}, \PPi_{1}) \ar[d]_{\cong}^{{}^{v}\!\sfb_{N}} \ar[r] &
  (N[1], \PPi_{1})[1] \ar[d]_{\cong} \\
A^{\vvee} \lotimes_{A} N^{\rvvee} 
\ar[r]_{  -y({}^{v}\!\ttheta_{N^{\rvvee} })} 
&  N^{\rvvee} \ar[r]_{{}^{v}\!\varrho_{N^{\rvvee}}} 
& {}^{v}\!\YYM_{1} \lotimes_{A} N^{\rvvee} \ar[r]_{-y^{-1}{}^{v}\!\ppi_{1, N^{\rvvee}}}
& \PPi_{1} \lotimes_{A} N^{\rvvee}
}\end{xymatrix} 
\end{equation}
where we use the  abbreviation  $(-,+) =\RHom_{\Pa^{\op}}(-, +)$, 
the top row is the exact triangle $\RHom( {}_{N}{}^{v}\!\sfAR, \PPi_{1})$ 
and we set
\[
y := \frac{{}^{v}\!\Euch(\tuD(A)[-1] \lotimes_{A} N^{\rvvee})}{{}^{v}\!\Euch(N^{\rvvee})}.
\]
\end{proposition}

\begin{proof}
First we claim that we have the following commutative diagram 
\begin{equation}\label{202105231814}
\begin{xymatrix}@C=70pt{
(N[1])^{\rvvee} \ar[d]_{\cong}\ar[r]^{(-{}_{N}\!\! {}^{v}\! \ttheta [1])^{\rvvee}} 
&  (N \lotimes \PPi_{1})^{\rvvee}\ar[d]_{\cong}\\ 
\tuD(A)[-1] \lotimes_{A}  A^{\vvee} \lotimes_{A} N^{\rvvee} 
\ar[r]^{  -y({}_{\tuD(A)[-1]}\!\!{}^{v}\!\ttheta_{N^{\rvvee} })} 
& \tuD(A) [-1]  \lotimes_{A} N^{\rvvee} 
}\end{xymatrix}
\end{equation} 
Indeed we obtain this diagram in the following way:  
\[
\begin{xymatrix}@C=70pt{
(N[1])^{\rvvee} \ar[d]_{\cong}\ar[r]^{(-{}_{N}\!\!{}^{v}\! \ttheta [1])^{\rvvee}} 
&  (N \lotimes \PPi_{1})^{\rvvee}\ar[d]_{\cong}\\ 
(A[1])^{\vvee} \lotimes_{A} \tuD(A) \lotimes_{A} N^{\rvvee} 
\ar[d]_{\cong}\ar[r]^{ (-{}^{v}\!\ttheta [1])^{\vvee}_{\tuD(A) \lotimes N^{\rvvee} }} 
&  (A^{\vvee} [1])^{\vvee} \lotimes_{A} \tuD(A) \lotimes_{A} N^{\rvvee} \ar[d]_{\cong}\\ 
A^{\vvee} \lotimes_{A} \tuD(A)[-1] \lotimes_{A} N^{\rvvee} 
\ar[d]_{\cong}\ar[r]^{ (-{}^{v}\!\ttheta)^{\vvee}_{\tuD(A)[-1]  \lotimes N^{\rvvee} }} 
&  (A^{\vvee} )^{\vvee} \lotimes_{A} \tuD(A) [-1]  \lotimes_{A} N^{\rvvee} \ar[d]_{\cong}\\ 
A^{\vvee} \lotimes_{A} \tuD(A)[-1] \lotimes_{A} N^{\rvvee} 
\ar[d]_{\cong}^{\gamma_{\tuD(A)[-1]} } \ar[r]^{ -{}^{v}\!\ttheta_{\tuD(A)[-1]  \lotimes N^{\rvvee} }} 
& \tuD(A) [-1]  \lotimes_{A} N^{\rvvee} \ar@{=}[d]\\ 
\tuD(A)[-1] \lotimes_{A}  A^{\vvee} \lotimes_{A} N^{\rvvee} 
\ar[r]^{  -y({}_{\tuD(A)[-1]}\!\!{}^{v}\!\ttheta_{N^{\rvvee} })} 
& \tuD(A) [-1]  \lotimes_{A} N^{\rvvee} 
}\end{xymatrix}
\]
where for the commutativity of the first square 
we use a  canonical isomorphism   
\[
(L \lotimes_{A} X)^{\rvvee} 
\cong X^{\rvvee} \lotimes_{A} L^{\rvvee} \cong X^{\vvee} \lotimes_{A} \tuD(A) \lotimes L^{\rvvee}. 
\]
which is natural in $L \in \sfD(A^{\op})$ and $X \in \sfD(A^{\mre})$ 
(see Section \ref{section: natural isomorphisms}). 
For the third square we use \eqref{202105111259}.  
Finally the commutativity of the fourth  square follows from Proposition \ref{202101131644}.

We apply $\PPi_{1}\lotimes_{\Pa} - $ to \eqref{202105231814} and use
  canonical isomorphisms 
\[
\RHom_{A^{\op}}(- ,\PPi_{1} ) \cong \PPi_{1} \lotimes_{A} (-)^{\rvvee}, \ 
\PPi_{1} \lotimes_{\Pa} \tuD(A) [-1] \cong \Pa,  
\]
then we obtain a commutative square that appeared in the  diagram \eqref{202103051621}.  
\end{proof}

We remark that $\sfb_{N}$ is not uniquely determined, but it is unique modulo $\rad$. 
We point out that $\sfb_{N}$ has functoriality modulo $\rad$.

\subsection{Bimodules version}

In this subsection, we establish isomorphisms over $\Pa^{\mre}$ that involve ${}^{v}\!\YYM_{1}$. 
The results of this sections are  used in Section \ref{section: the universal ladder} and the reader can postpone them until then. 

\subsubsection{} 
Combining Lemma \ref{202101131703} and Theorem \ref{202101081603} we deduce the following result.

\begin{proposition}\label{2021010816031}
Let $v \in \kk Q_{0}$ and $T$ be a two-sided tilting complex over $\Pa$. 
Then, 
there exists an isomorphism 
\[
{}^{v}\!\sfc_{ T}:  {}^{v}\!\YYM_{1} \lotimes_{\Pa} T \to T  \lotimes_{\Pa} {}^{\underline{T}^{t}(v)}\!\YYM_{1}
\]
in $\sfD(\Pa^{\mre})$ that gives the following isomorphism of exact triangles:
\[
\begin{xymatrix}@C=60pt{
\Pa^{\vvee} \lotimes_{\Pa} T \ar[r]^{{}^{v}\!\ttheta_{T} }  \ar[d]_{\gamma_{T}} 
& 
T   \ar[r]^{{}^{v}\!\rrho_{T}} \ar@{=}[d]  &
{}^{v}\!\YYM_{1} \lotimes_{A}T \ar[d]^{{}^{v}\!\sfc_{T} } \ar[r]^{{}^{v}\!\ppi_{1, T} } &
\PPi_{1} \lotimes_{\Pa} T  \ar[d]^{\tilde{\gamma}_{T}}  \\
T \lotimes_{\Pa} \Pa^{\vvee} \ar[r]_{ {}_{T}  {}^{\underline{T}^{t}(v)} \ttheta } & 
T  \ar[r]_{  {}_{T} ({}^{\underline{T}^{t}(v)} \rrho) }  &
T \lotimes_{\Pa}  {}^{\underline{T}^{t}(v)} \!\YYM_{1} \ar[r]_{ {}_{T} ({}^{\underline{T}^{t}(v)}\ppi_{1}) } &
T \lotimes_{\Pa} \PPi_{1}.   \\
}\end{xymatrix}
\]
where the right vertical morphism $\tilde{\gamma}_{T}$ is the canonical isomorphism 
interchanging $T$ and $\PPi_{1} = \Sigma \Pa^{\vvee}$. 
Here we  use $\Sigma$ to denote the shift functor instead of $[1]$. 
Moreover, $\tilde{\gamma}_{T}$ is given by the composition of the following canonical morphisms:
\[
\tilde{\gamma}_{T} : 
\PPi_{1} \lotimes_{\Pa} T 
= \Sigma \Pa^{\vvee} \lotimes_{\Pa} T 
\xrightarrow{\ \Sigma(\gamma_{T})\ } 
\Sigma T \lotimes_{\Pa} \Pa^{\vvee} 
\xrightarrow{\ (\sigma^{-1}_{T})_{\Pa^{\vvee}}\ } 
T \lotimes_{\Pa} \Sigma \Pa^{\vvee} 
= T \lotimes_{\Pa} \PPi_{1}.
\]
where $\sigma_{T}$ denotes the canonical isomorphism 
$T \lotimes_{\Pa} \Sigma \Pa \to \Sigma T$ 
(see Section~\ref{section: Happel's criterion}).
\end{proposition}

In the case where the weight $v \in \kk Q_{0}$ is an eigenvector of $\underline{T}^{t}$ we obtain the following corollary 
by using the isomorphism \eqref{202102222229}.

\begin{corollary}\label{2021010816032}
Let $v \in \kk Q_{0}$ and $T$ be a two-sided tilting complex over $\Pa$. 
Assume that $v$ is an eigenvector of  $\underline{T}^{t}$ with the eigenvalue $\lambda$. 
Then, 
there exists an isomorphism 
\[{}^{v}\!\sfc'_{T}: {}^{v}\!\YYM_{1} \lotimes_{\Pa} T \to T \lotimes_{\Pa} {}^{v}\!\YYM_{1}
\] in $\sfD(\Pa^{\mre})$ 
that gives the following isomorphism of exact triangles:
\[
\begin{xymatrix}@C=60pt{
\Pa^{\vvee} \lotimes_{\Pa} T \ar[r]^{{}^{v}\!\ttheta_{T} }  \ar[d]_{\gamma_{T}} 
& 
T   \ar[r]^{{}^{v}\!\rrho_{T}} \ar@{=}[d]  &
{}^{v}\!\YYM_{1} \lotimes_{A}T \ar[d]^{{}^{v}\!\sfc'_{T} } \ar[r]^{{}^{v}\!\ppi_{1, T} } &
\PPi_{1} \lotimes_{\Pa} T  \ar[d]^{\tilde{\gamma}_{T}}   \\
T \lotimes_{\Pa} \Pa^{\vvee} \ar[r]_{  \lambda {}_{T}  {}^{v}\!\ttheta } & 
T  \ar[r]_{  {}_{T} {}^{v}\!\rrho }  &
T \lotimes_{\Pa}  {}^{v}\!\YYM_{1} \ar[r]_{ \lambda^{-1} {}_{T}{}^{v}\!\ppi_{1}} &
T \lotimes_{\Pa} \PPi_{1}   \\
}\end{xymatrix}
\]
\end{corollary}

\subsubsection{}

We collect the case $T = \PPi_{1}$ of Corollary \ref{2021010816032}, since it plays a key role in Section \ref{section: the universal ladder}.

Let $\Phi $ be the Coxeter matrix of $\kk Q$ for left modules. 
For simplicity we set $\Psi := \Phi^{-t}= \underline{\PPi_{1}}^{t}$.

\begin{corollary}\label{20210108160321}
Assume that $v \in \kk Q_{0}$ is an eigenvector of  $\Psi$ with the eigenvalue $\lambda$. 
Then, 
there exists an isomorphism 
\[{}^{v}\!\sfc'_{\PPi_{1}}: {}^{v}\!\YYM_{1} \lotimes_{\Pa} \PPi_{1} \to \PPi_{1} \lotimes_{\Pa} {}^{v}\!\YYM_{1}
\] in $\sfD(\Pa^{\mre})$ 
that gives the following isomorphism of exact triangles:
\begin{equation}\label{202306211412}
\begin{xymatrix}@C=60pt{
\Pa^{\vvee} \lotimes_{\Pa} \PPi_{1} \ar[r]^{{}^{v}\!\ttheta_{\PPi_{1}} }  \ar[d]_{\gamma_{\PPi_{1}}} 
& 
\PPi_{1}   \ar[r]^{{}^{v}\!\rrho_{\PPi_{1}}} \ar@{=}[d]  &
{}^{v}\!\YYM_{1} \lotimes_{A}\PPi_{1} \ar[d]^{{}^{v}\!\sfc'_{\PPi_{1}} } \ar[r]^{{}^{v}\!\ppi_{1, \PPi_{1}} } &
\PPi_{1} \lotimes_{\Pa} \PPi_{1}  \ar@{=}[d]   \\
\PPi_{1} \lotimes_{\Pa} \Pa^{\vvee} \ar[r]_{  \lambda {}_{\PPi_{1}}  {}^{v}\!\ttheta } & 
\PPi_{1}  \ar[r]_{  {}_{\PPi_{1}} {}^{v}\!\rrho }  &
\PPi_{1} \lotimes_{\Pa}  {}^{v}\!\YYM_{1} \ar[r]_{ \lambda^{-1} {}_{\PPi_{1}}{}^{v}\!\ppi_{1}} &
\PPi_{1} \lotimes_{\Pa} \PPi_{1}   \\
}\end{xymatrix}
\end{equation}
\end{corollary}

\begin{proof}
The only non-trivial point is to show that $\tilde{\gamma}_{\PPi_{1}} = \id$, 
which follows from Lemma \ref{202102171838}(1) with $n = 2$.
\end{proof}

\subsubsection{The right duality}

Recall that
we set $\Psi := \Phi^{-t}= \underline{\nu_{1}}^{-t}$,
and 
that  $\nu_{1} =\tuD(A)[-1]\lotimes_{\Pa} -$.

\begin{lemma}\label{202102230750} 
There exists an isomorphism 
\[{}^{v}\!\sfd: ({}^{v}\!\YYM_{1})^{\rvvee} \to \tuD(A)[-1]\lotimes_{\Pa} ({}^{\Psi^{-1}(v)} \!\YYM_{1}) \]
in $\sfD(\Pa^{\mre})$  
that completes the  following commutative diagram
\begin{equation}\label{202105232003}
\begin{xymatrix}@C=40pt{
(A[1])^{\rvvee} \ar[d]_{\cong}\ar[r]^{(-{}^{v}\!\ttheta [1])^{\rvvee}} &
\PPi_{1}^{\rvvee}\ar[d]_{\cong}  \ar[r]^{ ({}^{v}\!\ppi_{1})^{\rvvee} } & 
({}^{v}\!\YYM_{1})^{\rvvee} \ar[d]^{{}^{v}\!\sfd} \ar[r]^{ ({}^{v}\!\rrho)^{\rvvee}} & 
A^{\rvvee} \ar[d]^{\cong} \\ 
\tuD(A)[-1] \lotimes_{A}  A^{\vvee}  
\ar[r]_-{  -{}_{\tuD(A)[-1] }{}^{\Psi^{-1}(v)}\ttheta } 
& \tuD(A) [-1]   \ar[r]_-{{}_{\tuD(A)[-1]} {}^{\Psi^{-1}(v)} \rrho} & 
\tuD(A)[-1] \lotimes_{\Pa} ({}^{\Psi^{-1}(v)} \!\YYM_{1} )  \ar[r]_-{- {}_{\tuD(A)[-1] } {}^{\Psi^{-1}(v)}\ppi_{1} }
 & 
 \tuD(A) [-1] \lotimes_{\Pa} \PPi_{1}.
}\end{xymatrix} 
\end{equation}
\end{lemma} 

\begin{proof}
We  can verify commutativity of  the left most  square of \eqref{202105232003} in the following way  
\[ 
\begin{xymatrix}@C=50pt{
(A[1])^{\rvvee} \ar[d]_{\cong}\ar[r]^{(-{}^{v}\!\ttheta [1])^{\rvvee}} 
&  \PPi_{1}^{\rvvee}\ar[d]_{\cong}\\ 
(A[1])^{\vvee} \lotimes_{A} \tuD(A) 
\ar[d]_{\cong}\ar[r]^{ (-{}^{v}\!\ttheta [1])^{\vvee}_{\tuD(A) }} 
&  (A^{\vvee} [1])^{\vvee} \lotimes_{A} \tuD(A) \ar[d]_{\cong}\\ 
A^{\vvee} \lotimes_{A} \tuD(A)[-1]  
\ar[d]_{\cong}\ar[r]^{ (-{}^{v}\!\ttheta)^{\vvee}_{\tuD(A)[-1]   }} 
&  (A^{\vvee} )^{\vvee} \lotimes_{A} \tuD(A) [-1]   \ar[d]_{\cong}\\ 
A^{\vvee} \lotimes_{A} \tuD(A)[-1] 
\ar[d]_{\cong}^{\gamma_{\tuD(A)[-1]} } \ar[r]^{ -{}^{v}\!\ttheta_{\tuD(A)[-1]   }} 
& \tuD(A) [-1] \ar@{=}[d]\\ 
\tuD(A)[-1] \lotimes_{A}  A^{\vvee}  
\ar[r]^{  - {}_{\tuD(A)[-1] }{}^{\Psi^{-1}(v)}\ttheta } 
& \tuD(A) [-1]   
}\end{xymatrix}
\]
where for the commutativity of the first square  we use  a canonical isomorphism
\[
X^{\rvvee} 
\cong X^{\vvee} \lotimes_{A} \tuD(A)
\]
which is natural  for $X \in \Dbmod{\Pa^{\mre}}$  (see Section \ref{section: natural isomorphisms}). 
For the third square we use \eqref{202105111259}.  
Finally the commutativity of the fourth  square follows from Theorem \ref{202101081603}.
\end{proof}

Since $\RHom_{\Pa^{\op}}(-, \PPi_{1}) \cong \PPi_{1} \lotimes_{\Pa} (-)^{\rvvee}$, 
applying $\PPi_{1} \lotimes_{\Pa} -$ to the diagram of Lemma \ref{202102230750} 
we deduce the following lemma. 

\begin{lemma}\label{2021022307501} 

There exists an isomorphism 
\[
{}^{v}\!\sfe: \RHom_{\Pa^{\op}} ( {}^{v}\!\YYM_{1}, \PPi_{1} ) \to {}^{\Psi^{-1}(v)} \!\YYM_{1} 
\]
in $\sfD(\Pa^{\mre})$  
that completes the  following commutative diagram
\[
\begin{xymatrix}@C=60pt{
(A[1], \PPi_{1}) \ar[r]^{(- {}^{v}\!\ttheta[1], \PPi_{1} ) }\ar[d]_{\cong}  &
 (\PPi_{1}, \PPi_{1}) \ar[d]_{\cong} \ar[r]^{({}^{v}\!\ppi_{1}, \PPi_{1} ) } &
 ({}^{v}\!\YYM_{1}, \PPi_{1}) \ar[d]_{\cong}^{{}^{v}\!\sfe} \ar[r]^{ ({}^{v}\!\rrho, \PPi_{1} ) } &
 (A, \PPi_{1}) \ar[d]_{\cong} \\
A^{\vvee} 
\ar[r]_{  -{}^{\Psi^{-1}(v)}\!\ttheta} 
& A  \ar[r]_{{}^{\Psi^{-1}(v)}\!\rrho } 
& {}^{\Psi^{-1}(v)} \!\YYM_{1} \ar[r]_{ -{}^{\Psi^{-1}(v)}\!\ppi_{1} }
& \PPi_{1} .
}\end{xymatrix} 
\]
where we use the  abbreviation  $(-,+) =\RHom_{\Pa^{\op}}(-, +)$.

\end{lemma} 

Using \eqref{202102222229}, we deduce the following corollary. 
We remark that  for the later quotation, in the statement  we use    $\Psi$.

\begin{corollary}\label{2021022307502} 
Assume that $v$ is an eigenvector of $\Psi$ with the eigenvalue $\lambda$. 
Then, there exists an isomorphism 
\[
{}^{v}\!\sfe': \RHom_{\Pa^{\op}} ( {}^{v}\!\YYM_{1}, \PPi_{1} ) \to {}^{v}\!\YYM_{1} 
\]
in $\sfD(\Pa^{\mre})$  
that completes the  following commutative diagram
\[
\begin{xymatrix}@C=60pt{
(A[1], \PPi_{1}) \ar[r]^{(- {}^{v}\!\ttheta[1], \PPi_{1} ) }\ar[d]_{\cong}  &
 (\PPi_{1}, \PPi_{1}) \ar[d]_{\cong} \ar[r]^{({}^{v}\!\ppi_{1}, \PPi_{1} ) } &
 ({}^{v}\!\YYM_{1}, \PPi_{1}) \ar[d]_{\cong}^{{}^{v}\!\sfe'} \ar[r]^{ ({}^{v}\!\rrho, \PPi_{1} ) } &
 (A, \PPi_{1}) \ar[d]_{\cong} \\
A^{\vvee} 
\ar[r]_{  -\lambda^{-1} {}^{v}\!\ttheta} 
& A  \ar[r]_{{}^{v}\!\rrho } 
& {}^{v}\!\YYM_{1} \ar[r]_{ -\lambda {}^{v}\!\ppi_{1} }
& \PPi_{1} .
}\end{xymatrix} 
\]
where we use the  abbreviation  $(-,+) =\RHom_{\Pa^{\op}}(-, +)$.

\end{corollary}

\section{The derived quiver Heisenberg algebras}\label{section: the derived quiver Heisenberg algebras}

In this section \ref{section: the derived quiver Heisenberg algebras}, 
we introduce the derived quiver Heisenberg algebra ${}^{v}\!\YYM(Q)$ of a quiver $Q$ and establish its basic properties. 
It may be worth mentioning that  all the results of this section hold for a finite quiver $Q$ which is not necessarily acyclic.

\subsection{The derived preprojective algebras}\label{derived preprojective algebra}
Before introducing the derived quiver Heisenberg algebra,  
we recall the construction of the derived preprojective algebra $\PPi(Q)$ of $Q$.

\subsubsection{The preprojective algebra of a quiver } 

Let $Q$ be a finite quiver and $\overline{Q}$ its double quiver. 
We may  identify the arrow module $\kk \overline{Q}_{1}$ with $V \oplus V^{*}$ 
and 
the path algebra $\kk \overline{Q}$ of $\overline{Q}$ with the tensor algebra  $\sfT_{\Pa_{0}} (V \oplus V^{*})$, 
which is isomorphic to 
$\sfT_{\Pa}(\Pa V^{*} \Pa)$ (see e.g, \cite[Lemma 2.1]{Grant-Iyama}). 
\begin{equation}\label{202006121845}
\kk \overline{Q} = \sfT_{\Pa_{0}}(V \oplus V^{*} ) \cong \sfT_{\Pa}(\Pa V^{*} \Pa) 
= \Pa \oplus \Pa V^{*} \Pa \oplus \Pa V^{*} \Pa V^{*} \Pa \oplus \cdots. 
\end{equation}
Under this isomorphism, the mesh relation $\rho_{i}$ \eqref{mesh relation} may be identified with the element below of $\Pa V^{*} \Pa$ 
\begin{equation}\label{201911262}
\rho_{i} = \sum_{\alpha: t(\alpha) = i} \alpha \otimes \alpha^{*} \otimes e_{i} 
- \sum_{\alpha: h(\alpha) =i } e_{i} \otimes \alpha^{*} \otimes \alpha. 
\end{equation}
Therefore,  we obtain the following isomorphism by which we identify these two algebras in the sequel. 
\[
\Pi(Q) \cong  \sfT_{A} (AV^{*} A)/(\rho). 
\]

\subsubsection{The derived preprojective algebras}

The derived preprojective algebra $\PPi=\PPi(Q)$ of $Q$ is defined to be 
the tensor dg-algebra $\sfT_{\Pa}( \PPa^{\vee}[1])$ of $ \PPa^{\vee}[1]$ over $\Pa$, 
that is 
a dg-algebra whose underlying algebra is
the tensor algebra $\sfT_{\Pa}( \PPa^{\vee}[1])$ of $ \PPa^{\vee}[1]$ over $\Pa$ 
and 
the differential is that of induced from the differential of the complex $ \PPa^{\vee}[1]$. 

For $i \in Q_{0}$, we denote by $s_{i}$ the element of the generator of $\Pa \Pa[1] = \Pa ( \Pa_{0}[1]) \Pa$ 
corresponding to $e_{i} \in \Pa_{0}$. 
We set $S := \Pa_{0}[1] = \bigoplus_{ i \in Q_{0}} \kk s_{i}$ and $s := \sum_{ i \in Q_{0}} s_{i}$.

We give  an explicit description of $\PPi$. 
The underlying algebra is 
the free algebra over $\Pa_{0}$ generated by $\alpha, \alpha^{*}$ and $s_{i}$. 
In other words, it is the tensor algebra 
$\sfT_{\Pa_{0}}(V \oplus V^{*} \oplus S) \cong \sfT_{\Pa}(\Pa V^{*} \Pa \oplus \Pa S \Pa)$. 
The differential is given by 
\[
d (\alpha ) := 0, d(\alpha^{*} ) := 0, d(s_{i}) := - \rho_{i}. 
\]
The values of $d$ for general homogeneous elements are determined from the Leibniz rule 
$d(xy) = d(x) y + ( -1)^{|x|}x d(y)$.  
Observe that the canonical morphism $\PPi \to \Pi$ induces an isomorphism  
\[
\tuH^{0}(\PPi) \xrightarrow{\cong } \Pi.
\]

We equip  $\PPi$ with an extra  grading, which we call the $*$-\emph{grading}, 
in the following way: 
\[
\deg^{*} e_{i} : =0, \ \deg^{*} s_{i} := 1 \ (i \in Q_{0}), \ 
\deg^{*} \alpha := 0, \  \deg^{*} \alpha^{*} := 1 \  ( \alpha \in Q_{1}).
\]

We give the table of the cohomological degrees and the $*$-degrees of the generators:

\centerline{
\begin{tabular}{c|c|c|c|c}
& $e_{i}$ & $\alpha$ & $\alpha^{*}$ & $s_{i}$ \\ \hline
$\textup{ch}\deg $& $0$ & $0$  & $0$ & $-1$ \\ \hline
$\deg^{*}$ & $0$ & $0$ & $1$ & $1$ 
\end{tabular}
}

\[
\begin{xymatrix}{ 
i \ar@(ld, lu)^{s_{i}}  \ar@/^10pt/[rr]^{\alpha} &&
 j  \ar@(rd,ru)_{s_{j}} \ar@/^10pt/[ll]^{\alpha^{*}} 
}\end{xymatrix}
\]

It is clear that the differential $d$ preserves the $*$-grading and 
$\PPi$ is a $*$-graded dg-algebra. 
It can be checked that the $*$-degree on $\PPi$ coincides with the tensor degree on $\PPi = \sfT_{\Pa}( \PPa^{\vee}[1])$. 
Hence 
if we denote  the $*$-degree $n$ part of $\PPi$ by $\PPi_{n}$, 
then $\PPi_{n} =   \PPa^{\vee}[1] \otimes_{\Pa} \cdots \otimes_{\Pa}  \PPa^{\vee}[1]$ ($n$-times). 
Thus in particular 
\begin{equation}\label{202006122059}
\PPi_{1} = \PPa^{\vee}[1] \cong  \left( \Pa V^{*} \Pa \oplus \Pa \Pa [1], \begin{pmatrix} 0 & -\hat{\rho}\uparrow  \\ 0 & 0 \end{pmatrix}\right)
\end{equation}
where for the second isomorphism we use Lemma \ref{202006121909}.

\subsection{The derived quiver Heisenberg algebras}

\subsubsection{The derived quiver Heisenberg algebras}

We define the \emph{derived quiver Heisenberg algebra}.   
\begin{definition}[The derived quiver Heisenberg algebras]\label{202111201445}
Let $Q$ be a finite quiver and $v \in \kk^{\times} Q_{0}$. 
We set
 \[
 {}^{v}\!\varrho := \sum_{i \in Q_{0}} v_{i}^{-1}\rho_{i}, 
 {}^{v}\!\eta_{a} := [a, {}^{v}\!\varrho]
 \] for $a \in \overline{Q}$.

We define the \emph{derived quiver Heisenberg algebra} $ {}^{v}\!\YYM(Q)$ in the following way. 

We set $A_{0} := \kk Q_{0}, V := \kk Q_{1},  V^{*} := \tuD(V)$ and $A := \kk Q$. 
Let $V^{\circ} := V^{*}[1],$  $V^{\circledast }:= V[1]$ and $T := \kk Q_{0}[2]$. 
The underlying cohomological graded algebra of ${}^{v}\!\YYM(Q)$ is defined to be 
\begin{equation}\label{20191127}
\sfT_{\Pa_{0}}(V \oplus V^{*} \oplus V^{\circ} \oplus V^{\circledast} \oplus T)  
\cong \sfT_{\Pa}(\Pa V^{*} \Pa \oplus \Pa V^{\circ} \Pa \oplus \Pa V^{\circledast} \Pa \oplus 
\Pa T \Pa ).
\end{equation}

The differential $d$ is  defined in the following way. 
We denote by $\alpha^{\circ}, \alpha^{\circledast}$ the elements of $V^{\circ}, V^{\circledast}$ corresponding to $\alpha \in Q_{1}$. 
We denote by $t_{i}$ the element of $T$ corresponding to $i \in Q_{0}$. 
We set $t := \sum_{ i \in Q_{0}} t_{i}$. 
Then the underlying cohomological graded algebra \eqref{20191127}
 is freely generated by $\alpha, \alpha^{*}, \alpha^{\circ}, \alpha^{\circledast}, t_{i}$ 
for $\alpha \in Q_{1}$ and $i\in Q_{0}$. 
The values of $d$ for these generators are defined by the formulas:
\begin{equation}\label{the differentials of derived QHA}
\begin{split}
&d(\alpha) := 0, d(\alpha^{*}) := 0,  d(\alpha^{\circ}) :  = - {}^{v}\!\eta_{\alpha^{*}},
 d(\alpha^{\circledast}) :  =  {}^{v}\!\eta_{\alpha}, \\
& d(t_{i}) := \sum_{\alpha \in Q_{1} } e_{i}[\alpha, \alpha^{\circ}]e_{i} + \sum_{\alpha \in Q_{1}}  e_{i}[\alpha^{*}, \alpha^{\circledast}] e_{i} \\
& \ \ \ \ \ \ \ = \sum_{\alpha: t(\alpha) = i} \alpha \alpha^{\circ}-  \sum_{\alpha: h(\alpha) = i} \alpha^{\circ} \alpha +
\sum_{\alpha: h(\alpha) = i} \alpha^{*} \alpha^{\circledast}-  \sum_{\alpha: t(\alpha) = i} \alpha^{\circledast} \alpha^{*}.
\end{split}
\end{equation}
The values of $d$ for general homogeneous elements are determined from the Leibniz rule 
$d(xy) = d(x) y + ( -1)^{|x|}x d(y)$. 
\end{definition} 

From now until the end of this section, we fix an element $v=(v_{i}) \in \kk^{\times}  Q_{0}$ 
and use abbreviation such as $\YYM = {}^{v}\!\YYM(Q), \varrho = {}^{v}\!\varrho$ and $\eta_{a} = {}^{v}\!\eta_{a}$. 
Moreover in the sequel, similarly we omit $v$ and use similar abbreviations.

We equip  $\YYM$ with the  $*$-grading  in the following way.
\[
\begin{split} 
\deg^{*} V := 0, \deg^{*} V^{*} = 1, \deg^{*} V^{\circledast} := 1,  \deg^{*} V^{\circ} := 2,  \deg^{*} T:= 2. 
\end{split}
\]

We give the table of the cohomological degrees and the $*$-degrees of the generators:

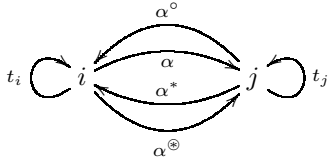
\centerline{
\begin{tabular}{c|c|c|c|c|c|c}
& $e_{i}$ & $\alpha$ & $\alpha^{*}$ &  $\alpha^{\circledast}$ &$\alpha^{\circ} $ & $t_{i}$ \\ \hline
$\textup{ch}\deg $& $0$ & $0$  & $0$ & $-1$  & $-1$ & $-2$ \\ \hline
$\deg^{*}$ & $0$ & $0$ & $1$ & $1$ & $2$ &  $2$ 
\end{tabular}}
\[
\begin{xymatrix}{ 
i \ar@(ld, lu)^{t_{i}} \ar@/_20pt/[rr]_{\alpha^{\circledast}} \ar@/^10pt/[rr]_{\alpha} && j  \ar@(rd,ru)_{t_{j}} \ar@/^10pt/[ll]_{\alpha^{*}} 
\ar@/_20pt/[ll]_{\alpha^{\circ}}
}\end{xymatrix}
\]

We denote by $\YYM_{n}$ the component of $\YYM$ having  $*$-degree $n$.
For example, $\YYM_{0} = \Pa = \kk Q$. 
It is straightforward to check  that the differential $d$ preserves the $*$-degree and $\YYM$ is a $*$-graded dg-algebra. 
Therefore $\YYM_{n}$ has a canonical structure of complex of bimodules over $\Pa$.

The following observation plays an important  role. 

\begin{observation}\label{202006122126}
 The $*$-degree $1$ part $\YYM_{1}$ and the differential $d_{\YYM_{1}}$ are   of the form 
 \[
 \YYM_{1} =\left( \Pa V^{*} \Pa \oplus  \Pa V^{\circledast}\Pa,  \begin{pmatrix} 0 & \hat{\eta}' \\ 0 & 0 \end{pmatrix} \right) 
\] 
where $\hat{\eta}'$ is the $\Pa^{\mre}$-homomorphism of cohomological degree $1$ 
such that $\hat{\eta}'(\alpha^{\circledast}) = \eta_{\alpha}$. 
\end{observation}

We denote by $\sfC(\YYM \Gr), \sfK(\YYM \Gr) $ and $\sfD(\YYM \Gr)$ 
the category of $*$-graded dg-$\YYM$-modules, its homotopy category and its derived category. 
We write $(n)$ to denote $*$-degree shift by $n \in \ZZ$. 

The following lemma is clear from the definition.

\begin{lemma}\label{202006121937}
The canonical map $\YYM \to \YM $ 
gives an isomorphism $\tuH^{0}(\YYM(Q)) \cong \YM(Q)$ 
of algebras. 
\end{lemma}

\subsubsection{The derived quiver Heisenberg algebras as Ginzburg dg-algebras}\label{202102071454} 
In this section \ref{202102071454}, we assume that $\chara \kk \neq 2$.

We set 
\[
W: = -\frac{1}{2} \varrho \rho = -\frac{1}{2} \sum_{i \in Q_{0}} v_{i}^{-1} \rho_{i}^{2}. 
\]
By a straightforward calculation (or using  cyclic Leibniz rule \cite[Lemma 3.8]{DWZ}), we can check  that 
\[
\partial_{\alpha}( W) =  -\eta_{\alpha^{*}}, \ \ \ \partial_{\alpha^{*}} (W)=  \eta_{\alpha}.
\]
Therefore, 
the quiver Heisenberg algebra $\YM(Q)$ is 
 the  Jacobi algebra of the double quiver $\overline{Q}$ with the potential $W$.
\[
\YM(Q) = \cP\left (\overline{Q}, W \right).\]

\begin{remark}
In \cite[p.604]{QVB}, 
another quiver with potential $(\hat{Q}, W')$ such that $\cP(\hat{Q}, W') = \YM(Q)$ is given. 

In subsequent work we prove that the derived quiver Heisenberg algebra and the Ginzburg dg-algebras $\cG(\hat{Q}, W')$ of this quiver with potential, are quasi-isomorphic to each other. 
In this sense, a point of this paper is that the derived quiver Heisenberg algebra has a smaller number of generators than that of $\cG(\hat{Q}, W')$. 
\end{remark}

It is straight forward to check that the derived quiver Heisenberg algebra $\YYM(Q)$ is isomorphic to 
the Ginzburg dg-algebra $\cG\left(\overline{Q}, W \right)$. 
\[
\YYM(Q)  = \cG\left(\overline{Q}, W \right).
\]
The point here is that  although the potential $W =- \frac{1}{2} \varrho\rho$ contains the fraction $\frac{1}{2}$, 
the differentials of $\cG\left(\overline{Q}, W \right)$ do not. 
Therefore,  the definition of the differentials even works for the case $\chara \kk =2$.

By Ginzburg, Keller and Van den Bergh \cite{Ginzburg, Keller: Calabi-Yau completion}, 
the Ginzburg dg-algebras for quivers with potentials  are $3$-Calabi-Yau. 
Hence, as a special case, we have 

\begin{proposition}\label{202111261255}
Assume that $\chara \kk \neq 2$. 
Then the derived quiver Heisenberg algebra $\YYM$ is $3$-Calabi-Yau. 
\end{proposition}

Later in Theorem \ref{202112122233}, 
we prove that $\YYM$ is $3$-Calabi-Yau even in the case $\chara \kk = 2$.

\subsubsection{The  morphism $\tilde{\pi}$}

We introduce the elements  $\vars_{i} := v_{i}^{-1} s_{i}$  and
 $\vars := \sum_{i \in Q_{0}} \vars_{i} =\sum_{i \in Q_{0}} v_{i}^{-1} s_{i}$ of $\PPi$. 
Note that $d(\vars) = -\varrho$. 
We also introduce the elements  $\vart_{i} := v_{i}^{-1} t_{i}$  and $\vart := \sum_{i \in Q_{0}}\vart_{i} =\sum_{i \in Q_{0}} v_{i}^{-1} t_{i}$ of $\YYM$.

We define a morphism $\tilde{\pi}: \YYM \to \PPi$ of algebras 
 over $\Pa_{0}$. 
On the generators, $\tilde{\pi}$ is defined by the formula 
\[
\begin{split}
&\tilde{\pi}(\alpha) := \alpha, \tilde{\pi}(\alpha^{*}) := \alpha^{*}, \\
&
\tilde{\pi}(\alpha^{\circledast}) := -[\alpha, \vars] =- v_{h(\alpha)}^{-1} \alpha s_{h(\alpha)} + v_{t(\alpha)}^{-1}s_{t(\alpha)} \alpha,\\
&\tilde{\pi}(\alpha^{\circ}) := [\alpha^{*},\vars] =  v_{t(\alpha)}^{-1} \alpha^{*} s_{t(\alpha)} - v_{h(\alpha)}^{-1} s_{h(\alpha)} \alpha^{*},\\
&\tilde{\pi}(\vart_{i}): = -\vars^{2}_{i} = - v_{i}^{-2}s_{i}^{2}.
\end{split}
\]
Since $\YYM$ is freely generated by these generators, the above formulas defines a morphism $\ppi: \YYM \to \PPi$ of algebras. 

Observe that $\tilde{\pi}$ preserves cohomological degree and $*$-degree. 
The  $*$-degree $0$ part $\tilde{\pi}_{0}: \YYM_{0} \to \PPi_{0}$ is just the identity map of $\Pa = \YYM_{0} = \PPi_{0}$.

\begin{lemma}\label{202102041200}
The morphism $\ppi$ is compatible with the differentials and hence is 
a morphism of $*$-graded  dg-algebras. 
\end{lemma}

\begin{proof}
It is enough to check that the equation $\ppi d= d \ppi $ on the generators. 

\[
\ppi d( \alpha) = 0 = d \ppi (\alpha), \ppi d( \alpha^{*}) = 0 = d \ppi (\alpha^{*}).
\]

\[
\begin{split}
&d \ppi(\alpha^{\circledast}) =-  d([\alpha, \vars]) =- [d(\alpha), \vars] - [\alpha, d(\vars)] = 
[\alpha, \varrho],  \\
& \ppi d(\alpha^{\circledast}) = \ppi(\eta_{\alpha}) = \ppi([\alpha, \varrho]) = [\alpha, \varrho]. 
\end{split}
\]

\[
\begin{split}
&d \ppi(\alpha^{\circ}) = d([\alpha^{*}, \vars]) = [d(\alpha^{*}), \vars] + [\alpha^{*}, d(\vars)] = -  [\alpha^{*}, \varrho],  \\
& \ppi d(\alpha^{\circ}) = - \ppi(\eta_{\alpha^{*}}) = - \ppi([\alpha^{*}, \varrho]) = - [\alpha^{*}, \varrho].
\end{split}
\]

\[
\begin{split}
& d \ppi(\vart_{i}) = - d(\vars_{i}^{2} ) = -d(\vars_{i})\vars_{i} + \vars_{i}d(\vars_{i}) = \varrho_{i} \vars_{i} - \vars_{i} \varrho_{i} = e_{i}[\rho, \vars]e_{i},\\
& \ppi d(\vart_{i}) =v^{-1}_{i} \ppi\left(  \sum_{\alpha \in Q_{1} } e_{i}[\alpha, \alpha^{\circ}]e_{i} + \sum_{\alpha \in Q_{1}}  e_{i}[\alpha^{*}, \alpha^{\circledast}] e_{i} \right)\\ 
& \ \ \ \ \ 
=  v^{-1}_{i} \sum_{\alpha \in Q_{1} } e_{i}[\alpha, [\alpha^{*},\vars] ]e_{i} -  \sum_{\alpha \in Q_{1}}  e_{i}[\alpha^{*}, [\alpha,\vars]] e_{i} \\
& \ \ \ \ \ 
= v^{-1}_{i}  \sum_{\alpha \in Q_{1} } e_{i}[[\alpha, \alpha^{*}] , \vars]e_{i}
= e_{i}[\varrho, \vars]e_{i}.
\end{split}
\]
\end{proof}

It is clear that under the isomorphisms $\tuH^{0} (\YYM) \cong \YM, \tuH^{0}(\PPi) \cong \Pi$, 
the $0$-th cohomology morphism $\tuH^{0}(\ppi): \tuH^{0}(\YYM) \to \tuH^{0}(\PPi) $ corresponds to 
the canonical projection $\pi: \YM \to \Pi$ of \eqref{202009212139}.
 
\subsubsection{The exact triangle $\sfU$} 

In Theorem \ref{exact triangle U} below,  
we give an exact triangle $\sfU: \YYM(-1)  \xrightarrow{ \sfr_{ \varrho } } \YYM \xrightarrow{ \ppi } \PPi \to \YYM(-1)[1]$  
the $0$-th cohomology group of which coincides with 
the canonical exact sequence  $\YM(-1)  \xrightarrow{ \sfr_{ \varrho } } \YM \xrightarrow{ \ppi } \Pi \to 0 $.  
To state the theorem, first we need to prove the following lemma.

\begin{lemma}\label{homotopy 20191203}
We have $\ppi \sfr_{\varrho} = -d\sfr_{\vars} \ppi -\sfr_{\vars}\ppi d$. 
In other words, the morphism $\ppi\sfr_{\varrho}: \YYM \to \PPi$ is homotopic to $0$ via the homotopy 
$-\sfr_{\vars}\ppi$. 
\[
\begin{xymatrix}@C=10mm@R=4mm{
\YYM(-1)  \ar[r]^{\sfr_{\varrho}} \ar@/_2pc/[rr]_{0}   &\YYM \ar[r]^{\ppi} \ar@{=>}[d]^{-\sfr_{\vars}\ppi} & \PPi\\
&& 
}\end{xymatrix}
\]
\end{lemma}

\begin{proof}
Let $x \in \YYM$ be a homogeneous element. 
Then, using Lemma \ref{202102041200}, we deduce the following equation
\[
\begin{split}
d (\sfr_{\vars}\ppi(x)) &= ( -1)^{|x|}d(\ppi(x)\vars) = ( -1)^{|x|} (d\ppi(x)) \vars +\ppi(x) d(\vars) \\
&=(-1)^{|x|} \ppi(dx)\vars -  \ppi(x) \varrho\\ 
&=-\sfr_{\vars} \ppi(dx) -  \sfr_{\varrho} \ppi(x).
\end{split}
\]
\end{proof}

By this lemma,  we have the following cochain map of dg-$\YYM$-modules 
\[
q_{\sfr_{\varrho}, \ppi, -\sfr_{\vars}\ppi } := ( \ppi, - \sfr_{\vars} \ppi \uparrow): \cone( \sfr_{\varrho}) \to \PPi  
\]
where we use the notation in \eqref{202012082037}. 

\begin{theorem}\label{exact triangle U}
The map $q_{\sfr_{\varrho}, \ppi, -\sfr_{\vars}\ppi }$ is a quasi-isomorphism. 
Therefore we have 
an exact triangle  in $\sfD(\YYM\Gr)$
\[
\sfU: \YYM( -1) \xrightarrow{ \sfr_{\varrho}} \YYM \xrightarrow{ \ppi } \PPi \to \YYM( -1)[1]. 
\]
\end{theorem}

A proof is given in Section \ref{subsection: U}. 
In Section \ref{a-infinity}, we show that there exists an exact triangle \[
\widehat{\sfU}: \YYM( -1) \xrightarrow{ \widehat{\sfr_{\varrho}} } \YYM \xrightarrow{ \ppi } \PPi \to \YYM( -1)[1]. 
\]
in $\sfD(\YYM^{\mre}\Gr)$ 
which is sent to $\sfU$  by the forgetful functor $\sfD(\YYM^{\mre}\Gr) \to \sfD(\YYM\Gr)$.

\subsubsection{The homotopy $\sfH$}

We define a derivation $\sfH$ of $\YYM$ over $\Pa_{0}$ of cohomological degree $-1$ to be 
\[
\sfH(\alpha) := \alpha^{\circledast}, 
\sfH(\alpha^{*}) := - \alpha^{\circ},
\sfH(\alpha^{\circ}) := - [\alpha^{*}, \vart], 
\sfH(\alpha^{\circledast}) := [ \alpha,\vart ], 
\sfH(\vart_{i}) := 0
\]
for all $\alpha \in Q_{1}$ and $i \in Q_{0}$. 
The values of $\sfH$ for general homogeneous elements are determined from the Leibniz rule 
$\sfH(xy) = \sfH(x) y + ( -1)^{|x|} x \sfH(y)$. 

We note  that the map $\sfH: \YYM \to \YYM$ increase  the $*$-degree by $1$. 
\[
\sfH :\YYM_{n -1} \to \YYM_{n}. 
\]

The weighted mesh relation $\varrho$ is not a central element of $\YYM$, but 
in the next lemma we see  that it may be said to be  cohomologically central.  

\begin{lemma}\label{homotopy proposition}
The morphism $\sfH$ is a homotopy from 
the right multiplication map  $\sfr_{\varrho}$ to  
the left multiplication map $\sfl_{\varrho}$. 

In particular the weighted mesh relation $\varrho$ is central in the cohomology algebra $\tuH(\YYM)$. 
\end{lemma}

\begin{proof}
We have to show  
\[
d\sfH(x) +\sfH d(x) =-\sfb_{\varrho}(x) = [x, \varrho]  
\]
for $x \in \YYM$. 
In other words, 
we have to prove the following equality  of morphisms $\YYM \to \YYM$
\begin{equation}\label{homotopy equation}
[d, \sfH]_{\Hom} = -\sfb_{\varrho}
\end{equation}
where $[-,+]_{\Hom}$  is the commutator inside the $\Hom$-complex  $\Hom_{\Pa_{0}}^{\bullet} (\YYM, \YYM)$.

Since the both sides of the equation \eqref{homotopy equation} 
are derivations of degree $0$, 
it is enough to check that 
the equality holds on the generators $\alpha, \alpha^{*}, \alpha^{\circ}, \alpha^{\circledast}, t_{i}$. 

The  equation  \eqref{homotopy equation} on $\alpha$ 
is checked as below.    
\[
[d, \sfH]_{\Hom}(\alpha) = d\sfH(\alpha) = d \alpha^{\circledast} = \eta_{\alpha}=[\alpha, \rho].
\]
A similar calculation works for $\alpha^{*}$.

We claim $\sfH (\rho_{i}) + d(t_{i})= 0$. 
\[
\begin{split}
\sfH(\rho_{i})
& = \sfH \sum_{\alpha \in Q_{1}}e_{i}[\alpha, \alpha^{*}]e_{i}
=  \sum_{\alpha \in Q_{1}}e_{i}[\sfH(\alpha), \alpha^{*}]e_{i} + e_{i}[\alpha, \sfH(\alpha^{*})]e_{i} 
=  \sum_{\alpha \in Q_{1}}e_{i}[\alpha^{\circledast}, \alpha^{*}]e_{i} + e_{i}[\alpha, -\alpha^{\circ}] e_{i}\\
& =  \sum_{\alpha \in Q_{1}} - e_{i}[\alpha^{*}, \alpha^{\circledast}]e_{i} -  e_{i}[\alpha, \alpha^{\circ}]e_{i} 
 =- d( t_{i}). 
\end{split}
\]
 
Using the equation  $\sfH(\varrho) + d(\vart) = 0$, which follows from the claim, 
we check \eqref{homotopy equation} on $\alpha^{\circ}$
\[
\begin{split}
[d, \sfH]_{\Hom}(\alpha^{\circ}) 
& = - d[\alpha^{*}, \vart] -\sfH[\alpha^{*}, \varrho] 
 = - [d(\alpha^{*}), \vart] - [\alpha^{*}, d(\vart)] -[\sfH (\alpha^{*}), \varrho]- [\alpha^{*}, \sfH(\varrho)]
= [\alpha^{\circ}, \varrho].
\end{split}
\]
A similar calculation works for $\alpha^{\circledast}$.

Finally, we check the equation \eqref{homotopy equation} on $t_{i}$. 
\[
\begin{split}
[d, \sfH]_{\Hom}(t_{i}) 
& = \sfH d (t_{i} ) = \sfH\left(  \sum_{\alpha \in Q_{1}}e_{i}[ \alpha, \alpha^{\circ}]e_{i} +  e_{i} [\alpha^{*}, \alpha^{\circledast}]e_{i}\right)
 \\ & 
  = \sum_{\alpha \in Q_{1}} 
e_{i} [\alpha^{\circledast}, \alpha^{\circ}] e_{i}- e_{i}[\alpha, [\alpha^{*}, \vart]]e_{i} 
-e_{i}[ \alpha^{\circ}, \alpha^{\circledast}]e_{i} +e_{i} [\alpha^{*}, [\alpha, \vart]] e_{i}
 \\&
 =e_{i} \left(  \sum_{\alpha \in Q_{1}} [[\vart, \alpha], \alpha^{*}] +  [\alpha, [\vart, \alpha^{*}]] \right) e_{i}
\\&
 = e_{i}\left(  \sum_{ \alpha \in Q_{1}}  [\vart, [\alpha, \alpha^{*}] ]\right) e_{i} =[t_{i}, \varrho].
\end{split}
\]
\end{proof}

\begin{lemma}\label{homotopy lemma s}
We have $\ppi \sfH = \sfb_{\vars} \ppi$. 
\end{lemma}

\begin{proof}
By direct computation we can check that 
 the equation holds on the generators $\alpha, \alpha^{*}, \alpha^{\circ}, 
\alpha^{\circledast}, \vart_{i}$. 
\[
\begin{split}
\ppi\sfH(\alpha) &= \ppi(\alpha^{\circledast}) =\sfb_{\vars}(\alpha) = \sfb_{\vars}\ppi(\alpha), \\
\ppi\sfH(\alpha^{*})& = -\ppi(\alpha^{\circ}) =\sfb_{\vars}(\alpha^{*}) = \sfb_{\vars}\ppi(\alpha^{*}), \\
\ppi\sfH(\alpha^{\circledast}) &= \ppi[\alpha, \vart] = -[ \alpha, \vars^{2} ] = [\vars^{2}, \alpha]  \\
& = \vars[\vars, \alpha] + [\vars, \alpha]\vars = [\vars, [\vars, \alpha] ] =\sfb_{\vars}\ppi(\alpha^{\circledast}) \\ 
\ppi\sfH(\alpha^{\circ}) & = -\ppi[\alpha^{*}, \vart] = [ \alpha^{*}, \vars^{2}] = - [\vars^{2}, \alpha^{*}] \\ 
&= -(\vars[\vars, \alpha^{*}] + [\vars, \alpha^{*}]\vars) = -[\vars, [\vars, \alpha^{*}] ] =\sfb_{\vars}\ppi(\alpha^{\circ}) \\ 
\ppi \sfH(\vart_{i}) & = 0 = \sfb_{\vars}\ppi(\vart_{i}).
\end{split}
\]

Observe that  the both sides of the equation are derivation along $\ppi$ of degree $-1$,  
i.e., the morphisms $F = \ppi\sfH, \sfb_{\vars}\ppi$ satisfy
$F(xy) = F(x) \ppi(y) +(-1)^{|x|}\ppi(x) F(y)$. 
It follows from  the above observation the equation holds for any $x \in \YYM$. 
\end{proof}

\subsection{The exact triangle $\mathsf{AR}$}\label{The exact triangle AR}

We recall from Section \ref{Section:trace formula}:
\[
\hat{v} := \sum_{ i \in Q_{0} } v_{i} \hat{e}_{i}, \ \ \ttheta := {}^{v}\!\ttheta = \sum_{i \in Q_{0}} v_{i} \tilde{e}_{i}.
\]
Explicitly, 
\begin{equation}\label{202008141807}
\ttheta: \small
 \PPa^{\vee}  = \left( \Pa V^{*} \Pa[-1] \oplus \Pa \Pa,  \begin{pmatrix} 0 & \uparrow \hat{\rho} \\ 0 & 0\end{pmatrix} \right)
\xrightarrow{ \tiny \begin{pmatrix} 0 & \hat{v} \\ 0 & 0 \end{pmatrix}}
\left( \Pa \Pa \oplus \Pa V\Pa [1], \begin{pmatrix} 0 & \hat{\mu} \uparrow \\ 0 & 0\end{pmatrix} \right) =
 \PPa. 
\end{equation}

Since now we are assuming  $v_{i} \neq 0 $ for all $i \in Q_{0}$, 
the morphism $\hat{v}: \Pa\Pa \to \Pa \Pa$ is an isomorphism. 
We set $\hat{\eta} := \hat{\rho} \hat{v}^{-1} \hat{\mu}:\Pa V\Pa \to \Pa V^{*} \Pa$. 
Then it is straightforward check that 
 for an arrow $\alpha \in Q_{1}$ we have 
\[
\begin{split}
\hat{\eta}( 1 \otimes \alpha\otimes 1) &= 
\sum_{\beta \in Q_{1}, h(\beta) = t(\alpha)} v_{h(\alpha)}^{-1}\alpha \beta \otimes \beta^{*} \otimes 1
- \sum_{\beta\in Q_{1},  t(\beta) = t(\alpha)} v_{h(\alpha)}^{-1} \alpha \otimes  \beta^{*}\otimes \beta \\
&
- \sum_{\beta\in Q_{1}, h(\beta) = h(\alpha)}  v_{t(\alpha)}^{-1} \beta \otimes \beta^{*} \otimes \alpha 
+\sum_{\beta\in Q_{1}, t(\beta) = h(\alpha)} v_{t(\alpha)}^{-1} 1 \otimes \beta^{*}\otimes \beta \alpha,
\end{split}\]
where in each term $\beta^{*}$ in the middle is belonging to $V^{*}$. 
Therefore, if we regard $\Pa V^{*} \Pa \subset \kk \overline{Q}$ via the isomorphism \eqref{202006121845}, 
then $\hat{\eta}( 1 \otimes \alpha \otimes 1) = \eta_{\alpha}$. 
We note that here  the symbol $\otimes$  denotes  tensor products over $\Pa_{0}$ which are suppressed in the sequel.

Comparing this observation  with
Observation \ref{202006122126}, 
we come to the following lemma. 

\begin{lemma}
The $*$-degree $1$ part $\YYM_{1}$ is isomorphic to  the 
cone $\cone(\hat{\eta})$. 
Moreover, the $*$-degree $1$ part $\tilde{\pi}_{1}: \YYM_{1} \to \PPi_{1}$ 
is identified with  
\[
\tilde{\pi}_{1}: \small
\YYM_{1} \cong \left( \Pa V^{*} \Pa \oplus \Pa V \Pa [1], 
\begin{pmatrix} 0 & \hat{\eta}\uparrow \\ 0 & 0\end{pmatrix} \right)
\xrightarrow{ \tiny \begin{pmatrix} \id & 0 \\ 0 & - \hat{v}^{-1}\hat{\mu} \end{pmatrix}}
\left( \Pa V^{*} \Pa \oplus \Pa\Pa [1], \begin{pmatrix} 0 & -\hat{\rho}\uparrow \\ 0 & 0\end{pmatrix} \right) =
\PPi_{1}.
\]
\end{lemma}

 We define a morphism $\rrho:  \PPa \to \YYM_{1}$ to be 
 \[
\rrho:  \PPa = \left( \Pa \Pa \oplus \Pa V \Pa [1], \begin{pmatrix} 0 & \hat{\mu}\uparrow \\ 0 & 0\end{pmatrix} \right)
\xrightarrow{ \tiny \begin{pmatrix} \rrho\hat{v}^{-1}  & 0 \\ 0 & \id \end{pmatrix}}
\left( \Pa V^{*} \Pa \oplus \Pa V\Pa [1], \begin{pmatrix} 0 & \hat{\eta}\uparrow \\ 0 & 0\end{pmatrix} \right) =
\YYM_{1}. 
\]
 The next lemma says that these morphisms $\rrho, \ppi_{1}$ are representatives of the morphisms in  the exact triangle ${}^{v}\!\sfAR$  
 given in Section \ref{subsection The universal Auslander-Reiten triangle} denoted by the same symbols.

\begin{lemma}\label{pre universal Auslander-Reiten triangle lemma}
The following diagram gives an exact triangle in the homotopy category $\sfK(\Pa^{\mre})$ 
\[
\mathsf{AR}:  \PPa \xrightarrow{ \ \ \rrho \ \ } \YYM_{1} \xrightarrow{ \ \ \tilde{\pi}_{1} \ \ } \PPi_{1} \xrightarrow{ -\ttheta[1] }  \PPa[1]. 
\]
Moreover, 
if we define 
a morphism 
$\sfL:  \PPa \to \PPi_{1}$ of degree $-1$ in $\sfC_{\DG}(\Pa^{\mre})$ as follows 
\[
\sfL(e_{i}) := -\vars_{i}, \ \sfL(\downarrow \alpha) : = 0, 
\]
for $i \in Q_{0}, \alpha \in Q_{1}$, 
then the following statements hold
\begin{enumerate}[(1)]
\item 
It is a homotopy from $\ppi_{1} \rrho$ to $0$ 
\item 
It induces the homotopy equivalences 
$q_{\rrho, \ppi_{1}, \sfL} =(\ppi_{1}, \sfL\hspace{-3pt}\uparrow): \cone(\rrho) \to \PPi_{1}$ and  
$j_{\rrho, \ppi_{1}, \sfL} =( \uparrow\hspace{-2pt}\sfL,  -\rrho)^{t}:  \PPa \to \cone(\ppi_{1})[-1]$. 

\item 
We have an equality $\sfL = \sfh_{\ppi_{1}} j_{\rrho, \ppi_{1}, \sfL }$ of morphisms $\PPa \to \PPi_{1}$ of degree $-1$ 
where $\sfh_{\ppi_{1}}$ is the canonical homotopy of \eqref{202007312045}. 
\end{enumerate} 
\end{lemma}

\begin{proof}
Applying  Lemma \ref{octahedral axiom} for $f= \hat{\mu}, g = \hat{\rho}, h= \hat{\eta}$ and $H = 0$, we obtain an exact triangle 
$\cone(\hat{\mu}) \xrightarrow{\Psi} \cone( \hat{\eta} ) \xrightarrow{\Upsilon} \cone( \hat{\rho}) \xrightarrow{ -\Phi[1]} $. 
Modifying this sequence by the isomorphism ${\tiny \begin{pmatrix} \id & 0 \\ 0 & -\hat{v}^{-1} \end{pmatrix}}: \cone ( \hat{\rho}\hat{v}^{-1}) 
\xrightarrow{\cong}  \cone(- \hat{\rho}) = \PPi_{1}$, 
we obtain the desired exact triangle.

\begin{equation}\label{202007311609}
\begin{xymatrix}{
&&\Pa V\Pa  \ar@{=}[rr] \ar[d]_{\hat{\mu} }  &&  \Pa V\Pa \ar[d]^{\hat{\eta} }  && && \\
\cone(\hat{\rho}\hat{v}^{-1} )[-1] \ar[rr] \ar@{=}[d] &&
\Pa\Pa \ar@{=>}[urr]^{0} \ar[d] \ar[rr]^{\hat{\rho}\hat{v}^{-1} } && 
\Pa V^{*} \Pa \ar[d]  \ar[rr] && \cone(\hat{\rho} \hat{v}^{-1}) \ar@{=}[d] 
\\
\cone(\hat{\rho}\hat{v}^{-1} )[-1] 
\ar[d]_{\tiny \begin{pmatrix} \id & 0\\ 0 & -\hat{v}^{-1} \end{pmatrix} }^{\cong}
\ar[rr]_{\tiny \begin{pmatrix} 0 & - \id \\ 0 & 0 \end{pmatrix} }
 &&\cone(\hat{\mu} ) \ar[rr]_{\tiny \begin{pmatrix} \hat{\rho}\hat{v}^{-1} & 0 \\ 0 & \id \end{pmatrix} } 
 \ar@{=}[d]  
 && \cone(\hat{\eta})  \ar@{=}[d] \ar[rr]_{\tiny \begin{pmatrix} \id  & 0 \\ 0 &\hat{\mu}[1] \end{pmatrix} }
  &&
\cone(\hat{\rho}\hat{v}^{-1}) 
\ar[d]^{\tiny \begin{pmatrix} \id & 0\\ 0 & -\hat{v}^{-1} \end{pmatrix} }_{\cong}  \\ 
\PPi_{1}[-1]  \ar[rr]_{\ttheta} && \PPa \ar[rr]_{\rrho} && \YYM_{1} \ar[rr]_{\ppi_{1}} && \PPi_{1}&& 
}\end{xymatrix}
\end{equation}

In a similar way, we obtain the desired homotopy from Lemma \ref{octahedral axiom}. 
\end{proof}

\subsection{Exact triangles $\sfV_{n}$ and the morphisms $\eeta^{*}, \zzeta$}\label{202106020843}

\subsubsection{The morphism $\zzeta$} 

We define a morphism $\zzeta: \YYM_{1} \otimes_{\Pa} \YYM \to \YYM$ to be the multiplication map 
\[
\zzeta: \YYM_{1} \otimes_{\Pa} \YYM \to \YYM, \ \zzeta(x \otimes y) := xy. 
\]
We denote the $*$-graded version  by the same symbol $\zzeta: \YYM_{1} \otimes_{\Pa} \YYM( -1) \to \YYM$. 
We denote the $*$-degree $n$-component by $\zzeta_{n}: \YYM_{1} \otimes_{\Pa} \YYM_{n -1}  \to \YYM_{n }$.

\subsubsection{The morphism $\eeta_{2}^{*}$}

Next we look at the $*$-degree $2$ part $\YYM_{2}$. 
For this we introduce several notations. 

For $i \in Q_{0}$, 
we define elements $\varrho^{\circ}_{i}, \varrho^{\circledast}_{i}$ of $\YYM_{2}$ to be 
 \[
\varrho^{\circ}_{i} := v_{i}^{-1} \sum_{\alpha \in Q_{0}} e_{i} [\alpha, \alpha^{\circ}]e_{i}, \ \   
\varrho^{\circledast}_{i} := v_{i}^{-1} \sum_{\alpha \in Q_{0}} e_{i} [\alpha^{*}, \alpha^{\circledast}] e_{i}. 
\]
Observe that  we can  regard $[\alpha^{*}, \alpha^{\circledast}] = \alpha^{*} \alpha^{\circledast} - \alpha^{\circledast}\alpha^{*}$ 
as an element of $\Pa V^{*}\Pa V^{\circledast} \Pa \oplus \Pa V^{\circledast} \Pa V^{*} \Pa$ 
which is the cohomological degree $-1$ part of $\YYM_{1} \otimes_{\Pa} \YYM_{1}$.

We set $\varrho^{\circ} := \sum_{i \in Q_{0}} \varrho^{\circ}_{i}$ 
and $\varrho^{\circledast} := \sum_{i \in Q_{0}} \varrho^{\circledast}_{i}$ 
so that $d(\vart) = \varrho^{\circ} + \varrho^{\circledast}$. 

We define  morphisms 
\[
\hat{\varrho}^{\circ}: \Pa T \Pa \to \Pa V^{\circ} \Pa, \ 
 \hat{\varrho}^{\circledast}: \Pa T \Pa \to \YYM_{1} \otimes_{\Pa} \YYM_{1}, \ 
\hat{\eta}^{\circ}: \Pa V^{\circ} \Pa \to \YYM_{1} \otimes_{\Pa} \YYM_{1}\]
  of cohomological degree $1$ in $\sfC_{\DG}(\Pa^{\mre})$  
  by  the formulas 
  \[
  \hat{\varrho}^{\circ}(\vart_{i}) := \varrho^{\circ}_{i}, \ 
\hat{\varrho}^{\circledast}(\vart_{i}) := \rho^{\circledast}_{i}, \ 
\hat{\eta}^{\circ}(\alpha^{\circ}) := \eta_{\alpha^{*}}.
\] 
  Then we obtain the following description of $\YYM_{2}$ 
  \[
  \YYM_{2} = \left( ( \YYM_{1} \otimes_{\Pa} \YYM_{1})  \oplus \Pa V^{\circ} \Pa \oplus \Pa T \Pa, 
  {\small \begin{pmatrix} d_{\YYM_{1} \otimes \YYM_{1}} & - \hat{\eta}^{\circ} & \hat{\varrho}^{\circledast} \\ 0 & 0& \hat{\varrho}^{\circ} \\ 
  0 & 0& 0 \end{pmatrix}}\right). 
  \]
Observe that 
the subcomplex $\PPi_{1}^{\circ}[1] :=\left( \Pa V^{\circ} \Pa \oplus \Pa T \Pa,
 \begin{pmatrix} 0 & \hat{\varrho}^{\circ} \\ 0 & 0 \end{pmatrix}\right)$
 is isomorphic to  $\PPi_{1}[1]$. 
  We fix an identification as follows 
  \begin{equation}\label{202102041839}
  \mathsf{idn}:  \PPi_{1}[1] \xrightarrow{ \cong} \PPi^{\circ}_{1}[1],  \alpha^{*} \mapsto - \alpha^{\circ}, \ \vars_{i} \mapsto - \vart_{i}.
  \end{equation}

We define a morphism $\eeta^{*}_{2}: \PPi_{1} \to \YYM_{1} \otimes_{\Pa} \YYM_{1}$ in $\sfC(\Pa^{\mre})$ 
by the formulas 
  \[
\eeta^{*}_{2}(\alpha^{*}) := \eta_{\alpha^{*}}, \eeta^{*}_{2}(\vars) := - \varrho^{\circledast}.
\]
Then we may identify  $\YYM_{2}$ with  the cone of $\eeta^{*}_{2}$  
  \[
  \YYM_{2} \cong \cone(\eeta^{*}_{2}) = 
  \left( (\YYM_{1} \otimes_{\Pa} \YYM_{1}) \oplus \PPi_{1}[1], 
  \begin{pmatrix} d_{\YYM_{1} \otimes_{\Pa} \YYM_{1}} & \eeta^{*}_{2}\uparrow  \\ 0 & d_{\PPi[1]} \end{pmatrix} \right).
  \]
  Observe that the $*$-degree $2$-component  $\zzeta_{2} : \YYM_{1} \otimes_{\Pa} \YYM_{1} \to \YYM_{2}$ 
  is the cone morphism $i_{1}^{\eeta^{*}_{2}}$ of $\eeta^{*}_{2}$. 
   Therefore we obtain the  sequence below in $\sfC(\Pa^{\mre})$ 
    that gives an exact triangle in $\sfK(\Pa^{\mre})$ 
\[
\begin{xymatrix}@C=10mm@R=4mm{
\sfV_{2} :&
 \PPi_{1}  \ar[r]^{\eeta^{*}_{n}} \ar@/_2pc/[rr]_{0} &
\YYM_{1} \otimes_{\Pa} \YYM_{1} \ar[r]^{\zzeta_{2}}  \ar@{=>}[d]_{\sfg_{\eeta^{*}_{2}}}&
\YYM_{2}  &\\ 
&&&
}\end{xymatrix}
\] 
where $\sfg_{\eeta^{*}_{2}}= (0, \downarrow)^{t}$ as given in \eqref{202007312045}.

\subsubsection{The exact triangle $\sfV_{n}$}

Let $ n \geq 3$.  
We set  $\eeta^{*}_{n} := ({}_{\YYM_{1} }\zzeta_{ n-1}) (\eeta^{*}_{2,\YYM_{ n-2} })$ in $\sfC(\Pa^{\mre})$.
\[
\eeta^{*}_{n}:  \PPi_{1} \otimes_{A} \YYM_{ n-2} \xrightarrow{ \eeta^{*}_{2, \YYM_{ n-2} } }
 \YYM_{1} \otimes_{A} \YYM_{1} \otimes_{A} \YYM_{n -1} 
\xrightarrow{ {}_{\YYM_{1} }\zzeta_{ n-1}  } 
\YYM_{1} \otimes_{A} \YYM_{n}.
\]
Observe that as cohomological graded modules, we have 
\[
\begin{split}
\YYM_{n } &
= \left( \YYM_{1}  \otimes_{\Pa} \YYM_{n-1} \right)  \oplus 
\left( \Pa V^{\circ} \Pa \otimes_{\Pa} \YYM_{ n-2} \right)  \oplus 
\left( \Pa T \Pa \otimes_{\Pa} \YYM_{ n-2}\right)\\
& = \left( \YYM_{1}  \otimes_{\Pa} \YYM_{n-1} \right)  \oplus 
\left( \PPi^{\circ}_{1} \otimes_{\Pa} \YYM_{ n-2}[1] \right)
\end{split}
\]
Using the identification \eqref{202102041839}, we obtain 
\[
\YYM_{n} =
\left( \left( \YYM_{1}  \otimes_{\Pa} \YYM_{n-1} \right) \oplus 
\left(\PPi_{1}\otimes_{\Pa} \YYM_{ n-2}[1]\right) , \begin{pmatrix} d & \eeta_{n}^{*}\uparrow\\ 
0 & d
\end{pmatrix}\right).
\]
Thus we may identify $\YYM_{n}$ with the cone $\cone(\eeta^{*}_{n})$ 
and the cone morphism $i_{1}^{\eeta^{*}_{n}}$ with $\zzeta_{n}$. 
We  obtain  a diagram in $\sfC_{\DG}(\Pa^{\mre})$ below 
which becomes an exact triangle in $\sfK(\Pa^{\mre})$ 
\[
\begin{xymatrix}@C=10mm@R=4mm{
\sfV_{n} :&
 \PPi_{1} \otimes_{\Pa} \YYM_{n -2} \ar[r]^{\eeta^{*}_{n}} \ar@/_2pc/[rr]_{0} &
\YYM_{1} \otimes_{\Pa} \YYM_{n -1} \ar[r]^{\zzeta_{n}}  \ar@{=>}[d]_{\sfg_{\eeta^{*}_{n}}}&
\YYM_{n} &\\ 
&&&
}\end{xymatrix}
\] 
where $\sfg_{\eeta^{*}_{n}}= (0, \downarrow)^{t}$ as given in \eqref{202007312045}.

By convention, we set $\zzeta_{1} := \id_{\YYM_{1}}$ and $\zzeta_{n -1}^{*} := 0, \eeta^{*}_{n } := 0$ for $n \leq 1$.
Taking the direct sums of $\zzeta_{n}$ and $\eeta^{*}_{n}$, 
we obtain the  diagram below in $\sfC_{\DG}(\Pa \otimes \YYM^{op}\Gr)$ below 
which becomes an exact triangle in $\sfK(\Pa \otimes \YYM^{\op}\Gr)$ 
\begin{equation}\label{20191202}
\begin{xymatrix}@C=10mm@R=4mm{
\sfV :&
 \PPi_{1} \otimes_{\Pa} \YYM(-2) \ar[r]^{\eeta^{*}} \ar@/_2pc/[rr]_{0} &
\YYM_{1} \otimes_{\Pa} \YYM(-1) \ar[r]^{\zzeta}  \ar@{=>}[d]_{\sfg_{\eeta^{*}}}&
\YYM_{\geq 1}. &\\ 
&&&
}\end{xymatrix}\end{equation}

\subsubsection{Koszul resolution}\label{2021121160751}

Let $\iota: \YYM_{\geq 1} \to \YYM$ be a canonical injection and $\epsilon: \YYM \to \Pa$ be a canonical projection. 
Then, we have an exact sequence $0 \to \YYM_{\geq 1}\xrightarrow{ \iota} \YYM \xrightarrow{ \epsilon } \Pa \to 0$. 
We define an object $P \in \sfC_{\DG}(\Pa\otimes \YYM^{\op} \Gr)$ to be 
\[
P := 
\left( \YYM \oplus  ( \YYM_{1} \otimes_{\Pa} \YYM(-1)[1] )\oplus ( \PPi_{1} \otimes_{\Pa}\YYM(-2)[2]), 
\begin{pmatrix} d & \iota \zzeta\uparrow &\iota \sfg \uparrow^{2} \\ 
0 & d &- \downarrow \eeta^{*}\uparrow^{2}\\
0& 0& d
\end{pmatrix}
\right).
\]
Combining the above exact sequence  and the diagram \eqref{20191202}, we see that the morphism
$
\epsilon_{P}:=(\epsilon, 0,0): P \to \Pa
$
is a quasi-isomorphism in $\sfC_{\DG}(\Pa \otimes \YYM^{\op} \Gr)$.

\subsubsection{Remark on Koszul duality}\label{202111261238}

The above observations show that 
the derived quiver Heisenberg algebra $\YYM$ has the following description: 
\[
\YYM = \left( \sfT_{\Pa}(\YYM_{1} \oplus \PPi_{1}[1]), d_{\sfT} + \delta\right)
\]
where $d_{\sfT}$ is the differential of the dg-tensor algebra $\sfT_{\Pa}(\YYM_{1} \oplus \PPi_{1}[1])$ 
and 
$\delta$ is a morphism of cohomological degree $1$ induced from $\eeta^{*}_{2}$. 
This description says that $\YYM$ is the \emph{Kosuzl dual} of the graded dg-coring 
$C = \Pa \oplus \YYM_{1}[1] \oplus \PPi_{1}[2]$ over $\Pa$ 
whose comultiplication is essentially $\eeta^{*}_{2}$. 

In view of Koszul duality, 
the quasi-isomorphism $\epsilon_{P} : P \to \Pa$ may be regarded as  the Koszul resolution of $\Pa$ 
\[
\begin{xymatrix}{
\PPi_{1} \otimes_{\Pa} \YYM( -2)  \ar[rr]^{\eeta^{*}} && 
\YYM_{1} \otimes_{\Pa} \YYM( -1) \ar[rr] \ar[dr]^{\zzeta} &&
\YYM \ar[rr]^{\epsilon}  && \Pa \ar[dlll]^{[1]} \\
&&& \YYM_{\geq 1}  \ar[ulll]^{[1]} \ar[ur]^{\iota}
}\end{xymatrix}.
\]

This point will be studied in a subsequent work \cite{Minamoto Smith} and in particular give a construction of Calabi-Yau algebras which generalize Keller's Calabi-Yau completion \cite{Keller: Calabi-Yau completion}.

\subsubsection{An explicit formula of $\eeta^{*}$}\label{202105211751}
We give  an explicit formula of $\eeta^{*}: \PPi \lotimes_{\Pa} \YYM \to \YYM_{1} \lotimes_{\Pa} \YYM$ 
which uses  the following  identifications of cohomological graded modules over $\Pa^{\mre}$
\begin{equation}\label{202105211917}
\PPi_{1} \otimes_{\Pa} \YYM \cong   \Pa V^{*}  \YYM \oplus \Pa S  \YYM, \ \ 
\YYM_{1} \otimes_{\Pa} \YYM \cong   \Pa V^{*}\YYM \oplus \Pa V^{\circ}\YYM.
\end{equation} 
Although the formula is obtained by a straightforward calculation, we explain it in detail,  since  similar explicit formulas for other morphisms play key roles in the sequel.

In  the identification of $\PPi_{1} \otimes_{\Pa} \YYM$ with $\Pa V^{*}  \YYM \oplus \Pa S  \YYM$, 
an element of $\PPi_{1} \otimes_{\Pa} \YYM$ is given by a sum of elements of the forms 
$p\alpha^{*} x, p\vars_{i} x$ where $p \in \Pa, x \in \YYM, \alpha \in Q_{1}, i \in Q_{0}$. 
Thus in particular, $\PPi_{1} \otimes_{\Pa} \YYM$ is generated over $\Pa^{\mre}$ 
by elements of the forms $\alpha^{*}x, \vars_{i} x$ for $ x \in \YYM, \alpha \in Q_{1}, i \in Q_{0}$. 
The values of the morphism $\eeta^{*}$ over $\Pa^{\mre}$ on these generators are given as below
\[
\eeta^{*}(\alpha^{*}  x) = [\alpha^{*}, \varrho]  x = \alpha^{*}\varrho   x - \varrho \alpha^{*} x , 
\ \eeta^{*}(\vars_{i}  x) = -\varrho_{i}^{\circledast} x 
\]
where we use the identification $\YYM_{1} \otimes_{\Pa} \YYM \cong   \Pa V^{*}\YYM \oplus \Pa V^{\circ}\YYM$ 
to write the right hand sides. More precisely, the right hand sides should be interpreted as follows.
For the first term $\alpha^{*}\varrho  x = \alpha^{*} (\varrho x)$ of the first equation, 
we   regard  the first factor $\alpha^{*}$ as an element of $V^{*}$
and the second factor  $\varrho  x$ as an element $\YYM$. 
Then using the identification  $\Pa V^{*}  \YYM \subset \YYM_{1} \otimes_{\Pa} \YYM$, 
we regard $\alpha^{*} \varrho x$ as an element of $\YYM_{1} \otimes_{\Pa} \YYM$.

As for the second  term $\varrho \alpha^{*}  x= \varrho (\alpha^{*} x)$ of the first equation, 
we   regard  the first factor $\varrho $ as an element of $\YYM_{1}$
and the second factor $\alpha^{*}  x$ as an element $\YYM$. 
Then we write $\varrho \otimes \alpha^{*} x$ as $\varrho \alpha^{*}x$.  
In a similar way, we regard $\varrho_{i}^{\circledast} p x$ as an element of $\YYM_{1} \otimes_{\Pa} \YYM$. 

We note that under the identification \eqref{202105211917}, the element  $\varrho \alpha^{*}  x $  
corresponds to the element  
\[
\sum_{i \in Q_{0}} v_{i}^{-1}\left(\sum_{\beta: t(\beta) = i } \beta\beta^{*} \alpha^{*}x 
-\sum_{\beta: h(\beta) = i} \beta^{*} \beta \alpha^{*} x\right) 
\]
where the first term $\beta\beta^{*} \alpha^{*}x$ is regarded  as an element of $\Pa V^{*} \YYM$ by viewing 
$\beta\beta^{*} \in \Pa V^{*}$ and $\alpha^{*} x \in \YYM$ 
and the second term $\beta^{*} \beta\alpha x$ 
is regarded  as an element of $\Pa V^{*} \YYM$ by viewing 
$\beta^{*} \in \Pa V^{*}$ and $\beta \alpha^{*} x \in \YYM$. 
In a similar way, we can write the  element of $\Pa V^{\circledast} \YYM$ which corresponds to  $\varrho_{i}^{\circledast} p x$.

\subsubsection{The homotopy $\sfK$}

Since $\ppi: \YYM \to \PPi$ is a dg-algebra homomorphism, 
we see that  the  square in the following diagram is commutative. 
\[
\begin{xymatrix}@C=30pt{
\PPi_{1} \otimes_{A} \YYM (-2) \ar@/_20pt/[drr]_{0} \ar[rr]^{\eeta^{*}} && 
\YYM_{1} \otimes_{A} \YYM ( -1) \ar[rr]^-{\zzeta} \ar[d]^{\ppi_{1} \otimes \ppi(-1)} \ar@{=>}[dl]_{\sfK}&&
\YYM_{ \geq 1} \ar[d]^{\ppi_{\geq 1}} \\ 
&& \PPi_{1} \otimes_{A} \PPi ( -1) \ar@{=}[rr] && \PPi_{\geq 1}
}\end{xymatrix}
\]

Decomposing the morphism $\ppi_{\geq 1}$ according to the decomposition 
$\YYM_{\geq 1} = \YYM_{1} \otimes_{A} \YYM(-1) \oplus \PPi_{1} \otimes_{A} \YYM( -2)[1]$ of underlying cohomological graded modules, 
we deduce the following lemma.

\begin{lemma}\label{202008021653} 
We define a morphism $\sfK: \PPi_{1} \otimes_{\Pa} \YYM( -2) \to \PPi_{\geq 1}$ in $\sfC(\Pa\otimes \YYM^{\op}\Gr^{*})$ 
of cohomological degree $-1$ by the formula 
\[
\sfK(\alpha^{*} p x) := [\vars, \alpha^{*}] p \ppi(x), \ \sfK(\vars_{i}  p x) := \vars_{i}^{2}p \ppi(x). 
\]
for $i \in Q_{0}, \alpha \in Q_{1}, p \in \Pa,   x \in \YYM$ 
where we use the identification $\PPi_{1}\otimes_{\Pa} \YYM \cong \Pa V^{*} \YYM \oplus \Pa S \YYM$ 
to exhibit elements in the domain of $\sfK$.

Then it is a homotopy from $(\ppi_{1} \otimes \ppi(-1)) \eeta^{*}$ to $0$ 
and we have $\ppi_{\geq 1} = q_{\eeta^{*}, \ppi_{1} \otimes \ppi( -1),  \sfK}$ 
as morphisms from $\YYM_{\geq 1} = \cone(\eeta^{*}) $ to $\PPi_{ \geq 1}$. 
\end{lemma}

\subsection{The exact triangle $\sfU$}\label{subsection: U}

The aim of this section \ref{subsection: U} is to prove Theorem \ref{exact triangle U} 
and Theorem \ref{exact triangle U2} below.

\subsubsection{}

Taking  the tensor product $\sfAR \otimes_{A} \YYM$, we obtain the upper row of the following diagram. 
It is clear that the right square is commutative. 
We set  $\brho: = \zzeta \rrho_{\YYM}$ to make the left square commutative.  
\begin{equation}\label{202007312226}
\begin{xymatrix}@C=60pt@R=6mm{
&&\\
\PPa \otimes_{\Pa}  \YYM  \ar@/^30pt/[rr]^{0}  \ar[r]^{\rrho_{\YYM} } \ar@{=}[d] & 
\YYM_{1} \otimes_{\Pa} \YYM \ar@{=>}[u]_{\sfL_{\YYM}} \ar[r]^{\ppi_{1, \YYM} } \ar[d]_{\zzeta} & 
\PPi_{1} \otimes_{\Pa} \YYM \ar[d]^{{}_{\PPi_{1} }\ppi} \\ 
\PPa \otimes_{\Pa} \YYM  \ar[r]_{\brho} &\YYM \ar[r]_{\ppi} & \PPi  
}\end{xymatrix}
\end{equation}

We set 
$\sfM := ({}_{\PPi_{1}} \pi)(  \sfL_{ \YYM}): \PPa \otimes \YYM \to \PPi$. 
It  is a homotopy from $\ppi \brho $ to $0$ in $\sfC_{\DG}(\Pa\otimes \YYM^{\op})$ by Lemma \ref{202008021445}. 
 \[
\begin{xymatrix}@C=10mm@R=4mm{
\PPa \otimes_{\Pa} \YYM   \ar[r]^{\brho } \ar@/_2pc/[rr]_{0}   &\YYM \ar[r]^{\ppi} \ar@{=>}[d]^{\sfM} & \PPi\\
&&. 
}\end{xymatrix}
\]

We  give explicit formulas for these morphisms. 
For this we use the identification $\PPa \otimes_{\Pa} \YYM \cong \Pa \Pa_{0} \YYM \oplus \Pa (V[1] )\YYM$. 
Observe that  $\PPa \otimes_{\Pa} \YYM$ is generated over $\Pa^{\mre}$ by elements 
of the forms $e_{i} x, \downarrow \alpha x$ where $i \in Q_{0}, \alpha \in Q_{1} $ and $x \in \YYM$. 
On these generators, the morphisms $\brho$ and $\sfM$ take the following values
 \[
 \brho(e_{i} x) = \varrho_{i} x, \ \brho(\downarrow \alpha x)= \alpha^{\circ} x, \ \sfM(e_{i}x ) = -\vars_{i} \ppi(x),  \ \sfM(\downarrow \alpha x) = 0. 
 \]
 
We remark that the  module $\Pa \Pa_{0} \YYM$ that appeared in the above identification has the simpler expression  $\Pa \YYM$. 
However if we use the latter, then the expression $px \ (p \in \Pa, x \in \YYM)$ of an element of $\Pa\YYM$  
could have two meanings that are  $ e_{t(p)} \otimes_{\Pa_{0}} px$ and $p \otimes_{\Pa_{0}} x$. 
To avoid this confusion, we insert $\Pa_{0}$ in between $\Pa$ and $\YYM$. 
  
For example, if we apply the differential $d$ to a homogeneous element of $\PPa \otimes_{\Pa} \YYM$ which has the form  
$\downarrow \alpha x$,  
then  the result is written as $d(\downarrow \alpha x) = \alpha e_{h(\alpha)} x - e_{t(\alpha)} \alpha x - \downarrow\alpha d(x)$.

 The  main result of Section \ref{subsection: U} is the following, which is a two-sided version of Theorem \ref{exact triangle U}.
 
\begin{theorem}\label{exact triangle U2}

The induced morphism
$q :=q_{\brho, \ppi_{1}, \sfM}: \cone(\brho) \to \PPi$ is a quasi-isomorphism in $\sfC(\Pa \otimes \YYM^{\op})$. 

Consequently, 
we have an exact triangle of the following form in $\sfD(\Pa \otimes \YYM^{\op}\Gr)$ 
\begin{equation}\label{202008141923}
\YYM(-1) \xrightarrow{\brho} \YYM \xrightarrow{\ppi} \PPi  \to \YYM(-1)[1].
\end{equation}

\end{theorem}

Before proceeding  a proof, we discuss its consequence.  
Let $n \geq 1$. 
We note that looking $*$-degree $n$ part of the exact triangle \eqref{202008141923} 
we obtain the following exact triangle in $\sfD(\Pa^{\mre})$,
\begin{equation}\label{202008141924}
\YYM_{n -1} \xrightarrow{ \brho_{n} } \YYM_{n } \xrightarrow{ \ppi_{n} } \PPi_{n } \to \YYM_{n -1}[1].
\end{equation}
Thus taking the tensor product with $M \in \Dbmod{\Pa}$, we obtain the following  exact triangle in $\Dbmod{\Pa}$
\begin{equation}\label{2020081419241}
\YYM_{n -1} \lotimes_{\Pa} M \xrightarrow{ \brho_{n, M } } \YYM_{n } \lotimes_{\Pa} M \xrightarrow{ \ppi_{n, M} } \PPi_{n } \lotimes_{\Pa} M \to \YYM_{n -1} \lotimes_{\Pa} M[1].   
\end{equation}
Since the weighted Euler characteristic is additive for an exact triangle, we deduce the following corollary.

\begin{corollary}\label{202102211730}
Let $u \in \kk Q, \ M \in \Dbmod{\Pa}, \ n \geq 0$. Then we have the following equality 
\[
{}^{u}\! \Euch(\YYM_{n} \lotimes_{\Pa} M) = \sum_{i = 0}^{n} {}^{u}\!\Euch(\PPi_{i} \lotimes_{\Pa} M).
\]

In particular in the case where $u $ is an eigenvector of $\Psi = \Phi^{-t}$ (the inverse of the transpose of  the Coxeter matrix) with the eigenvalue $\lambda$, 
we have
\[
\frac{ {}^{u}\! \Euch(\YYM_{n} \lotimes_{\Pa} M) }{ {}^{u}\! \Euch(M)} =  \sum_{i = 0}^{n} \lambda^{i}.
\]
\end{corollary}

\subsubsection{The morphism $\brho'$}

To prove Theorem \ref{exact triangle U}, first we need to introduce  
a morphism $\brho'  : \PPa \otimes_{\Pa} \YYM \to \YYM$ related to the morphism $\brho$.

We consider the co-cone $\cone(\ppi_{1}) \otimes_{A}\YYM[ -1]$ of $\ppi_{1, \YYM}$. 
We set $p := - p_{2}^{\ppi_{1}}$ and  $\brho' := \zzeta  p_{\YYM} : \PPa \otimes_{\Pa} \YYM \to \YYM$. 
\begin{equation}\label{2020073122261}
\begin{xymatrix}@C=60pt@R=6mm{
&&\\
\cone(\ppi_{1}) \otimes_{\Pa}  \YYM [-1]  \ar@/^30pt/[rr]^{0} \ar[r]^{p_{\YYM} } \ar@{=}[d] & 
\YYM_{1} \otimes_{\Pa} \YYM \ar@{=>}[u]_{(\sfh_{\pi_{1}})_{\YYM}} \ar[r]^{\ppi_{1, \YYM} } \ar[d]_{\zzeta} & 
\PPi_{1} \otimes_{\Pa} \YYM \ar[d]^{{}_{\PPi_{1} }\ppi} \\ 
\cone(\ppi_{1}) \otimes_{\Pa} \YYM [-1] \ar[r]_{\brho'} &\YYM \ar[r]_{\ppi} & \PPi  
}\end{xymatrix}
\end{equation}
Then setting 
$\sfM' = {}_{\PPi_{1} } \pi (\sfh_{\ppi_{1}})_{ \YYM}$,   
we obtain the following diagram in $\sfC_{\DG}(\Pa \otimes\YYM^{\op})$ by Lemma \ref{202008021445}.
 \[
\begin{xymatrix}@C=10mm@R=4mm{
\cone(\ppi_{1}) \otimes_{\Pa} \YYM [-1]  \ar[r]^-{\brho' } \ar@/_2pc/[rr]_{0}   &\YYM \ar[r]^{\ppi} \ar@{=>}[d]^{\sfM'} & \PPi.\\
&& 
}\end{xymatrix}
\]

\begin{lemma}\label{202007312238}
The  morphism
$q :=q_{\brho, \ppi_{1}, \sfM}: \cone(\brho) \to \PPi$ is a quasi-isomorphism 
if and only if so is 
the induced morphism
$q' :=q_{\brho', \ppi_{1}, \sfM}: \cone(\brho') \to \PPi$. 
\end{lemma}

\begin{proof}

We set $j: = j_{\rrho, \ppi_{1}, \sfL}$. 
First we claim  $\brho'  j_{\YYM} = \brho$. 
Indeed, we have $pj =\rrho$ 
and hence $p_{\YYM} j_{\YYM} = \rrho_{\YYM}$. 
\begin{equation}\label{2020073122263}
\begin{xymatrix}@C=60pt{
\PPa \otimes_{\Pa} \YYM \ar[dr]^{\rrho_{\YYM} } \ar[d]_{j_{ \YYM} }& & \\ 
\cone(\ppi_{1}) \otimes_{\Pa} \YYM[-1] \ar[r]_{p_{\YYM}} & 
\YYM_{1} \otimes_{\Pa} \YYM 
}\end{xymatrix}
\end{equation}
Thus we have $\brho' j_{\YYM} = \zzeta p_{\YYM}j_{\YYM} = \zzeta  \rrho_{\YYM} = \rrho$.

From the claim, we obtain the homotopy equivalence 
$\sfJ := \begin{pmatrix} \id_{\YYM} & 0 \\ 0 & j_{ \YYM }\end{pmatrix} : \cone(\brho) \to \cone(\brho')$. 
It follows from the equation $\sfh_{\ppi_{1}} j = \sfL$ that 
$q_{\brho', \ppi_{1}, \sfM'} \sfJ =  q_{\brho, \ppi_{1}, \sfM}$. 
\begin{equation}\label{202007312234}
\begin{xymatrix}@C=40pt{ 
\cone( \brho) \ar[r]^{q_{\brho, \ppi_{1}, \sfM} } \ar[d]_{\sfJ}^{\wr}  & \PPi \\
\cone( \brho') \ar[ur]_{q_{\brho', \ppi_{1}, \sfM'} } 
}\end{xymatrix} 
\end{equation}
Therefore we conclude that $q=q_{\brho, \ppi_{1}, \sfM}$ is quasi-isomorphism 
if and only if so is $q' = q_{\brho', \ppi_{1}, \sfM'}$. 
\end{proof}

\subsubsection{The homotopy $\sfN$} 

\begin{lemma}\label{homotopic lemma 5}

We define a morphism $\sfN: \PPi_{1} \otimes_{\Pa} \PPa \otimes_{\Pa}  \YYM_{ n-2} 
 \to \PPi_{1} \otimes_{\Pa} \YYM_{n -1}$ in $\sfC_{\DG}(\Pa^{\mre} \Gr)$ of cohomological  degree $-1$
by the formulas
\[
\begin{split}
&\sfN( \alpha^{*} p e_{i}  x  ) :=\vars \alpha^{*} px - \alpha^{*} \sfH(p)x, \ 
\sfN(\vars p e_{i} x ) :=  \vars \sfH(p)x, \\
&\sfN(\alpha^{*} p \downarrow \beta x) := 0, \ 
\sfN(\vars p \downarrow \beta x) := 0 
\end{split} 
\]
for $i \in Q_{0}, \alpha, \beta  \in Q_{1}, p \in \Pa$ and $x  \in \YYM_{ n-2}$ 
to which we use the identifications 
\begin{equation}\label{202105211909}
\begin{split}
\PPi_{1} \otimes_{\Pa} \PPa \otimes_{\Pa}  \YYM_{ n-2} 
& \cong \Pa V^{*} \Pa\Pa_{0} \YYM \oplus \Pa S \Pa \Pa_{0} \YYM \oplus  \Pa V^{*} \Pa(V[1]) \YYM \oplus \Pa S \Pa(V[1]) \YYM, \\
 \PPi_{1} \otimes_{\Pa} \YYM_{n -1} &\cong \Pa V^{*} \YYM \oplus \Pa S \YYM. 
 \end{split}
\end{equation}

Then $\sfN$ is a homotopy from $(\ppi_{1,\YYM_{n -1}} ) \eeta^{*}_{n}({}_{\PPi_{1}} \mu_{\YYM_{n-2}} ) $ to ${}_{\PPi_{1}}\brho_{n -1}$.
\begin{equation}\label{202007302239}
\begin{xymatrix}{
\PPi_{1} \otimes_{\Pa} \PPa \otimes_{\Pa} \YYM_{n -2} 
\ar[d]_{{}_{\PPi_{1}}\mu_{\YYM_{ n-2}}  }\ar@{=}[rrrr] &&&& 
\PPi_{1} \otimes_{\Pa} \PPa \otimes_{\Pa} \YYM_{n -2} 
 \ar[d]^{ {}_{\PPi_{1}}\brho_{n -1}  }
 \\
\PPi_{1} \otimes_{\Pa} \YYM_{n -2}
 \ar@{=>}[urrrr]^{\sfN}  
\ar[rr]_{\eeta^{*}_{n}} && \YYM_{1} \otimes_{\Pa} \YYM_{n - 1}
\ar[rr]_{\ppi_{1, \YYM_{ n-1}} } & &
\PPi_{1} \otimes_{\Pa} \YYM_{ n-1}  \\
}\end{xymatrix} 
\end{equation}

Thus, we have the following commutative diagram in $\sfD(\Pa^{\mre})$
\[
\begin{xymatrix}@C=60pt{
 \PPi_{1}\lotimes_{\Pa} \YYM_{n -2}\ar[d]_{\eeta^{*}_{n}} \ar@{=}[r] & \PPi_{1}\lotimes_{\Pa} \YYM_{n -2}  \ar[d]^{{}_{\PPi_{1} }\brho_{n -1} } & \\
\YYM_{1} \lotimes_{\Pa} \YYM_{n -1} \ar[r]_-{\ppi_{1, \YYM_{n -1} }} & \PPi_{1} \lotimes_{\Pa} \YYM_{ n-1}.
}\end{xymatrix}
\]

\end{lemma}

\begin{proof}
We check the equation 
$(\ppi_{1,\YYM_{n -1}} ) \eeta^{*}_{n}({}_{\PPi_{1}} \mu_{\YYM_{n-2}} ) - {}_{\PPi_{1}}\brho_{n -1}= d\sfN + \sfN d$ 
on the generators by direct calculation. 
For simplicity, we denote the left hand side and the right hand side of the equation by $L$ and $R$. 
We point out a key equation $\sfH(p) - \sfH(\beta p) =-  \sfH(\beta) p = - \beta^{\circledast}p$ 
for $p \in \Pa, \beta\in Q_{1}$. Also note that in the $\stackrel{(1)}{=}, \stackrel{(2)}{=}$ below we use Lemma \ref{homotopy lemma s}. 

\begin{align*}
L (\alpha^{*} p e_{i} x) 
&=  \ppi_{1,  \YYM_{n -1}}  \eeta^{*}_{n}(\alpha^{*}px) - \alpha^{*}p \varrho x \\
&=  \ppi_{1, \YYM_{n -1}}  ( [\alpha^{*}, \varrho] px ) - \alpha^{*}p \varrho x \\
&= \alpha^{*} \varrho p x - \varrho \alpha^{*}p x - \alpha^{*}p \varrho x  
\displaybreak[1] \\ 
R (\alpha^{*} p e_{i} x) 
&= d( \vars\alpha^{*} px - \alpha^{*} \sfH(p) x ) + \sfN ( \alpha^{*} p e_{i} d(x) ) \\
& =- \varrho \alpha^{*} px - \vars\alpha^{*} p d(x) - \alpha^{*} d(\sfH(p)) x + \alpha^{*}\sfH(p) d(x) 
 + \vars \alpha^{*} p d(x) - \alpha^{*}\sfH(p)d(x)  \\
& = - \varrho \alpha^{*} px  - \alpha^{*} d(\sfH(p)) x  \\
& \stackrel{(1)}{=} - \varrho \alpha^{*} px  + \alpha^{*} [\varrho, p] x  \\
&= \alpha^{*} \varrho p x - \varrho \alpha^{*}p x - \alpha^{*}p \varrho x.  
\displaybreak[1]\\
L(\vars_{j} pe_{i}x) 
& = (\ppi_{1, \YYM_{n -1}}) \eeta^{*}_{n}(\vars_{j} px) -\vars_{j} p\varrho x\\ 
& = (\ppi_{1,\YYM_{n-1}}) (- \varrho^{\circledast}_{j} px) - \vars_{j} p\varrho x\\ 
& = (\ppi_{1,\YYM_{n-1}}) \left(v_{j}^{-1} \sum_{\beta} \beta^{\circledast}\beta^{*} px-  v_{j}^{-1} \sum_{\beta} \beta^{*}\beta^{\circledast} px\right) -\vars_{j} p\varrho x\\ 
& = v_{j}^{-1} \sum_{\beta} [\vars,\beta] \beta^{*} px- v_{j}^{-1} \sum_{\beta} \beta^{*}\beta^{\circledast} px -\vars_{j} p\varrho x
\displaybreak[1]\\ 
R(\vars_{j} pe_{i} x) 
&= d (\vars \sfH(p) x) + \sfN(-\varrho_{j} pe_{i}x - \vars_{j} pe_{i} d(x)) 
\\ 
& =- \varrho_{j}  \sfH(p)x - \vars_{j} d\sfH(p) x  + \vars_{j} \sfH(p)d(x) 
+ \sfN\left(v_{j}^{-1}  \sum_{\beta}\beta^{*} \beta pe_{i}x - v_{j}^{-1} \sum_{\beta} \beta\beta^{*}pe_{i}x  - \vars_{j} pe_{i} d(x)\right) 
\displaybreak[1]\\
& = v_{j}^{-1}  \sum_{\beta} \beta^{*} \beta\sfH(p)x  - v_{j}^{-1}  \sum_{\beta} \beta \beta^{*}   \sfH(p)x - \vars_{j} d\sfH(p) x   + \vars_{j} \sfH(p)d(x)  \\ 
&\ \ \ \ + 
v_{j}^{-1}  \sum_{\beta}\left( \vars_{j}  \beta^{*} \beta px- \beta^{*}\sfH(\beta p) x \right) 
- v_{j}^{-1}  \sum_{\beta} \left( \beta \vars_{j}  \beta^{*}px - \beta \beta^{*} \sfH(p) x\right)   - \vars_{j} \sfH(p) d(x) 
\displaybreak[1]\\
& = v_{j}^{-1}  \sum_{\beta} \left( \beta^{*} \beta\sfH(p)x - \beta^{*}\sfH(\beta p) x\right)   - \vars_{j} d\sfH(p) x +
v_{j}^{-1}  \sum_{\beta} \vars_{j}  \beta^{*} \beta px -v_{j}^{-1}  \sum_{\beta} \beta \vars_{j}  \beta^{*}px   
\displaybreak[1]\\
& \stackrel{(2)}{=} - v_{j}^{-1} \sum_{\beta}   \beta^{*} \beta^{\circledast} px  +  \vars_{j} [\varrho ,p] x +
v_{j}^{-1}  \vars_{j} \left( \sum_{\beta} \beta\beta^{*}\right)   px - \vars_{j}  \varrho px  - v_{j}^{-1}  \sum_{\beta} \beta \vars_{j}  \beta^{*}px  
\displaybreak[1]\\
& = - v_{j}^{-1}  \sum_{\beta}   \beta^{*} \beta^{\circledast}p x  -  \vars_{j} p\varrho x +
v_{j}^{-1}  \vars_{j} \left( \sum_{\beta} \beta\beta^{*}\right)   px   -  v_{j}^{-1} \sum_{\beta} \beta \vars_{j}  \beta^{*}px 
\displaybreak[1]\\
& = v_{j}^{-1}  \sum_{\beta} [\vars_{j},\beta] \beta^{*} px- v_{j}^{-1} \sum_{\beta} \beta^{*}\beta^{\circledast} px -\vars_{j} p\varrho x
\displaybreak[1]\\ 
L(\alpha^{*} p \downarrow \beta x) & = - \alpha^{*} p \beta^{\circledast} x 
\displaybreak[1]\\ 
R(\alpha^{*} p \downarrow \beta x) & 
= \sfN(\alpha^{*} p \beta e_{t(\beta)} x - \alpha^{*} p e_{h(\beta)} \beta x - \alpha^{*} \downarrow \beta dx)  \\ 
&= \vars  \alpha^{*} p x- \alpha^{*} \sfH(p\beta)x  - \vars  \alpha^{*} p \beta x + \alpha^{*}\sfH(p)\beta x   \\ 
& = -\alpha^{*}p\beta^{\circledast} x
\displaybreak[1] \\ 
L(\vars  p \downarrow \beta x) & = - \vars  p \beta^{\circledast} x 
\displaybreak[1]\\ 
R(\vars  p \downarrow \beta x)  
& =  \sfN(-\varrho p \downarrow \beta x - \vars p \beta e_{t(\beta)} x + \vars  p e_{h(\beta)} \beta x + \vars  \downarrow \beta dx)\\ 
& = -\vars \sfH(p\beta)x + \vars  \sfH(p) \beta x\\ 
& = - \vars  p \beta^{\circledast} x.
\end{align*}
\end{proof}

Now we obtain the following diagram that appeared in Lemma \ref{202008021437}. 
\begin{equation}\label{202008021750}
\begin{xymatrix}{
\PPi_{1} \otimes_{\Pa} \PPa \otimes_{\Pa} \YYM_{n -2} 
\ar[d]_{{}_{\PPi_{1}}\mu_{\YYM_{n -2}} }\ar@{=}[rrrr] &&&& 
\PPi_{1} \otimes_{\Pa} \PPa \otimes_{\Pa} \YYM_{n -2} 
 \ar[d]^{ {}_{\PPi_{1}}\brho } \ar@/^70pt/[dd]^{0} &&\\
\PPi_{1} \otimes_{\Pa} \YYM_{n -2}
 \ar@{=>}[urrrr]^{\sfN}  
\ar[rr]_{\eeta^{*}}
 \ar@/_20pt/[drrr]_{0} && \YYM_{1} \otimes_{\Pa} \YYM_{n - 1}
\ar[rr]_{\ppi_{1, \YYM_{ n-1}} } & &
\PPi_{1} \otimes_{\Pa} \YYM_{ n-1}  \ar[d]^{{}_{\PPi_{1} }\ppi_{n -1} } \ar@{=>}[dll]^{\sfK} \ar@{=>}[r]^{{}_{\PPi_{1}}\sfM}&&\\
&&& \PPi_{n} \ar@{=}[r]^-{\sim} & \PPi_{1} \otimes_{\Pa} \PPi_{ n-1}&&
}\end{xymatrix} 
\end{equation}
A key step is to prove that the homotopies in the diagram satisfy the condition of Lemma \ref{202008021437}.

\begin{lemma}\label{202008021717}
We have the following equality. 
\[
\sfK({}_{\PPi_{1}}\mu_{\YYM_{n -2}} ) =  {}_{\PPi_{1}} \sfM  + ({}_{\PPi_{1} }\ppi_{n-1}) \sfN 
\]
\end{lemma} 

\begin{proof}
For simplicity, we denote the left hand side and the right hand side of the equation by $L$ and $R$. 
We check $L = R$ on  generators of $\PPi_{1} \otimes_{\Pa} \PPa \otimes_{\Pa} \YYM_{n -2}$ by direct calculation. 
We use the identification \eqref{202105211909} to exhibit  generators of $\PPi_{1} \otimes_{\Pa} \PPa \otimes_{\Pa}  \YYM_{n -2}$.  

First let $i \in Q_{0}, \alpha \in Q_{1}, p \in \Pa$ and $x \in \YYM_{ n-2}$. 
Then 
we have 
\[
\begin{split} 
L(\alpha^{*}p e_{i} x ) &= [\vars, \alpha^{*}] p\pi(x),\\ 
R(\alpha^{*} pe_{i} x)
 & = - \alpha^{*} p\vars  \ppi(x) + \vars \alpha^{*} p \ppi(x) - \alpha^{*} \ppi ( \sfH(p))\ppi(x)\\ 
 & \stackrel{(1)}{=} - \alpha^{*} p\vars  \ppi(x) + \vars \alpha^{*} p \ppi(x) - \alpha^{*} [\vars , p]\ppi(x)\\ 
& = \vars  \alpha^{*} p \ppi(x) - \alpha^{*}\vars p \ppi(x) \\
& = [\vars , \alpha^{*}]p \ppi(x). \\
L(\vars p e_{i} x) &=\vars ^{2} p \ppi(x)\\
R(\vars  pe_{i} x) &  \stackrel{(2)}{=} \vars p \vars \ppi(x) + \vars  \ppi\sfH(p) \ppi(x) \\
& \stackrel{(3)}{=} \vars  p \vars  \ppi(x) + \vars [\vars ,p] \ppi(x) \\
& = \vars ^{2} \ppi(x)
\end{split}
\]
where for $\stackrel{(1)}{=}, \stackrel{(3)}{=}$ we use Lemma \ref{homotopy lemma s}. 
For  the first term of  $\stackrel{(2)}{=}$ we use Koszul sign rule as below. 
\[
({}_{\PPi_{1}}\sfM )( \vars pe_{i} x) = - \vars p \sfM(e_{i} x) = -\vars p( -\vars \ppi(x)) = \vars p \vars \ppi(x). 
\]
 
 We can easily check that $L$ and $R$ both vanish 
 in the generators of the forms $\alpha^{*} p \downarrow \beta x, s p \downarrow \beta x$. 
\end{proof}

We use the  notations of Section \ref{202008041341}. 
By definition, we have the equality $\Phi_{\eeta^{*}, \ppi_{1} \otimes \YYM } = \brho'$ of morphisms from $\cone(\ppi_{1}) \otimes_{A} \YYM[-1]$ 
to $\YYM$. 

The next step is to identify the morphism $ \acute{q}_{\eeta^{*}, \ppi_{1}\otimes \YYM, \PPi_{1} \otimes \ppi, \sfK}:
\cone(\Phi_{\eeta^{*}, \ppi_{1} \otimes \YYM }) \to \PPi$ with 
$q': \cone(\brho') \to \PPi$. 

\begin{lemma}\label{202008041355} 
 We have \[
 \acute{q}_{\eeta^{*}, \ppi_{1}\otimes \YYM, \PPi_{1} \otimes \ppi, \sfK} = q'. 
 \]
 \end{lemma}

\begin{proof}
Recall that 
\[
q' = (\ppi, \sfM'\uparrow): \cone(\brho') = \YYM \oplus ( \cone(\ppi_{1} ) \otimes_{\Pa} \YYM) \to \PPi. 
\]
Since,
by definition $\sfM' = ({}_{\PPi_{1} } \ppi) (\sfh_{\ppi_{1}})_{ \YYM}$
and \[
(\sfh_{\ppi_{1}})_{ \YYM} = (\downarrow, 0) : \cone(\ppi_{1}) \otimes \YYM [-1] = \
 (\PPi_{1} \otimes_{\Pa} \YYM [-1])\oplus ( \YYM_{1} \otimes_{\Pa} \YYM) 
 \to 
 \PPi_{1} \otimes \YYM, 
 \]
we obtain the following  description of the second component $\sfM'\uparrow: \cone(\ppi_{1}) \otimes_{\Pa} \YYM \to \PPi$. 
\[
\sfM' \uparrow =({}_{\PPi_{1} }\ppi, 0) : 
\cone(\ppi_{1} ) \otimes_{\Pa} \YYM = (\PPi_{1} \otimes_{\Pa} \YYM )\oplus ( \YYM_{1} \otimes_{\Pa} \YYM[1]) 
\to  \PPi_{1} \otimes_{\Pa} \PPi = \PPi_{\geq 1}. 
\]

Consequently, we come to the desired equation.
\[
\begin{split}
q' &
= (\ppi, \sfM'\uparrow) = (\ppi_{1} \otimes \ppi, \sfK\uparrow, {}_{\PPi_{1} }\ppi, 0) \\
&= 
( ({}_{\PPi_{1}} \ppi) (\ppi_{1, \YYM}), \sfK\uparrow, {}_{\PPi_{1} } \ppi, 0)\\
&= \acute{q}_{\eeta^{*}, \ppi_{1} \otimes \YYM, \PPi_{1} \otimes\ppi, \sfK}.
\end{split}
\]

\end{proof}

Thanks to Lemma \ref{202008021717}, we can apply Lemma \ref{202008021437} 
to the diagram \eqref{202008021750}.  
By Lemma \ref{202008041355} we obtain the following commutative diagram
\begin{equation}\label{202008021755} 
\begin{xymatrix}{
\cone(\brho'_{n} ) \ar[dd]_{q'_{n}} & \cone(\Phi_{\eeta^{*}, \ppi_{1}\otimes \YYM_{n -1}} )
\ar[l]_-{\sfind} \ar[r]^-{\dot{q}}  & \cone(\PPi_{1} \otimes \brho_{n -1} ) \ar[d]^{\sfiso} \\ 
&& \PPi_{1} \otimes \cone(\brho_{n-1} ) \ar[d]^{{}_{\PPi_{1} } q_{n -1}} \\
\PPi_{n} \ar@{=}[rr] && \PPi_{1} \otimes_{\Pa} \PPi_{ n-1} 
}\end{xymatrix} 
\end{equation} 
where $\dot{q}$ is a homotopy equivalence and $\sfiso$ denotes a canonical isomorphism. 
The morphism $\sfind$ is induced from a quasi-isomorphism ${}_{\PPi_{1}} \mu_{\YYM_{n-2}}: \PPi_{1} \otimes_{\Pa} \PPa \otimes_{\Pa} \YYM_{ n-2} 
\to \PPi_{1} \otimes_{\Pa} \YYM_{n -2}$. 
Therefore it  is a quasi-isomorphism.

\subsubsection{Proof of Theorem \ref{exact triangle U2}}

We prove that $q_{n}$ and $q'_{n}$ 
are quasi-isomorphisms by induction on $n \geq 0$. 

Since $\brho_{0} =0, \brho'_{0} = 0, \ppi_{0} = \id_{\Pa}, \sfM_{0} = 0$, we see that $q_{0} = q'_{0} = \id_{\Pa}$ and 
hence they are quasi-isomorphisms.

The morphism $q_{1}$ and $q'_{1}$ are quasi-isomorphism by Lemma \ref{pre universal Auslander-Reiten triangle lemma}. 

We deal with the case $n \geq 2$. 
We assume $q_{n-1}$ is a quasi-isomorphism. 
Then,  
 all the morphisms in  the diagram \eqref{202008021755} are quasi-isomorphisms 
except $q'_{n}$.  
Therefore,   
$q'_{n}$ is also a quasi-isomorphism. 
Thanks to Lemma \ref{202007312238} that $q_{n}$ is also a quasi-isomorphism. 
\qed

\subsubsection{}

The commutative diagram given below plays a key role later.

\begin{corollary}\label{202008091343}
For $n \geq 2$, we have the following commutative diagram in $\sfD(\Pa^{\mre})$. 
\[\begin{xymatrix}@C=60pt{
 &
\PPi_{1} \lotimes_{\Pa} \YYM_{n -2} \ar@{=}[r] \ar[d]_{\eeta^{*}_{n}} 
 &
 \PPi_{1} \lotimes_{\Pa} \YYM_{n -2}  \ar[d]^{ {}_{\PPi_{1}} \brho_{n -1}} 
 &
\\
 \YYM_{n -1}   \ar@{=}[d]  \ar[r]^{ \rrho_{\YYM_{n -1} } } 
 &
\YYM_{1} \lotimes_{\Pa} \YYM_{n -1}  \ar[r]^{\ppi_{1, \YYM_{n -1} }  } \ar[d]_{\zzeta_{n} } 
&
\PPi_{1} \lotimes_{\Pa} \YYM_{n- 1} \ar[d]^{ {}_{ \PPi_{1} }\ppi_{n -1}} \ar[r]^{- \ttheta[1]_{ \YYM_{n -1}}}  
& \YYM_{n -1} [1] \ar@{=}[d]\\ 
\YYM_{n-1}   \ar[r]_{\brho_{n}}   &
\YYM_{n} \ar[r]_{\ppi_{n}}  \ar[d]
 & \PPi_{n} \ar[r] \ar[d]  & \YYM_{ n-1}[1]\\
 & \PPi_{1} \lotimes_{\Pa} \YYM_{n -2}[1] \ar@{=}[r] & \PPi_{1} \lotimes_{\Pa} \YYM_{n -2}[1] &
}\end{xymatrix}
\]
where two middle rows and two middle columns are exact triangles. 
\end{corollary}

\begin{proof}
It only remains to show that the two middle columns give a morphism of exact triangles. 
Recall that we are identifying $\YYM_{n}$ with $ \cone(\eeta^{*}_{n})$ and $\zzeta_{n}$ with a canonical morphism $i_{1}^{\eeta^{*}_{n}}$. 
For simplicity we set $\phi :=  (\eeta^{*}_{n})({}_{\PPi_{1}} \mu_{\YYM_{n-2}})$, 
then we have that following commutative diagram in $\sfC(\Pa)$
\[\begin{xymatrix}@C=40pt{
\PPi_{1} \otimes_{\Pa} \YYM_{n -2} \ar[d]_{\eeta^{*}_{n}}
 &
\PPi_{1} \otimes_{\Pa} \PPa \otimes_{\Pa}  \YYM_{n -2} \ar[l]_{{}_{\PPi_{1}} \mu_{\YYM_{n-2}} } \ar@{=}[r] 
\ar[d]_{ \phi } 
 &
 \PPi_{1} \otimes_{\Pa} \PPa \otimes_{\Pa} \YYM_{n -2}  \ar[d]^{ {}_{\PPi_{1}} \brho_{n -1} } \ar@{=}[r]
 &  \PPi_{1} \otimes_{\Pa} \PPa \otimes_{\Pa} \YYM_{n -2} \ar[d]^{ {}_{\PPi_{1}} \brho_{n -1} }
\\
 \YYM_{1} \otimes_{\Pa} \YYM_{n -1}   \ar@{=}[r]  \ar[d]_{ \zzeta_{n}  } 
 &
\YYM_{1} \otimes_{\Pa} \YYM_{n -1}  \ar[r]_{\ppi_{1, \YYM_{n -1} }  } \ar[d]_{i_{1}^{\phi}} \ar@{=>}[ur]^{\sfN} 
&
\PPi_{1} \otimes_{\Pa} \YYM_{n- 1} 
\ar[d]^{i_{1}^{{}_{\PPi_{1} } \brho_{n-1} }}   \ar@{=}[r]
& \PPi_{1} \otimes_{\Pa} \YYM_{n- 1} \ar[d]^{ {}_{ \PPi_{1} }\ppi_{n -1}} 
\\ 
\YYM_{n}  &
\cone(\phi)  \ar[l]^{\Psi_{1}}  \ar[r]_{\Psi_{2}}
 & \cone({}_{\PPi_{1}} \brho_{n -1} ) \ar[r]_{q}  & \PPi_{n} 
}\end{xymatrix}
\]
where we set 
\[
\Psi_{1} = \begin{pmatrix} \id & 0 \\ 0 & {}_{\PPi_{1}}\mu_{\YYM_{n -2}} \end{pmatrix}, \quad
\Psi_{2} = \begin{pmatrix} \ppi_{1, \YYM_{n -1}} & \sfN\uparrow \\ 0 & \id \end{pmatrix}, \quad
q = ({}_{\PPi_{1}} \ppi_{n -1}. \sfM\uparrow)
\]
Observe that all columns extend to exact triangles and morphisms between them induce morphisms between exact triangles. 
Moreover the morphisms from the second column to the first column are quasi-isomorphisms. 
Thus it is enough to verify the equality $\ppi_{n} = q \Psi_{2} \Psi_{1}^{-1}$ in $\sfD(\Pa)$. 
The last equality follows from the equality $\ppi_{n} \Psi_{1} = q \Psi_{2}$ 
that can be checked from Lemma \ref{202008021653} and Lemma \ref{202008021717}.
\end{proof}

\subsubsection{Proof of Theorem \ref{exact triangle U}}

We prove Theorem \ref{exact triangle U} by reducing it to Theorem \ref{exact triangle U2}.

\begin{lemma}\label{homotopic lemma 3}
Define a morphism $\sfO:  \PPa \otimes_{\Pa} \YYM( -1) \to \YYM$ in $\sfC_{\DG}(\Pa \Gr)$ of cohomological  degree $-1$
by the formula 
\[
\sfO(e_{i} x) := -e_{i} \sfH(x), \ \sfO(\downarrow \alpha x ) := 0
\]
for $i \in Q_{0}, \alpha \in Q_{1}$ and $x \in \YYM$ to which we use the identification 
$\PPa \otimes_{\Pa} \YYM \cong \Pa \Pa_{0} \YYM \oplus \Pa (V[1]) \YYM$. 

Then the following statements hold. 

\begin{enumerate}[(1)] 
\item 

$\sfO$ is a homotopy from 
$\brho $ to $\sfr_{\varrho} ( \mu\otimes \YYM(-1))$.

\item 
We have the following equality 
\[
\sfM = \ppi \sfO - \sfr_{\vars} \ppi(\mu_{ \YYM( -1)})
\]
of morphisms from $\PPa \otimes_{\Pa} \YYM(-1)$ to $\PPi$ of degree $-1$. 
\end{enumerate} 
\end{lemma}

Note that
we have the following diagram
\begin{equation}\label{202102051520}
\begin{xymatrix}@C=50pt@R=10pt{
&&\\
 \PPa \otimes_{\Pa} \YYM(-1)
 \ar@/^30pt/[rr]^{0}
 \ar[dd]_{\mu_{\YYM(-1)}} 
\ar[r]^{\brho} & 
\YYM \ar@{=}[dd] \ar[r]^{\ppi} \ar@{=>}[u]_{\sfM} \ar@{=>}[ddl]_{\sfO} 
& \PPi \ar@{=}[dd]\\
&&\\
\YYM(-1) \ar@/_30pt/[rr]_{0}
\ar[r]^{\sfr_{\varrho}} & \YYM \ar[r]_{\ppi} \ar@{=>}[d]^{- \sfr_{\vars} \ppi} & \PPi \\ 
&& .
}\end{xymatrix}
\end{equation}

\begin{proof}
(1) 
It is enough to show that the equation 
$ \zzeta (\rrho_{\YYM(-1)})  - \sfr_{\varrho} (\mu_{\YYM(-1)} ) =  d \sfO + \sfO d$. 
holds on the generators $e_{i} x$ and $\downarrow \alpha  x$. 
We check this by direct calculation. 
For simplicity, we denote by $L$ and $R$ the left hand side and the right had side of the equation. 
\[
\begin{split} 
L(e_{i}  x)  & = \varrho_{i} x -  e_{i} x \varrho, \\ 
R(e_{i}  x)  & = - e_{i} ((d \sfH + \sfH d)(x)) = -e_{i} (x \varrho - \varrho x) = \varrho_{i}x - e_{i} x \varrho, \\ 
L(\downarrow \alpha   x) & = \alpha^{\circledast} x,  \\ 
R(\downarrow \alpha  x) & = \sfO d( \downarrow \alpha x) = 
\sfO(\alpha e_{h(\alpha)}  x - e_{t(\alpha)}\alpha x- \downarrow \alpha  dx) \\
& = - \alpha \sfH(x) + \sfH(\alpha x)   = \sfH(\alpha) x = \alpha^{\circledast} x. 
\end{split}
\]

(2) 
We also check the equation holds on the generators $e_{i} x$ and $\downarrow \alpha  x$ by direct calculation. 
\[
\begin{split} 
(\ppi \sfO - \sfr_{\vars} \ppi(\mu_{ \YYM( -1) })) ( e_{i} x)
&= - e_{i} \ppi \sfH(x) - \sfr_{\vars}\ppi(x) \\ 
&= -[\vars, \ppi(x)]- \sfr_{\vars} \ppi(x) \\
& = - \vars\ppi(x) = \sfM(x), \\ 
(\ppi \sfO - \sfr_{\vars} \ppi(\mu_{\YYM( -1)} ))(\downarrow \alpha  x) & = 0 = \sfM(\downarrow\alpha   x)
\end{split}
\]
\end{proof}

We proceed a proof of Theorem \ref{exact triangle U}. 

\begin{proof}[Proof of Theorem \ref{exact triangle U}] 
By Lemma \ref{homotopic lemma 3}, the  diagram \eqref{202102051520} induces the following commutative diagram 
\[
\begin{xymatrix}@C=50pt@R=20pt{ 
\cone(\brho) \ar[r]^{q_{\brho, \ppi, \sfM}} \ar[d]_{\sfind} & \PPi \ar@{=}[d] \\ 
\cone(\sfr_{\rho}) \ar[r]_{q_{\sfr_{\varrho}, \ppi, - \sfr_{\vars} \ppi}} & \PPi 
}\end{xymatrix}
\]
where $\sfind$ is the morphism induced from $\mu_{ \YYM( -1)}$. 
By Theorem \ref{exact triangle U2}, the upper arrow $q_{\brho, \ppi, \sfM}$ is a quasi-isomorphism. 
Since $\mu_{\YYM( -1)}$ is a quasi-isomorphism, so is the induced morphism $\sfind$. 
Thus we conclude that $q_{\sfr_{\varrho}, \ppi,-\sfr_{\vars}\ppi}$ is a quasi-isomorphism. 
\end{proof}

\subsection{Lemmas}

We collect two lemmas  for the later quotations

\subsubsection{}

It follows from Lemma \ref{homotopy proposition} that the composite $\zzeta_{2}({}_{\YYM_{1}}\rrho- \rrho_{\YYM_{1}})$ becomes   a zero morphism
$\YYM_{1} \to \YYM_{2} $ in $\sfD(\Pa_{0}^{\mre})$. 
Since $\eeta^{*}_{2}$ is a co-cone morphism of $\zzeta_{2}$, 
there exists a morphism $f: \YYM_{1} \to \PPi_{1}$ such that ${}_{\YYM_{1}}\rrho- \rrho_{\YYM_{1}} = \eeta^{*}_{2} f$. 
The next lemma says that we can take $f$ to be $\ppi_{1}$.

\begin{lemma}\label{202001111736}
We define a morphism $\sfP: \PPa \otimes_{\Pa} \YYM_{1} \otimes_{\Pa} \PPa \to \YYM_{1}\otimes_{\Pa} \YYM_{1}$ of degree $-1$ in $\sfC_{\DG}(\Pa^{\mre})$ by the formula 
\[
\begin{split} 
\sfP(e_{i} p\alpha^{*} q e_{j}) & := \sfH(p) \alpha^{*} q + p \alpha^{*} \sfH(q), \
\sfP(e_{i} p\alpha^{\circledast} q e_{j} )  := \sfH(p) \alpha^{\circledast} q -p \alpha^{\circledast} \sfH(q), \\ 
\sfP(\downarrow \beta p\alpha^{*} q e_{j} ) & :=0,\  \sfP(\downarrow \beta p\alpha^{\circledast} q e_{j} ) :=0, 
\sfP(e_{j} p\alpha^{*} q \downarrow \gamma  )  :=0,\  \sfP(e_{i}p\alpha^{\circledast} q \downarrow \gamma )  :=0, \\
\sfP(\downarrow \beta p\alpha^{*} q \downarrow \gamma  ) & :=0,\  \sfP(\downarrow \beta p\alpha^{\circledast} q \downarrow \gamma ) :=0 
\end{split}
\]
for $i,j \in Q_{0}, \ \alpha,\beta, \gamma \in Q_{1}, \ p,q \in \Pa$ 
to which we use the identifications 
\[
\begin{split} 
\PPa \otimes_{\Pa} \YYM_{1} \otimes_{\Pa} \PPa  &\cong 
\Pa\Pa_{0} \Pa V^{*} \Pa\Pa_{0} \Pa \oplus \cdots \oplus \Pa(V[1]) \Pa V^{\circ} \Pa(V[1]) \Pa,\\
\YYM_{1}\otimes_{\Pa} \YYM_{1}  &\cong 
 \Pa V^{*} \Pa V^{*} \Pa \oplus \cdots \oplus \Pa V^{\circ } \Pa V^{\circ } \Pa\end{split}\]
 induced from 
$\PPa \cong \Pa\Pa \oplus \Pa (V[1])\Pa, \YYM_{1} \cong \Pa V^{*} \Pa \oplus \Pa V^{\circledast} \Pa$.
Then $\sfP$ is a homotopy from $\mu\otimes \YYM_{1} \otimes \rrho-\rrho \otimes \YYM_{1} \otimes \mu$ to 
$\eta^{*}_{2} \ppi_{1} (\mu \otimes \YYM_{1} \otimes \mu)$. 

\[
\begin{xymatrix}@C=50pt{
\YYM_{1} \ar[rr]^-{\ppi_{1}} && \PPi_{1} \ar[d]^{\eta^{*}_{2}} \\
\PPa \otimes_{\Pa} \YYM_{1} \otimes_{\Pa} \PPa \ar[u]^{\mu\otimes \YYM_{1} \otimes \mu}
\ar[rr]_-{\mu\otimes \YYM_{1} \otimes \rrho-\rrho \otimes \YYM_{1} \otimes \mu }&
\ar@{=>}[u]_{\sfP} & 
 \YYM_{1}\otimes_{\Pa} \YYM_{1}
}\end{xymatrix}\]

Therefore, we have the following commutative diagram in $\sfD(\Pa^{\mre})$. 
\[
\begin{xymatrix}@C=50pt{
 & \PPi_{1} \ar[d]^{\eta^{*}_{2}} \\
\YYM_{1} \ar[ur]^-{\ppi_{1}} 
\ar[r]_-{ {}_{\YYM_{1}}\rrho- \rrho_{\YYM_{1}} } & \YYM_{1}\otimes_{\Pa} \YYM_{1}
}\end{xymatrix}\]

\end{lemma}
Since it is verified by straightforward calculation, we leave the proof to the readers. 

\subsubsection{}

We may construct  a right version $\eeta^{*\textup{right}}_{n}: \YYM_{n -2} \lotimes_{\Pa} \PPi_{1} \to \YYM_{n-1} \lotimes_{\Pa} \YYM_{1}$ of the morphism $\eeta^{*}_{n}$ 
as below 
\[
\eeta^{*\textup{right}}_{n}: \YYM_{n -2} \lotimes_{\Pa} \PPi_{1} \xrightarrow{{}_{\YYM_{n-2}}\eeta^{*}_{2} }
\YYM_{n-2} \lotimes_{\Pa} \YYM_{1} \lotimes_{\Pa} \YYM_{1} 
\xrightarrow{ \ \ \ \ \ } 
\YYM_{n-1} \lotimes_{\Pa} \YYM_{1}
\]
where the second morphism is  the multiplication morphism.
The following is a right version of Lemma \ref{homotopic lemma 5}.

\begin{lemma}\label{homotopic lemma 5 right}

We have the following commutative diagram in $\sfD(\Pa^{\mre})$
\[
\begin{xymatrix}@C=60pt{
\YYM_{n -2} \lotimes_{\Pa}  \PPi_{1}\ar[d]
_{\eeta^{*, \textup{right} }_{n} } \ar@{=}[r] 
&
\YYM_{n -2} \lotimes_{\Pa}  \PPi_{1} \ar[d]^{ -\brho_{n -1, \PPi_{1} } } & \\
\YYM_{n-1} \lotimes_{\Pa} \YYM_{1} \ar[r]_-{{}_{\YYM_{n-1}}\ppi_{1 }} &
\YYM_{n -1} \lotimes_{\Pa} \PPi_{1} .
}\end{xymatrix}
\]

We use the case $n =2$ later. 
In this case, we have   
$\eeta^{*, \textup{right} }_{2} = \eeta^{*}_{2}$ 
and 
hence  the equality ${}_{\YYM_{1}}\ppi_{1}\eeta^{*}_{2} = -\brho_{1, \PPi_{1}}$.
\end{lemma}

\begin{proof} 
We can verify this lemma in a similar way of Lemma \ref{homotopic lemma 5} a left version of it. 
Here we take a different way. We show that we can deduce this lemma from 
Lemma \ref{homotopic lemma 5} for the opposite quiver $Q^{\op}$. 

We may  identify $(\kk Q)^{\op} \cong \kk (Q^{\op})$ and $\kk Q \cong (\kk (Q^{\op}))^{\op}$. 
We introduce an isomorphism 
$f: \Pa^{\mre} = \kk Q \otimes (\kk Q)^{\op} \to \kk ( Q^{\op}) \otimes (\kk ( Q^{\op}))^{\op} = \kk (Q^{\op})^{\mre}$ 
to be $f(p\otimes q)= q \otimes p$. 
 Let $f^{*}: \sfC_{\DG}(\kk (Q^{\op})^{\mre}) \to \sfC_{\DG}(\Pa^{\mre})$ be the induced isomorphism. 
Then we have $f^{*}(X \otimes_{\kk(Q^{\op})} Y) \cong Y \otimes_{\kk Q} X$.
The anti-algebra isomorphism  $\phi: \kk (Q^{\op}) \to \kk Q, \phi(\alpha_{1} \alpha_{2}\cdots, \alpha_{n}) 
:= \alpha_{n} \cdots \alpha_{2} \alpha_{1}$ induces an isomorphism $f^{*}\kk (Q^{\op}) \cong \kk Q$. 
Similar anti-$*$-graded-dg-algebra isomorphisms $\phi: \PPi(Q^{\op}) \to \PPi(Q), \phi: \YYM(Q^{\op}) \to \YYM(Q)$ denoted by the same symbol induces isomorphisms $f^{*}\PPi(Q^{\op}) \to \PPi(Q), f^{*} \YYM(Q^{\op}) \to \YYM(Q)$. 
We note that $\phi(\rho) = -\rho$ and $\phi(s_{i}) = -s_{i}$. 
We can check that under these isomorphisms $\phi$, the morphisms 
$f^{*}(\eeta^{*}_{Q^{\op}})$ and $f^{*}(\ppi_{Q^{\op}})$ correspond to $\eeta^{*, \textup{right}}$ and $\ppi$. 
Thus, from Lemma \ref{homotopic lemma 5} for $Q^{\op}$, we can deduce the desired commutativity. 
\end{proof}

\subsection{$3$-Calabi-Yau property}

\subsubsection{Calabi-Yau algebras}

Let $n, l\in \ZZ$. 
Recall that a DG-algebra $R = (\bigoplus_{i,j } R_{i}^{j}, d_{R})$ with an additional grading 
is called $n$-Calabi-Yau algebras of Gorenstein parameter $l$ 
if it is smooth and there exists an isomorphism 
$R^{\vvee}[n](-l) \cong R$ in $\sfD(R^{\mre})$.

\subsubsection{}
The aim of this theorem is to prove the following theorem. 
Note that the case $\chara \kk \neq 2$ is already shown in Proposition \ref{202112122233}. 
We give an alternative proof which works for arbitrary characteristic.

\begin{theorem}\label{202112122233}
The derived quiver Heisenberg algbebra $\YYM$ is a $3$-Calabi-Yau 
algebra of Gorenstein parameter $2$. 
\end{theorem}

\subsubsection{A preparation}

Let $M=\bigoplus_{i\in\ZZ} M^{i}$ be a cohomologically graded $\Pa_{0}$-$\Pa_{0}$-bimodule. 
Then by a DG-version of Lemma \ref{Grant-Iyama lemma}, 
the map 
\[
\begin{split}
F_{M}: & \YYM \tuD(M) \YYM \to ( \YYM M\YYM)^{\vee}, 
F_{M}(x \otimes f \otimes y) ( z \otimes m \otimes w) 
:= (-1)^{\epsilon }f(m) zy \otimes xw \\
\epsilon & := |x|(|f|+|y| + |z|+|m|) + |f|(|y|+|z|) +|y||z|
\end{split}
\]
is an isomorphism of cohomologically graded  $\YYM^{\mre}$-bimodules.

\begin{lemma}\label{202112131737}
Let $M=\bigoplus_{i \in \ZZ} M^{i}, N= \bigoplus_{i \in \ZZ} N^{i}$ be cohomologically graded $\Pa_{0}^{\mre}$-modules 
and $\{m_{i}\}, \{n_{i}\}$ homogeneous basis of $M$ and $N$. 
We identify $(\YYM M \YYM)^{\vee}, (\YYM N \YYM)^{\vee}$ with $\YYM \tuD(M) \YYM, \YYM\tuD(N)\YYM$ 
via $F_{M}, F_{N}$.  
If a homogeneous morphism 
$\phi: \YYM M \YYM \to \YYM N \YYM$ is given by  
\[
\phi(m_{j} ) = \sum_{i} a_{ij} n_{i} b_{ij}
\]
for some $a_{ij}, b_{ij} \in \YYM$, 
then the $\YYM^{\mre}$-dual $\phi^{\vee}: (\YYM N \YYM)^{\vee} \to (\YYM M \YYM)^{\vee}$ of $\phi$ is given by 
\[
\phi^{\vee}(n_{i}^{\vee} ) = \sum_{j} ( -1)^{|\phi||n_{i}| + |a_{ij}|(|m_{j}| +|n_{i}| + |b_{ij}| ) } b_{ij}m_{j}^{\vee} a_{ij}
\]
where  $\{ m_{j}^{\vee}\}, \{ n_{i}^{\vee}\}$ denote the dual basis of $\tuD(M)$ and $\tuD(N)$.
\end{lemma}

\subsubsection{Proof of Theorem \ref{202112122233}}

We prove the theorem by constructing a projectively cofibrant resolution $Q \xrightarrow{\sim} \YYM$ of $\YYM$ in 
$\sfC(\YYM^{\mre}\Gr)$ such that $Q$ is perfect and  $Q^{\vee}[3](-2) \cong Q$. 

Recall that  
$\PPi_{1} \cong \Pa V^{*} \Pa \oplus \Pa S \Pa, \ \YYM_{1} \cong \Pa V^{*} \Pa \oplus \Pa V^{\circledast} \Pa, 
\PPa \cong \Pa \Pa \oplus \Pa V[1] \Pa$ 
as cohomological graded modules over $\Pa^{\mre}$. 
It is straightforward to check  that the canonical isomorphisms $\tuD(S) =\Pa_{0}, \tuD(V^{*}) = V, \tuD(V^{*}) =V^{\circledast}$ 
induces  isomorphisms 
\begin{equation}\label{202112161145}
\PPi_{1}^{\vee}[1] \xrightarrow{\cong} \PPa,  \ \YYM^{\vee}_{1}[1] \xrightarrow{\cong}  \YYM_{1}, \ 
\PPa^{\vee}[1] \xrightarrow{\cong} \PPi_{1}   
\end{equation}
(where $(-)^{\vee}$ denotes $\Pa^{\mre}$-duality) in $\sfC(\Pa^{\mre})$.

We set $Q_{(2)}:= \YYM\otimes_{\Pa} \PPi_{1} \otimes_{\Pa} \YYM(-2), 
Q_{(1)} := \YYM \otimes_{\Pa} \YYM_{1} \otimes_{\Pa} \YYM(-1), \ 
Q_{(0)} := \YYM \otimes_{\Pa} \PPa \otimes_{\Pa} \YYM$.  
As cohomological graded modules over $\YYM^{\mre}$, we have 
$Q_{(2)}  \cong \YYM V^{*} \YYM \oplus \YYM S \YYM, 
\ Q_{(1)}\cong \YYM V^{*} \YYM \oplus \YYM V^{\circledast} \YYM, 
Q_{(0)} \cong \YYM  \YYM \oplus \YYM V[1] \YYM$.

We define a morphism $Y: Q_{(2)} \to Q_{(1)}$ in $\sfC(\YYM^{\mre}\Gr)$ 
 in the following way. For $\alpha \in Q_{1}$ and $i \in Q_{0}$, we set 
\[
\begin{split}
Y(\alpha^{*} )&
 := \alpha^{*} \otimes \varrho \otimes 1  - \varrho \otimes \alpha^{*} \otimes 1 
 + 1\otimes \alpha^{*} \otimes \varrho - 1 \otimes \varrho \otimes \alpha^{*},\\
 Y(\vars_{i})& := v_{i}^{-1} \sum_{\beta: t(\beta) = i} \beta^{\circledast} \otimes \beta^{*} \otimes 1 + 1 \otimes \beta^{\circledast} \otimes \beta^{*}- v^{-1}_{i} \sum_{\beta: h(\beta) = i} \beta^{*} \otimes \beta^{\circledast} \otimes 1 + 1\otimes \beta^{*} \otimes \beta^{\circledast}
 \end{split} 
\]
where $\otimes$ denotes $\otimes_{\Pa}$.

For $i \in Q_{0}$, 
we define an element $\varrho'_{i} \in Q_{(0)}$ of cohomological degree $-1$ 
to be
\[
\varrho'_{i} := v^{-1}_{i}\sum_{\beta: t(\beta) = i} 1 \otimes \downarrow \beta \otimes \beta^{*}  - v^{-1}_{i} 
\sum_{\beta: h(\beta) = i} \beta^{*} \otimes \downarrow \beta \otimes 1. 
\]
We define a morphism $Z: Q_{(1)} \to Q_{(0)}$ in $\sfC(\YYM^{\mre}\Gr)$ 
 in the following way. For $\alpha \in Q_{1}$ and $i \in Q_{0}$, we set 
\[
\begin{split}
Z(\alpha^{*})  &:= \alpha^{*} \otimes 1 \otimes 1 - 1 \otimes 1 \otimes \alpha^{*} \\
Z(\alpha^{\circledast})  & := 
\varrho' \alpha - \alpha \varrho' + \varrho \otimes \downarrow \alpha \otimes 1 - 1 \otimes \downarrow \alpha \otimes \varrho 
+ \alpha^{\circledast} \otimes 1 \otimes 1 - 1 \otimes 1 \otimes \alpha^{\circledast}
\end{split}
\]
where $\otimes$ denotes $\otimes_{\Pa}$.

We define a morphism $\sfG: Q_{(2) } \to Q_{(0)} $ 
in $\sfC_{\DG}(\YYM^{\mre}\Gr)$ of cohomological degree $-1$ in the following way. 
For $\alpha \in Q_{1}$ and $i \in Q_{0}$, we set 
\[
\begin{split}
\sfG(\alpha^{*}) & := -\alpha^{\circ} \otimes 1 \otimes 1 + 1 \otimes 1 \otimes \alpha^{\circ}- \alpha^{*}\varrho' + \varrho' \alpha^{*}\\
\sfG(\vars_{i} ) & := -\vart_{i} \otimes 1\otimes 1 + 1 \otimes 1\otimes \vart_{i}+ v^{-1}_{i}\sum_{\beta: t(\beta) = i} 1 \otimes \downarrow \beta \otimes \beta^{\circ}  +v^{-1}_{i} 
\sum_{\beta: h(\beta) = i} \beta^{\circ} \otimes \downarrow \beta \otimes 1.  
\end{split}
\]
where $\otimes$ denotes $\otimes_{\Pa}$.

The isomorphisms \eqref{202112161145}
induces isomorphisms 
\begin{equation}\label{202112152209}
Q_{(0)}^{\vee} (-2)[1] \cong Q_{2}, \ Q_{(1)}^{\vee} (-2)[1] \cong Q_{1}, 
\ Q_{(2)}^{\vee}(-2)[1] \cong Q_{0}
\end{equation}
(where $(-)^{\vee}$ denotes $\YYM^{\mre}$-duality) in $\sfC(\YYM^{\mre}\Gr)$.

We leave it to prove the following lemma to the readers.

\begin{lemma}\label{202112131642}

\begin{enumerate}[(1)] 
\item $\sfG$ is a homotopy from $ZY $ to $0$. 

\item The morphism  $Y^{\vee}(-2)[1]: Q_{(1)}^{\vee} (-2)[1]  
\to Q_{(2)}^{\vee} (-2)[1] 
$ corresponds to 
$Z$ under the isomorphisms \eqref{202112152209}. 

\item The morphism  $Z^{\vee}(-2)[1]: Q_{(0)}^{\vee} (-2)[1]  
\to Q_{(1)}^{\vee} (-2)[1] 
$ corresponds to 
$Y$ under the isomorphisms \eqref{202112152209}.

\item
The morphism  $\sfG^{\vee}(-2)[1]:
Q_{(0)}^{\vee}(-2)[1] \to Q_{(2)}^{\vee}(-2)[1]$ corresponds to 
$\sfG$ under the isomorphisms \eqref{202112152209}. 
\end{enumerate}
\end{lemma}

In other   words, 
we have the following diagram in $\sfC(\YYM^{\mre}Gr)$ 
\begin{equation}\label{20200731222611}
\begin{xymatrix}@C=60pt@R=6mm{
&&\\
Q_{(0)}^{\vee}(-2)[1]\ar@/^30pt/[rr]^{0} \ar[r]^{Z^{\vee}(-2)[1]} \ar[d]_{\cong} & 
Q_{(1)}^{\vee}(-2)[1] \ar@{=>}[u]_{\sfG^{\vee}(-2)[1]} 
\ar[r]^{Y^{\vee}(-2)[1]} \ar[d]_{\cong} 
& 
Q_{(0)}^{\vee}(-2)[1] \ar[d]^{\cong} 
\\
Q_{(2)} \ar[r]_{Y} \ar@/_30pt/[rr]_{0} & Q_{(1) }\ar[r]_{Z} \ar@{=>}[d]^{\sfG} & Q_{(0)}  \\
&&
}\end{xymatrix}
\end{equation}

Let $Q$ be the totalization of the lower half part of the above diagram \eqref{20200731222611}. 
Namely, we set 
\[
Q:= \left( 
Q_{(0)}\oplus Q_{(1)}[1] \oplus Q_{(2)}[2], 
\begin{pmatrix} 
d & Z\uparrow & G\uparrow^{2} \\
0 & d & -\downarrow Y \uparrow^{2} \\ 
0 & 0 & d 
\end{pmatrix} 
\right).
\] 
It follows from Lemma \ref{202112131642} that the isomorphism isomorphisms \eqref{202112152209} 
induces an isomorphism $Q^{\vee}(-2)[3] \cong Q$ in $\sfC(\YYM^{\mre})$. 
It is clear that $Q$ is projectively cofibrant in $\sfC(\YYM^{\mre})$. 
Therefore, to prove Theorem \ref{202112122233}, it is enough to show that $Q$ is quasi-isomorphic to $\YYM$ in $\sfC(\YYM^{\mre})$.

We set $Q'_{(0)} :=\YYM \YYM= \YYM \otimes_{\Pa} \Pa\otimes_{\Pa} \YYM$ and 
denote by $\mu': Q_{(0)} \to Q'_{(0)}$ the morphism induced from $\mu: \PPa \to \Pa$.  
We set $Z' := \mu' Z, \ \sfG' := \mu' \sfG$ 
and $Q'$ to be the totalization of the diagram below.
\[
\begin{xymatrix}@C=60pt@R=6mm{
Q_{(2)} \ar[r]_{Y} \ar@/_30pt/[rr]_{0} & Q_{(1) }\ar[r]_{Z'} \ar@{=>}[d]^{\sfG'} & Q'_{(0)}  \\
&&
}\end{xymatrix}
\]
Since $\mu$ is a quasi-isomorphism of projectively cofibrant DG-modules over $\Pa^{\mre}$, 
$\mu'$ is a quasi-isomorphism over $\YYM^{\mre}$ 
and hence  so is the morphism $Q \to Q'$ induced from $\mu'$. 

Let $\epsilon: \YYM \YYM \to \YYM, \ x\otimes y \mapsto xy$ be the multiplication map. 
Observe that $\epsilon Z' = 0, \epsilon \sfG = 0$ and hence that $\epsilon$ induces  
a  morphism $\epsilon_{Q'}:=(\epsilon, 0,0): Q' \to \YYM$ in $\sfC(\YYM^{\mre})$. 
We finish the proof by showing $\epsilon_{Q'}$ is a quasi-isomorphism. 

We define increasing  filtrations on $Q'$ and $\YYM$ as follows. 
For $i \in \ZZ$, we set 
\[
F_{i}(Q_{(2)}) = \YYM_{\geq -i} \otimes_{\Pa} \PPi_{1}\otimes_{\Pa} \YYM, 
F_{i}(Q_{(1)}) = \YYM_{\geq -i} \otimes_{\Pa} \YYM_{1}\otimes_{\Pa} \YYM, 
F_{i}(Q'_{(0)}) = \YYM_{\geq -i}  \YYM, 
\]
and $F_{i}(Q')$ denotes the induced filtration. 
We set $F_{i}(\YYM) = \YYM_{\geq -i}$ for $i \in \ZZ$. 
Then $\epsilon_{Q'}$ preserves filtrations. 
For $i > 0$, the graded quotients $G_{i}(Q'), G_{i}(\YYM)$ are zero. 
It is clear that $G_{0}(\YYM) \cong \Pa$ as  objects of $\sfC(\Pa\otimes\YYM^{\mre}\Gr)$. 
On the other hand,  $0$-th graded quotient $G_{0}(Q')$ is isomorphic to the object $P$ of $\sfC(\Pa\otimes \YYM^{\mre}\Gr)$ defined in 
Section \ref{2021121160751} and that the induced morphism $G_{0}(\epsilon_{Q'})$ corresponds to the quasi-isomorphism
 $\epsilon_{P}: P \to \Pa$. 
Observes that  for $i > 0$, we have an isomorphisms $G_{-i}(Q') \cong \YYM_{i} \otimes_{\Pa} P$ and $G_{-i}(\YYM) \cong \YYM_{i}$ 
 in $\sfC(\Pa\otimes \YYM^{\mre}\Gr)$ 
 and that the induced morphism $G_{-i}(\epsilon_{Q'})$ corresponds to $\YYM_{i} \otimes_{\Pa} \epsilon_{P}$. 
 Since $\YYM_{i}$ is projectively cofibrant, $\YYM_{i} \otimes_{\Pa} \epsilon_{P}$ is a quasi-isomorphism and so is $G_{-i}(\epsilon_{Q'})$. 
 Finally, since the filtration is exhausted and bounded in each $*$-degree, 
 we conclude that the morphism $\epsilon_{Q'}: Q'\to \YYM$ is 
 a quasi-isomorphism. \qed

\subsection{The exact triangle $\widehat{\sfU}$}\label{a-infinity}

\def\wsv{{\widehat{\sfr_{\varrho}}}}

The aim of this section is to show that there exists an exact triangle
 \[
\widehat{\sfU}: \YYM( -1) \xrightarrow{ \widehat{\sfr_{\varrho}} } \YYM \xrightarrow{ \ppi } \PPi \to
\]
in $\sfD(\YYM^{\mre}\Gr)$ which is sent to $\sfU$  by the forgetful functor $\sfD(\YYM^{\mre}\Gr) \to \sfD(\YYM\Gr)$. 
The reader can postpone this section until Section \ref{section: QHA Dynkin case}.

The problem here is that the morphism $\sfr_{\varrho}: \YYM( -1) \to \YYM$ dose not commute with the 
right action of $\YYM$ on $\YYM$ 
and hence it is not a morphism of dg-modules over $\YYM^{\mre}$. 
To overcome this problem, we use theory of A${}_{\infty}$-algebras for which we refer \cite{Hasegawa}.

For simplicity we set $R := \YYM^{\mre}$. 
We denote the category of $*$-graded A${}_{\infty}$-$R$-modules by $\sfC_{A_{\infty}}(R\Gr)$. 
Recall that we regard a  $*$-graded dg-$R$-module $M$ as a $*$-graded A${}_{\infty}$-$R$-modules 
by setting higher multiplication $\{ m^{M}_{n}\}_{n \geq 1}$ as 
\[
m^{M}_{n} := 
\begin{cases} 
d \ (\textup{the differential of $M$}) & ( n=1), \\
\textup{the left action map}: R\otimes M \to M & (n =2), \\
0 & (n \geq 3). 
\end{cases} 
\]
This assignment yields    a functor $\sfcan: \sfC(R\Gr) \to \sfC_{A_{\infty}}(R\Gr)$, 
which induces an equivalence $\sfD(R\Gr) \xrightarrow{\cong} \sfD_{A_{\infty}}(R\Gr)$ of derived categories. 
Thus we may identify the derived category $\sfD_{A_{\infty}}(R\Gr)$ of $*$-graded A${}_{\infty}$-$R$-modules 
with the derived category $\sfD(R\Gr)$ of $*$-graded dg-$R$-modules.

A point here is that the functor $\sfcan$ is faithful, but  not full.
In other words, 
even for $*$-graded dg-$R$-modules $M$ and $N$, 
there exists a morphism $g:=\{g_{n}\}_{n \geq 1} : M \to N$ of $*$-graded A${}_{\infty}$-$R$-modules 
which does not come from morphisms of $*$-graded dg-$R$-modules, i.e., $g_{n} \neq 0$ for some $n \geq 2$. 

We define a morphism $\wsv=\{ f_{n}\}_{n \geq 1} : \YYM(-1) \to \YYM$ of $*$-graded A${}_{\infty}$-$R$-modules 
whose $1$-st component  $f_{1}$ equals to $\sfr_{\varrho}$ in the following way. 
So, first we set $f_{1}: = \sfr_{\varrho}$. 
Next, we define a morphism $f_{2}: R \otimes \YYM(-1) \to \YYM$ of $\sfC_{\DG}(\kk\Gr)$ of cohomological degree $-1$ to be 
\[
f_{2}((r \otimes s) \otimes x) := ( -1)^{|r|+ (|s| + 1)|x|} rx\sfH(s)
\]
for $ r \otimes s \in R$ and $x \in \YYM$. 
Finally, for $n \geq 3$, we set $f_{n} := 0$.

\begin{lemma}\label{202106071504}
The collection $\widehat{\sfr_{\varrho}} := \{f_{n}\}_{n \geq 1}$ is a morphism of $*$-graded A${}_{\infty}$-$R$-modules.
\[
\widehat{\sfr_{\varrho}}: \YYM(-1) \to \YYM.
\]
\end{lemma}

\begin{proof}
Since almost all the relevant morphisms are zero, 
we only have to check the following three equations 
\[
\begin{split}
f_{1}m^{\YYM}_{1} - m^{\YYM}_{1}f_{1}  &= 0,\\
f_{1} m^{\YYM}_{2} -m^{\YYM}_{2}({}_{R} f_{1}) &=  m^{\YYM}_{1} f_{2}+f_{2} (\sfd_{\YYM}) + f_{2}({}_{R}m^{\YYM}_{1}),\\ 
f_{2}({}_{R} m^{\YYM}_{2}) - f_{2} \sfm_{\YYM} &= -m^{\YYM}_{2} ({}_{R}f_{2})  
\end{split}
\]
where  $\sfd: R \to R$ is the differential and $\sfm: R\otimes R \to R$ is the multiplication map. 

The first one is just says that $f_{1} = \sfr_{\varrho}$ is a cochain map. 
Thus it is already verified. 

We set the left hand side and the right hand side of the second equation by $L$ and $R$. 
Then 
\[
\begin{split}
L((r\otimes s) \otimes x) & = (-1)^{|s||x|}rxs \varrho - (-1)^{|s||x|}rx\varrho s= (-1)^{|s||x|} rx[s, \varrho], \\
R((r\otimes s) \otimes x) 
& = d(( -1)^{|r|+ (|s| + 1)|x|} rx\sfH(s)) +\\
& \ \ \ \ \ \ \  f_{2}\left((dr \otimes s) \otimes x + (-1)^{|r|}(r \otimes ds) \otimes x +(-1)^{|r|+ |s|}(r \otimes s ) \otimes x \right)\\
&  = (-1)^{|s||x|}rx(d\sfH(s) +\sfH d(s))\\
& =(-1)^{|s||x|} rx[s, \varrho]
\end{split}
\]
This shows $L=R$.

We set the left hand side and the right hand side of the third equation by $L$ and $R$. 
Then 
\[
\begin{split}
L((r\otimes s) \otimes (t \otimes u) \otimes x) 
& = (-1)^{|x||u|} f_{2}((r\otimes s) \otimes txu) -(-1)^{|s||t|+|u||s|} f_{2}((rt \otimes us ) \otimes x) \\
& = (-1)^{|x||u|+|r| + (|t|+ |x| + |u|)(|s| +1) } r txu\sfH(s)  \\
& \ \ \ \ \ \ \  -(-1)^{|s||t|+ |u||s| + |r| + |t| + |x|(|s| +|u| + 1) } rt  x\sfH(us) \\
& = (-1)^{|s||t|+ |u||s| + |r| + |t| + |x|(|s| +|u| + 1) }\left(( -1)^{|u|} r txu\sfH(s) - rt  x\sfH(us) \right) \\ 
& = (-1)^{|s||t|+ |u||s| + |r| + |t| + |x|(|s| +|u| + 1) +1} rt  x\sfH(u)s,  \\ 
R((r\otimes s) \otimes (t \otimes u) \otimes x) 
& =-(-1)^{|t| + (|u|+1)|x|+|r|+|s| } m_{2} \left( (r\otimes s) t x \sfH(u)\right) \\
& =(-1)^{|t| + (|u|+1)|x| +|r| + |s| +1+(|t| + |x| + |u| -1)|s| } rt x \sfH(u)s 
\end{split}
\]We check that the exponents of $(-1)$ of both equations coincide and verify $L =R$. 
\end{proof}

By the definition of composition of morphisms of A${}_{\infty}$-modules, 
the composition 
$\ppi \widehat{\sfr_{\varrho}} :\YYM \to \PPi$ is the collection  $\ppi \widehat{\sfr_{\varrho}} =\{ \ppi f_{n}\}_{n \geq 1}$.

Next  we construct a homotopy $\cH$ from $\ppi\wsv$ the composition  to 0. 
We define morphisms $\cH_{n}: R^{\otimes n -1} \otimes \YYM \to \YYM$ in $\sfC_{\DG}(\kk\Gr)$ of cohomological  degree $ -n$ 
to be 
\[
\cH_{n}: =
\begin{cases} 
- \sfr_{\vars}\ppi & ( n = 1), \\
0 & ( n \geq 2).
\end{cases}
\]

\begin{lemma}\label{202106071527}
The collection $\cH := \{ \cH_{n} \}_{n \geq 1}$ 
is a homotopy from $\ppi \wsv$ to $0$ of morphisms of A${}_{\infty}$-$R$-modules.
\end{lemma}  

\begin{proof}
Since almost all the relevant morphisms are zero, 
we only have to check the following two equations 
\[
\begin{split}
\ppi f_{1}&= m^{\PPi}_{1}\cH_{1} + \cH_{1} m^{\YYM}_{1}, \\
\ppi f_{2}  &=  -m^{\PPi}_{2} ({}_{R} \cH_{1}) + \cH_{1} m^{\YYM}_{2}.\\  
\end{split}
\]
The first one is already verified in Lemma \ref{homotopy 20191203}. 

We set the left hand side and the right hand side of the second equation by $L$ and $R$. 
Then, we have 
\[
\begin{split}
L((r\otimes s) \otimes x) &= 
(-1)^{|r| +(|s| +1)|x|} \ppi(r) \ppi(x) \ppi(\sfH(s)) \\
& \stackrel{(*)}{=} (-1)^{|r| +(|s| +1)|x|} \ppi(rx) [\vars, \ppi(s)]\\
R((r\otimes s) \otimes x) 
& = -(-1)^{|r|+ |s|}m^{\PPi}_{2}((r\otimes s) \otimes \cH_{1}(x) ) + (-1)^{|s||x|} \cH_{1}(r xs)\\
& = (-1)^{|r|+ |s|}m^{\PPi}_{2}((r\otimes s)\otimes  \sfr_{\vars}\ppi(x)) + (-1)^{|s||x|+1} \sfr_{\vars}\ppi(r xs)\\
& = (-1)^{|r|+|s| +|x| +(|x|+1)|s|}\ppi(r)\ppi(x)\vars  \ppi(s) + (-1)^{|s||x| +1 + |r| + |x| + |s|} \ppi(rxs) \vars\\
& = (-1)^{|r|+|x|+|s||x|}\left( \ppi(rx)\vars \ppi(s) - ( -1)^{|s|}\ppi(rx) \ppi(s) \vars\right)\\
& =  (-1)^{|r| +(|s| +1)|x|} \ppi(rx) [\vars, \ppi(s)]
 \end{split}
 \]
where for $\stackrel{(*)}{=}$ we use Lemma \ref{homotopy lemma s}. 
This shows $L =R$.  
\end{proof}

Let $C := \cone(\wsv)$ be the cone of $\wsv=\{ f_{n}\}_{n \geq 1} : \YYM \to \YYM$. 
Namely, it is a $*$-graded A${}_{\infty}$-$R$-module (which is actually a $*$-graded dg-$R$-module) whose underlying cohomological graded object 
is $\YYM \oplus \YYM[1]$ with the higher multiplications $\{m^{C}_{n} \}_{n \geq 1}$ given as below
\[
m^{C}_{1} = \begin{pmatrix} d_{\YYM} & f_{1} \uparrow\\ 0 & d_{\YYM[1]} \end{pmatrix}, \ 
m^{C}_{2} = \begin{pmatrix} m_{2}  &- f_{2} \uparrow \\ 0 & m_{2} \end{pmatrix}, \
m^{C}_{n } = 0  \ \ \ (n \geq 3).
\]

The homotopy $\cH$ induces a morphism $\hat{q}=\{ \hat{q}_{n}\}_{n \geq 1} : C\to \PPi$ of A${}_{\infty}$-$R$-modules, 
which is defined by 
\[
\hat{q}_{1} := (\ppi, \cH_{1} \uparrow), \ \hat{q}_{n} = 0  \ \ \ (n \geq 2). 
\]

\begin{theorem}\label{202106071837}
The morphism $\hat{q}: C \to \PPi$ is a quasi-isomorphism of A${}_{\infty}$-$\YYM^{\mre}$-modules. 
Hence we have an exact triangle 
\[
\widehat{\sfU}: \YYM( -1) \xrightarrow{ \widehat{\sfr_{\varrho}} } \YYM \xrightarrow{ \ppi } \PPi \to \YYM( -1)[1]. 
\]
in $\sfD(\YYM^{\mre}\Gr)$ 
which is sent to the exact triangle $\sfU$ by the forgetful functor $\sfD(\YYM^{\mre}\Gr) \to \sfD(\YYM\Gr)$. 
\end{theorem}

\begin{proof}
Recall that a morphism $g=\{ g_{n}\}_{n \geq 1}: M \to N$ of A${}_{\infty}$-modules is called a \emph{quasi-isomorphism} 
if the $1$-st component $g_{1}$  is a quasi-isomorphism of complexes. 
Since  $\hat{q}_{1} = q_{\sfr_{\varrho}, \ppi, - \sfr_{\vars} \ppi}$  is a quasi-isomorphism by Theorem \ref{exact triangle U}, 
the morphism $\hat{q}: C \to \PPi$ is a quasi-isomorphism. 

 The forgetful functor $\sfD(\YYM^{\mre}\Gr) \to \sfD(\YYM\Gr)$ is the restriction functor along the morphism 
 $a: \YYM \to \YYM^{\mre}, r \mapsto r \otimes 1$. 
 It follows that the morphism $\wsv: \YYM( -1) \to \YYM$ of $\sfC_{A_{\infty}}(\YYM^{\mre})$ 
 is sent to $\sfr_{\varrho}: \YYM( -1) \to \YYM$. 
 Thus we conclude that the exact triangle $\widehat{\sfU}$ is sent to $\sfU$ by the forgetful functor $\sfD(\YYM^{\mre}\Gr) \to \sfD(\YYM\Gr)$. 
\end{proof}

\section{QHA of non-Dynkin type  for non-sincere weight}\label{202207111353}

In this section, we study the  QHA ${}^{v}\!\YM$ of a non-Dynkin quiver $Q$ for arbitrary weight $v \in \kk Q_{0}$. 
So throughout this section $Q$ denotes a non-Dynkin quiver.

As we point out in Proposition \ref{202207111600} below, 
 the derived QHA ${}^{v}\!\YYM$ defined only for a sincere weight $v$,  
 is quasi-isomorphic to the (non-derived) QHA ${}^{v}\!\YM$. 
The aim of this section is to show that almost all results about ${}^{v}\!\YM \simeq {}^{v}\!\YYM$ given in the previous section, holds true for ${}^{v}\!\YM$ even in the case where the weight $v$ is not sincere.

\subsection{Preparations}

\subsubsection{}
First we point out that the derived QHA and the non-derived QHA are quasi-isomorphic.

\begin{proposition}\label{202207111600}
Let $v \in \kk^{\times}Q_{0}$ be a sincere weight. 
Then the derived QHA ${}^{v}\!\YYM$ is concentrated in the $0$-th cohomological grading and 
consequently the canonical morphism ${}^{v}\!\YYM \to {}^{v}\!\YM$ is a quasi-isomorphism. 
\end{proposition}

\begin{proof}
By Theorem \ref{exact triangle U}, there exists an exact triangle 
${}^{v}\!\YYM_{n} \xrightarrow{ {}^{v}\!\varrho} {}^{v}\!\YYM_{n+1} \xrightarrow{ {}^{v}\ppi } \PPi_{n+1} \to$ 
for $n \geq 0$. 
Since $Q$ is non-Dynkin, $\PPi_{n}$ is concentrated in $0$-th cohomological degree. 
Thus by induction on $n$, we can show that ${}^{v}\!\YYM_{n}$ is concentrated in $0$-th cohomological degree. 
\end{proof}

Thanks to  Theorem \ref{202112122233}, we deduce the following theorem.  

\begin{theorem}\label{202111261228}
Let $Q$ be a non-Dynkin quiver and $v\in \kk^{\times} Q_{0}$ is sincere. 
Then the QHA ${}^{v}\!\YM(Q)$ is $3$-Calabi-Yau. 
\end{theorem}

\subsubsection{Flatness of the deformation family $\Pi_{\bullet}$ of the preprojective algebras over $\kk Q_{0}$}

Recall that the deformation family of the preprojective algebras is defined to be 
\[
\Pi_{\bullet} :=\Pi(Q)_{\bullet} := \frac{\kk [x_{1}, \ldots, x_{r}]\overline{Q}}{ (\rho_{i} -x_{i}e_{i} \mid i \in Q_{0}) } 
\] 
where $r = \# Q_{0}$. 
Let $R:=\kk [x_{1}, \dots, x_{r}]$  be the coordinate ring of $\kk Q_{0}$. 
Then the deformation family $\Pi_{\bullet}$  is an algebra over $R$. 

We give a proof of  the following well-known result.

\begin{theorem}\label{2022071015211}
The algebra $\Pi_{\bullet}$ is flat over $R$. 
\end{theorem}

\begin{proof}
We may assume that $\kk$ is algebraically closed. 
In this proof, we equip $\overline{Q}$ with the path length grading and $R$ is regraded as in $0$-th degree with respect to this grading. 
Then the grading gives a filtration $\{\Pi_{\bullet, \leq n} \mid n \geq 0 \}$ on $\Pi_{\bullet}$.
Since it is exhaustive, it is enough to show that $\Pi_{\bullet, \leq n}$ is flat over $R$ for any $n \geq 0$.

For $\lambda=(\lambda_{i})_{i} \in  \kk Q_{0}$, 
we denote by  $\Pi_{\lambda} := \kk \overline{Q}/(\rho_{i} -\lambda_{i} \mid i \in Q_{0})$ the deformed preprojective algebra. 
Then we have $\Pi_{\bullet}\otimes_{R} \kappa(\lambda)  \cong \Pi_{\lambda}$ 
where $\kappa(\lambda) = R/(x_{i} -\lambda_{i} \mid i \in Q)$ is the residue field and we identify it with $\kk$.  
The grading also gives a filtration  $\{ \Pi_{\lambda, \leq n} \mid n \geq 0\}$ on $\Pi_{\lambda}$ and 
the above isomorphism compatible with filtrations. 
Namely we have  $\Pi_{\bullet, \leq n }\otimes_{R} \kappa(\lambda)  \cong \Pi_{\lambda, \leq n}$ for $n \geq 0$. 

By \cite[Lemma 2.3]{Crawley-Boevey-Kimura}, the $n$-th graded quotient $\Pi_{\lambda, \leq n}/\Pi_{\lambda, \leq n -1}$ 
is isomorphic to the $n$-th degree part $\Pi_{n}$ of $\Pi$ with respect to the path length grading. 
It follows that  $\dim_{\kk} \Pi_{\lambda, \leq n}$ is independent on $\lambda$ 
and hence $\Pi_{\lambda, \leq n}$ is flat over $R$ by Lemma \ref{202103141428}.  
\end{proof}

\subsection{}

\subsubsection{}

We recall from Lemma \ref{202011242056} that 
if a weight $v \in \kk^{\times} Q_{0}$ is  sincere, then 
\[
{}^{v}\!\YM(Q) \cong \frac{\kk[z]\overline{Q}}{( \rho_{i} -v_{i}z e_{i} \mid i \in Q_{0} )}.
\]
We identify these two algebras. 
In the case $v \in \kk Q_{0}$ is  not necessarily sincere, we set 
\begin{equation}\label{202404181706}
{}^{v}\!\YM(Q) := \frac{\kk[z]\overline{Q}}{( \rho_{i} -v_{i}z e_{i} \mid i \in Q_{0} )}.\end{equation}
We regard this algebra as $*$-graded algebra by setting 
$\deg^{*} \alpha := 0, \ \deg^{*} \alpha^{*} :=1 \textup{ for } \alpha \in Q_{1}$ and $\deg^{*} z: =1$. 
Recall that in the case where  $v$ is sincere, the element $z \in {}^{v}\!\YM$ coincides with the weighted mesh relation ${}^{v}\!\varrho$ 
under the above identification. 
Abusing the notation, even in the case $v$ is not sincere, we denote by ${}^{v}\!\varrho: {}^{v}\!\YM \to {}^{v}\!\YM$ 
the multiplication by $z \in {}^{v}\!\YM$.
\[
{}^{v}\!\varrho: {}^{v}\!\YM \to {}^{v}\!\YM, \ x \mapsto xz.
\]

\begin{lemma}\label{202207111415}
The element $z \in {}^{v}\!\YM$ is regular, i.e., the multiplication map ${}^{v}\!\varrho$ by $z$ is injective. 
\end{lemma} 

\begin{proof}
Let $v \in \kk Q_0$. 
We define a homomorphism $\psi_{v}: R \to \kk[z]$ of algebras to be $\psi_{v}(x_{i}):= v_{i} z \ \ (i \in Q_{0})$. 
Then  ${}^{v}\!\YM$ is isomorphic to $\Pi_{\bullet} \otimes_{R} {}_{\psi_{v}}\kk[z]$ as algebras over $\kk[z]$ and it is flat over $\kk[z]$ by Theorem \ref{2022071015211}. 
Thus in particular the multiplication by $z$ is injective. 
\end{proof}

Let ${}^{v}\!\pi: {}^{v}\!\YM \to \Pi$ be the canonical projection. 
As is the same with the sincere case, the kernel $\Ker {}^{v}\!\pi$ is generated by $z$ and we have the following exact sequence 
\begin{equation}\label{202009212139grns}
 0\to {}^{v}\!\YM(-1) \xrightarrow{ \  {}^{v}\!\varrho  \ }  {}^{v}\!\YM \xrightarrow{ \ {}^{v}\!\pi \ } \Pi \to 0
\end{equation}
of graded ${}^{v}\!\YM$-modules. 
Looking at the $*$-degree $1$ part of this exact sequence, 
we obtain an exact sequence of  $\Pa$-bimodules
\begin{equation}\label{202207131440}
0 \to {}^{v}\!\YM_{n -1}  
\xrightarrow{ \ {}^{v}\!\varrho \ } {}^{v}\!\YM_{n} \xrightarrow{ \ {}^{v}\!\pi_{n} \ } \Pi_{n}  \to 0.
\end{equation}
In particular the $*$-degree $1$-parti is of the following form:
\begin{equation}\label{universal AR-sequence ns}
0 \to \Pa \xrightarrow{ \ {}^{v}\!\varrho \ } {}^{v}\!\YM_{1} \xrightarrow{ \ {}^{v}\!\pi_{1} \ } \Pi_{1}  \to 0.
\end{equation}
Since we are assuming $Q$ is non-Dynkin, we have $\Pi_{1}\cong \PPi_{1} \cong \Pa^{\vvee}[1]$.
We denote the Yoneda class of this exact sequence by ${}^{v}\!\theta' \in \Ext_{\Pa^{\mre}}^{1}(\Pi_{1}, \Pa)\cong \Hom_{\Pa^{\mre}}(\Pa^{\vvee}, \Pa)$. 

On the other hand, 
we introduced the element ${}^{v}\!\ttheta \in \Hom_{\Pa^{\mre}}(\Pa^{\vvee}, \Pa)$ in Definition \ref{202102061701}.
In the case $v$ is sincere, we have the equality  ${}^{v}\!\ttheta = {}^{v}\!\theta'$. 
In the next lemma, we prove this equality holds even for non-sincere weight. 

\begin{lemma}\label{2022007111441}
We have the following equality in $\Hom_{\Pa^{\mre}}(\Pa^{\vvee}, \Pa)$: 
\[
{}^{v}\!\ttheta ={}^{v}\!\theta'.
\]
\end{lemma}

Proof is given in Section \ref{202207111421}.
We point out the following corollary. 

\begin{corollary}\label{202207112325}
For $n \geq 0$, the $*$-degree $n$-part ${}^{v}\!\YM_{n}$ is 
preprojective as  an $\Pa$-module (resp. as an $\Pa^{\op}$-module). 
\end{corollary}

\begin{proof}
For $n \geq 0$, we have an exact sequence 
$ 0 \to {}^{v}\!\YM_{n} \to {}^{v}\!\YM_{n +1} \to \Pi_{n+1} \to 0$. 
Since the class of preprojective modules is closed under extension, 
inductively we can show that ${}^{v}\!\YM_{n}$ is preprojective as an $\Pa$-module. 
\end{proof}

\subsubsection{}

We define a morphism ${}^{v}\!\zeta: {}^{v}\!\YM_{1} \otimes_{\Pa} {}^{v}\!\YM \to {}^{v}\!\YM$ to be the multiplication map 
\[
{}^{v}\!\zeta: {}^{v}\!\YM_{1} \otimes_{\Pa} {}^{v}\!\YM \to {}^{v}\!\YM, \ \zzeta(x \otimes y) := xy. 
\]
We denote the $*$-graded version  by the same symbol ${}^{v}\!\zeta: {}^{v}\!\YM_{1} \otimes_{\Pa} {}^{v}\!\YM( -1) \to {}^{v}\!\YM$. 
We denote the $*$-degree $n$-component by ${}^{v}\!\zeta_{n}: {}^{v}\!\YM_{1} \otimes_{\Pa} {}^{v}\!\YM_{n -1}  \to {}^{v}\!\YM_{n }$. 
Since ${}^{v}\!\YM$ is generated by $*$-degree $0,1$-part, the map ${}^{v}\!\zeta$ is surjective.

Observe that the homomorphism $\eta^{*}_{2}: AV^{*} A \to {}^{v}\!\YM \otimes_{\Pa} {}^{v}\!\YM$ of $A^{\mre}$-modules 
given by ${}^{v}\!\eta^{*}_{2}(\alpha^{*}) := \alpha^{*} \otimes z - z\otimes \alpha^{*}$ 
induces a homomorphism ${}^{v}\!\eta^{*}_{2}: \Pi_{1} \to {}^{v}\!\YM \otimes_{\Pa} {}^{v}\!\YM$ of $\Pa^{\mre}$ denoted by the same symbol.

We set  ${}^{v}\!\eta^{*}_{n} := ({}_{{}^{v}\!\YM_{1} }{}^{v}\!\zeta_{ n-1}) ({}^{v}\!\eta^{*}_{2, {}^{v}\!\YM_{ n-2} })$ in $\Pa^{\mre}\Mod$.
\[
{}^{v}\!\eta^{*}_{n}:  \Pi_{1} \otimes_{A} {}^{v}\!\YM_{ n-2} \xrightarrow{ {}^{v}\!\eta^{*}_{2, {}^{v}\!\YM_{ n-2} } }
 {}^{v}\!\YM_{1} \otimes_{A} {}^{v}\!\YM_{1} \otimes_{A} {}^{v}\!\YM_{n -2} 
\xrightarrow{ {}_{{}^{v}\!\YM_{1} }{}^{v}\!\zeta_{ n-1}  } 
{}^{v}\!\YM_{1} \otimes_{A} {}^{v}\!\YM_{n-1}.
\]
Combining ${}^{v}\!\eta_{n}$ for $n \geq 2$, we obtain a morphism ${}^{v}\!\eta: \Pi_{1}\otimes_{\Pa}{}^{v}\!\YM \to {}^{v}\!\YM_{1} \otimes_{A} {}^{v}\!\YM$ in $\Pa \otimes \YM^{\op}\Mod$.

We note that these morphisms ${}^{v}\!\zeta, \ {}^{v}\!\eta^{*}$ coincide with ${}^{v}\zzeta, \ {}^{v}\eeta^{*}$ in the case where the weight $v$ is sincere.

\begin{lemma}\label{202008091343ns}
For $n \geq 2$, we have the following commutative diagram where  rows and  columns are exact sequences in $\Pa^{\mre}\mod$: 
\[\begin{xymatrix}@C=60pt{
&& 0 \ar[d] & 0 \ar[d] & \\
 &&
\Pi_{1} \otimes_{\Pa} {}^{v}\!\YM_{n -2} \ar@{=}[r] \ar[d]_{{}^{v}\!\eta^{*}_{n}} 
 &
 \Pi_{1} \otimes_{\Pa} {}^{v}\!\YM_{n -2}  \ar[d]^{ {}_{\Pi_{1}} {}^{v}\!\varrho_{n -1}} 
 &
\\
0\ar[r] 
& {}^{v}\!\YM_{n -1}   \ar@{=}[d]  \ar[r]^{ {}^{v}\!\varrho_{{}^{v}\!\YM_{n -1} } } 
 &
{}^{v}\!\YM_{1} \otimes_{\Pa} {}^{v}\!\YM_{n -1}  \ar[r]^{{}^{v}\!\pi_{1, {}^{v}\!\YM_{n -1} }  } \ar[d]_{{}^{v}\!\zeta_{n} } 
&
\Pi_{1} \otimes_{\Pa} {}^{v}\!\YM_{n- 1} \ar[d]^{ {}_{ \Pi_{1} }{}^{v}\!\pi_{n -1}} \ar[r]  
& 0\\ 
0 \ar[r] &{}^{v}\!\YM_{n-1}   \ar[r]_{{}^{v}\!\varrho_{n}}   &
{}^{v}\!\YM_{n} \ar[r]_{{}^{v}\!\pi_{n}}  \ar[d]
 & \Pi_{n} \ar[r] \ar[d]  & 0\\
 & & 0 & 0 &
}\end{xymatrix}
\]
Note that by Corollary \ref{202207112325}, 
the tensor products $\otimes_{\Pa}$ in the above diagram compute the derived tensor products $\lotimes_{\Pa}$. 
\end{lemma}

\begin{proof}
It is straightforward to check the commutativity of the three squares of the diagram.
The bottom row is exact by Lemma \ref{202207111415}. 
Since $\Tor_{>0}^{A}(\Pi_{1}, {}^{v}\!\YM_{n -1}) = 0 $ by Corollary \ref{202207112325},  
it follows from Lemma \ref{202207111415} that the top row is exact. 
Similarly, since $\Tor_{>0}^{A}(\Pi_{1}, \Pi_{n -1}) = 0 $,   
it follows from Lemma \ref{202207111415} that the right column is exact. 
Now by Snake Lemma we see that the  left column is exact. 
\end{proof}

\subsubsection{Equations}

By a straight forward computation we can check the following  versions of Lemma \ref{202001111736} and Lemma \ref{homotopic lemma 5 right}.

\begin{lemma}\label{202207131538}
We have the following equalities 
\[ 
\begin{split}
({}^{v}\!\eta^{*}_{2})({}^{v}\pi_{1}) & ={}_{{}^{v}\!\YM_{1}}{}^{v}\!\varrho- {}^{v}\!\varrho_{{}^{v}\!\YM_{1}}\\
({}_{{}^{v}\!\YM_{1}}{}^{v}\pi_{1})({}^{v}\!\eta^{*}_{2})  &= - {}^{v}\!\varrho
\end{split}
\]
In other words, the following diagram is  commutative in $\Pa^{\mre}\mod$. 
\[
\begin{xymatrix}@C=50pt{
 & \Pi_{1} \ar[d]^{{}^{v}\!\eta^{*}_{2}} \ar[dr]^{-{}^{v}\!\varrho_{\Pi_{1}}} \\
{}^{v}\!\YM_{1} \ar[ur]^-{{}^{v}\!\pi_{1}} 
\ar[r]_-{ {}_{{}^{v}\!\YM_{1}}{}^{v}\!\varrho- {}^{v}\!\varrho_{{}^{v}\!\YM_{1}} } & {}^{v}\!\YM_{1}\otimes_{\Pa} {}^{v}\!\YM_{1} 
\ar[r]_{{}_{{}^{v}\!\YM_{1}}{}^{v}\pi_{1} } 
& {}^{v}\!\YM_{1} \otimes_{\Pa} \Pi_{1}
}\end{xymatrix}\]
\end{lemma}

\subsection{Proof of Lemma \ref{2022007111441}}\label{202207111421}
We may assume that $\kk$ is algebraically closed.

Let $S:= \kk [y_{1}, \dots, y_{r}]$ be (a copy of) the  coordinate ring of $\kk Q_{0}$.   
Finally we set $T:=S[z]$. 
Note that there is a canonical injection  $S \hookrightarrow T$ and a canonical sujection $T \to S$ that sends $z$ to $0$. 
Let $f: R \to T$ be the algebra homomorphism given by $f(x_{i}) := y_{i} z$ for all $i \in Q_{0}$. 

We set 
\[
{}^{S}\!\YM := \Pi_{\bullet} \otimes_{R}{}_{f}T =\frac{T\overline{Q}}{(\rho_{i} - y_{i}ze_{i} \mid i \in Q_{0})}
\]
We regard ${}^{S}\!\YM$ as a $*$-graded algebra by setting $\deg^{*} y_{i} =0, \ \deg^{*} z :=1$ for $i \in Q_{0}$. 
Note that the $*$-degree $0$-part $({}^{S}\!\YM)_{0}$ coincides with the path algebra $SQ$ of $Q$ with the coefficients in $S$. 

Observe that the canonical homomorphism $T\overline{Q} \to S\overline{Q}$ of algebras 
 induces a surjective homomorphism ${}^{S}\!\pi: {}^{S}\!\YM  \to  \Pi \otimes_{\kk} S$ of algebras 
 whose kernel is generated by $z$.  
Since ${}^{S}\!\YM$ is flat over $T$, the multiplication by $z$ is injective. 
In other words, there is an exact sequence of ${}^{S}\!\YM$-modules 
\begin{equation}\label{202207101845}
0\to {}^{S}\!\YM \xrightarrow{  \   \ z  \ \   } {}^{S}\!\YM \xrightarrow{ \ \ {}^{S}\!\pi \ \ } \Pi \otimes_{\kk} S \to 0.
\end{equation}

We set $SQ^{\mre_{/S}} := SQ \otimes_{S} (SQ)^{\op}$. 
Looking at the $*$-degree $0$ part of \eqref{202207101845}, we obtain an exact sequence of $SQ^{\mre_{/S}}$-modules: 
\begin{equation}\label{2022071018451}
0\to SQ \xrightarrow{  \   \ z  \ \   } {}^{S}\!\YM_{1} \xrightarrow{ \ \ {}^{S}\!\pi \ \ } \Pi_{1} \otimes_{\kk} S \to 0.
\end{equation} 
We denote by  ${}^{S}\!\theta' \in \Ext_{SQ^{\mre_{/S}}}^{1}(\Pi_{1} \otimes_{\kk} S, SQ)$  the corresponding Yoneda class. 

Note that we have $\Ext_{SQ^{\mre_{/S}}}^{1}(\Pi_{1} \otimes_{\kk} S, SQ) \cong \Ext_{\Pa^{\mre}}^{1}(\Pi_{1}, \Pa)\otimes_{\kk} S$, since $SQ \cong \Pa \otimes_{\kk} S$.
We  define the element ${}^{S}\!\ttheta$ of $\Ext_{\Pa^{\mre}}^{1}(\Pi_{1}, \Pa)\otimes_{\kk} S$ to be  
\[
{}^{S}\!\ttheta := \sum_{i\in Q_{0}}\tilde{e}_{i} \otimes y_{i}
\]
where we regard  the elements $\tilde{e}_{i} \in \Hom_{\Pa^{\mre}}(\Pa^{\vvee}, \Pa)$  given  in Section \ref{Section:trace formula}   
as  elements of $\Ext_{\Pa^{\mre}}^{1}(\Pi_{1}, \Pa)$.

Let  $v \in \kk Q_{0}$. 
We may identify the residue field $\kappa(v) :=S/(y_{i} - v_{i} \mid i\in Q_{0})$ with $\kk$. 
We extends the canonical projection $S \to \kappa(v)=\kk$ to $g_{v}: T=S[z] \to \kk[z]$.
We have an isomorphism 
${}^{S}\!\YM \otimes_{S} \kappa(v) \cong {}^{v}\!\YM$ of algebras. 
Moreover taking $- \otimes_{T}{}_{g_{v}} \kk[z]$ to the exact sequence 
\eqref{202207101845} we obtain the exact sequence \eqref{202009212139grns}. 
It follows that ${}^{S}\!\theta' \otimes_{S} \kappa(v) ={}^{v}\!\theta'$ in 
$\Ext_{SQ}^{1}(\Pi_{1} \otimes_{\kk} S, SQ) \otimes_{S} \kappa(v) \cong \Ext_{\Pa}^{1}(\Pi_{1}, \Pa)$. 

On the other hand, it is clear that ${}^{S}\!\ttheta \otimes_{S} \kappa(v) ={}^{v}\!\ttheta$ in $\Ext_{\Pa^{\mre}}^{1}(\Pi_{1}, \Pa)$.

Let $U:= \kk[y_{i}, y_{i}^{-1} \mid i \in Q_{0}]$ be the coordinate ring of the space  $\kk^{\times} Q_{0}$ of sincere weights. 
By construction, if $v$ is sincere, then we have ${}^{v}\!\ttheta = {}^{v}\!\theta'$. 
It follows that ${}^{S}\!\ttheta \otimes_{S} U = {}^{S} \!\theta' \otimes_{S} U$ in $\Ext_{\Pa^{\mre}}^{1}(\Pi_{1}, \Pa) \otimes_{\kk} U$. 
Since the localization $S \to U$ is injective, we conclude that ${}^{S}\!\ttheta  = {}^{S} \!\theta'$ as desired. 
\qed

\section{Minimal right and left   $\rad^{n}$-approximations}\label{section: right rad-n approximation}

It is convenient to set 
\[
\We_{Q} := 
\begin{cases} \kk^{\times }Q_{0} & (Q: \textup{ Dynkin}), \\ 
\kk Q_{0} & (Q: \textup{ non-Dynkin}). \\ 
\end{cases}
\]
From now until the end of Section \ref{section: minimal left rad^n-approximations}, we fix an element $v=(v_{i}) \in \We_{Q}$ 
and omit $v$ from most of our notation. 

\subsection{Minimal right $\rad^{n}$-approximations}
It follows from 
Theorem \ref{universal Auslander-Reiten triangle}  that for $M \in \sfD^{\mrb}(\Pa \mod)$, 
the morphism $\ppi_{1, M}  : \YYM_{1} \lotimes_{\Pa} M \to \PPi_{1} \lotimes_{\Pa} M$ is a minimal right  $\rad$-approximation. 
The aim of this subsection is to prove a higher version of this statement. 

For this purpose, we need to introduce the following condition.  

\subsubsection{The property (I)${}_{M, n}$}

\begin{definition}
Let $M$ be an indecomposable  object of  $\Dbmod{\Pa}$ and $ n \geq 1$. 
We say that $v \in \We_{Q}$ has the property (I)${}_{M, n}$ if 
we have 
$
{{}^{v}\!\Euch(N)} \neq 0$  
for any $N \in \ind\add \{ \YYM_{m } \lotimes_{\Pa} M \mid 0 \leq m \leq n-1 \}$. 
\end{definition}
 
 We note that the locus of $v \in \We_{Q}$ that has the property (I)${}_{M, n}$ is determined by  finite number of linear equations.

 \subsubsection{} 

Let $M \in \Dbmod{\Pa}$. 
Applying $- \lotimes_{\Pa} M$ to the exact triangle \eqref{202008141924} (or \eqref{202207131440}),  
we obtain the following exact triangle in $\Dbmod{\Pa}$. 
\[
\YYM_{n -1} \lotimes_{\Pa} M \xrightarrow{\brho_{n, M}} \YYM_{n } \lotimes_{\Pa} M  \xrightarrow{ \ppi_{n, M}} \PPi_{n} \lotimes_{\Pa} M 
\xrightarrow{- \ttheta_{n, M}[1]}
\YYM_{n -1} \lotimes_{\Pa} M[1]
\] 
where we set  $\ttheta_{n, M}: \PPi_{n}[-1] \lotimes_{\Pa} M \to \YYM_{n -1} \lotimes_{\Pa} M$ 
to be the connecting morphism of the above exact triangle. 
 We remark that in this notation 
 the AR-coconnecting morphism $\ttheta_{M}$ of \eqref{202008141807} is denoted by $\ttheta_{1, M}$. 

To state the next theorem, we use the subset $N_{Q} \subset \NN$ given in Definition \ref{definition: NQ}.

\begin{theorem}\label{right approximation theorem}
Let $Q$ be a finite acyclic quiver,  $M \in\ind  \sfD^{\mrb}(\Pa\mod)$ and $n \in N_{Q}^{\geq 1}$. 
Assume that $v \in \We_{Q}$ has the property (I)${}_{M, n}$. 
Then, 
the following statements hold. 
\begin{enumerate}[(1)] 
\item 
The morphism $\ppi_{n, M}: \YYM_{n} \lotimes_{\Pa} M \to \PPi_{n} \lotimes_{\Pa} M$ 
is a  minimal right $\rad^{n}$-approximation of $M$.

\item 
The morphism $\eeta^{*}_{n, M}$ is a split-monomorphism.

\item $\YYM_{n} \lotimes_{\Pa} M \neq 0$.

\item 
The morphism $\brho_{n, M} : \YYM_{n -1} \lotimes_{\Pa} M \to \YYM_{n } \lotimes_{\Pa} M $ satisfies 
the left  and the right $\rad$-fitting condition. 
\end{enumerate}
\end{theorem}

\begin{remark}

It is clear that $v \in \kk Q_{0}$ is regular  
if and only if it has property (I)${}_{M, n}$ for all $M \in \ind\Dbmod{\Pa}$ and $n \geq 1$.  
Therefore, if $v\in \kk Q_{0}$ is regular,  
then the conclusion of Theorem \ref{right approximation theorem} hold for 
all objects $M \in \Dbmod{\Pa}$. 

\end{remark}

\begin{proof}
We use induction on $n$. 
The case $n=1$ follows from Theorem \ref{universal Auslander-Reiten triangle}. 

We deal with the case $n \geq 2$. 
We assume that the case $n -1$ is already verified. 

Applying $- \lotimes_{\Pa} M$ to the diagram of Corollary \ref{202008091343} or Lemma \ref{202008091343ns}, 
we obtain the following commutative diagram. 
\begin{equation}\label{202008141848}
\begin{xymatrix}@C=40pt{
 &
\PPi_{1} \lotimes_{\Pa} \YYM_{n -2} \lotimes_{\Pa} M 
\ar@{=}[r] \ar[d]_{\eeta^{*}_{n, M}} 
 &
 \PPi_{1} \lotimes_{\Pa} \YYM_{n -2}   \lotimes_{\Pa} M 
 \ar[d]^-{ {}_{\PPi_{1} }\brho_{n-1, M}} 
 &
\\
 \YYM_{n -1}\lotimes_{\Pa} M    \ar@{=}[d]  \ar[r]^{ \rrho_{\YYM_{n-1} \lotimes M} } 
 &
\YYM_{1} \lotimes_{\Pa} \YYM_{n -1}  \lotimes_{\Pa} M 
\ar[r]^{\ppi_{1, \YYM_{n -1} \lotimes M}  } \ar[d]_{\zzeta_{n, M} } 
&
\PPi_{1} \lotimes_{\Pa} \YYM_{n- 1} \lotimes_{\Pa} M\ar[d]^{{}_{\PPi_{1}}\ppi_{n -1, M}} \ar[r]^{- \ttheta_{1,  \YYM_{n -1} \lotimes M} [1] }  
& \YYM_{n -1} \lotimes_{\Pa} M [1] \ar@{=}[d]\\ 
\YYM_{n-1} \lotimes_{\Pa} M   \ar[r]_{\brho_{n, M}}   &
\YYM_{n} \lotimes_{\Pa} M \ar[r]_{\ppi_{n, M}}  \ar[d]
 & \PPi_{n}\lotimes_{\Pa} M  \ar[r]_{-\ttheta_{n, M}[1]} \ar[d]  & \YYM_{ n-1}\lotimes_{\Pa} M[1]\\
 & \PPi_{1} \lotimes_{\Pa} \YYM_{n -2}\lotimes_{\Pa} M[1] \ar@{=}[r] & \PPi_{1} \lotimes_{\Pa} \YYM_{n -2}\lotimes_{\Pa} M[1] &
}\end{xymatrix}
\end{equation}
where two middle rows and two middle columns are exact triangles.
It follows from statement (4) for $n-1$  that the morphism ${}_{\PPi_{1}} \brho_{n -1, M}$ at the top of the third column  satisfies the right $\rad$-fitting condition. 
By Lemma \ref{202105171508}, the morphism $\eeta^{*}_{n, M}$ is a split monomorphism. 
This proves statement (2) for $n$. 
It follows that $\zzeta_{n, M}$ is a split epimorphism. 
Since the second row is a direct sum of Auslander-Reiten triangles, 
by the right version of Lemma \ref{202012191919}, we conclude that $\ppi_{n, M}$ is a minimal right $\rad^{n}$-approximation. 
Thanks to a right version of Theorem \ref{kQ approximations theorem}, we conclude that statements (3) and (4) for $n$ hold. 
\end{proof}

\subsubsection{}

The following assertion is a consequence of Theorem  \ref{right approximation theorem}.  

\begin{corollary}\label{202007181611} 
Let $M \in \ind \sfD^{\mrb}(\Pa\mod)$ and $n \in N_{Q}^{\geq 2}$. 
Assume that $v \in \We_{Q}$ has property ${(I)}_{M, n}$. 
Then the following statements hold. 

\begin{enumerate}[(1)] 
\item There exists an isomorphism 
\[
\YYM_{1} \lotimes_{\Pa} \YYM_{n-1} \lotimes_{\Pa} M \cong (\YYM_{n } \lotimes_{\Pa} M) \oplus ( \PPi_{1} \lotimes_{\Pa} \YYM_{n-2} \lotimes_{\Pa}M ) 
\]
under which the morphisms 
\[
\begin{split}
&\eeta^{*}_{n, M}: \PPi_{1} \lotimes_{\Pa} \YYM_{n -2} \lotimes_{\Pa} M \to \YYM_{1} \lotimes_{\Pa} \YYM_{n-1} \lotimes_{\Pa} M, \\
&\zzeta_{n, M} :\YYM_{1} \lotimes_{\Pa} \YYM_{n -1} \lotimes_{\Pa} M  \to \YYM_{n } \lotimes_{\Pa} M. 
\end{split}
\] 
correspond to 
the canonical injection and the canonical projection. 

\item 
We have a direct sum of Auslander-Reiten triangle starting from $\YYM_{n-1} \lotimes_{\Pa} M$ which is of the following form. 
\begin{equation}\label{202007192007}
\YYM_{ n-1} \lotimes_{\Pa} M \to 
(\YYM_{n } \lotimes_{\Pa} M) \oplus ( \PPi_{1} \lotimes_{\Pa} \YYM_{n-2} \lotimes_{\Pa}M ) 
  \to  \PPi_{1} \lotimes_{\Pa} \YYM_{n -1} \lotimes_{\Pa} M \to
  \YYM_{ n-1} \lotimes_{\Pa} M[1].  
\end{equation}
where, the first morphism is of the forms 
\[
\begin{pmatrix} 
\rho_{n , M} \\ -\alpha_{n, M} \end{pmatrix}: 
\YYM_{ n-1} \lotimes_{\Pa} M \to 
(\YYM_{n } \lotimes_{\Pa} M) \oplus ( \PPi_{1} \lotimes_{\Pa} \YYM_{n-2} \lotimes_{\Pa}M ) 
\]
for some morphism $\alpha_{n, M}$,   
and the second is 
\[
(\beta_{n, M}, {}_{\PPi_{1} }\rho_{n -1, M} ) : (\YYM_{n } \lotimes_{\Pa} M) \oplus ( \PPi_{1} \lotimes_{\Pa} \YYM_{n-2} \lotimes_{\Pa}M ) 
  \to  \PPi_{1} \lotimes_{\Pa} \YYM_{n -1} \lotimes_{\Pa} M 
\]
for some $\beta_{n ,M}$ such that $({}_{\PPi_{1}}\ppi_{n-1, M} ) \beta_{n, M} = \ppi_{n}$.

\item 
The morphisms $\alpha_{n, M}$ and $\beta_{n, M}$ given in (2)  provide  a morphism of exact triangles of the following form
\[
\begin{xymatrix}@C=40pt{
\YYM_{n -1} \lotimes_{\Pa} M \ar[d]_{\alpha_{n ,M} }  \ar[r]^{\brho_{n ,M} } & \YYM_{n} \lotimes_{\Pa} M \ar[d]^{\beta_{n, M}} \ar[r]^{\ppi_{n}} & 
\PPi_{n}\lotimes_{\Pa} M  \ar@{=}[d] \ar[r] & \YYM_{n -1} [1] \ar[d]^{\alpha_{n, M} [1]} \\
\PPi_{1} \lotimes_{\Pa} \YYM_{n -2} \lotimes_{ \Pa} M \ar[r]_{{}_{\PPi_{1} }\brho_{n -1 ,M} } & 
\PPi_{1} \lotimes_{\Pa} \YYM_{n -1} \lotimes_{ \Pa} M \ar[r]_-{ {}_{\PPi_{1}} \ppi_{n -1, M} } & 
\PPi_{n } \lotimes_{\Pa} M  \ar[r] & \PPi_{1} \lotimes_{\Pa} \YYM_{n -2}[1]  
}\end{xymatrix}
\]
\end{enumerate} 
\end{corollary}

\begin{proof}
(1) and (2) are left to the readers. 

(3) The exact triangle \eqref{202007192007} gives that the left square of the following diagram is 
a homotopy Cartesian square (see Section \ref{section: homotopy Cartesian square}) 
\[
\begin{xymatrix}@C=40pt{
\YYM_{n -1} \lotimes_{\Pa} M \ar[d]_{\alpha_{n ,M} }  \ar[r]^{\brho_{n ,M} } & 
\YYM_{n} \lotimes_{\Pa} M \ar[d]^{\beta_{n, M}} \ar@{-->}[r] & 
\PPi_{n}\lotimes_{\Pa} M  \ar@{=}[d] \ar@{-->}[r] & \YYM_{n -1} [1] \ar[d]^{\alpha_{n, M} [1]} \\
\PPi_{1} \lotimes_{\Pa} \YYM_{n -2} \lotimes_{ \Pa} M \ar[r]_{ {}_{\PPi_{1}  }\brho_{n ,M} } & 
\PPi_{1} \lotimes_{\Pa} \YYM_{n -1} \lotimes_{ \Pa} M \ar[r]_{ {}_{\PPi_{1} }\ppi_{n -1, M} } & 
\PPi_{n } \lotimes_{\Pa} M  \ar[r] & \PPi_{1} \lotimes_{\Pa} \YYM_{n -2}[1]  
}\end{xymatrix}
\]
Therefore by a dual version of \cite[Lemma 1.4.4]{Neeman}, 
there exist dotted arrows that make the top row exact. 
Since $({}_{ \PPi_{1}} \ppi_{n-1, M} ) \beta_{n, M} = \ppi_{n}$, the first dotted arrow is $\ppi_{n, M}$. 
Thus we obtain the desired diagram. 
\end{proof}

\subsubsection{}

For later quotation,
we collect the following lemma that follows from  the above commutative diagram \eqref{202008141848}. 
 
 \begin{lemma}\label{202012211305}
 There exists the following commutative diagram
 \[
 \begin{xymatrix}@C=50pt{
 \PPi_{1} \lotimes_{\Pa} \YYM_{ n-1} \lotimes_{\Pa} M[-1] 
 \ar[d]_{ {}_{\PPi_{1}}\ppi_{n -1, M}[-1]}  \ar[r]^-{\ttheta_{1, \YYM_{n -1} \lotimes M}} & \YYM_{n -1}\lotimes_{\Pa} M\ar@{=}[d] \\
 \PPi_{n } \lotimes_{\Pa} M[-1]  \ar[r]_{\ttheta_{n, M}} & \YYM_{n -1} \lotimes_{\Pa} M.  
}\end{xymatrix} 
 \] 
 
 In other words, we have
 the following equality of morphisms $\PPi_{1} \lotimes_{\Pa} \YYM_{ n-1} \lotimes_{\Pa} M \to \YYM_{n-1}\lotimes_{\Pa} M$: 
\[
\ttheta_{n, M} ({}_{\PPi_{1} }\ppi_{n -1, M}[-1] )= \ttheta_{1, \YYM_{n -1} \lotimes M}.
\]
\end{lemma}

\subsection{Left $\rad^{n}$-approximation}

Next we discuss left $\rad^{n}$-approximations. 
For this we introduce the following conditions. 

\subsubsection{The property (I')${}_{M, n}$}

\begin{definition}
Let $M$ be an indecomposable  object of  $\Dbmod{\Pa}$ and $ n \geq 1$. 
We say that $v \in \We_{Q}$ has the property (I')${}^{}_{M, n}$ if 
it has the property (I)${}_{M ,n}$ and its right version, i.e.,  
 we have 
$
{{}^{v}\!\Euch(N)} \neq 0$  
for any $N \in \ind\add \{ M^{\lvvee} \lotimes_{\Pa} \YYM_{m } \mid 0 \leq m \leq n-1 \}$. 
\end{definition}
 
 We note that if $v \in \We_{Q}$ is regular (resp. semi-regular), then 
 it has the property (I')${}_{M, n}$ for all $M \in \ind \Dbmod{\Pa}$ (resp. $M \in \ind \cU_{\Pa} [\ZZ]$). 
 
 \subsubsection{Left $\rad^{n}$-approximation theorem} 
\begin{theorem}\label{pre left approximation theorem}
Let $M \in \ind \Dbmod{\Pa}$ and $n \in N_{Q}$. 
Assume that the weight $v \in \We_{Q}$ has the property (I')${}_{M, n}$. 

\begin{enumerate}[(1)] 
\item

We have an isomorphism below in $\sfD^{\mrb}(\Pa \mod)$ 
\[
\YYM_{n}\lotimes_{\Pa} M \cong \RHom_{\Pa^{\op}} ( M^{\lvvee} \lotimes_{\Pa} \YYM_{n}, \PPi_{n}). 
\]

\item 

There exists a minimal left $\rad^{n}$-approximation  $\beta_{M}^{(n)}: M \to \YYM_{n} \lotimes_{\Pa} M$. 

\item 
There exists the following  isomorphism  in $\Dbmod{\Pa}$
\[
\YYM_{n} \lotimes_{ \Pa} M \cong \bigoplus_{ N \in \ind \Dbmod{\Pa}} N^{d_{N}}
\]
where 
\[
d_{N} := \dim_{\ResEnd(N)} \irr^{n}(M,N).
\]
\end{enumerate}
\end{theorem}

\begin{remark}
In the case $Q$ is Dynkin, by Theorem \ref{right approximation theorem}(4) we have  $\YYM_{h -1} \lotimes_{\Pa} M=0$. 
On the other hand, by Theorem \ref{kQ approximations theorem add}, the zero morphism 
$M \to 0$ is a minimal left $\rad^{h-1}$-approximation. 
It follows that the statement (2) of above theorem also  holds for $n = h -1$. 
\end{remark}

\begin{proof}

(1)
We prove the statement for $n \geq 0$ by induction on $n$. 
The case $n = 0$ is clear and the case $n =1$ follows from Proposition \ref{202103051610}. 

We deal with the case $n \geq 2$ by assuming the cases $n-1$ and $n-2$ are proved. 
Applying  the right versions of  Theorem \ref{right approximation theorem} and Corollary \ref{202007181611}(1) to $Q^{\op}$ and  the object $M^{\lvvee} \in \Dbmod{\Pa^{\op}}$, 
we see that $M^{\lvvee} \lotimes_{\Pa} \YYM_{n}$ fits the following 
direct sum of Auslander-Reiten triangle 
\[
M^{\lvvee} \lotimes_{\Pa} \YYM_{n -1} \to (M^{\lvvee} \lotimes_{\Pa} \YYM_{n}) \oplus (M^{\lvvee} \lotimes_{\Pa} \YYM_{ n-2} \lotimes_{\Pa} \PPi_{1} ) 
\to 
M^{\lvvee} \lotimes_{\Pa} \YYM_{n -1} \lotimes_{\Pa} \PPi_{1} \to
M^{\lvvee} \lotimes_{\Pa} \YYM_{n -1}[1].  
\]
Applying  $\RHom_{\Pa^{\op}}(-, \PPi_{n})$ to the above exact triangle and using the induction hypothesis, 
we obtain a direct sum of  Auslander-Reinten triangle
\begin{equation}\label{202104141734}
\YYM_{ n-1} \lotimes_{\Pa} M \to 
\RHom_{\Pa^{\op}} (  M^{\lvvee} \lotimes_{\Pa} \YM_{n}, \PPi_{n}) \oplus ( \PPi_{1} \lotimes_{\Pa} \YYM_{n-2} \lotimes_{\Pa}M ) 
  \to  \PPi_{1} \lotimes_{\Pa} \YYM_{n -1} \lotimes_{\Pa} M \to
\YYM_{ n-1} \lotimes_{\Pa} M[1]  .  
\end{equation}
Thus by uniqueness of the middle term of an Auslander-Reiten triangle and the Krull-Schmit property of $\sfD^{\mrb}(\Pa\mod)$, 
we deduce the following  isomorphism by comparing \eqref{202007192007} with \eqref{202104141734}.
\[
\RHom_{\Pa^{\op}}(  M^{*} \lotimes_{\Pa} \YM_{n}, \PPi_{n}) \cong \YYM_{n} \lotimes_{\Pa} M.
\]

(2) 
By the right version of Theorem \ref{right approximation theorem}, 
the morphism $
{}_{M^{\lvvee}} \ppi_{n}: M^{\lvvee} \lotimes_{\Pa} \YYM_{n} \to M^{\lvvee} \lotimes_{\Pa} \PPi_{n} $ 
is a minimal right $\rad^{n}$-approximation of $M^{\lvvee} \lotimes_{\Pa} \PPi_{n}$ in $\sfD^{\mrb}(\Pa^{\op} \mod)$. 
Thanks to (1),  applying $\RHom_{\Pa^{\op}}(-, \PPi_{n})$ to it, we obtain a minimal left  $\rad^{n}$-approximation
\[
\begin{split}
M \cong \RHom_{\Pa}&(M^{\lvvee}\lotimes_{\Pa}\PPi_{n}, \PPi_{n} ) \\ 
&\xrightarrow{ \ \RHom_{\Pa^{\op}}( {}_{M^{\lvvee}} \ppi_{n}, \PPi_{n}) \ } 
\RHom_{\Pa}(M^{\lvvee}\lotimes_{\Pa}\YYM_{n}, \PPi_{n} )\cong \YYM_{n} \lotimes_{\Pa} M. 
\end{split}
\]

(3) follows from (2) and a derived category version of Theorem \ref{202008172145}.
\end{proof}

Combining this result with  Theorem \ref{20211111751}, 
we come to the following conclusion.

\begin{theorem}\label{202111110809} 
Let $M \in \ind \Dbmod{\Pa}$ and $\cC_{M} \subset \Dbmod{\Pa}$ the full subcategory that consists of objects belonging to the 
same components with $M$ in the AR-quiver. 
Assume that the weight $v \in \We_{Q}$ has the property (I')${}_{M, n}$ for any $n \in N_{Q}$ 
and that the following condition does not  hold: $Q$ is wild and $M$ is a shift of a regular module. 

Then we have an isomorphism 
\[
\bigoplus_{n \in N_{Q}}  \YYM_{n}  \lotimes_{\Pa} M  \cong \bigoplus_{N \in \ind \cC_{M} } N^{\oplus \dim\Hom(M, N)}
\]
in $\sfD(\Pa)$. 
\end{theorem}

\subsubsection{Proof of Theorem \ref{description of QHA as kQ-modules}}\label{202111261251}

We give a proof of Theorem \ref{description of QHA as kQ-modules} stated in the introduction. 

(1) Let $i \in Q_{0}$ be a vertex and $P_{i} := \Pa e_{i}$. 
Since $\Hom_{\Pa}(P_{i}, N) = \tuH^{0}(e_{i} N)$
 and $\YYM_{n} \lotimes_{\Pa} P_{i} = \YYM e_{i}$, 
 it follows from Theorem \ref{202111110809} that 
 $\bigoplus_{n \in N_{Q}} \YYM_{n} e_{i} \cong \bigoplus_{ N \in \ind \cP(Q)} N^{\dim e_{i} N}$. 
 In particular, this shows that the left hand side is concentrated in cohomological degree $0$.
 
 Since $\YM = \tuH^{0}(\YYM)$, 
 it only remains to show that for $n \in \NN \setminus N_{Q}$ we have $\tuH^{0}(\YYM_{n} e_{i}) =0$. 
  The case where $Q$ is non-Dynkin, it is trivial since $N_{Q} =\NN$. 
The case where $Q$ is Dynkin follows from Proposition \ref{formula 1}(4) below. 

(2) is an immediate consequence of (1). \qed

\subsubsection{The Dynkin case}

We give descriptions of $\YYM_{h -2}, \YYM_{h -1}$ for a Dynkin quiver $Q$ with the Coxeter number $h$.

\begin{proposition}\label{202110241503}
Assume that  $Q$ is Dynkin with the Coxeter number $h$. 
Let $M \in \ind\Dbmod{\Pa}$. 
Assume that $v \in \kk^{\times}Q_{0}$ has the property (I)${}_{M, h-2}$. 
Then, 
the morphism $\ppi_{ h-1, M} : \YYM_{h -1} \lotimes_{\Pa} M \to \PPi_{ h-1} \lotimes_{\Pa} M$ is a minimal right $\rad^{h -1}$-approximation. 
Consequently $\YYM_{h -1} \lotimes_{\Pa} M = 0$. 
\end{proposition} 

\begin{proof} 
By Theorem \ref{right approximation theorem},   
the morphism ${}_{\PPi_{1}} \brho_{h -2, M}$ satisfies the right $\rad$-fitting condition. 
Thus, we can apply the above argument to the case $n = h-1$ and conclude that $\ppi_{h-1, M}$ is a minimal right $\rad^{h-1}$-approximation. 
The second statement is an immediate consequence of Theorem \ref{kQ approximations theorem add}. 
\end{proof}

\begin{corollary}\label{vanishing corollary}
Let $Q$ be a Dynkin quiver with the Coxeter number $h$. 
Assume that $v \in \kk Q_{0}$ is regular. 
Then $\YYM_{h -2} \cong \tuD(\Pa)$ and $\YYM_{h -1} = 0$ in $\sfD(\Pa)$. 
\end{corollary}

\begin{proof}
The first statement follows from Theorem \ref{kQ approximations theorem add} and Theorem \ref{pre left approximation theorem}. 
The second follows from Proposition \ref{202110241503}. 
\end{proof}

\subsection{When is the multiplication $\brho_{M}^{n}: M \to \YM_{n} \lotimes_{\Pa} M$ a minimal left $\rad^{n}$-approximation?}\label{202112011709}

Here we restate Problem \ref{introduction:left approximation problem} to fix notations. 

The minimal right $\rad^{n}$-approximation morphism $\ppi_{n, M}: \YYM_{n} \lotimes_{\Pa} M \to \PPi_{n} \lotimes_{\Pa} M$ is explicitly given. 
In contrast 
a minimal left $\rad^{n}$-approximation $\beta_{M}^{(n)}: M \to \YYM_{n} \lotimes_{\Pa} M$ is only known to exists 
and we do not have an explicit description at this moment. 

In an optimistic expectation, a natural candidate for a minimal left $\rad^{n}$-approximation $\beta_{M}^{(n)}: M \to \YYM_{n} \lotimes_{\Pa} M$ 
is  the $n$-th power $\brho^{n}_{M}$ of $\varrho$ 
\[
\brho^{n}_{M}: M \xrightarrow{ \rrho_{M}} \YYM_{1} \lotimes_{\Pa} M \xrightarrow{ \brho_{2, M}} \YYM_{2} \lotimes_{\Pa} M \xrightarrow{} \cdots 
\xrightarrow{} \YYM_{n -1} \lotimes_{\Pa} M \xrightarrow{ \brho_{n, M}} \YYM_{n} \lotimes_{\Pa} M .
\]
Indeed, this is the case when $n =1$ by Theorem \ref{semi-universal Auslander-Reiten triangle}.
However as is shown in Example \ref{202103221959} and Example \ref{202103051509}, this expectation is not fulfilled even for $n =2$.

\section{QHA of Dynkin type}\label{section: QHA Dynkin case}

In this section we investigate basic properties of (derived) quiver Heisenberg algebras of Dynkin type. 
So throughout this section, $Q$ denotes a Dynkin quiver. 

\subsection{Quiver Heisenberg algebras of Dynkin type}

\subsubsection{Finiteness of  dimension in the Dynkin case}

By Corollary \ref{202111191825} if the weight $v$ is regular, then   the algebra $\YM$ is  finite dimensional. 
We prove the converse.

 A. Chan and R. Marczinzik first   observed that at a special  value of $v$, the algebra $ \YM$ becomes infinite dimensional. 
The authors thank them for sharing  their observation.

\begin{theorem}\label{202109131544}
Let $Q$ be a Dynkin quiver. 
Then $\YM$ is of finite dimension if and only if $v$ is regular. 
\end{theorem}

We need to use another description of $\YM$, verification of which is left to the readers.

\begin{lemma}\label{202109131603}
Let $\overline{Q}^{\circlearrowleft}$ be a quiver obtained from $\overline{Q}$ 
by adding a loop $r_{i}$ to each vertex $i$.
For $i \in Q_{0}$ and $a \in \overline{Q}_{1}$,  
we set  $\eta'_{i} := \rho_{i} - v_{i}r_{i}, \  \zeta'_{a} = r_{t(a)}a- ar_{h(a)}$. 
\[
\begin{xymatrix}{ 
i \ar@(ld, lu)^{r_{i}}  \ar@/^10pt/[rr]^{\alpha} &&
 j  \ar@(rd,ru)_{r_{j}} \ar@/^10pt/[ll]^{\alpha^{*}} 
}\end{xymatrix}, \ \ 
\begin{tabular}{c|c|c|c|c}
& $e_{i}$ & $\alpha$ & $\alpha^{*}$ & $r_{i}$ \\ \hline
$\deg' $& $0$ & $1$  & $1$ & $2$ \\
\end{tabular}
\]
Then, the algebra $ \YM' := \kk \overline{Q}^{\circlearrowleft}/(\eta'_{i}, \zeta'_{a} \mid i \in Q_{0}, a \in \overline{Q}_{1})$ 
is isomorphic to $\YM$.

If moreover we equip  $\overline{Q}^{\circlearrowleft}$ with the grading $\deg'$ given in the table above, 
then 
$\deg'\eta'_{i} =2, \ \deg' \zeta'_{a} = 3$ and 
the graded algebra $\YM'$ is isomorphic to $\YM$ with the path length grading.
\end{lemma}

\begin{proof}[Proof of Theorem \ref{202109131544}]
We only have to prove that if there exists an indecomposable module $M$ such that ${}^{v}\!\Euch(M) = 0$, 
then $\dim  \YM = \infty$. 

First assume that  ${}^{v}\!\Euch(S_{i}) = 0$ for some simple module $S_{i}$ corresponding to a vertex $ i \in Q_{0}$ 
or equivalently assume $v_{i} = 0$. 
We prove $r_{i}^{n} \neq 0$ for any $n \geq 1$ inside $\YM'$ of Lemma \ref{202109131603}.
Assume that $r_{i}^{n} = 0$ for some $n \geq 1$ inside $\YM'$. 
Then inside the path algebra $\kk \overline{Q}^{\circlearrowleft}$, the element $r_{i}^{n}$ equals 
a sum of elements of the forms (1) $p \eta'_{j} q$ for some $j \in Q_{0}$  or (2) $r \zeta'_{a} s$  for some $a \in \overline{Q}_{1}$. 
Moreover we may assume that these elements are homogeneous with respect to the grading $\deg'$. 
Therefore in the case of (1) we have $\deg' p + \deg' q= 2n -2$, in the case (2) we have $\deg' r + \deg's = 2n -3$. 
By the assumption we have $\eta'_{i} = \rho_{i}$. It follows from degree reasons that monomials in elements of the forms (1) or (2) 
cannot contain $n$  $r_{i}$'s as their factors. A contradiction.

Next we assume that  ${}^{v}\!\Euch(M) = 0$ for some indecomposable module $M$, 
but   ${}^{v}\!\Euch(S_{i}) \neq 0$ for each simple module $S_{i}$. 
Note that the second assumption implies that $v$ is sincere. 

We claim that for any $n \geq 1$, the object $\YYM_{n} \lotimes_{\Pa} M$ of $\Dbmod{\Pa}$ contains $M$ as its direct summand.
We use the induction on $n$. 
By the assumption ${}^{v}\!\Euch(M) = 0$. Thus the case $n =1$  follows from Theorem \ref{semi-universal Auslander-Reiten triangle}. 

We deal with the case $n \geq 2$. Assume that the claim is verified in the case $n -1$. 
Then it follows from the case $n =1$ that $ \YYM_{1} \lotimes_{\Pa} \YM_{n -1} \lotimes_{\Pa} M$ contains $M$ 
as its direct summand. 
On the other hand, we have the exact triangle $\sfV_{n, M}$ below in $\Dbmod{\Pa}$
\[
\PPi_{1} \lotimes_{\Pa} \YM_{n -2} \lotimes_{\Pa} M \to  \YM_{1} \lotimes_{\Pa}  \YM_{n -1} \lotimes_{\Pa} M \to 
\YM_{n } \lotimes_{\Pa} M \to 
\PPi_{1} \lotimes_{\Pa} \YM_{n -2} \lotimes_{\Pa} M[1].
\]
Thus we only have to show 
$\Hom_{\Dbmod{\Pa}}(\PPi_{1} \lotimes_{\Pa} \YM_{n -2} \lotimes_{\Pa} M, M )= 0$. 

Note that $\Hom_{\Dbmod{\Pa}}(\PPi_{n } \lotimes_{\Pa} M, M) = \Hom_{\Dbmod{\Pa}}(\nu_{1}^{-n}(M), M) = 0$ for $n \geq 1$ since $Q$ is Dynkin.
By Theorem \ref{exact triangle U2} 
we have the exact triangle $\sfU_{n-1, M}$ below in $\Dbmod{\Pa}$ 
\[
\PPi_{1} \lotimes_{\Pa} \YYM_{n -3} \lotimes_{\Pa} M \to \PPi_{1} \lotimes_{\Pa}  \YYM_{n -2} \lotimes_{\Pa} M \to \PPi_{n -1}\lotimes_{\Pa} M \to
\PPi_{1} \lotimes_{\Pa} \YYM_{n -3} \lotimes_{\Pa} M[1].
\] 
Using them, we can deduce the desired vanishing condition by induction on $n$. 
This finishes the proof of the claim.

Recall that $\YM_{n} = \tuH^{0}(\YYM_{n})$ and $\tuH^{> 0}(\YYM) = 0$.
Thus we have $\YM_{n} \otimes_{\Pa} M = \tuH^{0}(\YYM_{n} \lotimes_{\Pa} M ) \neq 0$.
It follows from the claim that $ \YM_{n} \neq 0$ for any $n \geq 0$. 
\end{proof}

\subsubsection{Support $\tau$-tilting modules of $\YM$}

\begin{theorem}
Let $Q$ be a Dynkin quiver. Assume that  the weight $v \in \kk Q_{0}$ is regular. 
Then there is a bijection between the following sets. 

\begin{enumerate}[(1)]
\item The Weyl group $W_{Q}$ of $Q$. 

\item The set $\textup{s}\tau\textup{-tilt}\Lambda(Q)$ of isomorphism classes of basic support $\tau$-tilting $\Lambda$-modules
\end{enumerate}

\end{theorem}

\begin{proof}
Mizuno \cite[Theorem 1.1]{Mizuno} established  the first bijection 
\[
W_{Q} \cong \textup{s}\tau\textup{-tilt}\Pi(Q) \cong \textup{s}\tau\textup{-tilt}\YM(Q).   
\]
Since $\Pi(Q) = \YM(Q)/(\varrho)$ and $\varrho \in \tuZ(\YM(Q)) \cap \rad (\YM(Q))$, 
the second bijection is obtained by  a result \cite[Theorem 4.1]{EJR} by Eisele-Janssens-Raedschelders. 
\end{proof}

In subsequent work we study $\tau$-tilting theory of $\YM(Q)$ and in particular give a description of support $\tau$-tilting modules 
corresponding to $w \in W_{Q}$.

\subsection{The cohomology algebras of a derived quiver Heisenberg algebras of Dynkin type}

We study the cohomology algebra $\tuH(\YYM)$ of $\YYM(Q)$ of a Dynkin quiver. 

\begin{proposition}\label{formula 1}
Assume that $v \in \kk^{\times} Q_{0}$ is regular. 

Then, the following assertions hold in $\sfD(\Pa^{\mre})$.  
\begin{enumerate}[(1)]
\item 
For $0 \leq n \leq h -2$, 
the $*$-degree $n$ part $\YYM(Q)_{n}$ is concentrated in $0$-th cohomological degree 
and $\YYM(Q)_{n}$ is isomorphic to $\YM(Q)_{n}$.

\item $\YYM_{h-1} = 0.$

\item 
For a non-negative integer $m$, we have 
\[
\YYM_{m + h}  \cong \YYM_{m}[2].
\]

\item 
For $a \geq 0$ and $0 \leq b \leq h-1$, 
we have 
\[ 
\YYM_{ah +b} \cong \YYM_{b}[2a].
\]
In particular, 
\[
\tuH^{c}(\YYM_{ah +b} ) =
\begin{cases}
\YM_{b} & c = -2a\\ 
 0 & c \neq -2a
 \end{cases}
\]
\end{enumerate}
\end{proposition}

\begin{proof}
(1) follows from Theorem \ref{pre left approximation theorem}. 

(2) follows from  Corollary \ref{vanishing corollary}.

(3) 
First we deal with the case $m =0$. 
Taking the $*$-degree $h$ part of the exact triangle $\sfU$ of Theorem \ref{exact triangle U}, 
we obtain an exact triangle below  in $\sfD(\Pa)$ 
\[
\sfU_{h}: \YYM_{h -1} \xrightarrow{\sfr_{\rho}}  \YYM_{h} \xrightarrow{\ppi_{h}} \PPi_{h} \to \YYM_{h -1}[1].  
\]
It follows from (2) that $\ppi_{h}: \YYM_{h} \to \PPi_{h}$ is a quasi-isomorphism. 
Since $\ppi$ is a morphism of dg-algebras over $\Pa$, 
we conclude that $\ppi_{h}: \YYM_{h} \to \PPi_{h}$ is a quasi-isomorphism over $\Pa^{\mre}$. 

Let $m> 0$ be a positive integer. 
Generalizing the construction of the exact triangle $\sfV_{n}$ given in Section \ref{202106020843},  
we can construct an exact triangle $\sfV_{m,h}$ in $\sfK(\Pa^{\mre})$ of the following form
\[
\sfV'_{m,h}: \YYM_{m -1} \otimes_{ \Pa} \PPi_{1} \otimes_{\Pa} \YYM_{ h-1} 
\xrightarrow{ \eeta^{*}_{m,h}}  
\YYM_{m}  \otimes_{\Pa} \YYM_{h}  \xrightarrow{\zzeta_{m,h}} \YYM_{ m+h} \to \YYM_{m -1} \otimes_{ \Pa} \PPi_{1} \otimes_{\Pa} \YYM_{ h-1}[1].
\]
Using  (2) and Theorem \ref{Miyachi-Yekutieli lemma}, 
we conclude that the second morphism $\zzeta_{m,h}$ gives the desired isomorphism $\YYM_{m}[2] \cong \YYM_{m +h}$.

(4) immediately follows from (1) and (3).
\end{proof}

We determine the algebra structure of the cohomology algebra $\tuH(\YYM)$. 
Note that $\tuH(\YYM)$ acquires the cohomological grading and the $*$-grading.

\begin{theorem}\label{cohomology algebra structure theorem}
Let $Q$ be a Dynkin quiver and $\YYM := \YYM(Q)$. 
Assume that the weight  $v \in \kk Q_{0}$ is regular.

We identify $\tuH^{0}(\YYM) $ with $\YM = \YM(Q)$. 
Then the cohomology algebra $\tuH(\YYM)$ has a central element $u \in \tuH^{-2}(\YYM_{h})$ of 
cohomological degree $-2$ and of $*$-degree $h$ 
which induces an isomorphism 
\[
\tuH(\YYM) \cong \YM[u]
\]
of algebras with cohomological degrees and $*$-gradings 
where the right hand side denotes the polynomial algebra in a single variable $u$. 
\end{theorem}

We point out an immediate consequence. 

\begin{corollary}\label{cohomology algebra structure corollary}
Let $Q$ be a Dynkin quiver. 
Then, for each $m \geq 0$, 
the $-2m$-th cohomology group $\tuH^{-2m}(\YYM(Q))$ is isomorphic to $\YM(Q)$ 
as a bimodule over $\YM(Q) = \tuH^{0}(\YYM(Q))$. 
\end{corollary}

\begin{question}
Is there an explicit formula of a cycle $\tilde{u}$ of $\YYM(Q)$ 
whose cohomology class coincides with $u$?
\end{question}

To prove Theorem \ref{cohomology algebra structure theorem}, we need a preparation. 
\begin{lemma}\label{202006221721}
The subset $\{ v_{t(\alpha)}^{-1}\alpha \alpha^{*} - v_{h(\alpha)}^{-1}\alpha^{*} \alpha \in  \YM \mid \alpha \in Q_{1} \}$ of $\YM$ is linearly independent over $\kk$. 
\end{lemma}

\begin{proof}
We  equip $\overline{Q}$ with a bigrading 
by setting $\bideg \alpha :=(1, 0) \ \bideg \alpha^{*} =  (0,1)$. 
Then $\bideg  \eta_{\alpha} =( 2,1), \bideg \eta_{\alpha^{*}} = (1,2)$ 
and the bigrading descends to $\YM$. 
It follows from the bidegrees of relations that the bidegree $(1,1)$-component $\YM_{11}$ 
is isomorphic to $\kk \overline{Q}_{11}$. 
Since the  subset $\{ v_{t(\alpha)}^{-1}\alpha \alpha^{*} - v_{h(\alpha)}^{-1}\alpha^{*} \alpha \in \kk \overline{Q} \mid \alpha \in Q_{1} \}$ is 
a linear independent subset of $\kk \overline{Q}_{11}$, 
we deduce the desired conclusion. 
\end{proof}

\begin{proof}[Proof of Theorem \ref{cohomology algebra structure theorem}]
First note that by Proposition \ref{formula 1}(4), 
the cohomology algebra $\tuH(\YYM)$ is concentrated in non-positive even cohomological degrees. 

We fix an isomorphism $\YYM_{h} \cong A[2]$ of Proposition \ref{formula 1}(3) in $\sfD(\Pa^{\mre})$. 
Let $u \in \tuH^{-2}(\YYM_{h})$ be an element that corresponds to $1 \in A$ via $\YYM_{h} \cong A[2]$. 
We claim that $u$ and $\YM$ generates $\tuH(\YYM)$ as an algebra. 

Let $m \geq 0$. In the proof of Proposition \ref{formula 1}(3), 
it was shown that the multiplication morphisms $\YYM_{m} \lotimes_{\Pa} \YYM_{h} \to \YYM_{m +h}$ is an  isomorphism. 
It follows that the right multiplication  by $u$ gives 
 isomorphisms 
\[
\sfr_{u}: \tuH^{i}(\YYM_{m}) \to \tuH^{i -2}(\YYM_{m +h}), \ \ x \mapsto xu 
\]
for $i \in \ZZ$. 
For $0 \leq a < b$, we set $\YYM_{[a,b]} := \bigoplus_{i = a}^{b} \YYM_{i}$. 
Let $m \geq 0$. By Proposition \ref{formula 1}, we have $\tuH^{-2m}(\YYM) = \tuH^{-2m}(\YYM_{[mh, (m +1)h-1]})$ for $m \geq 0$. 
It follows from the above observation that 
the $m$-times iteration $\sfr_{u}^{m}$ gives 
an isomorphism 
\[
\sfr_{u}^{m}: \YM= \tuH^{0}(\YYM)  \to \tuH^{-2m}(\YYM_{[mh, (m +1)h -1]}) = \tuH^{-2m}(\YM), \ x \mapsto xu^{m}. 
\]
We finish the proof of claim. 

To prove theorem, we only have to show that $u$ is a central element. 
In the same way of the above argument, we can check that
the left multiplication  by $u$ gives 
 isomorphisms 
\[
\sfl_{u}: \tuH^{i}(\YYM_{m}) \to \tuH^{i -2}(\YYM_{m+h}), \ x \mapsto ux 
\]
for $i \in \ZZ$. 
We define an isomorphism $\phi: \tuH(\YYM) \to \tuH(\YYM)$  to be  $\phi:= \sfl_{u}^{-1} \sfr_{u}$. 
In other words, $\phi$ is such a map that sends an element $x \in \tuH^{i}(\YYM_{m})$ to 
the unique element $\phi(x) \in \tuH^{i}(\YYM_{m})$ which satisfies  
\[
xu = u \phi(x). 
\]
We only have to show that $\phi $ is the identity map on $\tuH(\YYM)$.

It is easy to check that $\phi$ is an automorphism  of the algebra $\tuH(\YYM)$. 
It is also clear that $\phi(u) =u$. 
Therefore, it is enough to show that  $\phi$ acts identically on $\YM$. 

Since the isomorphism $\YYM_{h} \cong \Pa[-2]$ is isomorphism in $\tuD(\Pa^{\mre})$ 
we have $pu = up$ for $p \in \Pa$. 
It follows that  $\phi$ is identity on $\YM_{0}= \Pa$. 

Now it is easy to see that $\phi$ induces an isomorphism $\YM_{1} \to \YM_{1}$ of $\Pa$-bimodules. 
Recall that the $\Pa$-bimodule $\YM_{1}$ is generated by opposite arrows $\alpha^{*}$ corresponds to $\alpha \in Q_{1}$. 
Since $Q$ has no multiple arrows, for each $\alpha\in Q_{1}$ there exists a  non-zero scalar $c_{\alpha} \in \kk$ 
such that $\phi(\alpha^{*} ) = c_{\alpha}\alpha^{*}$. 
Since the weighted mesh relation $\varrho \in \YM_{1}$ is a central element of $\tuH(\YYM)$ by Lemma \ref{homotopy proposition}, 
we have $\varrho u = u \varrho$ and hence $\phi(\varrho) = \varrho$. 
On the other hand applying $\phi$ to the definition of $\varrho_{i}$, we obtain 
\[
\phi(\varrho_{i}) = v_{i}^{-1} \phi\left(  \sum_{\alpha: t(\alpha) = i}  \alpha \alpha^{*} -\sum_{\alpha: h(\alpha) = i} \alpha^{*} \alpha\right) 
= v_{i}^{-1} \left(  \sum_{\alpha: t(\alpha) = i} c_{\alpha} \alpha \alpha^{*} -\sum_{\alpha: h(\alpha) = i} c_{\alpha} \alpha^{*} \alpha\right) 
\]
Therefore, we have the following equation in $\YM_{1}$
\[
0=\phi(\varrho) - \varrho=
 \sum_{\alpha \in Q_{1}} ( c_{\alpha}- 1) (v_{t(\alpha)}^{-1}\alpha \alpha^{*} - v_{h(\alpha)}^{-1}\alpha^{*} \alpha). 
\]
Thus by Lemma \ref{202006221721} we see that $c_{\alpha} =1$ for all $\alpha \in Q_{1}$ 
and $\phi(\alpha^{*} ) = \alpha^{*}$ for all $\alpha \in Q_{1}$. 

Since the algebra $\YM$ is generated by $\{ \alpha, \alpha^{*}\mid \alpha \in Q_{1} \}$, 
we conclude that $\phi$ is the identity map on $\YM$  as desired. 
\end{proof}

\subsubsection{The cohomology algebra with respect to the path grading}

We study the cohomology algebra $\tuH(\YYM)$ with respect to the path grading. 
We denote the grading with respect to the path length by $\pldeg$. 
By the definition, $\pldeg \alpha = 1, \pldeg \alpha^{*} = 1$ for $\alpha \in Q_{1}$. 
Therefore $\pldeg \eta_{\alpha} = 3, \pldeg \eta_{\alpha^{*}} = 3$ and $\YM$ has path length grading. 

We can extend the path-length grading $\pldeg\!$ to the derived quiver Heisenberg algebra $\YYM$ 
as shown in the table below. 
The aim is to prove the following theorem.

\begin{theorem}\label{cohomology algebra structure theorem path grading}
Let $Q$ be a Dynkin quiver and $\YYM := \YYM(Q)$. 
Assume that the weight  $v \in \kk Q_{0}$ is regular.

We identify $\tuH^{0}(\YYM) $ with $\YM = \YM(Q)$. 
Then the cohomology algebra $\tuH(\YYM)$ has a central element $u \in \tuH^{-2}(\YYM_{h})$ of 
cohomological degree $-2$ and of path length degree $2h$ 
which induces an isomorphism 
\[
\tuH(\YYM) \cong \YM[u].
\]
\end{theorem}

We use the bigrading introduced in the proof of Lemma \ref{202006221721}. 
Recall that we set $\bideg \alpha :=(1,0)$ and $\bideg \alpha^{*}:=(0,1)$ for $\alpha \in Q_{1}$. 
We note that the bidegree of other generators of the derived quiver Heisenberg algebra $\YYM$ are determined 
from the requirement that the differential $d_{\YYM}$ preserves the bigrading. 
We also note that if an element  $x \in \YYM$ is homogeneous with respect to the bigrading and $\bideg  x =(a,b)$, 
then it is also homogeneous with respect to the $*$-grading and path length grading and $\deg^{*} x= b, \pldeg x = a+b$. 
We give the table of degrees  of the generators of the derived quiver Heisenberg algebra $\YYM$:

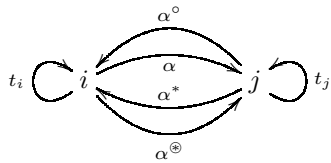
\centerline{
\begin{tabular}{c|c|c|c|c|c|c}
& $e_{i}$ & $\alpha$ & $\alpha^{*}$ &  $\alpha^{\circledast}$ &$\alpha^{\circ} $ & $t_{i}$ \\ \hline
$\bideg $& $0$ & $(1,0)$  & $(0,1)$ & $(2,1)$  & $(1,2)$ & $(2,2)$ \\ \hline
$\pldeg$ & $0$ & $1$ & $1$ & $3$ & $3$ &  $4$ 
\end{tabular}}
\[
\begin{xymatrix}{ 
i \ar@(ld, lu)^{t_{i}} \ar@/_20pt/[rr]_{\alpha^{\circledast}} \ar@/^10pt/[rr]_{\alpha} && j  \ar@(rd,ru)_{t_{j}} \ar@/^10pt/[ll]_{\alpha^{*}} 
\ar@/_20pt/[ll]_{\alpha^{\circ}}
}\end{xymatrix}
\]

We denote the homogeneous component of $\YYM$ of the bidegree $(a,b)$ by $\YYM_{(a,b)}$. 
For a bigraded module $M$ and $a \in \ZZ$, we set $M_{(\bullet, a)} := \bigoplus_{b\in \ZZ}M_{(b,a)}$ and 
$M_{(a,\bullet)} := \bigoplus_{b \in \ZZ} M_{(a,b)}$. 
Note that since the second component of bigrading counts the $*$-degree, we have $\YYM_{(\bullet, b)} = \YYM_{b}$ for $b \in \ZZ$.

We need the following lemma that allows us to exchange the first and the second components of bidegrees of $\YYM$. 

\begin{lemma}\label{202303171804}
Let $Q$ be a quiver. 
There exists an isomorphism $f: \YYM(Q^{\op}) \xrightarrow{\cong} \YYM(Q)$ of dg-algebras 
which induces an isomorphism $f_{(a,b)}: \YYM(Q^{\op})_{(a,b)} \xrightarrow{\cong} \YYM(Q)_{(b,a)}$ of $\kk$-vector spaces for $(a,b) \in \ZZ^{\oplus 2}$. 
\end{lemma}

\begin{proof}
We denote the opposite arrow of $\alpha \in Q_{1}$ by $\alpha^{\op}$. 
Then it is straightforward to check that by setting a morphism $f: \YYM(Q^{\op}) \to \YYM(Q)$ of dg-algebras 
as below, we obtain a dg-algebra isomorphism having desired property:  
\[
\begin{split}
f(\alpha^{\op}) := \alpha^{*}, \ f(\alpha^{\op, *}) := \alpha, \ 
f(\alpha^{\op, \circ}) := -\alpha^{\circledast}, \ 
f(\alpha^{\op, \circledast}) := -\alpha^{\circ}, \ 
f(t_{i}^{\op}) := -t_{i}
\end{split}
\]
for $i \in Q_{0}$ and $\alpha \in Q_{1}$. 
\end{proof}

\begin{proof}[Proof of Theorem \ref{cohomology algebra structure theorem path grading}]

It is enough to show that 
the element $u \in \tuH^{-2}(\YYM_{h})$  in the proof of Theorem \ref{cohomology algebra structure theorem} 
can be chosen  as homogeneous with respect to the bidegree  with $\bideg u =(h,h)$. 
Recall that $u$ is given as the image of $1$ by an isomorphism $A \xrightarrow{\cong} \tuH^{-2}(\YYM_{h})$ of $A^{\mre}$-modules.

For simplicity we set $M:= \tuH^{-2}(\YM)$.  
We may regard $\YM^{\mre}$ as a bigraded algebra by setting $(\YM^{\mre})_{(a,b)}:= \bigoplus_{c+d =a, e+f=g}\YM_{(c,e)} \otimes \YM^{\op}_{(d,f)}$. 
The above isomorphism extends to an isomorphism $f: \YM \xrightarrow{\cong} M$ of (ungraded) $\YM^{\mre}$-modules. 
Since the algebra $\YM$ is indecomposable as an algebra (i.e., it does not have a non-trivial central idempotent element), 
it is indecomposable as $\YM^{\mre}$-modules. Hence by bigraded version of \cite[Theorem 3.2]{Gordon-Green}, 
 $\YM$ and $M$ are indecomposable as bigraded $\YM^{\mre}$-modules. 
Applying  bigraded version of \cite[Theorem 4.1]{Gordon-Green} to $f$, 
we obtain an isomorphism $g: \YM((s,t)) \xrightarrow{\cong} M$ as bigraded $\YM^{\mre}$-modules for some $(s,t) \in \ZZ^{\oplus 2}$. 

We claim that $(s,t) = (-h,-h)$. 
Note that by Lemma \ref{202303171804} and Proposition \ref{formula 1}, 
we have $M_{(a,b)} = 0$ unless $(a,b)$ belongs to the rectangle $[h, 2h-2]\times [h, 2h-2]$. 
On the other hand, we have $\YM_{(0,0)}= \kk Q_{0} \neq 0$. 
By Corollary \ref{vanishing corollary}, we have $ \YM_{(\bullet,h-2)} = \YM_{h-2} \cong \tuD(\Pa) \neq 0$. 
Thanks to Lemma \ref{202303171804}, the same corollary implies  $\YM_{(h-2, \bullet)} \neq 0$. 
Now it is straightforward to see $(s,t) =(-h,-h)$. 

Looking at the $(\bullet, h)$ component of $g$ we obtain an isomorphism 
$g_{(\bullet, h)}: A \xrightarrow{\cong} \tuH^{-2}(\YYM_{h})$ of $\Pa^{\mre}$-modules.  
Thus, if we  set $u$ to be the image of $1 \in A$ by $g_{(\bullet, h)}$, it is homogeneous with $\bideg u= (h,h)$ as desired. 
\end{proof}

\subsection{Symmetric property}

As was mentioned in the introduction, 
Etingof-Rains \cite{Etingof-Rains} proved that 
 the QHA $\YM(Q)$  is Frobenius if the weight $v \in \kk Q_{0}$ is regular and $\chara \kk = 0$.  
It was proved by Etingof-Latour-Rains \cite{ELR} 
that if $\chara \kk = 0$ and $v \in \kk Q_{0}$ is generic, then  the QHA $\YM(Q)$ is symmetric.

In subsequent work we prove the following statement. 
\begin{theorem}[{\cite{Herschend-Minamoto: tilting theory of QHA}}]\label{202111300947}
Let $Q$ be a Dynkin quiver. Assume that  the weight $v \in \kk Q_{0}$ is regular. 
Then $\YM(Q)$ is a symmetric algebra.
\end{theorem}

Recall that
the right multiplication $\sfr_{\varrho}$ and the left multiplication $\sfl_{\varrho}$ of $\varrho$ on $\YM$ 
coincide to each other and 
their cokernel  is  the preprojective algebra $\Pi$. 
In the next proposition, we identify the kernel of $\sfr_{\varrho} = \sfl_{\varrho}$ with the $\kk$-dual $\tuD(\Pi)$ of $\Pi$.

\begin{corollary}\label{202006221820}
Let $Q$ be a Dynkin quiver. Assume that  the weight $v \in \kk Q_{0}$ is regular. 
We have the following exact sequence of bimodules over $\YM$. 
\[
0 \to \tuD(\Pi) \to \YM \xrightarrow{ \sfr_{\varrho}} \YM \to \Pi \to 0
\]
\end{corollary}

\begin{proof}
By Theorem \ref{202111300947} 
that there exists an isomorphism $\iota: \tuD(\YM) \xrightarrow{ \cong } \YM$ over $\YM^{\mre}$. 
Under this isomorphism, 
the right multiplication map  $\sfr_{\varrho}: \YM \to \YM$ by $\varrho$ 
 corresponds to $\tuD(\sfl_{\varrho}) : \tuD(\YM) \to \tuD(\YM)$. 
Thus, we have the following isomorphisms of bimodules over $\YM$.
\[
\Ker \sfr_{\varrho} \cong \Ker \tuD(\sfl_{\varrho} ) \cong \tuD( \Coker  \sfl_{\varrho}) \cong \tuD(\Pi). 
\]
We note that the kernel morphism $\tuD(\Pi) \to \YM$ is the composition $\iota \tuD(\pi)$. 
\end{proof}

\begin{remark}\label{202603182140}

As in the proof of Theorem \ref{cohomology algebra structure theorem path grading} any isomorphism $D(\YM) \xrightarrow{ \cong } \YM$ can be turned into a graded morphism either with respect to the $*$-grading or path-length grading. The largest degree for which $\YM$ has a non-zero homogeneous component is $h-2$ for the $*$-grading and $2h-4$ for the path-length grading. Denote degree shift by $n$ by $(n)$ for the $*$-grading and by $(n)_{\textup{pl}}$ for the path-length grading. Then we obtain $\tuD(\YM) \cong \YM(h -2)$ and $\tuD(\YM) \cong \YM(2h -4)_{\textup{pl}}$ as graded $\YM^{\mre}$-modules. Since $\deg^{*} \varrho = 1$ and $\pldeg \varrho = 2$ the exact sequence of Corollary \ref{202006221820} gives the following exact sequences of graded $\YM^{\mre}$-modules
\[
0 \to \tuD(\Pi)(-h+1) \to \YM(-1) \xrightarrow{ \sfr_{\varrho}} \YM \to \Pi \to 0,
\]
\[
0 \to \tuD(\Pi)(-2h+2)_{\textup{pl}} \to \YM(-2)_{\textup{pl}} \xrightarrow{ \sfr_{\varrho}} \YM \to \Pi \to 0.
\]
for the $*$-grading and path-length grading respectively.

\end{remark}

\subsubsection{}
Combining above corollary with Theorem \ref{cohomology algebra structure theorem}, 
we obtain a description of bimodule structure of the cohomology algebra  of the derived preprojective algebra $\PPi$ of Dynkin type.

\begin{corollary}
Let $Q$ be a Dynkin quiver. 
We have the following isomorphism of $\Pi = \tuH^{0}(\PPi)$-bimodules
\[
\tuH^{n}(\PPi) 
\cong 
\begin{cases} 
\Pi & n \textup{ is non-positive even}, \\
\tuD(\Pi) & n \textup{ is non-positive odd},\\
0 & n \textup{ is positive}. 
\end{cases}
\]
\end{corollary}

\begin{proof}
We use the exact triangle $\widehat{\sfU}: \YYM \xrightarrow{\hat{\sfr}_{\varrho}} \YYM \xrightarrow{\ppi} \PPi \to \YYM[1]$ 
given in Section \ref{a-infinity}. 
A point here is that it is an exact triangle in $\sfD(\YYM^{\mre})$. 
Taking the cohomology long exact sequence of $\widehat{\sfU}$, 
we obtain the following exact sequence 
\[
0\to \tuH^{ -2n -1}(\PPi) \to \tuH^{-2n} (\YYM) \xrightarrow{\tuH^{-2n}(\sfr_{\varrho})} \tuH^{ -2n}(\YYM) \xrightarrow{ \tuH^{-2n}(\ppi)} 
\tuH^{-2n}(\PPi) \to 0
\]
of $\YM^{\mre}$-modules.
Under the isomorphism $\tuH(\YYM) \cong \YM[u]$ of Theorem \ref{cohomology algebra structure theorem}, 
the map $\tuH^{-2n}(\sfr_{\varrho})$ corresponds to the multiplication map $\sfr_{\varrho}: \YM \to \YM$. 
Thus by Proposition \ref{202006221820}, we obtain isomorphisms 
$\tuH^{-2n}(\PPi) \cong \Pi$ and $\tuH^{-2n -1}(\PPi) \cong \tuD(\Pi)$ over $\YM^{\mre}$. 
\end{proof}

\section{The middle terms of AR-sequences starting from middle terms}\label{section: the middle term of the middle terms}

The aim of this section is to prove Theorem \ref{2020071920551}, 
from which Theorem \ref{202111181854} follows. 
As a corollary, we obtain a sufficient condition for Problem \ref{introduction:left approximation problem} (restated in Section \ref{202112011709}) in the case $n =2$ (Theorem \ref{202103021651}).

From now on  we assume the base field $\kk$ is algebraically closed. 
Thus,  in particular   we have   $\dim \ResEnd_{\Pa}(M) = 1$ 
for any  indecomposable object $M$ of $\Dbmod{\Pa}$.

\subsection{The middle terms of middle terms} 

The next result says that the morphism $\eeta^{*}_{2, M}: \PPi_{1} \lotimes_{\Pa} M \to \YYM_{1} \lotimes_{\Pa} \YYM_{1} \lotimes_{\Pa} M$ 
has AR-theoretic meaning. 

The precise statement is the following. 

\begin{theorem}\label{2020071920551}
Let $M \in \ind \Dbmod{\Pa}$ and $v \in \We_{Q}$. Assume that
${}^{v}\!\Euch(M) \neq 0, {}^{v}\!\Euch(\PPi_{1}\lotimes_{\Pa} M) \neq 0$ and   ${}^{v}\!\Euch(\YYM_{1} \lotimes_{\Pa} M) \neq 0$. 

Then 
there exists a morphism $\xi_{2, M}: \YYM_{1} \lotimes_{\Pa} \YYM_{1} \lotimes_{ \Pa} M \to \PPi_{1} \lotimes_{\Pa} M$ 
which satisfies   the following equations. 

\begin{enumerate}[(1)] 

\item 
$
\xi_{2, M} \eeta^{*}_{2, M}  = -\frac{{}^{v}\!\Euch(\YYM_{1} \lotimes_{\Pa} M ) }{{}^{v}\!\Euch(M)} \id_{\PPi_{1} \lotimes M}$.

\item  $\xi_{2, M} \rrho_{\YYM_{1} \lotimes M}= \ppi_{1, M}$.

\item $\xi_{2, M} {}_{\YYM_{1}} \rrho_{  M}= -\frac{{}^{v}\!\Euch(\PPi_{1} \lotimes_{\Pa} M ) }{{}^{v}\!\Euch(M)}  \ppi_{1, M}$. 

\end{enumerate}

Namely we have the following commutative  diagrams. 
\[
\begin{xymatrix}@C=40pt{ 
& \PPi_{1} \lotimes_{\Pa} M \ar[d]^{\eeta^{*}_{2,M} }  
\ar@/^20pt/[ddr]_{\cong}^{  -\frac{{}^{v}\!\Euch(\YYM_{1} \lotimes_{\Pa} M ) }{{}^{v}\!\Euch(M)}\id} 
& 
\\ 
\YYM_{1}\lotimes_{\Pa} M \ar@/_20pt/[rrd]_{\ppi_{1, M} } \ar[r]^{\rrho_{ \YYM_{1} \lotimes M}}
 &
 \YYM_{1} \lotimes_{\Pa} \YYM_{1} \lotimes_{\Pa} M \ar[dr]^{\xi_{2, M}}
 & \\
&& \PPi_{1} \lotimes_{\Pa} M,
}
\end{xymatrix}\]
\[
\begin{xymatrix}@C=40pt{ 
\YYM_{1}\lotimes_{\Pa} M \ar@/_20pt/[rrd]_{-\frac{{}^{v}\!\Euch(\PPi_{1} \lotimes_{\Pa} M ) }{{}^{v}\!\Euch(M)} \ppi_{1, M} } \ar[r]^{{}_{\YYM_{1}}\rrho_{ M}}
 &
 \YYM_{1} \lotimes_{\Pa} \YYM_{1} \lotimes_{\Pa} M \ar[dr]^{\xi_{2, M}}
 & \\
&& \PPi_{1} \lotimes_{\Pa} M.
}
\end{xymatrix}\]
\end{theorem}

The proof of Theorem~\ref{2020071920551} is given in Section~\ref{202204121618} after some preparation. First however, we point out several corollaries. 

\begin{corollary}\label{202102221805}
Let $M \in \ind \Dbmod{\Pa}$ and $v \in \We_{Q}$. Assume that
${}^{v}\!\Euch(M) \neq 0, {}^{v}\!\Euch(\PPi_{1} \lotimes_{\Pa} M ) \neq 0$ and   ${}^{v}\!\Euch(\YYM_{1} \lotimes_{\Pa} M) \neq 0$. 

If a morphism $\xi'_{2, M}: \YYM_{1} \lotimes_{\Pa} \YYM_{1} \lotimes_{ \Pa} M \to \PPi_{1} \lotimes_{\Pa} M$ 
satisfies the equation  $\xi'_{2, M} \rrho_{\YYM_{1} \lotimes M}= \ppi_{1, M}$, 
then 
the following equality holds in  $\ResEnd_{\Pa}(\PPi_{1} \lotimes_{\Pa} M )$:
\[
\xi'_{2, M} \eeta^{*}_{2, M}  =-\frac{{}^{v}\!\Euch(\YYM_{1} \lotimes_{\Pa} M ) }{{}^{v}\!\Euch(M)} \id_{\PPi_{1} \lotimes M}.
\]
Thus in particular, the endomorphism $\xi'_{2, M} \eeta^{*}_{2, M} $ is an automorphism of $\PPi_{1} \lotimes_{\Pa} M$. 
\end{corollary}

\begin{proof}
Let $\xi_{2,M}$ be the morphism obtained in Theorem \ref{2020071920551}. 
Since $(\xi'_{2, M} - \xi_{2, M}) \rrho_{\YYM_{1} \lotimes M} = 0$,  we see that $\xi'_{2, M} - \xi_{2, M}$ belongs to $\rad$. 
Hence $\xi'_{2, M} \eeta^{*}_{2, M} - \xi_{2, M} \eeta^{*}_{2, M}$ belongs to $\rad$. 
Thus we deduce the desired conclusions. 
\end{proof}

\begin{remark}
If we fix a generic weight $v \in \We_{Q}$,   then the fraction 
$-\frac{{}^{v}\!\Euch(\YYM_{1} \lotimes_{\Pa} M )}{{}^{v}\!\Euch(M)}$ depends on the indecomposable object $M$. 
It follows that 
for a generic  $v$, 
there exists  no  morphism $\xxi_{2}: \YYM_{1} \lotimes_{\Pa} \YYM_{1}\to \PPi_{1}$ in $\sfD(\Pa^{\mre})$
such that $\xxi_{2} \rrho_{\YYM_{1}} = \ppi_{1}$ and $\xi_{2, M} = \xxi_{2} \lotimes M$ for all $M$. 
Indeed, assume that such a morphism $\xxi_{2}: \YYM_{1} \lotimes_{\Pa} \YYM_{1}\to \PPi_{1}$ in $\sfD(\Pa^{\mre})$ exists. 
Let $f: M \to N$ be a non-zero morphism of indecomposable objects of $\Dbmod{\Pa}$. 
We set $g:= (\xxi_{2,M})(\eeta^{*}_{2, M}), \ h =(\xxi_{2, N})(\eeta^{*}_{2, N})$ and 
$x:= - \frac{ {}^{v}\!\Euch(\YYM_{1} \lotimes_{\Pa} M)}{{}^{v}\!\Euch(M)}
, \ y:= - \frac{ {}^{v}\!\Euch(\YYM_{1} \lotimes_{\Pa} N)}{{}^{v}\!\Euch(N)}$. 
Then  using \eqref{202111281249},  we obtain $(h -x)({}_{\PPi_{1}}f) = ({}_{\PPi_{1}}f)(g-x)$. 
By Corollary \ref{202102221805}, we have $(g -x)^{n}= 0$ for $n$ large enough 
and hence $(h -x)^{n}({}_{\PPi_{1}}f) = 0$. 
On the other hand, it follows from  Corollary \ref{202102221805} that if $x \neq y$ 
then $h-x$ is an automorphism of $N$. 
Thus we deduce $x = y$.

However if $v$ is an eigenvector of $\Psi$ with the eigenvalue $\lambda$, then 
\[
-\frac{{}^{v}\!\Euch(\YYM_{1} \lotimes_{\Pa} M )}{{}^{v}\!\Euch(M)}= -(1 + \lambda)
\]
and  it does  not depend on $M$. 
In this case, we prove that there exists  a morphism $\xxi :\YYM_{1} \lotimes_{\Pa} \YYM_{1}\to \PPi_{1}$ in $\sfD(\Pa^{\mre})$ that has the above properties in  Proposition \ref{202102231445}.
\end{remark}

We use the following corollary which holds  for  any not necessary indecomposable object $M$ of  $\Dbmod{\Pa}$. 
\begin{corollary}\label{2021022218051}
Let $M \in \Dbmod{\Pa}$ and $v \in \We_{Q}$. Assume that
${}^{v}\!\Euch(N) \neq 0, \ {}^{v}\!\Euch(\PPi_{1} \lotimes_{\Pa} N) \neq 0$ and   ${}^{v}\!\Euch(\YYM_{1} \lotimes_{\Pa} N) \neq 0$ for any indecomposable direct summand $N$ of $M$.

Then,  
there exists a morphism $\xi_{2, M}: \YYM_{1} \lotimes_{\Pa} \YYM_{1} \lotimes_{ \Pa} M \to \PPi_{1} \lotimes_{\Pa} M$ 
which satisfies   the following equations. 
\begin{enumerate}[(1)] 
\item 
$
\xi_{2, M} \eeta^{*}_{2, M}$ is an isomorphism.

\item  $\xi_{2, M} \rrho_{\YYM_{1} \lotimes M}= \ppi_{1, M}$.  
\end{enumerate}

Namely we have the following commutative  diagrams. 
\[
\begin{xymatrix}@C=40pt{ 
& \PPi_{1} \lotimes_{\Pa} M \ar[d]^{\eeta^{*}_{2,M} }  
\ar@/^20pt/[ddr]^{\cong}
& 
\\ 
\YYM_{1}\lotimes_{\Pa} M \ar@/_20pt/[rrd]_{\ppi_{1, M} } \ar[r]^{\rrho_{ \YYM_{1} \lotimes M}}
 &
 \YYM_{1} \lotimes_{\Pa} \YYM_{1} \lotimes_{\Pa} M \ar[dr]^{\xi_{2, M}}
 & \\
&& \PPi_{1} \lotimes_{\Pa} M
}
\end{xymatrix}\]
\end{corollary} 

\begin{proof}
Let $M = \bigoplus_{i =1}^{p} N_{i}$ be an indecomposable decomposition. 
For each $i =1,2, \ldots, p$, we have the morphism $\xi_{2, N_{i}}$ of Theorem \ref{2020071920551}. 
Then the morphism $\xi_{2, M}: \YYM_{1}\lotimes_{\Pa} \YYM_{1} \lotimes_{\Pa} M \to \PPi_{1} \lotimes_{\Pa} M$ induced from $\xi_{2,N_{1}},\xi_{2,N_{2}}, \ldots, \xi_{2,N_{p}}$ 
has the desired properties.
\end{proof}

\begin{remark}
\begin{enumerate}[(1)]
\item 
If 
the number 
$\lambda := \frac{{}^{v}\!\Euch(\PPi_{1} \lotimes_{\Pa}N )}{{}^{v}\!\Euch(N)}$ is independent of the indecomposable direct summand $N$ of $M$, 
then in (1) of the above corollary, we have 
\[
\xi_{2, M} \eeta^{*}_{2, M} = -(1+ \lambda). 
\]

\item  
The correspondence $ M \mapsto \xi_{2, M}$ is functorial with respect to split epimorphisms and split monomorphisms. 
Namely, if $f: M\to N$ is a split epimorphism or  a split monomorphism in $\Dbmod{\Pa}$, 
then we have 
$(\xi_{2, N})({}_{\YYM_{1}\lotimes \YYM_{1}} f) =({}_{\PPi_{1}} f)(\xi_{2, M})$. 
\end{enumerate}
\end{remark}

\begin{remark}
Even if the base field $\kk$ is not an algebraically closed, 
we can show in the same way as the proof of Theorem \ref{2020071920551} 
that there exists a morphism $\xi_{2, M}: \YYM_{1} \lotimes_{\Pa} \YYM_{1} \lotimes_{\Pa} M \to \PPi_{1} \lotimes_{\Pa} M$ 
such that $\xi_{2, M} \rrho_{\YYM_{1} \lotimes M}= \ppi_{1, M}$ and 
that $\xi_{2,M} \eeta^{*}_{2, M}$ is an automorphism of $\PPi_{1} \lotimes_{\Pa} M$. 
\end{remark}

\subsection{Preparations}

\subsubsection{}
Recall that $\PPi_{n} = \PPi_{1} \lotimes_{\Pa} \PPi_{1} \lotimes_{\Pa} \cdots \lotimes_{\Pa} \PPi_{1}$. 
We consider the following  isomorphisms that exchange  the shift functor  $[-1]$. 
\[
(\PPi_{n} \lotimes_{\Pa} M)[-1] 
\cong (\PPi_{1}[-1]) \lotimes_{\Pa} \PPi_{n-1} \lotimes_{\Pa} M 
\cong \PPi_{n -1} \lotimes_{\Pa} (\PPi_{1}[-1]) \lotimes_{\Pa} M. 
\]
We note that thanks to Lemma \ref{202102171838}(2)  
the second isomorphism coincides with the isomorphism $(\gamma_{\PPi_{n -1}})_{M}$ 
 associated to the inverse of the Serre functor $\sfS^{-1} =\PPi_{1}[-1] \lotimes_{\Pa} -$.

Using these isomorphisms, we regard $\ttheta_{1, \PPi_{ n-1} \lotimes M}$ and ${}_{\PPi_{n-1}}\ttheta_{1, M}$ as morphisms 
from $(\PPi_{n} \lotimes_{\Pa} M)[-1] $ to $\PPi_{n -1} \lotimes_{\Pa} M$. 
\[
\begin{split}
&\ttheta_{1, \PPi_{ n-1} \lotimes M}:
(\PPi_{n} \lotimes_{\Pa} M)[-1] 
\cong (\PPi_{1}[-1]) \lotimes_{\Pa} \PPi_{n-1} \lotimes_{\Pa} M 
\to \PPi_{n -1} \lotimes_{\Pa} M, \\
&{}_{\PPi_{n-1}}\ttheta_{1, M}:
(\PPi_{n} \lotimes_{\Pa} M)[-1]
\cong \PPi_{n -1} \lotimes_{\Pa} (\PPi_{1}[-1]) \lotimes_{\Pa} M
\to \PPi_{n -1} \lotimes_{\Pa} M. 
\end{split}
\]

By Proposition \ref{202101131644} we obtain the following lemma.

\begin{lemma}\label{202103021524}
Let $M \in \ind \Dbmod{\Pa}$.  
Assume that ${}^{v}\!\Euch(M) \neq 0$. 
Then, 
we have the following equality of morphisms $\PPi_{n} \lotimes_{\Pa} M [-1] \to \PPi_{n -1} \lotimes_{\Pa} M $
\[
\ttheta_{1, \PPi_{ n-1} \lotimes M}= \frac{{}^{v}\!\Euch(\PPi_{n -1} \lotimes_{\Pa} M)}{  {}^{v}\!\Euch(M)} {}_{\PPi_{n-1}}\ttheta_{1,  M}.
\]
\end{lemma}

Although we only use the case of $n=2$ of the following proposition to prove  Theorem \ref{2020071920551}, 
we provide a general statement  and proof.

\begin{proposition}\label{202012171832} 
Let $n \in N_{Q}^{n  \geq 2}$ and $M \in \ind \Dbmod{\Pa}$. 
Assume that $v$ has property (I)${}_{M, n-1}$  and  
$
{}^{v}\!\Euch(\YYM_{n -1}\lotimes_{\Pa} M) \neq 0. $ 
Then, we have
\[
\ppi_{n -1, M}\ttheta_{n, M} =
\frac{  {}^{v}\!\Euch(\YYM_{n -1} \lotimes_{\Pa} M)}{{}^{v}\!\Euch( M)}{}_{\PPi_{n-1}}\ttheta_{1,M}.
\]
In particular,  
the composition $\ppi_{n -1, M} \ttheta_{n, M}: \PPi_{n}\lotimes_{\Pa} M [-1] \to \PPi_{n -1} \lotimes_{\Pa} M$ is 
AR-coconnecting to $\PPi_{n -1} \lotimes_{\Pa} M$. 
\end{proposition}

\begin{proof}
We claim that $\ppi_{n -1, M} \ttheta_{n, M}$ is AR-coconnecting.  

We use  Happel's criterion (Theorem \ref{Happel's criterion}). 
We have the following equality
\begin{equation}\label{202102211751}
\begin{split} 
\isagl{\id_{\PPi_{n-1} \lotimes M}, \ppi_{n-1, M} \ttheta_{n ,M}  } 
& = \isagl{ \id_{\YYM_{ n-1} \lotimes M}, \ttheta_{n, M} (\PPi_{1}[-1] \lotimes \ppi_{n -1, M} )  } \\
& = \isagl{  \id_{\YYM_{ n-1} \lotimes M}, \ttheta_{1, \YYM_{ n-1} \lotimes M} } \\
& = {}^{v}\!\Euch(\YYM_{n -1} \lotimes_{\Pa} M) \neq 0 
\end{split}
\end{equation}
where the first equality follows from functoriality of Serre functor, 
 for the second equality we use Lemma \ref{202012211305} 
and for the third equality we use Theorem \ref{trace formula}.

Let $f$ be an endomorphism of $\PPi_{n -1}\lotimes_{\Pa} M$ belonging to $\rad$. 
Then we have the following commutative diagram 
\[
 \begin{xymatrix}@C=50pt{
 \sfS^{-1}(\YYM_{n-1} \lotimes_{\Pa}M) \ar[d]_{\sfS^{-1}(f \ppi_{ n-1, M})} \ar@{-->}[r]^{g}&
 \PPi_{1} \lotimes_{\Pa} \YYM_{ n-1} \lotimes_{\Pa} M[-1] 
 \ar[d]^{\PPi_{1} \lotimes \ppi_{n -1, M}[-1]}  \ar[r]^-{\ttheta_{1, \YYM_{n -1} \lotimes M}} & \YYM_{n -1}\lotimes_{\Pa} M\ar@{=}[d] \\
\sfS^{-1}(\PPi_{ n-1} \lotimes_{\Pa} M) \ar@{=}[r]&  \PPi_{n } \lotimes_{\Pa} M[-1]  \ar[r]_{\ttheta_{n, M}} & \YYM_{n -1} \lotimes_{\Pa} M.  
}\end{xymatrix} 
 \] 
By Theorem \ref{right approximation theorem}, the morphism 
$\PPi_{1} \lotimes \ppi_{n-1, M}[-1]$ is a minimal right $\rad^{n-1}$-approximation. 
Since $\sfS^{-1}(f\ppi_{n-1, M})$ belongs to $\rad^{n}$, there exists a morphism $g: \sfS^{-1}(\YYM_{n-1} \lotimes_{\Pa} M) \to \PPi_{1} \lotimes_{\Pa} \YYM_{n-1} \lotimes_{\Pa} M$ that completes the above commutative diagram. 
It follows from a version of Lemma \ref{202105252216}  that $g$ belongs to $\rad$. 
Therefore we have $\ttheta_{1, \YYM_{ n-1} \lotimes M}g = 0$ and consequently 
\[
\begin{split} 
\isagl{f, \ppi_{n-1, M} \ttheta_{n ,M} } 
& = \isagl{  \id_{\YYM_{ n-1} \lotimes M}, \ttheta_{n, M}\sfS^{-1}(f\ppi_{n -1, M}) } \\
& = \isagl{ \id_{\YYM_{ n-1} \lotimes M}, \ttheta_{n, M}(\PPi_{1}[-1] \lotimes \ppi_{n -1, M}) g    } \\
& = \isagl{ \id_{\YYM_{ n-1} \lotimes M},  \ttheta_{1, \YYM_{ n-1} \lotimes M}g } \\
& =  0.
\end{split}
\]
This completes the proof of the claim.

Using  Lemma \ref{202102071323} and the above calculation \eqref{202102211751}, 
we obtain the first equality of the equation below 
\[
\ppi_{n -1, M}\ttheta_{n, M} = \frac{  {}^{v}\!\Euch(\YYM_{n -1} \lotimes_{\Pa} M)}{{}^{v}\!\Euch(\PPi_{n -1} \lotimes_{\Pa} M)}\ttheta_{1, \PPi_{ n-1} \lotimes M}=
\frac{  {}^{v}\!\Euch(\YYM_{n -1} \lotimes_{\Pa} M)}{{}^{v}\!\Euch( M)}{}_{\PPi_{n-1}}\ttheta_{1,M}.
\]
Thanks to  Lemma \ref{202103021524}, 
we deduce the second equality and we obtain the desired conclusion. 
\end{proof}

\subsubsection{A lemma}

For the proof of Theorem \ref{2020071920551} we use  Lemma \ref{2020071917401} about homotopy Cartesian squares. 
The next lemma checks  the vanishing condition of this lemma. 

\begin{lemma}\label{202012171908}
Assume that $Q$ is not an $A_{2}$-quiver. 
Let $M \in \ind \Dbmod{\Pa}$. 
Then $\Hom_{\Pa}(\YYM_{1} \lotimes_{\Pa}M, \PPi_{2} \lotimes_{\Pa} M[-1]) = 0$.
\end{lemma}

\begin{proof}
First we deal with the case where $Q$ is a non-Dynkin quiver. 
We may assume that $M \in \Pa \mod$. 
Since $\YYM_{1} \lotimes_{\Pa} M $ and $\PPi_{2} \lotimes_{\Pa} M$ also belong to $\Pa \mod$, 
the statement is clear.

Next we deal with the case where $Q$ is a Dynkin quiver. 
We choose a vertex $i_{0} \in Q_{0}$ and use the function $p: \ind \Dbmod{\Pa} \to \ZZ$ defined in Section 
\ref{202105252310}.  
Let $h$ be the Coxeter number of $Q$. 
Then we have $p(\YYM_{1} \lotimes_{\Pa} M) = p(M) +1$ and $p(\PPi_{2}\lotimes_{\Pa} M [-1] ) = p(M) + 4 -h$. 
Thus if $h > 3$, 
we have $p(\YYM_{1}\lotimes_{\Pa} M) > p(\PPi_{2} \lotimes_{\Pa}M[-1])$ 
and $\Hom_{\Pa}(\YYM_{1} \lotimes_{\Pa}M, \PPi_{2} \lotimes_{\Pa} M[-1]) =0$. 

We have $h \leq 3$ precisely when $Q$ is an $A_{1}$ or $A_{2}$ quiver. 
The $A_{1}$ case is clear.
\end{proof}

\subsection{Proof of Theorem \ref{2020071920551}}\label{202204121618}

\subsubsection{Proof of (1)(2) of Theorem \ref{2020071920551} the case where $Q$ is not an $A_{2}$-quiver}
First we deal with the case where  $Q$ is not an $A_{2}$-quiver. 
For simplicity we set $\epsilon':= \frac{{}^{v}\!\Euch( M){}_{\PPi_{1} }}{  {}^{v}\!\Euch(\YYM_{1} \lotimes_{\Pa} M)} $. 
By the case $n =2$ of  Proposition \ref{202012171832}, 
the left square of the diagram below commutative:
\begin{equation}\label{202012171836}
\begin{xymatrix}@C=40pt{
\PPi_{2} \lotimes_{\Pa} M [-1] \ar[r]^{\ttheta_{2, M}} \ar@{=}[d] &
\YYM_{1} \lotimes_{\Pa} M \ar[d]_{\epsilon' \ppi_{1 ,M} }  \ar[r]^{\brho_{2 ,M} }
 & \YYM_{2} \lotimes_{\Pa} M \ar@{-->}[d]^{\varpi_{2, M}} \ar[r]^{\ppi_{2, M}} & 
\PPi_{2} \lotimes_{\Pa} M  \ar@{=}[d]  \\
\PPi_{2} \lotimes_{\Pa} M [-1] \ar[r]_{{}_{\PPi_{1}}\ttheta_{1, M}}  & 
\PPi_{1} \lotimes_{ \Pa} M \ar[r]_{ {}_{ \PPi_{1} }\rrho_{M} } & 
\PPi_{1} \lotimes_{\Pa} \YYM_{1} \lotimes_{ \Pa} M \ar[r]_{ {}_{ \PPi_{1} } \ppi_{1, M} } & 
\PPi_{2 } \lotimes_{\Pa} M.
}\end{xymatrix}
\end{equation}
It follows from \cite[Lemma 1.4.3]{Neeman} that 
there exists a morphism $\varpi_{2, M}: \YYM_{2} \lotimes_{\Pa} M \to \PPi_{1} \lotimes_{\Pa} \YYM_{1} \lotimes_{\Pa} M$ 
that completes the commutative diagram  \eqref{202012171836} and 
 makes the middle square of \eqref{202012171836} homotopy cartesian 
 that is folded to  an exact triangle 
 \begin{equation}\label{202105241753}
 \YYM_{1} \lotimes_{\Pa} M \xrightarrow{ \tiny\begin{pmatrix}-\epsilon' \ppi_{1, M}\\ \brho_{2, M}\end{pmatrix}  } 
 (\PPi_{1} \lotimes_{\Pa} M )\oplus (\YYM_{2} \lotimes_{\Pa} M) \xrightarrow{ ({}_{\PPi_{1}} \rrho_{ M}, \varpi_{2,M})} 
 \PPi_{1} \lotimes_{\Pa} \YYM_{1} \lotimes_{\Pa}M \to .
 \end{equation}
whose connecting morphism $\PPi_{1} \lotimes_{\Pa} \YYM_{1} \lotimes_{\Pa}M \to  \YYM_{1} \lotimes_{\Pa} M[1]$
 is $-\ttheta_{2, M}[1] ({}_{\PPi_{1}} \ppi_{1, M}) = -\ttheta_{1, \YYM_{1}\lotimes M}[1]$. 
 We remark that  we  modify the definition of homotopy Cartesian square from \cite{Neeman} (see Remark \ref{homotopy Cartesian remark}). 
 According to this, the sign of the connecting morphism is different from that of \cite[Lemma 1.4.2]{Neeman}.

We consider the case $n =2$ of the diagram \eqref{202008141848}. 
Since $\eeta^{*}_{2, M}$ is a split monomorphism, 
there exist morphisms    
\[
p: \YYM_{1} \lotimes_{\Pa}\YYM_{1} \lotimes_{\Pa} M \to \PPi_{1} \lotimes_{\Pa} M, \ 
i: \YYM_{1} \lotimes_{\Pa} M \to \YYM_{1} \lotimes_{\Pa}\YYM_{1} \lotimes_{\Pa} M
\]
that satisfy the equations 
\begin{equation}\label{202207131533}
p\eeta^{*}_{2, M}  = \id, \zzeta_{2, M} i = \id, pi =0, \eeta^{*}_{2, M} p+ i \zzeta_{2, M} = \id. 
\end{equation}
It follows that there exists   the following isomorphism of exact triangles 
\begin{equation}\label{202012171855}
\begin{xymatrix}@C=60pt{
\YYM_{1} \lotimes_{\Pa} M \ar@{=}[d] \ar[r]^{\rrho_{\YYM_{1} \lotimes M} } & 
\YYM_{1} \lotimes_{\Pa} \YYM_{1} \lotimes_{\Pa} M \ar[d]^{ \tiny \begin{pmatrix} p \\ \zzeta_{2, M} \end{pmatrix} }_{\cong} \ar[r]^{\ppi_{1, \YYM_{1} \lotimes M} }& 
\PPi_{1} \lotimes_{\Pa} \YYM_{1} \lotimes_{\Pa} M \ar@{=}[d] \\
\YYM_{1} \lotimes_{\Pa} M  
\ar[r]_-{ \tiny \begin{pmatrix}- \alpha  \\ \brho_{2 , M}  \end{pmatrix}} &
 ( \PPi_{1}  \lotimes_{\Pa}M )  \oplus 
(\YYM_{2 } \lotimes_{\Pa} M)
\ar[r]_-{ ( {}_{ \PPi_{1} }\rrho_{M}, \ \beta ) } & 
\PPi_{1} \lotimes_{\Pa} \YYM_{1} \lotimes_{\Pa} M 
}\end{xymatrix}
\end{equation}
where we set $\alpha :=  -p\rrho_{\YYM_{1} \lotimes M}$ and $\beta:= (\ppi_{1, \YYM_{1} \lotimes M}) i$.

We check $ ({}_{ \PPi_{1} } \ppi_{1}) \beta = \ppi_{2, M}$ as follows:
\[
({}_{\PPi_{1} } \ppi_{1, M}) \beta = ({}_{ \PPi_{1} }\ppi_{1, M }) \ppi_{1, \YYM_{1}\lotimes M} i = \ppi_{2, M} \zzeta_{2, M} i = \ppi_{2, M}
\]
where the second equality follows from the commutativity of the middle square of \eqref{202008141848}.
Therefore, we have the following commutative diagram 
whose left square is homotopy cartesian
\begin{equation}\label{202012171906}
\begin{xymatrix}@C=40pt{
 &
\YYM_{1} \lotimes_{\Pa} M \ar[d]_{\alpha } \ar[r]^{\brho_{2 ,M} }
 & \YYM_{2} \lotimes_{\Pa} M \ar[d]^{\beta} \ar[r]^{\ppi_{2, M}} & 
\PPi_{2} \lotimes_{\Pa} M  \ar@{=}[d] \\
\PPi_{2} \lotimes_{\Pa} M [-1] \ar[r]_{ {}_{ \PPi_{1}}\ttheta_{1, M}}  & 
\PPi_{1} \lotimes_{ \Pa} M \ar[r]_{ {}_{ \PPi_{1} }\rrho_{1 ,M} } & 
\PPi_{1} \lotimes_{\Pa} \YYM_{1} \lotimes_{ \Pa} M \ar[r]_{ {}_{ \PPi_{1} }\ppi_{1, M} } & 
\PPi_{2 } \lotimes_{\Pa} M.
}\end{xymatrix}
\end{equation}

Thanks to Lemma \ref{202012171908}, we can apply Lemma \ref{2020071917401} to the homotopy cartesian squares 
\eqref{202012171836} and \eqref{202012171906}. 
There exists a morphism $s: \YYM_{2} \lotimes_{\Pa} M  \to \PPi_{1} \lotimes_{\Pa} M$ that gives an isomorphism of exact triangle as below:
\[ 
\begin{xymatrix}@C=60pt{
\YYM_{1} \lotimes_{\Pa} M \ar@{=}[d] \ar[r]^-{ \tiny \begin{pmatrix}- \alpha  \\ \brho_{2 , M}  \end{pmatrix}} & 
( \PPi_{1}  \lotimes_{\Pa}M )  \oplus 
(\YYM_{2 } \lotimes_{\Pa} M)
\ar[r]^-{ ( {}_{ \PPi_{1} }\rrho_{M}, \ \beta ) }  
\ar[d]^{\tiny \begin{pmatrix} \id & s \\ 0 & \id \end{pmatrix} } 
&
\PPi_{1} \lotimes_{\Pa} \YYM_{1} \lotimes_{\Pa} M \ar@{=}[d] \\
\YYM_{1} \lotimes_{\Pa} M  
\ar[r]_-{ \tiny \begin{pmatrix}- \epsilon' \ppi_{1, M}  \\ \brho_{2 , M}  \end{pmatrix}} &
 ( \PPi_{1}  \lotimes_{\Pa}M )  \oplus 
(\YYM_{2 } \lotimes_{\Pa} M)
\ar[r]_-{ (  {}_{\PPi_{1}} \brho_{1, M}, \ \varpi_{2, M}) } & 
\PPi_{1} \lotimes_{\Pa} \YYM_{1} \lotimes_{\Pa} M 
}\end{xymatrix}
\]
Compositing the  isomorphism  with \eqref{202012171855}, we obtain the following diagram 
\[ 
\begin{xymatrix}@C=60pt{
\YYM_{1} \lotimes_{\Pa} M \ar@{=}[d] \ar[r]^{\rrho_{\YYM_{1} \lotimes M} } & 
\YYM_{1} \lotimes_{\Pa} \YYM_{1} \lotimes_{\Pa} M 
\ar[d]^{\tiny \begin{pmatrix} p +s \zzeta_{2,M} \\ \zzeta_{2, M} \end{pmatrix} =: \phi} 
\ar[r]^{\ppi_{1, \YYM_{1} \lotimes M} } & 
\PPi_{1} \lotimes_{\Pa} \YYM_{1} \lotimes_{\Pa} M \ar@{=}[d]\\
\YYM_{1} \lotimes_{\Pa} M  
\ar[r]_-{ \tiny \begin{pmatrix}- \epsilon' \ppi_{1, M}  \\ \brho_{2 , M}  \end{pmatrix}} &
 ( \PPi_{1}  \lotimes_{\Pa}M )  \oplus 
(\YYM_{2 } \lotimes_{\Pa} M)
\ar[r]_-{ (  {}_{\PPi_{1}} \brho_{1, M}, \ \varpi_{2, M}) } & 
\PPi_{1} \lotimes_{\Pa} \YYM_{1} \lotimes_{\Pa} M 
}\end{xymatrix}
\]
Let $\pr_{1}=(\id, 0): ( \PPi_{1}  \lotimes_{\Pa}M )  \oplus 
(\YYM_{2 } \lotimes_{\Pa} M) \to \PPi_{1} \lotimes_{\Pa} M$ be the first projection. 
Now it is straight forward to check that
 the composition $\xi_{2,M} := -(\epsilon')^{-1} \pr_{1} \phi$ satisfies the desired equations (1) and (2).
 \[
\begin{xymatrix}@C=40pt{ 
& \PPi_{1} \lotimes_{\Pa} M \ar[d]^{\eeta^{*}_{2,  M} }  \ar@/^20pt/[ddr]^{-(\epsilon')^{-1} } 
& 
\\ 
\YYM_{1}\lotimes_{\Pa} M \ar@/_20pt/[rrd]_{\ppi_{1, M} } \ar[r]^{\rrho_{ \YYM_{1} \lotimes M}}
 &
 \YYM_{1} \lotimes_{\Pa} \YYM_{1} \lotimes_{\Pa} M \ar[dr]^{\xi_{2, M}}
 & \\
&& \PPi_{1} \lotimes_{\Pa} M
}
\end{xymatrix}.\]
\qed

\subsubsection{Proof of (1)(2) of Theorem \ref{2020071920551} the case where $Q$ is an $A_{2}$-quiver}

We assume that $Q$ is an $A_{2}$-quiver. 
In this case the Coxeter number $h =3$ and hence $\YYM_{2} = 0$. 
Thus the exact triangle \eqref{202105241753} is of the following form  
 \[ \YYM_{1} \lotimes_{\Pa} M \xrightarrow{ -\epsilon' \ppi_{1, M}  } 
 \PPi_{1} \lotimes_{\Pa} M  \xrightarrow{ {}_{\PPi_{1}} \rrho_{ M} } 
 \PPi_{1} \lotimes_{\Pa} \YYM_{1} \lotimes_{\Pa}M \xrightarrow{ -\ttheta_{\YYM_{1} \lotimes M}[1] } \YYM_{1} \lotimes_{\Pa} M[1].
 \]
 
It follows from $\YYM_{2} = 0$ that $\zzeta_{2, M} = 0$. 
Thus the equation \eqref{202207131533} implies that $\eeta^{*}_{2, M}$ is an isomorphism and 
the exact triangle in the bottom of \eqref{202012171855} is of the following form 
 \[ \YYM_{1} \lotimes_{\Pa} M \xrightarrow{ (\eeta_{2, M}^{*})^{-1} \rrho_{\YYM_{1} \lotimes M }   } 
 \PPi_{1} \lotimes_{\Pa} M  \xrightarrow{ {}_{\PPi_{1}} \rrho_{ M} } 
 \PPi_{1} \lotimes_{\Pa} \YYM_{1} \lotimes_{\Pa}M \xrightarrow{ -\ttheta_{\YYM_{1} \lotimes M}[1] } \YYM_{1} \lotimes_{\Pa} M[1].
 \]
Since $\End(\PPi_{1} \lotimes_{\Pa} M ) \cong \kk$ and ${}_{\PPi_{1}}\rrho_{M} \neq 0$, comparing two exact triangles  
we have $-\epsilon' \ppi_{1, M}= (\eeta_{2, M}^{*})^{-1} \rrho_{\YYM_{1} \lotimes M }$.
 Thus we conclude that $\xi_{2, M} := - (\epsilon')^{-1} (\eeta_{2, M}^{*})^{-1}$ has the desired property.
 \qed

\subsubsection{Proof of (3) of Theorem \ref{2020071920551} }
We keep  the condition of Theorem \ref{2020071920551}. 
We have already proved that existence of $\xi_{2, M}$ that satisfies the equation (1) and (2). 
We prove this morphism also satisfies the equation (3).

\begin{lemma}\label{202104091845}
We have 
\[
{}_{\YYM_{1}}\rrho_{M} = (\id + \eeta^{*}_{ 2, M} \xi_{2, M}) \rrho_{\YYM_{1} \lotimes M}.
\]
\end{lemma}

\begin{proof}
We deduce the desired equality in the following way. 
\[
\begin{split} 
\eeta^{*}_{ 2, M} \xi_{2, M} \rrho_{\YYM_{1} \lotimes M} 
 & = \eeta^{*}_{2, M} \ppi_{1, M} \\
 & = {}_{\YYM_{1}}\rrho_{M} - \rrho_{\YYM_{1} \lotimes M} 
\end{split}
 \]
where for the first equality we use (1) and for the second we use Lemma \ref{202001111736} or Lemma \ref{202207131538}. 
\end{proof}

We set $\epsilon_{2, M} := - \epsilon' := - \frac{{}^{v}\!\Euch(M)}{  {}^{v}\!\Euch(\YYM_{1}\lotimes_{\Pa} M) }$. 
By (1), we have 
$\epsilon_{2, M}\xi_{2, M} \eeta^{*}_{2, M} = \id_{\PPi_{1} \lotimes M}$. 
In other words the morphisms  $\epsilon_{2, M} \xi_{2, M}: \YYM_{1} \lotimes_{\Pa} \YYM_{1} \lotimes_{\Pa} M \to \PPi_{1} \lotimes_{\Pa} M$ is a retraction  of $\eeta^{*}_{2, M}$
We denote by $\omega_{2, M}: \YYM_{2} \lotimes_{\Pa} M \to \YYM_{1} \lotimes_{\Pa} \YYM_{1} \lotimes_{\Pa} M$ the corresponding section of $\zzeta_{2, M}$ 
so that the morphisms below are inverse to each other 
\begin{equation}\label{2021040918451}
\begin{split}
& {\small \begin{pmatrix} \epsilon_{2, M} \xi_{2, M} \\ \zzeta_{2, M} \end{pmatrix}  }:
\YYM_{1} \lotimes_{\Pa} \YYM_{1} \lotimes_{\Pa} M \to 
(\PPi_{1} \lotimes_{\Pa}M ) \oplus (\YYM_{2} \lotimes_{\Pa} M), \\
&
(\eeta^{*}_{2, M}, \omega_{2,M} ): (\PPi_{1} \lotimes_{\Pa}M ) \oplus (\YYM_{2} \lotimes_{\Pa} M) 
\to \YYM_{1} \lotimes_{\Pa} \YYM_{1} \lotimes_{\Pa} M. 
\end{split}
\end{equation}
For  latter reference we note the equality  $\varpi_{2, M} = (\ppi_{1, \YYM_{1} \lotimes M} ) (\omega_{2, M})$.

Observe that if we identify those objects by using these isomorphisms, 
then the endomorphism $\eeta^{*}_{2, M} \xi_{2, M}$ of $\YYM_{1} \lotimes_{\Pa} \YYM_{1} \lotimes_{\Pa} M$ corresponds to 
the endomorphism 
${\small \begin{pmatrix} \epsilon_{2, M}^{-1}  \id& 0 \\ 0 & 0 \end{pmatrix}}$ of 
$(\PPi_{1} \lotimes_{\Pa}M ) \oplus (\YYM_{2} \lotimes_{\Pa} M)$. 
Hence, 
the endomorphism $\id + \eeta^{*}_{ 2, M} \xi_{2, M}$ corresponds to 
${\small \begin{pmatrix} ( 1+ \epsilon_{2, M}^{-1} )\id  & 0 \\ 0 & \id \end{pmatrix}}$. 
We have the following commutative diagram. 
\[
\begin{xymatrix}@C=60pt{
\YYM_{1} \lotimes_{\Pa} M \ar[r]^{\rrho_{\YYM_{1} \lotimes M} } \ar[dr]_{{}_{\YYM_{1}}\rrho_{ M}} & 
\YYM_{1} \lotimes_{\Pa} \YYM_{1} \lotimes_{\Pa} M \ar[r]^{\tiny \begin{pmatrix} \epsilon_{2, M} \xi_{2, M} \\ \zzeta_{2, M} \end{pmatrix}  } 
\ar[d]^{\id + \eeta^{*}_{ 2, M} \xi_{2, M}} & 
(\PPi_{1} \lotimes_{\Pa}M ) \oplus (\YYM_{2} \lotimes_{\Pa} M) \ar[d]^{\tiny \begin{pmatrix} ( 1+ \epsilon_{2, M}^{-1} )\id  & 0 \\ 0 & \id \end{pmatrix}}\\
& 
\YYM_{1} \lotimes_{\Pa} \YYM_{1} \lotimes_{\Pa} M \ar[dr]_{\xi_{2,M} } & 
(\PPi_{1} \lotimes_{\Pa}M ) \oplus (\YYM_{2} \lotimes_{\Pa} M) \ar[d]^{( \epsilon_{2, M}^{-1} \id , 0)} \ar[l]_{(\eeta^{*}_{2, M}, \omega_{2, M})} \\
&& \PPi_{1} \lotimes_{\Pa} M 
}\end{xymatrix} 
\]
Using this diagram, we complete the proof of (3) 
\[
(\xi_{2, M} )({}_{\YYM_{1}}\rrho_{M}) = ( 1 + \epsilon_{2, M}^{-1}) \xi_{2,M}\rrho_{\YYM_{1} \lotimes M} =  -\frac{{}^{v}\!\Euch(\PPi_{1} \lotimes_{\Pa} M )}{{}^{v}\!\Euch(M)}  \ppi_{1, M}. 
\]
\qed

\subsubsection{A corollary}

We have proved that 
the morphism $\rrho_{\YYM_{1} \lotimes M}: \YYM_{1} \lotimes_{\Pa} M \to \YYM_{1} \lotimes_{\Pa} \YYM_{1} \lotimes_{\Pa} M$ 
is a minimal left $\rad$-approximation under a certain condition. 
In the next corollary we show that  the morphism ${}_{\YYM_{1}} \rrho_{M} : \YYM_{1} \lotimes_{\Pa} M \to \YYM_{1} \lotimes_{\Pa} \YYM_{1} \lotimes_{\Pa} M$ that inserts $\varrho$ in the middle of $\YYM_{1}$ and $M$ also gives a minimal left $\rad$-approximation.

\begin{corollary}\label{202105251650}
Let $M \in \ind \Dbmod{\Pa}$ and $v \in \We_{Q}$. 
Assume that ${}^{v}\!\Euch(M) \neq 0$,  $ {}^{v}\!\Euch(\PPi_{1} \lotimes_{\Pa} M) \neq 0$ 
and moreover that ${}^{v}\!\Euch(N) \neq 0$ for any indecomposable direct summand $N$ of $\YYM_{1} \lotimes_{\Pa} M$.
Then, 
the morphism ${}_{\YYM_{1}} \rrho_{M} : \YYM_{1} \lotimes_{\Pa} M \to \YYM_{1} \lotimes_{\Pa} \YYM_{1} \lotimes_{\Pa} M$ 
is a left minimal $\rad$-approximation. 
\end{corollary}

\begin{proof} 
Since $1 + \epsilon_{2, M}^{-1} =- \frac{{}^{v}\!\Euch(\PPi_{1} \lotimes_{\Pa} M )}{{}^{v}\!\Euch(M)} \neq 0$, 
the endomorphism $\id + \eeta^{*}_{ 2, M} \xi_{2, M}$ is an automorphism of $\YYM_{1} \lotimes_{\Pa} \YYM_{1} \lotimes_{\Pa} M$.  
It follows from the second assumption that 
$\rrho_{\YYM_{1} \lotimes M}$ is a minimal left $\rad$-approximation of $\YYM_{1} \lotimes_{\Pa} M$. 
From the equality ${}_{\YYM_{1}}\rrho_{M} = (\id + \eeta^{*}_{ 2, M} \xi_{2, M}) \rrho_{\YYM_{1} \lotimes M}$ of   Lemma \ref{202104091845}, 
we conclude that 
the morphism ${}_{\YYM_{1}} \rrho_{M}$ 
is a left minimal $\rad$-approximation of $\YYM_{1} \lotimes_{\Pa} M$. 
\end{proof}

\subsection{Left $\rad^{2}$-approximation}

\begin{theorem}\label{202103021651}
Let $M \in \ind \Dbmod{\Pa}$ and $v \in \We_{Q}$. 
Assume that ${}^{v}\!\Euch(M) \neq 0$  and  ${}^{v}\!\Euch(\YYM_{1} \lotimes_{\Pa} M) \neq 0$. 
Then the morphism $\brho^{2}_{M} : M \to \YYM_{2} \lotimes_{\Pa} M $ is a minimal left $\rad^{2}$-approximation of $M$ 
whose cone is $\PPi_{1} \lotimes_{\Pa} \YYM_{1} \lotimes_{\Pa} M$. 
\[
M \xrightarrow{ \brho_{M}^{2} } \YYM_{2} \lotimes_{\Pa} M \to \PPi_{1} \lotimes_{\Pa} \YYM_{1} \lotimes_{\Pa} M. 
\]
\end{theorem}

\begin{proof}
By Theorem \ref{universal Auslander-Reiten triangle}  the morphism $\rrho_{M}: M \to \YYM_{1} \lotimes_{\Pa} M$ is a minimal left $\rad$-approximation of $M$. 
On the other hand,  
we established the following commutative diagram in \eqref{202012171836} 
where $\epsilon'$ is an automorphism  of $\PPi_{1} \lotimes_{\Pa} M$.  
\[ 
\begin{xymatrix}@C=40pt{
\YYM_{1} \lotimes_{\Pa} M \ar[d]_{\epsilon' \ppi_{1 ,M} }  \ar[r]^{\brho_{2 ,M} }
 & \YYM_{2} \lotimes_{\Pa} M \ar[d]^{\varpi_{2, M}} \\ 
\PPi_{1} \lotimes_{ \Pa} M \ar[r]_{\PPi_{1} \lotimes \rrho_{1 ,M} } & 
\PPi_{1} \lotimes_{\Pa} \YYM_{1} \lotimes_{ \Pa} M. 
}\end{xymatrix}\]
Since this square is homotopy Cartesian that is folded to a direct sum of AR-triangles, 
we conclude by Lemma \ref{2021050917551}  that the composition
 $\brho^{2}_{M} = (\brho_{2, M})(\rrho_{M})$ is a minimal left $\rad^{2}$-approximation of $M$ 
which has $\psi$ as a cone morphism. 
\end{proof} 

\subsubsection{Remark on the assumptions}

We provide an example that shows we need to assume that ${}^{v}\!\Euch(\YYM_{1} \lotimes_{\Pa} M) \neq 0$.

\begin{example}\label{202103051509}

Let $Q : 1 \to 2 \to 3$. 
Assume that the weight $v \in \kk Q_{0}$ is regular. 
 
Let  $M := P_{2}= \Pa e_{2}$ be the indecomposable projective  $\Pa$-module corresponding to the vertex $2$. 
Then we have $\YYM_{1} \lotimes_{\Pa} M \cong P_{3} \oplus S_{2}$ 
and 
\[
{}^{v}\!\Euch(\YYM_{1} \lotimes_{\Pa} M ) = v_{1} + 2v_{2} + v_{3}. 
\]
 
 As is  explained  in Example \ref{202103221959}, if $v_{1} + 2 v_{2} + v_{3} = 0$, then we can directly check that $\varrho_{2}^{2} = 0$. 
It follows that   the morphism $\brho^{2}_{M}: M \to \YYM_{2} \lotimes_{\Pa} M $  is zero morphism 
 and in particular does not give a minimal left $\rad^{2}$-approximation of $M$. 
Below we explain there is another way to see this. 

First we claim that 
$\ppi_{1, M} \ttheta_{2, M} =0$. 
Indeed, from  \eqref{202102211751} and the assumption,  we deduce   
\[
\isagl{\ppi_{1, M} \ttheta_{2, M} , \id_{\PPi_{1} \lotimes M} } ={}^{v}\!\Euch(\YYM_{1} \lotimes_{\Pa} M )=  0.
\]
Since $\Hom_{\Pa}(\PPi_{2} [-1] \lotimes_{\Pa} M, \PPi_{1} \lotimes_{\Pa} M ) \cong \tuD\Hom_{\Pa}(M, M)$ is one-dimensional over $\kk$, 
the above equation implies $\ppi_{1, M} \ttheta_{2, M} = 0$. 

By the octahedral axiom, we obtain the following commutative diagram whose right column is exact. 
\[
\begin{xymatrix}{
&&  M  \ar@{=}[rr] \ar[d]_{\rrho_{ M} }  &&  M  \ar[d]^{\brho_{M}^{2}}  && && \\
\PPi_{2}\lotimes_{\Pa} M [-1] \ar[rr]^{\ttheta_{2, M}} \ar@{=}[d] &&
\YYM_{1} \lotimes_{\Pa} M  \ar[d]_{ \ppi_{1, M} } \ar[rr]^{\brho_{2, M}} && 
\YYM_{2} \lotimes_{\Pa} M  \ar[d]^{ f  = { \tiny \begin{pmatrix} f_{1}  \\ f_{2}  \end{pmatrix}  }} 
\\
\PPi_{2} \lotimes_{\Pa} M[-1] \ar[rr]_{0}  
&& \PPi_{1 } \lotimes_{\Pa} M \ar[rr]_{ \tiny \begin{pmatrix} \id \\0 \end{pmatrix}  } \ar[d]
&&  ( \PPi_{1} \lotimes_{\Pa} M ) \oplus ( \PPi_{2} \lotimes_{\Pa} M ) \ar[d]
 \\ 
&& M [1] \ar@{=}[rr] &&  M [1] && && 
}\end{xymatrix}
\]
Using this diagram we can show that $f$ is a split monomorphism and hence $\brho^{2}_{M} =0$. 

\end{example}

\subsection{Remark on right versions }

We can obtain right versions of results obtained in this section by replacing $\Pa$ with $\Pa^{\op}$. 
We give few remarks on them  and their relationship to left versions, 
to avoid potential confusions for the formulas of Proposition \ref{202102231445}.

A right version of Theorem \ref{2020071920551} for an indecomposable object is the following:

Let $N \in \ind \Pa^{\op}$. 
Then 
there exists
 a morphism ${}_{N}\xi^{\textup{right}} : \ N \lotimes_{\Pa} \YYM_{1} \lotimes_{\Pa} \YYM_{1}  \to N \lotimes_{\Pa} \PPi_{1}$ 
such that $({}_{N}\xi^{\textup{right}} )({}_{N \lotimes \YYM_{1} } \rrho)  =- {}_{N}\ppi_{1}$ 
and  
we have 
\[
({}_{N}\xi^{\textup{right}})({}_{N} \eeta^{*})  = -\frac{{}^{v}\!\Euch(N \lotimes_{\Pa} \YYM_{1} ) }{{}^{v}\!\Euch(N)} \id_{N \lotimes \PPi_{1}}.
\]
Comparing 
$(\xi_{M})(\rrho_{\YYM_{1} \lotimes M}) = \ppi_{1, M}$ with 
$({}_{N}\xi^{\textup{right}} )({}_{N \lotimes \YYM_{1} } \rrho)  =- {}_{N}\ppi_{1}$, 
the signs are different.  
The difference   comes from the difference between Lemma \ref{homotopic lemma 5} and Lemma \ref{homotopic lemma 5 right}. 

We also point out the following difference.  
Assume that the weight $v \in \kk Q_{0}$ is an eigenvector of $\Psi = - C^{-1} C^{t}$ with the eigenvalue $\lambda$. 
Then it follows from Lemma \ref{202103021744} that 
\[
\frac{{}^{v}\!\Euch(\YYM_{1}\lotimes_{\Pa} M ) }{{}^{v}\!\Euch(M)} = 1 + \lambda, \ \frac{{}^{v}\!\Euch(N \lotimes_{\Pa} \YYM_{1} ) }{{}^{v}\!\Euch(N)} = 1+ \lambda^{-1}.
\]

\section{A sufficient criterion concerning $\rad^{n}$-approximations}\label{section: minimal left rad^n-approximations}

The aim of this Section \ref{section: minimal left rad^n-approximations} is to prove Theorem \ref{202104091259} 
that gives sufficient conditions that the morphism $\brho^{n}_{M}: M \to \YYM_{n} \lotimes_{\Pa} M$ is a minimal left $\rad^{n}$-approximation. 

\subsection{Main Theorem}

\subsubsection{The property (II)${}_{M,\lambda}$}

\begin{definition}
Let $M$ be an indecomposable  object of  $\Dbmod{\Pa}$ and $\lambda \in \kk \setminus \{ -1\}$. 
We denote by $\cC_{M}$ the connected component of the AR-quiver of $\Dbmod{\Pa}$ to which $M$ belongs. 

We say that $v \in \We_{Q}$ has the property (II)${}_{M, \lambda}$ if 
for any $ N \in \cC_{M}$ 
we have 
\[
{}^{v}\!\Euch(N) \neq 0, \ \ \frac{{}^{v}\!\Euch(\PPi_{1} \lotimes_{\Pa} N ) }{{}^{v}\!\Euch(N)} = \lambda.
\]
\end{definition}

We note that the equation above together with the assumption  $\lambda \neq -1$ implies 
\[
\frac{{}^{v}\!\Euch(\YYM_{1} \lotimes_{\Pa} N ) }{{}^{v}\!\Euch(N)} = 1 +\lambda \neq 0. 
\]

We point out  an immediate consequence. 
\begin{lemma}
Let $M \in \ind \Dbmod{\Pa}$ and $\lambda \in \kk \setminus \{ -1\}$. 
Assume that  $v \in \kk^{\times} Q_{0}$ has the property (II)${}_{M, \lambda}$. 
Then it has the property (I')${}_{M, n}$ for any 
$n \geq 1$. 
Moreover the assumption of Corollary \ref{202105251650} holds for any $L \in \cC_{M}$. 
\end{lemma}

\begin{example}
Let $\Phi$ be the Coxeter matrix of $Q$ and set $\Psi := \Phi^{-t}$. 

Let $M \in \ind\Dbmod{\Pa}$. 
If $v \in \kk Q_{0}$ is an eigenvector of $\Psi $ with the eigenvalue $\lambda \neq -1$  
and has the property (I)${}_{M}$, 
then it has the property (II)${}_{M, \lambda}$ by \eqref{202111101530}.

Thus, 
if  a regular  (resp. semi-regular) weight $v \in \kk Q_{0}$ is an eigenvector of $\Psi$ with the eigenvalue $\lambda \neq -1$, 
then it has the property (II)${}_{M, \lambda}$ for all $M \in \ind \Dbmod{\Pa}$ (resp. $\ind\cP[\ZZ]$). 

\end{example}

In the proof of  Proposition \ref{202105241610}, we provide an example of  $v \in \kk Q_{0}$ which is not an eigenvector of $\Psi$, but has 
the property (II)${}_{M, 1}$ for some $M$.  

\subsubsection{}

For simplicity we set $V_{n}(t) := 1 + t + t^{2} + \cdots + t^{n-1}$. 

\begin{theorem}\label{202104091259}
Let $M$ be an indecomposable  object of  $\Dbmod{\Pa}$ and $ \lambda \in \kk\setminus \{ -1\}$. 
Assume that the weight $v \in \kk^{\times} Q_{0}$ has the property (II)${}_{M, \lambda}$. 
Then the following statement holds.

If  a natural number  $n \geq 2$ 
satisfies the condition: 
\[ 
V_{m}(\lambda )  \neq 0, \textup{ for } 1 \leq m \leq n, 
\]
then the morphism $\brho_{M}^{n}: M\to \YYM_{n} \lotimes_{\Pa} M$ 
is a minimal left $\rad^{n}$-approximation whose cone is isomorphic to $\PPi_{1}\lotimes_{\Pa} \YYM_{n -1} \lotimes_{\Pa}M$.
\[
M \xrightarrow{ \brho_{M}^{n}} \YYM_{n} \lotimes_{\Pa} M \xrightarrow{ \varpi_{n, M}} \PPi_{1} 
\lotimes_{\Pa} \YYM_{n -1} \lotimes_{\Pa} M \to .
\]
\end{theorem}

\subsection{Main Proposition} 
In this section we prove Proposition \ref{2021040913541} which is a key for the proof of  Theorem \ref{202104091259}

\subsubsection{Main proposition}

\begin{proposition}\label{2021040913541}
Let $M$ be an indecomposable  object of  $\Dbmod{\Pa}$ and $\lambda \in \kk\setminus \{ -1\}$. 
Assume that the weight $v \in \kk Q_{0}$ has the property (II)${}_{M, \lambda}$. 
 
If  a natural number  $n \geq 2$ 
satisfies the condition: 
\begin{equation}\label{2021030112531}
V_{m}(\lambda )  \neq 0, \textup{ for } 1 \leq m \leq n -1, 
\end{equation}
then 
there exists a morphism of $\omega_{n -1, M}: \YYM_{n -1} \lotimes_{\Pa} M  \to \YYM_{1} \lotimes_{\Pa} \YYM_{n -2} \lotimes_{\Pa }M $ in $\Dbmod{\Pa}$ 
that has the following properties 

\begin{enumerate}[(1)] 
\item $\zzeta_{n -1, M} \omega_{n -1, M} = \id_{\YYM_{n -1} \lotimes M }$. 

\item If we set $\xi_{n, M} := (\xi_{2, \YYM_{ n-2} \lotimes M })({}_{\YYM_{1}} \omega_{n -1, M})$, where
\[
\xi_{n, M} 
: \YYM_{1} \lotimes_{\Pa} \YYM_{n -1} \lotimes_{\Pa} M 
\xrightarrow{{}_{\YYM_{1}} \omega_{n -1, M} } 
\YYM_{1} \lotimes_{\Pa} \YYM_{1} \lotimes_{\Pa} \YYM_{n -2} \lotimes_{\Pa} M 
\xrightarrow{\xi_{2, \YYM_{n -2} \lotimes M } } 
\PPi_{1} \lotimes_{\Pa} \YYM_{n -2} \lotimes_{\Pa}M,
\]
then the following equality holds  in $\ResEnd_{\Pa } (\PPi_{1} \lotimes_{\Pa} \YYM_{n -2}\lotimes_{\Pa} M )$ 
\[
\xi_{n, M} \eeta^{*}_{n, M} =-  \frac{V_{n}(\lambda)}{V_{n-1}(\lambda) }.
\]
Therefore if $V_{n}(\lambda) \neq 0$, 
the composition $\xi_{n, M}\eeta^{*}_{n, M}$ 
is an automorphism of $\PPi_{1} \lotimes_{\Pa} \YYM_{n -2} \lotimes_{\Pa} M $.

\item We have 
\[
(\xi_{n, M})( \rrho_{\YYM_{ n-1} \otimes  M }) =(\ppi_{1, \YYM_{n-2} \lotimes M })(\omega_{n-1, M}) .
\]
\end{enumerate}
\end{proposition}

We prove Proposition~\ref{2021040913541} in Section~\ref{proof of Proposition} after some preparation.

\subsubsection{}

We use the following easy observation. 

\begin{lemma}\label{2021022716261}
Let $f: X \to Y, g: Y \to Y, h : Y \to X$ be morphisms in $\Dbmod{\Pa}$. 
Assume we have 
$g= a\id_{Y}$ in $\ResEnd(Y)$ and $hf= b \id_{X}$ in $\ResEnd(X)$ for some $a, b\in \kk$. 
Then we have $hgf = ab\id_{X}$ in $\ResEnd(X)$. 
\end{lemma}

\begin{proof}
Let $\mathcal{C}$ be the ideal quotient of the category $\Dbmod{\Pa}$ by the ideal of radical morphisms. Then $\ResEnd(X)$ is the endomorphism algebra of $X$ in $\mathcal{C}$ and the claim follows as $\mathcal{C}$ is a $\kk$-category.
\end{proof}

\subsubsection{}

\begin{lemma}\label{2021022718051}
Let $M$ be an indecomposable  object of  $\Dbmod{\Pa}$ and $\lambda \in \kk$. 
Assume that the weight  $v \in \kk Q_{0}$ has the property (II)${}_{M, \lambda}$. 
Let $n \geq 2$. We set $L := \YYM_{n} \lotimes_{\Pa} M$. 
Then, we have the following equality in $\ResEnd_{\Pa}(\PPi_{1} \lotimes_{\Pa} \YYM_{1}\lotimes_{\Pa} L)$
\[
 ({\xi}_{2, \YYM_{1}\lotimes L})({}_{\YYM_{1}}\eeta_{2, L}^{*})({}_{\YYM_{1}}\xi_{2, L})(\eeta^{*}_{2, \YYM_{1} \lotimes L })  = \lambda \id_{\PPi_{1} \lotimes_{\Pa} \YYM_{1}\lotimes_{\Pa} L}
\]
where $\xi_{2, \YYM_{1} \lotimes L}$ and $\xi_{2, L}$ are the morphisms obtained in Corollary \ref{2021022218051}. 
\end{lemma}

\begin{proof} 
First we claim 

\begin{claim}\label{202102231643}
\begin{enumerate}[(1)]
\item
We have the following equality 
in  $ 
\Hom_{\Pa^{\mre}}(\PPi_{1}\lotimes_{\Pa} L, \YYM_{1}\lotimes_{\Pa} \PPi_{1} \lotimes_{\Pa} L)$
\[
({}_{\YYM_{1}}\xi_{2, L})(\eeta^{*}_{2, \YYM_{1}\lotimes L })( {}_{\PPi_{1}} \rrho_{L})
  =\lambda \rrho_{\PPi_{1} \lotimes L}.  
\]

\item 
We have the following equality in 
$\Hom_{\Pa^{\mre}}(\PPi_{1} \lotimes_{\Pa} L, \YYM_{1}\lotimes_{\Pa} \PPi_{1} \lotimes_{\Pa} L)$
\[
( \xi_{2, \YYM_{1}\lotimes L  } )({}_{\YYM_{1}}\eeta^{*}_{2, L})( {\rrho}_{\PPi_{1}\lotimes L } ) 
=  {}_{\PPi_{1}}\rrho_{L}. 
\]

\end{enumerate} 
\end{claim}

\begin{proof}[Proof of Claim]
We note that since we are assuming (II)${}_{M, \lambda}$, it follows from 
the remark after Corollary \ref{2021022218051}, that $\xi_{2, L} \eeta^{*}_{2, L} = - (1+ \lambda) \id_{\PPi_{1}\lotimes L} $ and 
$\xi_{2, \PPi_{1} \lotimes L} \eeta^{*}_{2,\PPi_{1} \lotimes  L} = - (1+ \lambda) \id_{\PPi_{1}\lotimes \PPi_{1} \lotimes L} $.

(1) We deduce the desired equality in the following way
\[
\begin{split}
({}_{\YYM_{1}}\xi_{2, L})(\eeta^{*}_{2, \YYM_{1}\lotimes L })( {}_{\PPi_{1}} \rrho_{L}) & 
= ({}_{\YYM_{1}}\xi_{2, L})(\eeta^{*}_{2, \YYM_{1}\lotimes L })( \ppi_{1, \YYM_{1} \lotimes L})(\eeta^{*}_{2, L} ) \\
&= ({}_{\YYM_{1}}\xi_{2, L})( {}_{\YYM_{1}} \rrho_{ \YYM_{1} \lotimes  L} - \rrho_{ \YYM_{1} \lotimes  \YYM_{1}\lotimes L} )(\eeta^{*}_{2, L} ) \\
&= ({}_{\YYM_{1}}\xi_{2, L})( {}_{\YYM_{1}} \rrho_{ \YYM_{1} \lotimes  L} )(\eeta^{*}_{2, L} )- ({}_{\YYM_{1}}\xi_{2, L})(  \rrho_{ \YYM_{1} \lotimes  \YYM_{1}\lotimes L} )(\eeta^{*}_{2, L} ) \\
&= ({}_{\YYM_{1}}\ppi_{1, L})(\eeta^{*}_{2, L} )- 
(\rrho_{\PPi_{1} \lotimes L}) (\xi_{2, L}) (\eeta^{*}_{2, L} ) \\
& = -\rrho_{\PPi_{1} \lotimes L}  + (1 + \lambda) \rrho_{\PPi_{1} \lotimes L}\\ 
& = \lambda  \rrho_{\PPi_{1} \lotimes L}
\end{split}
\]
where 
for the first equality we use Lemma \ref{homotopic lemma 5} or Lemma \ref{202008091343ns}, 
for the second we use Lemma \ref{202001111736} or Lemma \ref{202207131538}, 
for the fourth we use Theorem \ref{2020071920551}(2) and the exchange law mentioned in  \eqref{202111281249}, 
 for the fifth we use Lemma \ref{homotopic lemma 5 right} or Lemma \ref{202207131538} and Theorem \ref{2020071920551}(1).

The above computation is summarized in the following diagram: 
\[
\begin{xymatrix}@C=80pt@R=30pt{ 
\PPi_{1} \lotimes_{\Pa} L
\ar[rr]^{{}_{\PPi_{1}} \rrho_{L}} \ar[dr]_{\eeta^{*}_{2, L}}
\ar@/_40pt/[rrdd]_{-\rrho_{\PPi_{1}\lotimes_{\Pa} L } - \rrho  (\xi_{2, L} \eeta^{*}_{2, L}) \hspace{30pt}  }
& & \PPi_{1}\lotimes_{\Pa} \YYM_{1} \lotimes_{\Pa} L \ar[d]^{\eeta^{*}_{2, \YYM_{1} \lotimes_{\Pa} L} } \\ 
&\YYM_{1} \lotimes_{ \Pa} \YYM_{1}  \lotimes_{\Pa} L 
\ar[ur]^{\ppi_{1, \YYM_{1}\lotimes_{\Pa} L } } 
\ar[r]_{{}_{\YYM_{1}}\rrho_{\YYM_{1} \lotimes_{\Pa} L } - \rrho_{\YYM_{1} \lotimes \YYM_{1} \lotimes_{\Pa} L} } 
\ar[dr]_{{}_{\YYM_{1}}\ppi_{1, L } -\rrho_{\PPi_{1} \lotimes L} \xi_{2, L }} & 
\YYM_{1} \lotimes_{\Pa} \YYM_{1} \lotimes_{\Pa} \YYM_{1} \lotimes_{\Pa} L \ar[d]^{{}_{\YYM_{1}}\xi_{2, L}} \\
& & \YYM_{1} \lotimes_{\Pa} \PPi_{1} \lotimes_{\Pa} L
}\end{xymatrix} 
\]

(2) is proved in a similar way to (1) which is  summarized in the diagram below. 
\[\begin{xymatrix}@C=80pt@R=30pt{ 
\PPi_{1} \lotimes_{\Pa} L  
\ar[rr]^{ {\rrho}_{\PPi_{1} \lotimes L } } \ar[dr]_{\eeta^{*}_{2, L}}
\ar@/_40pt/[rrdd]_{ - (\xi_{2, \YYM_{1} \lotimes L} \eeta^{*}_{2, \YYM_{1} \lotimes L}) ({}_{\PPi_{1}}\rrho_{L}  )-  \lambda {}_{\PPi_{1}}\rrho_{L} \hspace{50pt}  }
& & \YYM_{1}\lotimes_{\Pa} \PPi_{1} \lotimes_{\Pa} L  \ar[d]^{{}_{\YYM_{1}} \eeta^{*}_{2, L} } \\ 
&\YYM_{1} \lotimes_{ \Pa} \YYM_{1} \lotimes_{\Pa} L  
\ar[ur]^{- {}_{\YYM_{1}} \ppi_{1, L} } 
\ar[r]_{ - {}_{\YYM_{1} \lotimes \YYM_{1} }\rrho_{L} +{}_{\YYM_{1}} \rrho_{\YYM_{1} \lotimes_{\Pa} L }  } 
\ar[dr]_{ - ( \xi_{2,\YYM_{1} \lotimes L } )({}_{\YYM_{1} \lotimes \YYM_{1}} \rrho_{L})  -\lambda \ppi_{1, \YYM_{1} \lotimes L } }& 
\YYM_{1} \lotimes_{\Pa} \YYM_{1} \lotimes_{\Pa} \YYM_{1} \lotimes_{\Pa} L  \ar[d]^{ \xi_{2,\YYM_{1}\lotimes L  } } \\
& & \PPi_{1} \lotimes_{\Pa} \YYM_{1}\lotimes_{\Pa} L   
}\end{xymatrix} 
\]
\end{proof}

Using the claim it is immediate to check the equality 
\begin{equation}\label{2021031218361}
\left(  ({\xi}_{2, \YYM_{1}\lotimes L})({}_{\YYM_{1}}\eeta_{2, L}^{*})({}_{\YYM_{1}}\xi_{2, L})(\eeta^{*}_{2, \YYM_{1} \lotimes L }) - 
\lambda\id_{\PPi_{1} \lotimes \YYM_{1} \lotimes L}\right) 
({}_{\PPi_{1} }\rrho_{L})= 0
\end{equation}
in $\Hom_{\Pa}(\PPi_{1} \lotimes_{\Pa} L , \PPi_{1} \lotimes_{\Pa} \YYM_{1} \lotimes_{\Pa} L)$.  
It follows  that 
the morphism 
\[
 ({\xi}_{2, \YYM_{1}\lotimes L})({}_{\YYM_{1}}\eeta_{2, L}^{*})({}_{\YYM_{1}}\xi_{2, L})(\eeta^{*}_{2, \YYM_{1} \lotimes L }) - \lambda\id_{\PPi_{1} \lotimes \YYM_{1} \lotimes L} 
\]
factors through the cone  morphism ${}_{\PPi_{1}} \ppi_{L}$ of ${}_{\PPi_{1}} \rrho_{L}$, which belongs to $\rad$.  
Thus we conclude that it
belongs to $\rad_{\Pa}( \PPi_{1} \lotimes_{\Pa} \YYM_{1} \lotimes_{\Pa} L, \PPi_{1} \lotimes_{\Pa} \YYM_{1} \lotimes_{\Pa} L) = 
\rad\End_{\Pa}(\PPi_{1} \lotimes_{\Pa} \YYM_{1} \lotimes_{\Pa} L)$ as desired. 
\end{proof}

\subsubsection{} 

We prove that the functor $\YYM_{1} \lotimes_{\Pa}- $ preserves radical morphisms. 
\begin{lemma}\label{2021031218431}
Assume  that $v \in \We_{Q}$ has the property (II)${}_{M, \lambda}$. 
Let $L ,N$  be objects of $\add\cC_{M}$ i.e., direct sums of indecomposable objects which belonging to the same connected components of AR-quiver of $\Dbmod{\Pa}$ with $M$.  
Then the following assertions hold. 

\begin{enumerate}[(1)]

\item 
If a morphism  $f: L \to N$ belongs to $\rad(L, N)$, 
then ${}_{\YYM_{1}}f $ belongs to $\rad_{\Pa}(\YYM_{1} \lotimes_{\Pa}L , \YYM_{1} \lotimes_{\Pa} N)$. 

\item 
The algebra homomorphism $\End_{\Pa}(L) \to \End_{\Pa}(\YYM_{1}\lotimes_{\Pa}L)$ associated to the functor $\YYM_{1} \lotimes -$ 
preserves the radicals and hence induces an algebra homomorphism 
$\ResEnd_{\Pa}(L) \to \ResEnd_{\Pa}(\YYM_{1}\lotimes_{\Pa} L)$. 
\end{enumerate}

\end{lemma}

\begin{proof}
(1) 
By Theorem \ref{semi-universal Auslander-Reiten triangle},  
 $L \xrightarrow{ \rrho_{L}} \YYM_{1} \lotimes_{\Pa} L$ is a minimal left $\rad$-approximation. 
Thus, a radical  morphism  $f: L \to N$ is factored as $L \xrightarrow{ \rrho_{L}} \YYM_{1} \lotimes_{\Pa} L \xrightarrow{f'} N$ for some $f'$. 
It follows from Corollary \ref{202105251650} that the morphism ${}_{\YYM_{1}} \rrho_{L}$ belongs to $\rad$. 
Therefore, we conclude that ${}_{\YYM_{1}}f =({}_{\YYM_{1}}f' ) ( {}_{\YYM_{1}}\rrho_{L} )$ is  a radical morphism as desired. 

(2) follows from  (1).
\end{proof}

\subsubsection{Proof of Proposition  \ref{2021040913541} }\label{proof of Proposition}

We proceed by induction on $n$.
First assume $n =2$. 
Note that since $V_{1}(t) =1$, the condition \eqref{2021030112531} is always satisfied. 
Since $\zzeta_{1, M} = \id_{\YYM_{1} \lotimes M }$ and $-\frac{V_{2}(\lambda)}{V_{1}(\lambda)} = -(1 +\lambda)$, 
it follows from Theorem  \ref{2020071920551} that  
$\oomega_{1, M } := \id_{\YYM_{1} \lotimes M}$ satisfies the desired properties.

Next assume $n \geq 3$. 
By induction hypothesis and 
setting   
\[
\xi_{n -1, M} := 
(\xi_{2, \YYM_{n -3} \lotimes M })({}_{ \YYM_{1} } \omega_{n -2, M })
: \YYM_{1} \lotimes_{\Pa} \YYM_{n -2} \lotimes M  \to
 \YYM_{1} \lotimes_{\Pa}\YYM_{1} \lotimes_{\Pa} \YYM_{n -3} \lotimes M  
\to \PPi_{1} \lotimes_{\Pa} \YYM_{ n-3} \lotimes M. 
\]
we obtain the following  equation in $\ResEnd_{\Pa}(\PPi_{1} \lotimes_{\Pa} \YYM_{n -3} \lotimes M )$. 
\[
\xi_{ n-1, M } \eeta^{*}_{ n-1, M } =- \frac{V_{ n-1}(\lambda)}{V_{n-2}(\lambda)}. 
\]
By assumption it is a non-zero scalar. Therefore the endomorphism $\xi_{ n-1, M} \eeta^{*}_{ n-1, M} $ is an automorphism of $\PPi_{1} \lotimes_{\Pa} \YYM_{ n-3} \lotimes_{\Pa} M$. 
We set $\epsilon_{n -1, M}:= (\xi_{ n-1, M} \eeta^{*}_{ n-1, M})^{-1}$. 
Note that we have 
the following  equation in $\ResEnd_{\Pa}(\PPi_{1} \lotimes_{\Pa} \YYM_{n -3}\lotimes_{\Pa} M )$. 
\begin{equation}\label{2021031219001}
\epsilon_{n -1, M} =- \frac{V_{ n-2}(\lambda)}{V_{n-1}(\lambda)}. 
\end{equation}
We define $\omega_{n -1, M}: \YYM_{n -1}\lotimes_{\Pa} M \to \YYM_{1} \lotimes_{\Pa} \YYM_{n -2}\lotimes_{\Pa} M $ to be the section of $\zzeta_{ n-1, M}$ 
corresponding to  the splitting  
\[\eeta^{*}_{n -1, M } \epsilon_{n -1, M } : \PPi_{1} \lotimes_{\Pa} \YYM_{n -3}\lotimes_{\Pa} M \rightleftarrows \YYM_{1} \lotimes_{\Pa}\YYM_{n -2}\lotimes_{\Pa} M : \xi_{ n-1, M }.\] 
Thus, statement (1) in Proposition \ref{2021040913541} automatically holds.

Applying Corollary \ref{2021022218051} and the remark after it, we have a morphism \[
\xi_{2, \YYM_{n -2}\lotimes M}: \YYM_{1} \lotimes_{\Pa} \YYM_{1} \lotimes_{\Pa} \YYM_{ n-2} \lotimes_{\Pa} M 
\to \PPi_{1} \lotimes_{\Pa} \YYM_{ n-2} \lotimes_{\Pa} M 
\] 
such that 
$\xi_{2, \YYM_{n -2} \lotimes M} \eeta_{2, \YYM_{n -2} \lotimes M}= - (1+\lambda) \id_{\PPi_{1} \lotimes_{\Pa} \YYM_{ n-2} \lotimes_{\Pa} M }$ 
and 
that 
the diagram 
below is commutative 
\begin{equation}\label{202104091842}
\begin{xymatrix}@C=60pt{
\YYM_{1} \lotimes_{\Pa} \YYM_{1} \lotimes_{\Pa} \YYM_{1} \lotimes_{\Pa} \YYM_{n -3}\lotimes_{\Pa} M 
 \ar[r]^{{}_{\YYM_{1} \lotimes \YYM_{1}} \zzeta_{n -2, M}} \ar[d]_{\xi_{2, \YYM_{1} \lotimes \YYM_{n -3} \lotimes M} } & 
\YYM_{1} \lotimes_{\Pa} \YYM_{1} \lotimes_{\Pa} \YYM_{n -2 } \lotimes_{\Pa} M\ar[d]^{\xi_{2, \YYM_{n -2}\lotimes M }} \\
\PPi_{1} \lotimes_{\Pa} \YYM_{1} \lotimes_{\Pa} \YYM_{n -3} \lotimes_{\Pa} M \ar[r]_{{}_{\PPi_{1}} \zzeta_{n -2, M}} & 
\PPi_{1} \lotimes_{ \Pa} \YYM_{n -2}.\lotimes_{\Pa} M.
}\end{xymatrix} 
\end{equation}

We consider  the following  diagram. 
\[
\begin{xymatrix}@C=30pt{ 
& \PPi_{1} \lotimes_{\Pa} \YYM_{ n-2 } \lotimes_{\Pa} M 
\ar@<5pt>[d]^{ \eeta^{*}_{2, \YYM_{n -2} \lotimes M } } 
\ar@/^10pt/[dr]^{\eeta^{*}_{n, M } }  & \\
\YYM_{1} \lotimes_{\Pa} \PPi_{1} \lotimes_{\Pa} \YYM_{ n-3} \lotimes_{\Pa} M 
\ar@<5pt>[r]^{{}_{\YYM_{1}} \eeta^{*}_{n-1, M } }
\ar@(l,u)^{ {}_{\YYM_{1}} \epsilon_{n-1, M} }
& 
\YYM_{1} \lotimes_{\Pa} \YYM_{1} \lotimes_{\Pa} \YYM_{n -2} \lotimes_{\Pa} M 
\ar@<5pt>[r]^{{}_{\YYM_{1}}\zzeta_{n -1, M} }
 \ar@<5pt>[l]^{{}_{\YYM_{1} } \xi_{n -1, M}}
\ar@<5pt>[u]^{\xi_{2, \YYM_{n-2} \lotimes M } }
&
\YYM_{1} \lotimes_{\Pa} \YYM_{n-1} \lotimes_{\Pa} M  \ar@<5pt>[l]^{{}_{\YYM_{1}}\omega_{ n-1, M} }
}\end{xymatrix} 
\]

We have 
\begin{equation}\label{2021030114411}
\begin{split}
\xi_{n, M} \eeta^{*}_{n, M } 
& = 
(\xi_{2, \YYM_{n -2} \lotimes M } ) ({}_{\YYM_{1}} \omega_{n -1, M})({}_{\YYM_{1}} \zzeta_{n -1, M }) ( \eeta^{*}_{\YYM_{n -2, M} }) \\
& = (\xi_{2, \YYM_{n -2} \lotimes M } ) \left( \id_{\YYM_{1} \lotimes \YYM_{1} \lotimes \YYM_{n -2}\lotimes M } - ({}_{\YYM_{1}}\eeta^{*}_{n -1, M })({}_{\YYM_{1}} \epsilon_{ n-1, M })
({}_{\YYM_{1}}  \xi_{ n-1, M}) \right)( \eeta^{*}_{\YYM_{n -2}\lotimes M } ) \\
& = (\xi_{2, \YYM_{n -2} \lotimes M } ) ( \eeta^{*}_{\YYM_{n -2}\lotimes M })
- (\xi_{2, \YYM_{n -2} \lotimes M } )  ({}_{\YYM_{1}}\eeta^{*}_{n -1, M })({}_{\YYM_{1}} \epsilon_{ n-1, M })({}_{\YYM_{1}}  \xi_{ n-1, M})
( \eeta^{*}_{\YYM_{n -2} \lotimes M }) \\
& = -(1+ \lambda)\id - (\xi_{2, \YYM_{n -2} \lotimes M } )  ({}_{\YYM_{1}}\eeta^{*}_{n -1, M })( {}_{\YYM_{1}}\epsilon_{ n-1, M })({}_{\YYM_{1}}  \xi_{ n-1, M})
( \eeta^{*}_{\YYM_{n -2} \lotimes M })
\end{split}
\end{equation}
where for the first equality we unwind the definitions of $\xi_{n, M}$ and $\eeta^{*}_{n, M}$, 
for the second we use the splitting identity $\omega_{n -1, M}\zzeta_{n -1, M} = \id_{\YYM_{1} \lotimes \YYM_{n -2}} - \eeta^{*}_{n -1, M}\epsilon_{ n-1, M} \xi_{ n-1, M }$.

To proceed we need to compute the value of the final term in \eqref{2021030114411} when passing to
$\ResEnd_{\Pa^{\mre}}(\PPi_{1} \otimes_{\Pa} \YYM_{n -2})$. 

\begin{claim}\label{2021022718581}
We have the following equality in $\ResEnd_{\Pa}(\PPi_{1} \lotimes_{\Pa} \YYM_{ n-2}\lotimes_{\Pa} M)$. 
\[
(\xi_{2, \YYM_{n -2} \lotimes M }) ( {}_{\YYM_{1}} \eeta^{*}_{n-1, M})( {}_{\YYM_{1}} \epsilon_{n -1, M })(  {}_{\YYM_{1}} \xi_{n-1, M})(\eeta^{*}_{\YYM_{n -2} \lotimes M}) = 
 - \lambda\frac{ V_{n-2}(\lambda) }{V_{n-1}(\lambda) }\id. 
\]
\end{claim}

\begin{proof}
First we compute the composition of the first two morphisms
\[
\begin{split}
(\xi_{2, \YYM_{n -2} \lotimes M}) ( {}_{\YYM_{1}} \eeta^{*}_{n-1, M}) 
&= 
(\xi_{2, \YYM_{n -2} \lotimes M }) ( {}_{\YYM_{1} \lotimes_{\Pa} \YYM_{1} }\zzeta_{n -2, M} )({}_{\YYM_{1}} \eeta^{*}_{2, \YYM_{n -3} \lotimes M })
 \\
&= 
 ( {}_{\PPi_{1} }\zzeta_{n -2, M} )
 (\xi_{2, \YYM_{1} \lotimes_{\Pa} \YYM_{n-3}\lotimes_{\Pa} M })
 ({}_{\YYM_{1}} \eeta^{*}_{2, \YYM_{n -3} \lotimes M})
\end{split} 
\]
where for the first equality we unwind the definition of ${}_{\YYM_{1}} \eeta^{*}_{n-1, M}$ 
for the second we use the commutativity \eqref{202104091842}.

Next we compute the composition of the last two morphisms 
\[
\begin{split}
(  {}_{\YYM_{1}} \xi_{n-1, M})(\eeta^{*}_{\YYM_{n -2} \lotimes M }) 
&= (  {}_{\YYM_{1}} \xi_{2, \YYM_{n -3} \lotimes M })({}_{\YYM_{1}\lotimes \YYM_{1} } \omega_{ n-2, M} )(\eeta^{*}_{\YYM_{n -2}  \lotimes M }) \\
&= (  {}_{\YYM_{1}} \xi_{2, \YYM_{n -3}   \lotimes M  }) (\eeta^{*}_{\YYM_{1} \lotimes_{\Pa} \YYM_{n -3} \lotimes M   } ) ({}_{\PPi_{1} } \omega_{ n-2, M} )\\
\end{split} 
\]
where for the first equality we unwind the definition of $\xi_{n-1, M}$
for the second we use the exchange law 
\[
\begin{xymatrix}@C=60pt{
\PPi_{1} \lotimes_{\Pa} \YYM_{n -2}  \lotimes_{\Pa} M 
\ar[r]^{\eeta^{*}_{2, \YYM_{n -2}  \lotimes M  } } \ar[d]_{ {}_{\PPi_{1}} \omega_{n -2, M} } & 
\YYM_{1} \lotimes_{\Pa} \YYM_{1} \lotimes_{\Pa} \YYM_{n -2} \lotimes_{\Pa} M  \ar[d]^{{}_{\YYM_{1}\lotimes \YYM_{1} } \omega_{n -2, M}} \\
\PPi_{1} \lotimes \YYM_{1} \lotimes \YYM_{ n-3}  \lotimes_{\Pa}  M \ar[r]_{\eeta^{*}_{2, \YYM_{1} \lotimes \YYM_{n -3}  \lotimes M } } & 
\YYM_{1} \lotimes_{\Pa} \YYM_{1} \lotimes_{\Pa} \YYM_{1} \lotimes_{\Pa} \YYM_{n -3}  \lotimes_{\Pa} M . 
}\end{xymatrix}
\]

Combining these results, we obtain 
\begin{equation}\label{2021022719311}
\begin{split} 
(\xi_{2, \YYM_{n -2} \lotimes M }) 
( {}_{\YYM_{1}} \eeta^{*}_{n-1, M})( {}_{\YYM_{1}} \epsilon_{n -1, M })(  {}_{\YYM_{1}} \xi_{n-1, M})(\eeta^{*}_{\YYM_{n -2} \lotimes M})= 
 ( {}_{\PPi_{1} }\zzeta_{n -2, M} )\Upsilon  ({}_{\PPi_{1} } \omega_{ n-2, M} )
\end{split}
\end{equation} 
where we set 
\[
\Upsilon :=  \left( (\xi_{2, \YYM_{1} \lotimes_{\Pa} \YYM_{n-3}\lotimes_{\Pa} M })
 ({}_{\YYM_{1}} \eeta^{*}_{2, \YYM_{n -3} \lotimes M})\right)
 ( {}_{\YYM_{1}} \epsilon_{n -1, M })
\left(   (  {}_{\YYM_{1}} \xi_{2, \YYM_{n -3}   \lotimes M  }) (\eeta^{*}_{\YYM_{1} \lotimes_{\Pa} \YYM_{n -3} \lotimes M   } ) \right). 
\]

By  Lemma \ref{2021031218431}  and \eqref{2021031219001}, we have 
\[
 {}_{\YYM_{1}} \epsilon_{n -1} =- \frac{V_{n -2}(\lambda)}{V_{n -1}(\lambda)}\id \] 
 in $\ResEnd_{\Pa}(\YYM_{1} \lotimes_{\Pa} \PPi_{1} \lotimes_{\Pa} \YYM_{n -3})$. 
Applying Lemma \ref{2021022718051} 
to $L = \YYM_{n -3} \lotimes_{\Pa} M$,  
we have the equality 
\[
\left( (\xi_{2, \YYM_{1} \lotimes_{\Pa} \YYM_{n-3}\lotimes_{\Pa} M })
 ({}_{\YYM_{1}} \eeta^{*}_{2, \YYM_{n -3} \lotimes M})\right)
\left(   (  {}_{\YYM_{1}} \xi_{2, \YYM_{n -3}   \lotimes M  }) (\eeta^{*}_{\YYM_{1} \lotimes_{\Pa} \YYM_{n -3} \lotimes M   } ) \right)   = \lambda\id
\]
 in $\ResEnd_{\Pa}(\PPi_{1} \lotimes_{\Pa} \YYM_{1} \lotimes_{\Pa} \YYM_{n -3}\lotimes M)$. 
Thus by  Lemma \ref{2021022716261} 
we obtain the equality 
\[
\Upsilon 
=- \lambda  \frac{V_{n -2}(\lambda)}{V_{n -1}(\lambda)}\id
\]
in $\ResEnd_{\Pa}(\PPi_{1} \lotimes_{\Pa} \YYM_{1} \lotimes_{\Pa} \YYM_{n -3}\lotimes_{\Pa} M)$. 

Since $({}_{\PPi_{1}} \zzeta_{ n-2, M})({}_{\PPi_{1}}\omega_{n-2, M}) = \id_{\PPi_{1} \lotimes \YYM_{n -3}}$, 
applying Lemma \ref{2021022716261} to \eqref{2021022719311} 
we come to the desired equality in $\ResEnd_{\Pa}(\PPi_{1} \lotimes_{\Pa} \YYM_{n -2})$
\[
\begin{split} 
&(\xi_{2, \YYM_{n -2} \lotimes M }) 
( {}_{\YYM_{1}} \eeta^{*}_{n-1, M})( {}_{\YYM_{1}} \epsilon_{n -1, M })(  {}_{\YYM_{1}} \xi_{n-1, M})(\eeta^{*}_{\YYM_{n -2} \lotimes M}) \\
 &= 
 ( {}_{\PPi_{1} }\zzeta_{n -2, M} )\Upsilon ({}_{\PPi_{1} } \omega_{ n-2, M})\\
& =- \lambda  \frac{V_{n -2}(\lambda)}{V_{n -1}(\lambda)}\id. 
\end{split}
\]
\end{proof}

We continue the equation \eqref{2021030114411}
by using Claim  \ref{2021022718581},  and  deduce the desired equality in $\ResEnd_{\Pa}(\PPi_{1} \lotimes_{\Pa} \YYM_{n -2} )$ as follows :
\[
\begin{split}
\xi_{n, M } \eeta^{*}_{n, M} 
&  =-(1+\lambda)\id -(\xi_{2, \YYM_{n -2} \lotimes M }) 
( {}_{\YYM_{1}} \eeta^{*}_{n-1, M})( {}_{\YYM_{1}} \epsilon_{n -1, M })(  {}_{\YYM_{1}} \xi_{n-1, M})(\eeta^{*}_{\YYM_{n -2} \lotimes M}) \\
&=-( 1+ \lambda)\id + \lambda  \frac{ V_{n -2}(\lambda) }{V_{n -1}(\lambda)}\id \\
&=-\frac{ (1 + \lambda) V_{n -1}(\lambda) - \lambda V_{n -2}(\lambda) }{ V_{n -1}(\lambda) }\id  \\
&=-\frac{ V_{n}(\lambda)}{ V_{n -1}(\lambda) }\id. 
\end{split} 
\]
This proves that (2) of Proposition \ref{2021040913541} holds.

Since $\xi_{n, M} = (\xi_{2, \YYM_{n -2 } \lotimes M})({}_{\YYM_{1}} \omega_{n-1, M})$ by definition, 
the equation of (3)  is deduced from the following commutative diagram 
\[
\begin{xymatrix}@C=80pt{
\YYM_{n -1}\lotimes_{\Pa} M  \ar[r]^{\rrho_{\YYM_{n-1}\lotimes M} } \ar[d]_{\omega_{n -1, M}} &
 \YYM_{1} \lotimes_{\Pa} \YYM_{n -1}\lotimes_{\Pa} M \ar[d]^{{}_{\YYM_{1}} \omega_{n -1, M} } \\
\YYM_{1} \lotimes_{\Pa} \YYM_{n -2} \lotimes_{\Pa} M \ar[r]^{\rrho_{\YYM_{1} \lotimes \YYM_{n-2} \lotimes M} } \ar[dr]_{\ppi_{1, \YYM_{n-2} \lotimes M}} &
\YYM_{1} \lotimes_{\Pa} \YYM_{1} \lotimes_{\Pa} \YYM_{n-2} \lotimes_{\Pa} M \ar[d]^{\xi_{2, \YYM_{n-2} \lotimes M}} \\
& \PPi_{1} \lotimes_{\Pa} \YYM_{n-2} \lotimes_{\Pa} M
}\end{xymatrix} 
\]
where the commutativity of the square is deduced from the exchange law \eqref{202111281249}.

\qed

\subsection{Proof of Theorem \ref{202104091259} }

We proceed a proof of Theorem \ref{202104091259}.

We use the induction on $n\geq 2$. 
The case $n =2$ follows from Theorem \ref{202103021651}. 
As was pointed out after  \eqref{2021040918451}, 
if we set $\varpi_{2, M} := (\ppi_{1, \YYM_{1} \lotimes M})(\omega_{2, M})$, 
then it is a cone morphism of $\brho_{M}^{2}: M \to \YYM_{2}\lotimes_{\Pa} M$.

We deal with the case $n \geq 3$. 
For $m = 2, \ldots, n$, we inductively  construct the morphisms 
\[
 \xi_{m, M} =(\xi_{2, \YYM_{m -2} \lotimes M } )({}_{\YYM_{1}}\omega_{m -1, M} ), \  \epsilon_{m, M}:= (\xi_{m, M} \eeta^{*}_{m, M})^{-1}, \omega_{m, M}
 \]  as in the proof of Proposition  \ref{2021040913541}. 
Assume that we have proved that the morphism $\brho^{n-1}_{M}: M \to \YYM_{n -1} \lotimes_{\Pa} M$ is a minimal left $\rad^{n-1}$-approximation 
and the morphism $$\varpi_{n-1, M} :=( -1)^{n -1}(\ppi_{1, \YYM_{n-2} \lotimes M })(\omega_{n-1, M})$$ is a cone morphisms of $\brho_{M}^{n -1}$.

Then by Proposition \ref{2021040913541} 
we have the following diagram whose row is a direct sum of AR-triangles 
and column gives a splitting of the middle term $\YYM_{1}\lotimes_{\Pa} \YYM_{n-1}\lotimes_{\Pa} M$.
\[
\begin{xymatrix}@C=60pt{
& \PPi_{1} \lotimes_{\Pa} \YYM_{n -2} \lotimes_{\Pa} M \ar@<-5pt>[d]_{ ( \eeta^{*}_{n, M})(\epsilon_{n, M}) } &  &\\ 
\YYM_{n-1} \lotimes_{\Pa} M  \ar[r]^{\rrho_{\YYM_{n-1} \lotimes M } } & 
\YYM_{1} \lotimes_{\Pa} \YYM_{n-1} \lotimes_{\Pa} M  
\ar[r]^{ \ppi_{1, \YYM_{n-1} \lotimes M } } \ar@<-5pt>[u]_{\xi_{n, M}} 
\ar@<-5pt>[d]_{\zzeta_{n, M} } & \PPi_{1} \lotimes_{\Pa} \YYM_{n-1}\lotimes_{\Pa} M \\ 
 &  \YYM_{n}\lotimes_{\Pa} M   \ar@<-5pt>[u]_{\omega_{n, M}} &&
}\end{xymatrix}
\]

Recall that 
 $( \ppi_{1, \YYM_{n-1} \lotimes M }) (\eeta^{*}_{n, M})(\epsilon_{n, M}) = {}_{\PPi_{1}}\brho_{ n-1, M} \epsilon_{n, M}$ 
and $\zzeta_{n, M} \rrho_{\YYM_{n-1} \lotimes M} = \brho_{n, M}$. 
By Proposition \ref{2021040913541}(3), we have $
\xi_{n, M} \rrho_{\YYM_{ n-1} \otimes  M } =(\ppi_{1, \YYM_{n-2} \lotimes M })(\omega_{n-1, M}) = ( -1)^{n -1} \varpi_{n -1, M}$.
Thus, 
setting $\varpi_{n, M} := ( -1)^{n} (\ppi_{1, \YYM_{n-1} \lotimes M })(\omega_{n, M})$, we obtain 
the following commutative  diagram  
\begin{equation}\label{202105261807}
\begin{xymatrix}@C=60pt{
& \PPi_{1} \lotimes_{\Pa} \YYM_{n -2} \lotimes_{\Pa} M \ar@<5pt>[d]^{ ( \eeta^{*}_{n, M})(\epsilon_{n, M}) } \ar@/^5pt/[dr]^{({}_{\PPi_{1}} \rrho_{\YYM_{n -1} \lotimes M})(\epsilon_{n, M})} &  &\\ 
\YYM_{n-1} \lotimes_{\Pa} M  \ar[r]^{\rrho_{\YYM_{n-1} \lotimes M } } 
\ar@/^5pt/[ur]^{(-1)^{n -1} \varpi_{n-1, M} }  \ar@/_5pt/[dr]_{\brho_{n, M}} 
& 
\YYM_{1} \lotimes_{\Pa} \YYM_{n-1} \lotimes_{\Pa} M  
\ar[r]^{ \ppi_{1, \YYM_{n-1} \lotimes M } } \ar@<5pt>[u]^{\xi_{n, M}} 
\ar@<-5pt>[d]_{\zzeta_{n, M} } & \PPi_{1} \lotimes_{\Pa} \YYM_{n-1}\lotimes_{\Pa} M \\ 
 &  \YYM_{n}\lotimes_{\Pa} M   \ar@<-5pt>[u]_{\omega_{n, M}} \ar@/_5pt/[ur]_{( -1)^{n}\varpi_{n, M} } &&
}\end{xymatrix}
\end{equation}
Hence we obtain the following homotopy Cartesian square whose totalization is a direct sum of AR-triangles
\begin{equation}\label{202105261808}
\begin{xymatrix}@C=60pt{
\YYM_{n-1} \lotimes_{\Pa} M  \ar[d]_{ \varpi_{n -1, M}  } \ar[r]^{\brho_{n, M }} 
& \YYM_{n}\lotimes_{\Pa} M  \ar[d]^{\varpi_{n, M}} \\
\PPi_{1} \lotimes_{\Pa} \YYM_{n -2} \lotimes_{\Pa} M 
\ar[r]_{ ( {}_{\PPi_{1}} \brho_{n -1, M} )(\epsilon_{n ,M })  } & \PPi_{1} \lotimes_{\Pa} \YYM_{n-1} \lotimes_{\Pa} M.  
}\end{xymatrix} 
\end{equation}

Recall that $\brho_{M}^{n} = \brho_{n, M} \brho_{M}^{n -1}$ 
and $\brho_{M}^{n-1}$ is a minimal left $\rad^{n -1}$-approximation by the induction hypothesis. 
Thanks to  Proposition \ref{2021050917551} we come to  the desired conclusion that  
the morphism $\brho^{n}_{M}: M \to \YYM_{n } \lotimes_{ \Pa} M$ is a minimal left $\rad^{n}$-approximation of $M$ in $\Dbmod{\Pa}$ 
and  we have an exact triangle 
\[
M \xrightarrow{ \brho^{n}_{M} } \YYM_{n} \lotimes_{\Pa} M   \xrightarrow{ \varpi_{n, M}} \PPi_{1} \lotimes_{\Pa} \YYM_{n -1} \lotimes_{\Pa} M  \to. 
\]
\qed

We note we have a left ladder of the following form in which we omit  indexes for simplicity. 
\begin{equation}\label{202105261824}
{\tiny
\begin{xymatrix}@C=28pt{
M \ar[r]^{\brho \hspace{10pt} } & \YYM_{1} \lotimes_{\Pa} M \ar[r]^{\brho} \ar[d]_{\ppi} & 
\YYM_{2} \lotimes_{\Pa} M \ar[r]^{\brho } \ar[d]_{\varpi} &  \hspace{16pt} \cdots \hspace{-16pt}  &\ar[r] &
\YYM_{n-1} \lotimes_{\Pa} M  \ar[d]_{\varpi  } \ar[r]^{\brho} 
& \YYM_{n}\lotimes_{\Pa} M  \ar[d]^{ \varpi} \\
& \PPi_{1} \lotimes_{\Pa} M \ar[r]_-{{}_{\PPi_{1}} \!\brho \epsilon } & \PPi_{1} \lotimes_{\Pa} \YYM_{1} \lotimes_{\Pa} M \ar[r]_-{{}_{\PPi_{1}} \!\brho \epsilon }  &
\hspace{16pt} \cdots \hspace{-16pt} &\ar[r] 
& 
\PPi_{1} \lotimes_{\Pa} \YYM_{n -2} \lotimes_{\Pa} M 
\ar[r]_{  {}_{\PPi_{1}} \!\brho \epsilon } & \PPi_{1} \lotimes_{\Pa} \YYM_{n-1} \lotimes_{\Pa} M.  
}\end{xymatrix}} 
\end{equation}

\section{A deformation argument}\label{section: deformation}

The aim of this section is to prove Theorem \ref{generic theorem}, 
which says that 
for  fixed $M \in \ind\Dbmod{\Pa}$ and $n \in N_{Q}$, 
the locus of $v\in \kk Q_{0}$ where  the multiplication ${}^{v}\!\brho_{M}^{n}$ is a minimal left $\rad^{n}$-approximation, 
contains a Zariski open set. Thus from now on we include weights $v \in \kk Q_{0}$ in our notation again.

\subsection{Statement of the main theorem}

Let $M \in \ind \Dbmod{\Pa}$ and $n \in N_{Q}$.  
Then we define  a subset $\cI_{M, n}$ to be the subset of $\We_{Q}$ that consists of points $v$ 
that has the property (I)${}_{M, n}$. 
Note that the subspace $\cI_{M, n}$ is defined by  finite number of  linear equations inside $\kk Q_{0}$ and is  a Zariski open subset of $\kk Q_{0}$.

We point out the following two properties of $\cI_{M, n}$. 

\begin{lemma}\label{202103151408} 

\begin{enumerate}[(1)]
\item If  the weight $v$ belongs to $\cI_{M, n}$, then the morphism ${}^{v}\!\brho_{M}^{n}: M \to {}^{v}\!\YYM_{n} \lotimes_{\Pa} M$ belongs to $\rad^{n}(M, {}^{v}\!\YYM_{n }\lotimes_{\Pa} M)$. 
\item The isomorphism class of ${}^{v}\!\YYM_{n} \lotimes_{\Pa} M$  is independent of  $v \in \cI_{n ,M}$.  
\end{enumerate}
\end{lemma}

\begin{proof}
(1) By Theorem \ref{semi-universal Auslander-Reiten triangle}, 
$\brho_{m, M}: {}^{v}\!\YYM_{ m-1} \lotimes_{\Pa} M \to {}^{v}\!\YYM_{m} \lotimes_{\Pa}M$ belongs to $\rad$ for $m =1,2, \ldots, n$.
Therefore, the composition $\brho_{M}^{n}=\brho_{n, M}\brho_{n-1, M} \cdots \brho_{1, M}$ belongs to $\rad^{n}$. 

(2) By  Theorem \ref{right approximation theorem}, 
if  the weight  $v$ belongs to $\cI_{M, n}$, then ${}^{v}\!\YYM_{n}\lotimes_{\Pa} M$ gives a minimal right $\rad^{n}$-approximation of $\PPi_{n} \lotimes_{\Pa} M$.  
Thus, if $u,v \in \cI_{n ,M}$, then ${}^{u} \!\YYM_{n } \lotimes_{\Pa} M \cong {}^{v}\!\YYM_{n } \lotimes_{\Pa} M$. 
\end{proof}

\begin{theorem}\label{generic theorem}
Let $M \in \ind \Dbmod{\Pa}$ and $n \in N_{Q}$. 
We set 
\[
\cL_{M, n} := \{ v \in \cI_{M, n} \mid {}^{v}\!\brho_{M}^{n} \textup{ is a minimal left $\rad^{n}$-approximation of }M \}.  
\] 
Then  $\cL_{M, n}$ is  a (possibly empty) Zariski open set of $\kk Q_{0}$. 
\end{theorem}

\begin{remark}
The above theorem says that the set $\cL_{M, n}$ of  weights $v \in \We_{Q}$ that satisfy the two conditions
(1)  the morphism ${}^{v}\!\brho_{M}^{n}$ is a minimal left $\rad^{n}$-approximation and 
(2) $v$ has the property (I)${}_{M, n}$, 
is open subset of $\kk Q_{0}$. 
We do not know if the first condition solely  provides an open subset. 
\end{remark}

\subsection{A criterion for locally triviality}

We recall a well-known criteria of locally freeness of modules over a commutative ring. 
For a commutative algebra $R$, we denote by $\maxSpec R$ the set of maximal ideals of $R$. 

\begin{lemma}[{e.g., \cite[II. Exercise 5.8]{Hartshorne}}]\label{202103141428}
Let $R$ be a commutative integral Noetherian algebra of finite type.
Then  an $R$-module $M$ is locally free
if and only if the function $\maxSpec R \to \NN, \frkm \mapsto \dim_{\kappa(\frkm)}M \otimes_{R} \kappa(\frkm)$
is constant.
\end{lemma}

We establish a similar statement for a module $\cM$ over $RQ= \kk Q \otimes_{\kk} R$. 
Namely if all the fibers $\cM \otimes_{R} \kappa(\frkm)$ at $\frkm \in \maxSpec R$ are isomorphic to each other as $\kappa(\frkm) Q$-modules, 
then $\cM$ is locally trivial. 

\begin{lemma}\label{202103141418}
Let $R$ be a commutative integral Noetherian algebra of finite type 
and $\cM$ a finitely generated $RQ$-module. 
Assume that there exists a finitely generated $\kk Q$-module $M$ such that 
$\cM \otimes_{R} \kappa(\frkm) \cong M \otimes \kappa(\frkm)$ as $\kappa(\frkm) Q$-modules for all $\frkm \in \maxSpec R$. 
 Then, for each $\frkm \in \maxSpec R$ there exists $f \in R \setminus \frkm$ such that 
 $\cM \otimes_{R} R_{f} \cong M \otimes R_{f}$ as $R_{f}Q$-modules.
\end{lemma}

\begin{proof}

\begin{claim}\label{202103141429}
For any $\frkm \in \maxSpec R$, the canonical morphism below is an isomorphism.
\begin{equation}\label{202103141430}
\Hom_{RQ}(M \otimes R, \cM) \otimes_{R} \kappa(\frkm) \to \Hom_{\kappa(\frkm) Q}(M \otimes \kappa(\frkm), \cM \otimes_{R} \kappa(\frkm)). 
\end{equation}
\end{claim}

\begin{proof}[Proof of Claim \ref{202103141429}]
Let $\mathsf{Res}: 0 \to P_{1} \xrightarrow{ f} P_{0} \to M \to 0$ be a projective resolution of $M$ over $\kk Q$. 
Applying $- \otimes R$, we obtain a projective resolution 
$\mathsf{Res} \otimes R: 0 \to P_{1} \otimes R \xrightarrow{ f \otimes R} P_{0} \otimes R \to M \otimes R \to 0$ 
of the $RQ$-module $M \otimes R$. 
Taking the long cohomology sequence of $\Hom_{R Q}(\mathsf{Res} \otimes R, \cM)$, we obtain 
the exact sequence 
\begin{equation}\label{202107021835}
0 \to \Hom_{RQ} ( M \otimes R, \cM ) \to\Hom_{RQ}( P_{0} \otimes R, \cM )\xrightarrow{f^{*}} \Hom_{RQ}(P_{1} \otimes R, \cM) 
\xrightarrow{ \delta } \Ext^{1}_{RQ}(M \otimes R, \cM) \to 0.
\end{equation}
where 
we set $f^{*} := \Hom_{RQ}(f \otimes R, \cM)$ 
and $\delta$ is the connecting morphism.

Let $\frkm \in \maxSpec R$.
By right exactness of the functor $-\otimes_{R} \kappa(\frkm)$,  the upper row of the following diagram is exact
{\tiny\[
\begin{xymatrix}{ 
\Hom_{RQ}( P_{0} \otimes R, \cM )\otimes_{R} \kappa(\frkm) \ar[r]^{f^{*} \otimes\kappa(\frkm)}  \ar[d]_{\cong} & 
\Hom_{RQ}(P_{1} \otimes R, \cM)\otimes_{R} \kappa(\frkm)  \ar[r]^{\delta \otimes \kappa(\frkm)} \ar[d]^{\cong} &
\Ext^{1}_{RQ}(M \otimes R, \cM) \otimes_{R} \kappa(\frkm)  \ar[d] \ar[r] & 0 \\
\Hom_{\kappa(\frkm) Q}( P_{0}\otimes\kappa(\frkm) , \cM\otimes_{R} \kappa(\frkm)  )\ar[r]  & 
\Hom_{\kappa(\frkm) Q}(P_{1}\otimes\kappa(\frkm), \cM\otimes_{R} \kappa(\frkm) ) \ar[r] & 
\Ext^{1}_{\kappa(\frkm) Q}(M\otimes\kappa(\frkm), \cM\otimes_{R} \kappa(\frkm) ) \ar[r] & 0. 
}\end{xymatrix} 
\] }
where the bottom row is a part of the cohomology long exact sequence of $\Hom_{\kappa(\frkm) Q}(\mathsf{Res}\otimes\kappa(\frkm), \cM \otimes_{R} \kappa(\frkm))$. 
Observe that the left and the middle vertical arrows are isomorphisms, 
and hence so is the right vertical  arrow. 
By the assumption $\Ext^{1}_{\kappa(\frkm) Q}(M\otimes\kappa(\frkm), \cM\otimes_{R} \kappa(\frkm) )$ is isomorphic to $\Ext^{1}_{\kappa(\frkm) Q}(M\otimes\kappa(\frkm), M\otimes\kappa(\frkm))\cong\Ext^{1}_{\kk Q}(M, M)\otimes\kappa(\frkm)$ for all $\frkm \in\maxSpec R$. 
Hence by Lemma \ref{202103141428}, the $R$-module $\Ext^{1}_{RQ}(M \otimes R, \cM)$ is locally free. 

It also follows from Lemma \ref{202103141428} and the assumption $\cM \otimes_{R} \kappa(\frkm) \cong M\otimes\kappa(\frkm)$ that 
$\cM$ is locally free as an $R$-module. 
Since the $R$-module $\Hom_{RQ}(P_{1} \otimes R, \cM)$ is a direct summand of a direct sum of copies of $\cM$ as an $R$-module, 
it is  locally free as an $R$-module. 
It follows that   $\ker \delta$ is locally free as an $R$-module.  
Therefore applying $-\otimes_{R} \kappa(\frkm)$ to the exact sequence \eqref{202107021835}, 
we see that  the upper row of the following diagram is exact 
{\tiny\[
\begin{xymatrix}{ 
0 \to \Hom_{RQ}(M \otimes R, \cM) \otimes_{R}\kappa(\frkm) \ar[r] \ar[d] & 
\Hom_{RQ}( P_{0} \otimes R, \cM )\otimes_{R} \kappa(\frkm) \ar[r]^{f^{*} \otimes\kappa(\frkm)}  \ar[d]_{\cong} & 
\Hom_{RQ}(P_{1} \otimes R, \cM)\otimes_{R} \kappa(\frkm)   \ar[d]^{\cong} &
 \\
 0 \to \Hom_{\kappa(\frkm)Q}(M \otimes \kappa(\frkm), \cM \otimes_{R}\kappa(\frkm) )  \ar[r] &
\Hom_{\kappa(\frkm) Q}( P_{0}\otimes \kappa(\frkm) , \cM\otimes_{R} \kappa(\frkm)  )\ar[r]  & 
\Hom_{\kappa(\frkm) Q}(P_{1}\otimes \kappa(\frkm), \cM\otimes_{R} \kappa(\frkm) )  & 
}\end{xymatrix} 
\]}
where the bottom row is a part of the cohomology long exact sequence of $\Hom_{\kappa(\frkm) Q}(\mathsf{Res}\otimes \kappa(\frkm), \cM \otimes_{R} \kappa(\frkm))$. 
Since the middle and the right vertical arrows are isomorphisms, we conclude that 
the left vertical arrow which is the canonical morphism \eqref{202103141430} is an isomorphism. 
\end{proof}

It follows from Claim \ref{202103141429} and the assumption $\cM \otimes_{R} \kappa(\frkm) \cong M \otimes \kappa(\frkm) Q$ that 
there exists $\phi \in \Hom_{RQ}(M \otimes R, \cM)$
such that $\phi \otimes_{R} \kappa(\frkm)$ is an isomorphism. 
Since $M \otimes R$ and $\cM$ are locally free as $R$-modules, 
there exists $f \in R \setminus \frkm$ such that $\phi \otimes R_{f}$ is an isomorphism of $R_{f}$-modules 
and is an isomorphism of $R_{f}Q$-modules. 
\end{proof}

\subsection{Proof of Theorem \ref{generic theorem}}
We only deal with the case that $Q$ is a Dynkin quiver, so that $\We_{Q} = \kk^{\times} Q_{0}$. 
The non-Dynkin case is proved in a similar way by using ${}^{v}\!\YM$ given in   \eqref{202404181706}.

We may assume that $M $ is indecomposable  and belongs to $\Pa \mod$.

Let $T: = \kk[x_{1}^{\pm 1}, x_{2}^{\pm 1}, \ldots, x_{r}^{\pm 1}]$ be the coordinate algebra of $\kk^{\times} Q_{0}$.
We denote by ${}^{T}\!\YYM$ the dg-algebra obtained from the formula \eqref{the differentials of derived QHA} by replacing $\kk$ with $T$ and $v_{i}^{-1}$ with $x_{i}^{-1}$ for all $i \in Q_{0}$. 
We define a morphism  ${}^{T}\!\brho_{M}^{n} : M \to {}^{T}\! \YYM_{n} \otimes_{\Pa} M$ in $\sfC(T Q)$ by replacing $v_{i}$ in the definition of ${}^{v}\!\brho_{n}^{M}$ 
with $x_{i}$. 
The complex ${}^{T}\! \YYM_{n} \otimes_{\Pa} M$ computes the derived tensor product ${}^{T}\! \YYM_{n} \lotimes_{\Pa} M$. 
We denote the induced morphism  ${}^{T}\!\brho_{M}^{n} : M \to {}^{T}\! \YYM_{n} \lotimes_{\Pa} M$ in $\sfD(T Q)$ by the same symbol.

Let $R$  be the coordinate ring of $\cI_{M, n}$. 
Since $R$ is a localization of $T$ it satisfies the assumption of Lemma \ref{202103141418}. 
We may assume that $\cL_{M, n}$ is not empty. 
Let $v \in \cL_{M, n}$ and $\frkm$ the corresponding maximal ideal of $R$. 
 For simplicity we set $M' := {}^{v}\!\YYM_{n} \lotimes_{\Pa} M$. 
Thanks to Lemma \ref{202103151408}, we can apply  Lemma \ref{202103141418} and  deduce that 
there exists $f \in R\setminus \frkm$ such that 
if we set $S := R_{f}$ and ${}^{S}\!\YYM := {}^{T}\!\YYM \otimes_{T} S$, 
then there exists an isomorphism  $\phi: {}^{S}\! \YYM_{n } \lotimes_{\Pa} M \to M' \otimes S$
of $SQ$-modules. 
The isomorphism $\phi$ induces an isomorphism 
\[
\phi_{*}: \Hom_{SQ}(M\otimes S, {}^{S}\! \YYM_{n } \lotimes_{\Pa}M) \to \Hom_{S Q}(M\otimes S, M'\otimes S) \to \Hom_{\Pa}(M, M') \otimes S.\]
We set ${}^{S}\!\brho_{M}^{n}: = {}^{T}\! \brho_{M}^{n} \otimes_{T} S$. 
We claim that the image $\phi_{*}({}^{S}\! \brho^{n}_{M})$ belongs to $\rad^{n}_{\Pa}(M, M') \otimes S$. 

First observe that the morphism $\phi_{*}$ is compatible with specializations. 
Namely for $u \in \maxSpec S$, the following diagram is commutative 
\[
\begin{xymatrix}@C=60pt{ 
\Hom_{SQ}(M, {}^{S}\! \YYM_{n } \lotimes_{\Pa}M) \ar[r]^{\phi_{*}} \ar[d] &  \Hom_{\Pa}(M, M') \otimes S \ar[d] \\
\Hom_{\Pa}(M, {}^{u}\! \YYM_{n } \lotimes_{\Pa}M) \ar[r] &  \Hom_{\Pa}(M, M') \otimes \kappa(u) 
}\end{xymatrix}
\]
where  the bottom arrow is the post-composition by $\phi \otimes_{S} \kappa(u)$ 
and the vertical arrows are the specialization maps. 

On the other hand, we have observed in Lemma \ref{202103151408} that ${}^{u}\!\brho^{n}_{M}$ belongs to $\rad^{n}_{\Pa}(M, {}^{u}\!\YYM_{n} \lotimes_{\Pa} M)$. 
It follows that $\phi_{*}({}^{S}\!\brho^{n}_{M}) $ belongs to $\rad^{n}(M, M')\otimes \kappa(u)$ after the  specialization to the point $u$. 
Thus we see that $\phi_{*}({}^{S}\!\brho^{n}_{M}) $ belongs to $\rad^{n}_{\Pa}(M, M')\otimes S$ as desired. 

Let $\psi: \Spec S \to \rad^{n}(M, M')$ be the morphism of schemes induced  from  the element  $\phi_{*}({}^{S}\!\brho^{n}_{M}) \in  \rad^{n}_{\Pa}(M, M') \otimes S$.
\[
\psi: \Spec S \to \rad^{n}(M, M'), \ u \mapsto {}^{u}\!\brho^{n}_{M}. 
\]
By Theorem  \ref{202008172145} the subspace $\mlap^{n}(M,M') \subset \rad^{n}(M ,M)$ of minimal left $\rad^{n}$-approximations  is open. 
Thus  the inverse image $\psi^{-1}(\mlap^{n}(M, M'))$ is an open neighborhood of  $v$ in $ \kk Q_{0}$ which is contained in $\cL_{n ,M}$. 
This shows that $\cL_{M, n}$ is an open subset of $\kk Q_{0}$ as desired. 
\qed

\section{Minimal left $\rad^{n}$-approximations in the case $\chara \kk = 0$}\label{section: check the case char k is 0}

The aim of this section is to establish the desired result in the case $\chara \kk = 0$  that the morphism ${}^{v}\!\varrho_{M}^{n}: M \to {}^{v}\!\YYM_{n} \lotimes_{\Pa} M$ is a minimal left $\rad^{n}$-approximation of $M \in \ind \Dbmod{\Pa}$ for a generic sincere weight $v \in \kk^{\times}Q_{0}$.

\subsection{Statements}

\begin{theorem}\label{202105241547}
Assume $\chara \kk = 0$. 
Let $Q$ be a  quiver, 
$M \in \ind \Dbmod{\Pa}$ and $n\in N_{Q}$.  

Then for a generic sincere weight $v \in \kk^{\times} Q_{0}$ 
the morphism 
\[
{}^{v}\!\varrho_{M}^{n}: M \to {}^{v}\!\YYM_{n} \lotimes_{\Pa} M
\]
is a minimal left $\rad^{n}$-approximation. 
\end{theorem}

In the case  $Q$ is Dynkin, since $\Pa$ is of finite representation type, it is immediate to deduce the following corollary. 

\begin{corollary}\label{202105241546}
Assume $\chara \kk = 0$. 
Let $Q$ be a Dynkin quiver with the Coxeter number $h$. 
Then for a generic sincere weight $v \in \kk^{\times} Q_{0}$
the following assertion holds: 

For $M \in \ind\Dbmod{\Pa}$ and $n = 1,2, \ldots, h -2$, 
the morphism 
\[
{}^{v}\!\varrho_{M}^{n}: M \to {}^{v}\!\YYM_{n} \lotimes_{\Pa} M
\]
is a minimal left $\rad^{n}$-approximation. 
\end{corollary}

By Theorem \ref{generic theorem}, 
to prove Theorem \ref{202105241547},  
it is enough to establish existence of a sincere weight $v \in \kk^{\times }Q_{0}$ that has property (I)${}_{M, n}$ and 
is such that the morphism ${}^{v}\!\brho^{n}_{M}$ is a minimal left $\rad^{n}$-approximations for fixed $M ,n$.
Thanks to Theorem \ref{202104091259}, this problem is reduced to establish 
existence of a  weight $v \in \We_{Q}$ that has the property (II)${}_{M, \lambda}$ 
for some $\lambda\in \kk\setminus \{-1\}$ such that $V_{m}(\lambda) \neq 0$ for $m= 2,3,\ldots, n$. 
We solve this problem in the following two propositions. The first one covers all the cases except when
$Q$ is extended Dynkin and $M$ is a shift of a regular module. These cases are dealt with in the second proposition.

\begin{proposition}\label{eigenvector theorem} 
Let $Q$ be a finite acyclic quiver. 
Then, the following holds. 

\begin{enumerate}[(1)] 

\item 
If $Q$ is  a Dynkin  quiver with the Coxeter number $h$. 
Then, there exists a regular weight  $v \in \kk Q_{0}$ which is an eigenvector of  
$\Psi$ whose eigenvalue $\lambda$ is a $h$-th root of unity.

\item 
Assume that  $Q$ is an extended Dynkin quiver.  
Then  there exists a semi-regular weight  $v \in \kk Q_{0}$  which is an eigenvector of  
$\Psi$ whose eigenvalue is $1$. Moreover there exists no regular weight  $v \in \kk Q_{0}$ which is an eigenvector of $\Psi$. 

\item 
If $Q$ is wild, then 
 exists a regular weight $v \in \kk Q_{0}$ which is an eigenvector of  
$\Psi$ whose eigenvalue $\lambda$ is not a root of unity different from $1$. 

\end{enumerate}

\end{proposition}

\begin{proposition}\label{202105241610} 
Let $Q$ be an extended Dynkin quiver and $M\in \Dbmod{\Pa}$ a shift of a regular module. 
Then there exists a sincere weight $v \in \kk^{\times} Q_{0}$ that has the property (II)${}_{M, 1}$. 
\end{proposition}

\subsection{A proof of Proposition \ref{eigenvector theorem}. Dynkin case}

We give a proof of Proposition \ref{eigenvector theorem} for a Dynkin quiver $Q$. First we extend our ground field $\kk$ to a field $\kk'$ such that $\CC \subset \kk'$. 

Let $Q$ be a Dynkin quiver with the Coxeter number $h > 2$. 
We set $\zeta := \exp\left(\frac{2\pi\sqrt{-1}}{h}\right)$ and $\xi_{n} := \frac{\zeta^{ n+1 } - 1}{\zeta - 1}$ for $n =0,1, \ldots, h-2$. 
We note that $\xi_{0} = 1, \xi_{h -2} = - \zeta^{-1}$ and $1 -\xi_{h -2}^{-1} \xi_{n -1} = \xi_{n}$. 
We fix  a vertex $i \in Q_{0}$ and  set $w_{n}  := \Psi^{n} C^{-1} \Euvect(P_{i})$. 
We prove that 
$v:= \sum_{n =0}^{h -2} \xi_{n} w_{n} $ is a regular weight 
which  is an eigenvector of $\Psi$ with the eigenvalue $\zeta^{-1}$.

 First we prove $\Psi(v) = \zeta^{-1} v$.  
The order of the Coxeter matrix $\Phi$ is $h$. 
Moreover $\Phi$  does not have $1$ as its eigenvalue by \cite{Dlab-Ringel: Indecomposable}.
Therefore, $\sum_{n = 0}^{h -1} \Psi^{n} =0$. 
Thus, we have 
$
\Psi ( w_{n}) = w_{n+1} 
$
for $n =0, 1, \ldots, h-3$ and 
$\Psi(w_{h -2} ) = - \sum_{n =0}^{h -2}w_{n}$. 
It follows that 
\[
\begin{split}
\Psi(v) & = \sum_{n =0}^{h -3} \xi_{n}w_{n +1} -\xi_{ h-2} \sum_{n =0}^{h -2} w_{n} \\
 & = - \xi_{h -2} \left( w_{0} +  \sum_{n =1}^{h -2} ( 1- \xi_{h -2}^{-1} \xi_{n-1})w_{n}\right) \\
& = \zeta^{-1} v.
 \end{split}
 \]
 
 We claim that ${}^{v}\!\Euch(\nu_{1}^{-1}(M)) = \zeta^{-1} {}^{v}\!\Euch(M)$ for all $M \in \Dbmod{A}$. 
 Indeed, we can deduce it by a straightforward calculation as below 
 \[
 \begin{split}
 {}^{v}\!\Euch(\nu_{1}^{-1}(M)) & = v^{t} \Euvect(\nu_{1}^{-1}(M))  = v^{t} \Phi^{-1}\Euvect(M) \\ 
 & = (\Psi v)^{t} \Euvect(M) = \zeta^{-1} v^{t} \Euvect(M) = \zeta^{-1} {}^{v}\!\Euch(M).  
 \end{split}
 \]
 
Next we prove that $v$ is regular.  
Recall that 
 an indecomposable module $M \in \ind Q$ belongs to the $\nu_{1}^{-1}$ orbit of some indecomposable projective module $P_{j}$. 
Thus by the above claim,  
it is enough to show that ${}^{v}\!\Euch(P_{j}) \neq 0$ for each vertex $j \in Q_{0}$. 

We take a vertex $j \in Q_{0}$.
We claim ${}^{v}\!\Euch(P_{j}) = \sum_{n =0}^{h -2} \xi_{n} \Euch(e_{i} \PPi_{n} e_{j})$. 
First we check 
\[
\begin{split}
\Euvect(P_{j})^{t} w_{n} & = \Euvect(P_{j})^{t}  (\Phi^{ -n})^{t}C^{-1} \Euvect(P_{i})\\
& = \left( \Phi^{-n} \Euvect(P_{j})\right)^{t}C^{-1}  \Euvect(P_{i})\\
 & = \Euvect(\nu_{1}^{-n} P_{j})^{t}C^{-1}  \Euvect(P_{i})\\
 & =\Euch(\RHom_{\kk Q}(P_{i}, \nu_{1}^{-n} P_{j})) \\
 & = \Euch(e_{i} \PPi_{n} e_{j} ).  
 \end{split} 
\]
Thus we have 
\[
{}^{v}\!\Euch(P_{j}) = \Euvect(P_{j})^{t}v = \sum_{n =0}^{h -2} \xi_{n} \Euch(e_{i} \PPi_{n} e_{j}). 
\]
Let  $m \in \{ 0, 1, \ldots, h-2 \}$ be the integer that 
such that $\PPi_{n} e_{j}= \nu_{1}^{-n}(P_{j})$ concentrated in cohomological degree $0$ for $n= 0,1, \ldots, m$ and 
$\PPi_{n} e_{j}= \nu_{1}^{-n}(P_{j})$ concentrated in cohomological degree $-1$ for $n= m +1, \ldots, h$. 
Then  $\Euch(e_{i}\PPi_{n} e_{j} ) \geq 0$  for $n = 0, \ldots, m$ 
and $\Euch(e_{i} \PPi_{n} e_{j} ) \leq 0$ for $n = m +1, \ldots, h-2$. 

Let $\arg: \CC^{\times} \to (-\pi, \pi]$ be the principal branch of the argument function.
It is easy to check that 
\[
 0 = \arg (\xi_{0}) < \arg (\xi_{1}) < \cdots < \arg (\xi_{h -2}) = \pi- \frac{ 2\pi}{h} < \pi. 
\]
Then  we have 
\[
\arg( -\xi_{m +1} )  < \arg(- \xi_{m +2} ) < \cdots < \arg(- \xi_{h -2})  < \arg(\xi_{0}) < \cdots < \arg( \xi_{m})
\]
and $\arg( \xi_{m} ) - \arg( -\xi_{m+ 1}) < \pi$. 
Therefore, ${}^{v}\!\Euch(P_{j})$ is given as a sum of complex numbers  belonging to  the half plane 
$\{ z \in \CC^{\times} \mid 0\leq \arg(z)-\arg( -\xi_{m+1})< \pi \}\sqcup\{ 0\}$ 
 as below
\[
{}^{v}\!\Euch(P_{j}) = \sum_{n =0}^{m } \xi_{n} \Euch(e_{i} \PPi_{n} e_{j}) +\sum_{n =m +1}^{h -2} ( -\xi_{n})\left( - \Euch(e_{i} \PPi_{n} e_{j})\right).  
\]
By Lemma \ref{202103301920} below,  $\Pi e_{j}$ is sincere. Therefore there exists $n$ such that $\Euch(e_{i} \Pi_{n} e_{j} ) > 0$. 
Thus we conclude ${}^{v}\!\Euch(P_{j}) \neq 0$  as desired. 
It follows that $v \neq 0$ and hence $v$ is an eigenvector of $\Psi$. 

Finally note that the coefficients of $v$ are in the algebraic closure of $\QQ$ and so $v \in \kk Q_{0}$.
 \qed

\begin{lemma}\label{202103301920}
Let $Q$ be a Dynkin quiver and $i \in Q_{0}$. 
Then the projective module $\Pi(Q)e_{i}$ is sincere.
\end{lemma}

\begin{proof}
Let $j \in Q_{0}$. 
Then the underlying graph $|Q|$ admits an orientation $\Omega$ such that 
in the quiver $Q' := (|Q|, \Omega)$ has a path from $j $ to $i$.
Therefore,  $e_{j} \kk Q' e_{i} \neq 0$. 
Since 
$\kk Q'$ is a (ungraded) subalgebra  of $\Pi(Q)$, we conclude $e_{j} \Pi(Q) e_{i} \neq 0$ as desired.
\end{proof}

\subsection{A proof of Proposition \ref{eigenvector theorem}  non-Dynkin case}
 
Recall that 
$\Phi := - C^{t} C^{-1}$, $\Psi := \Phi^{-t}= -C^{-1}C^{t}$. 
It follows that  $\Psi  = C^{-1} \Phi C$. 
Therefore an element $ v \in \kk Q_{0}$ is an eigenvector of $\Psi$ with the eigenvalue $\lambda$ 
if and only if $w: = C v$ is an eigenvector of $\Phi$ with the eigenvalue $\lambda$. 
Recall that the Euler-Ringel form  $\eragl{-,+}$  is a bilinear form on $\kk Q_{0}$ which is defined to be 
$\eragl{v, w} := w^{t}C^{-1} v$, which  satisfies the formula $\eragl{\Euvect(M), \Euvect(N)} = \dim\Hom_{\Pa}(M, N) - \dim\Ext_{\Pa}^{1}(M, N)$ for $M, N \in \Pa \mod$. 
It follows that for $u, v \in \kk Q_{0}$, we have $u^{t} v= u^{t} C^{-1} (Cv) = \eragl{Cv, u}$. 
Thus in particular ${}^{v}\!\Euch(M) = \eragl{C v, \Euvect(M)}$ for $M \in \Dbmod{A}$. 

Using this observation, we deduce the non-Dynkin case of Theorem \ref{eigenvector theorem} from results by 
Dlab-Ringel \cite{Dlab-Ringel: Indecomposable},  de la Pena-Takane \cite{Pena-Takane} and Takane \cite{Takane}. Note that by \cite{Takane} we can choose the eigenvalue $\lambda$ to be the spectral radius of $\Phi$ and in particular not a root of unity different from $1$.

\subsubsection{Extended-Dynkin case}
 
 Let $Q$ be an extended Dynkin quiver.

 First we prove that $\Psi$ has a semi-regular eigenvector with the eigenvalue $1$. 
Let $R$ be an indecomposable quasi-simple regular module such that $\nu_{1}^{-1} (R) = R$. 
We set $\delta := \Euvect(R)$ and $v:= C^{-1}\delta$. 
Then for $M \in A \mod$, we have ${}^{v}\!\Euch(M) = \eragl{ \delta, \Euvect(M) }$.
It follows that $v$ is semi-regular.

 Next we prove that $\Psi$ does not have a regular eigenvector. 
Assume that $\Psi$ has a regular eigenvector $u$ with the eigenvalue $\lambda$.  
Let $R$ be an indecomposable quasi-simple regular module such that $\nu_{1}^{-1} (R) = R$. 
From the following equation we see that $\lambda = 1$. 
\[
\lambda \left(  {}^{u}\!\Euch(R) \right)= {}^{\Psi(u) }\!\Euch(R) = {}^{u}\!\Euch(\nu_{1}^{-1}(R)) = {}^{u}\! \Euch(R). 
\]
 By \cite{Dlab-Ringel: Indecomposable}, the eigenspace of $\Phi$ corresponding to the eigenvalue $1$ is one-dimensional 
 and generated by $\delta= \Euvect( R)$. 
 It follows from the equation $\Psi = C^{-1} \Phi C$ that 
 there exists $c \in \kk\setminus\{0\} $ such that $u = cC^{-1} \delta$. 
 Now we reach a contradiction by
 \[
 {}^{u}\!\Euch(R) =\Euvect(R)^{t} u = c \delta^{t} C^{-1} \delta = c \eragl{\delta, \delta } = 0. 
 \]
 \qed

 \subsubsection{Wild case}
 
 Let $Q$ be a wild quiver. 
 By Takane \cite[Theorem 1.4, Theorem 2.1]{Takane},
  $\Phi$ has an eigenvector $y^{-}$ such that $\eragl{y^{-}, \Euvect(M)} \neq 0$ for all $M \in \ind Q$. 
  Thus the element $v:= C^{-1} y^{-}$ is a regular eigenvector of $\Psi$. 
 \qed

 \subsection{Proof of Proposition \ref{202105241610}}

 We deal with the remaining case that $Q$ is extended Dynkin and $M$ is a shifted of a regular module. 
 We may assume that $M$ is a regular module. 
 Let $\cC$ be the connected component of the AR-quiver of $\Pa \mod$ to which $M$ belong. 
 We take (an isomorphism class of ) all the quasi-simple regular modules  $L_{1}, L_{2}, \ldots, L_{p}$ in $\cC$. In other words $\cC$ looks as follows:
 
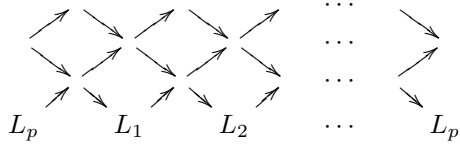
\[
\begin{xymatrix}@R=4pt@C=8pt{
&  \ar[dr] & & \ar[dr] &&&\cdots &\ar[dr]& \\
\ar[ur] \ar[dr] & & \ar[ur] \ar[dr] & & \ar[ur] \ar[dr] & &  \cdots & &\\
& \ar[ur] \ar[dr] & & \ar[ur] \ar[dr] & & & \cdots  & \ar[ur] \ar[dr] & \\
L_{p} \ar[ur] && L_{1} \ar[ur] && L_{2} \ar[ur] && \cdots && L_{p}
}\end{xymatrix}
\]

\begin{lemma}\label{202105071824}
 The dimension vectors $\Euvect(L_{1}), \Euvect(L_{2}), \ldots, \Euvect(L_{p})$ are linearly independent in $\kk Q_{0}$. 
\end{lemma}

\begin{proof}
We can check the statement for extended Dynkin quivers with particular orientations from 
the table given in \cite[XIII 2.4, 2.6, 2.12, 2.16, 2.20]{Simson-Skowronski: II}. 
Using the APR-tilting equivalences, we can reduce the general case to this case. 
\end{proof}

\begin{question}
Is there a direct proof which does not rely on any  explicit description of regular modules on the mouth?

Does the same statement holds true for  $n$-tame algebras in  the sense of Herschend-Iyama-Oppermann \cite{HIO}?
\end{question}

\begin{lemma}\label{202105241859}
Let $u_{1}, u_{2}, \ldots, u_{p} \in \kk^{r}$ be linear independent elements 
and $x \in \kk^{\times}$. 
Then there exists $v \in \kk^{r}$ such that ${}^{t} u_{i} v = x$ for all $i =1,2, \ldots, p$. 
\end{lemma}

\begin{proof} 
Let $\xi := (x, x, \ldots, x)^{t} \in \kk^{p}$. 
Since the map $F: \kk^{r} \to \kk^{p}, F(w) := ({}^{t}u_{1}w, \ldots, {}^{t}u_{p}w)^{t}$ is linear and surjective, 
the inverse image $F^{-1}(\xi)$ is non-empty. 
\end{proof}

It follows from Lemma \ref{202105071824} and Lemma \ref{202105241859} that 
 there exists $v \in \kk Q_{0}$ such that 
all the weighted Euler characteristics ${}^{v}\!\Euch(L_{1}), {}^{v}\!\Euch(L_{2}), \ldots, {}^{v}\!\Euch(L_{p})$ have the same non-zero  value. 
If we set $x := {}^{v}\!\Euch(L_{1})=  {}^{v}\!\Euch(L_{2}) = \cdots = {}^{v}\!\Euch(L_{p})$, then 
the numbers ${}^{v}\!\Euch(N)$ for $N \in \cC$ are given in the diagram below 
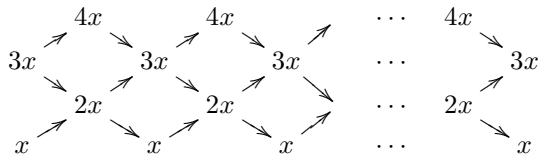
\[ 
\begin{xymatrix}@R=4pt@C=8pt{
& 4x \ar[dr] & & 4x \ar[dr] &&&\cdots & 4x \ar[dr]& \\
3x \ar[ur] \ar[dr] & & 3x \ar[ur] \ar[dr] & & 3x \ar[ur] \ar[dr] & &  \cdots & &3x\\
& 2x \ar[ur] \ar[dr] & & 2x\ar[ur] \ar[dr] & & & \cdots  & 2x \ar[ur] \ar[dr] &  \\
x \ar[ur] && x \ar[ur] && x \ar[ur] && \cdots && x
}\end{xymatrix}
\]

We have ${}^{v}\!\Euch(\PPi_{1} \lotimes_{\Pa} N ) ={}^{v}\!\Euch(N)$ and 
consequently,   
 \[
 \frac{{}^{v}\!\Euch(\PPi_{1} \lotimes_{\Pa} N) }{{}^{v}\!\Euch(N) } = 1
 \]
 for any indecomposable modules $N$ belonging to $\cC$.
 
 Thus we see that for any indecomposable module $K$ belonging to $\cC$ 
 the  weight  $v \in \kk Q_{0}$ has the property (II)${}_{K, 1}$.

 \bigskip 
 Now we have completed the proofs of Proposition \ref{eigenvector theorem} and Proposition \ref{202105241610}.
 Hence we have completed the proofs of Theorem \ref{202105241546} and Theorem \ref{202105241547} as well.

\section{QHA of Dynkin type $A_{N}$}\label{section: check the AN case}
We study the case where $Q$ is  an $A_{N}$-quiver and the base field $\kk$ is of arbitrary characteristic.

\begin{theorem}\label{202106081239}
Let $N \geq1$ and  $Q$  an $A_{N}$-quiver.
Assume that $\kk$ has a primitive $(N+1)$-th root of unity. 
Then for a generic $v \in \kk^{\times} Q_{0}$, 
the morphism 
\[
{}^{v}\!\varrho_{M}^{n} :M \to {}^{v}\!\YYM_{n} \lotimes_{\Pa} M
\]
is a minimal left $\rad^{n}$-approximation 
for all $M \in \Dbmod{\Pa}$ and $n = 1,2, \ldots, N$.
\end{theorem}

\begin{proof} 
By Theorem \ref{generic theorem} and Theorem \ref{202105241547},  
it is enough to show that 
the matrix $\Psi= \Phi^{-t}$ has a regular eigenvector $v\in \kk Q_{0}$.  
whose eigenvalue $\lambda$ satisfies 
\[
V_{m}(\lambda) \neq 0 \ \textup{ for } m = 2, \ldots, N-1. 
\]

Using APR-tilting equivalences, we may assume that $Q$ is   a directed $A_{N}$-quiver $Q$. 
\[
Q: 1 \to 2 \to 3 \to \cdots \to N.
\]
It is straight forward to check 
\[
\Psi=
\begin{pmatrix} 
-1 & -1 & -1& \cdots & -1 & -1 \\
1 & 0 & 0 & \cdots & 0 & 0\\ 
0 & 1 & 0 & \cdots & 0 & 0\\ 
0 & 0 & 1 & \cdots & 0 & 0\\ 
\vdots & \vdots & \vdots & \ddots &\vdots & \vdots \\
0 & 0 & 0 & \cdots & 1 & 0\\ 
\end{pmatrix}
\]
and the eigenpolynomial $F_{\Psi}(t)$ is 
\[
F_{\Psi}(t) = t^{N}+ t^{N-1} + \cdots +  t + 1 = V_{N +1}(t). 
\]

Let $\lambda$  be  a primitive $(N+1)$-th  root unity, then it is a root   of $F_{\Phi}(t)$. 
Then 
the vector  $v := (\lambda^{N-1}, \lambda^{N-2}, \ldots,   \lambda, 1)^{t} \in \kk Q_{0}$ 
is an eigenvector of $\Psi$  belonging to the eigenvalue $\lambda$.  

Recall that  isomorphism classes of indecomposable modules over $\Pa$ is 
parameterized by pairs  $(i,j)$ of integers such that $1 \leq i \leq j \leq N$.  
Namely, for such a pair $(i,j)$, 
there exists a unique indecomposable module $M_{(i,j)}$ such that 
$\dim e_{k} M= 1 \ ( i \leq k \leq j), \ \dim e_{k}M = 0 \ $(otherwise) 
and these modules form a set of representatives of $\ind Q$. 
Since 
\[
{}^{v}\!\Euch(M_{(i,j)} ) = \lambda^{N-i} + \lambda^{N -i-1} + \cdots + \lambda^{N -j},   
\]
we conclude that ${}^{v}\!\Euch(M_{(i,j)}) \neq 0$ for all $(i,j)$. 
\end{proof}

\section{Bimodule versions and the universal ladder}\label{section: the universal ladder}
The aim of this section is to prove Theorem \ref{202105261457} which is a bimodule version of Theorem \ref{202104091259}. Throughout the section we again suppress the weight in our notation and write $\YYM = {}^{v}\!\YYM$.

\subsection{}
First we establish a bimodule version of Theorem \ref{2020071920551}.

\begin{proposition}\label{202102231445} 
If the weight $v \in \kk^{\times} Q_{0}$  is a regular eigenvector of $\Psi= \Phi^{-t}$ with the eigenvalue $\lambda$. 
Then, there exists a morphism $\xxi_{2}: \YYM_{1} \lotimes_{\Pa} \YYM_{1} \to \PPi_{1}$ in $\sfD(\Pa^{\mre})$ 
that satisfies the following conditions:
\begin{enumerate}[(1)]
\item 
We have the following equalities in $\Hom_{\Pa^{\mre}} (\YYM_{1}, \PPi_{1})$ 
\[
\xxi_{2} (\rrho_{\YYM_{1}}) = \ppi_{1}, \ \xxi_{2} ({}_{\YYM_{1}}\rrho) = -\lambda \ppi_{1}.
\]

\item 
We have the following equality in $\End_{\Pa^{\mre}}(\PPi_{1}) \cong \kk$ 
\[
\xxi_{2} \eeta^{*}_{2}= - (1 + \lambda) \id_{\PPi_{1}}.
\]

\item 
If $\lambda \neq -1$, then the composition $\xxi_{2} \eeta^{*}_{2}$ is an automorphism of $\PPi_{1}$ in $\sfD(\Pa^{\mre})$. 

\end{enumerate} 
\end{proposition}

We need a preparation. 
Let $X$ be an object of $\Dbmod{\Pa^{\mre}}$. 
We provide a lemma that compares the endomorphism algebra $\End_{\Pa^{\mre}}(X)$ over $\Pa^{\mre}$ and 
the endomorphism algebra $\End_{\Pa}(X)$ over $\Pa$.  
Note that the forgetful functor $\sfD(\Pa^{\mre}) \to \sfD(\Pa)$ is not fully faithful. 

\begin{lemma}\label{202103040935}
Let $X \in \Dbmod{\Pa^{\mre}}$ be such an object  that  $\tuH^{< 0}(\RHom_{\Pa}(X, X)) = 0$. 
 Then the canonical map $i: \End_{\Pa^{\mre}}(X) \to \End_{\Pa}(X)$ is injective. 
 Moreover we have $\rad \End_{\Pa}(X) \cap \End_{\Pa^{\mre}}(X) \subset \rad\End_{\Pa^{\mre}}(X)$. 
\end{lemma}

\begin{proof}
Recall that $\Hom_{\Pa^{\mre}}(X, X) \cong \RHom_{\Pa^{\mre}}(\Pa, \RHom_{\Pa}(X, X))$. 
Applying the functor $$\RHom_{\Pa^{\mre}}(-, \RHom_{\Pa}(X, X))$$ to the exact sequence $0\to \Pa V\Pa \to \Pa^{\mre} \to \Pa \to 0$, 
 we obtain  the  exact sequence
\[
\Hom_{\Pa^{\mre}}(\Pa V \Pa,  \RHom_{\Pa}(X, X) [-1]) \to \Hom_{\Pa^{\mre}}(X, X) \xrightarrow{ i} \Hom_{\Pa}(X, X). 
\]
Since the left term vanishes by the assumption, we conclude that $i$ is injective. 

The second statement follows from the fact that inside a finite dimensional algebra $R$ the radical $\rad R$  is characterized as 
the largest nilpotent left ideal (see \cite[Theorem 4.12]{Lam}). 
\end{proof}

\begin{proof}[Proof of Proposition \ref{202102231445}]
The idea of  the proof is to take the $\PPi_{2}$-dual $\RHom_{\Pa^{\op}}(\eeta^{*}, \PPi_{2})$ of $\eeta^{*}$.
To complete this idea, we need to fix  two isomorphisms $\sfiso$ and $\sfcan$. 

First
we fix isomorphisms 
\[
\sfc'_{\PPi_{1}} : \YYM_{1} \lotimes_{\Pa} \PPi_{1} \to \PPi_{1} \lotimes_{\Pa} \YYM_{1}, \ \ 
\sfe': \RHom_{\Pa^{\op}}(\YYM_{1}, \PPi_{1}) \cong \PPi_{1} \lotimes_{\Pa} \YYM_{1}^{\rvvee} \to \YYM_{1}
\]
in $\sfD(\Pa^{\mre})$ obtained in Corollary \ref{2021010816032} and Corollary \ref{2021022307502}. 
We define an isomorphism 
$\sfiso: \RHom_{\Pa^{\op}}(\YYM_{1} \lotimes_{\Pa} \YYM_{1}, \PPi_{1} ) \to \YYM_{1} \lotimes_{\Pa} \YYM_{1}$ 
to be the composition
\[
\begin{split}
\RHom_{\Pa^{\op}}(\YYM_{1} \lotimes_{\Pa} \YYM_{1}, \PPi_{1} )& \xrightarrow{\cong }
\PPi_{1} \lotimes_{\Pa}  \RHom_{\Pa^{\op}}( \YYM_{1}, \PPi_{1} ) \lotimes_{\Pa}  ( \YYM_{1})^{\rvvee} \\
&\xrightarrow{ {}_{\PPi_{1}}\sfe'_{ \YYM_{1}^{\rvvee} } } 
\PPi_{1}\lotimes_{\Pa}  \YYM_{1} \lotimes_{\Pa} (\YYM_{1})^{\rvvee} \\
&\xrightarrow{(\sfc'_{\PPi_{1}} )^{-1}_{(\YYM_{1})^{\rvvee}  }} 
\YYM_{1} \lotimes_{\Pa} \PPi_{1}\lotimes_{\Pa} (\YYM_{1})^{\rvvee} \\
&\xrightarrow{ {}_{  \YYM_{1}}  \sfe' } 
 \YYM_{1} \lotimes_{\Pa} \YYM_{1}.\end{split}
\]
Let $\sfcan: \RHom_{\Pa^{\op}}(\PPi_{1}, \PPi_{2}) \to \PPi_{1}$  be the canonical isomorphism. 
We prove the composition  $\xxi_{2} := \lambda^{-1} \sfcan \RHom_{\Pa^{\op}}(\eeta^{*}, \PPi_{1}) \sfiso^{-1}: \YYM_{1} \lotimes_{\Pa} \YYM_{1} \to \PPi_{1}$
has the desired properties.
\[
\xxi_{2} : \YYM_{1} \lotimes_{\Pa} \YYM_{1} \xrightarrow{ \sfiso^{-1}} \RHom_{\Pa^{\op}}(\YYM_{1} \lotimes_{\Pa} \YYM_{1}, \PPi_{2}) 
\xrightarrow{ \RHom_{\Pa^{\op}}(\eeta^{*}, \PPi_{1}) } 
\RHom_{\Pa^{\op}}(\PPi_{1}, \PPi_{2}) \xrightarrow{ \lambda^{-1}  \sfcan} \PPi_{1}. 
\]

Applying  $\RHom_{\Pa^{\op}}(-, \PPi_{2})$ to the morphism 
${}_{ \YYM_{1}}\ppi_{1}:  \YYM_{1} \lotimes_{\Pa} \YYM_{1} \to \YYM_{1} \lotimes_{\Pa} \PPi_{1}$, we obtain the following commutative diagram 
\begin{equation}\label{2021022216151}
\begin{xymatrix}@C=60pt{
 (   \YYM_{1} \lotimes_{\Pa} \PPi_{1}, \PPi_{2} ) 
\ar[r]^{({}_{  \YYM_{1}} \ppi_{1}, \PPi_{2} )} \ar[d] & 
 (   \YYM_{1} \lotimes_{\Pa} \YYM_{1}, \PPi_{2} ) \ar[d] \\
\PPi_{1} \lotimes_{\Pa}  ( \PPi_{1}, \PPi_{1} )   \lotimes_{\Pa} ( \YYM_{1})^{\rvvee} 
\ar[r]^{ {}_{\PPi_{1}} (\ppi_{1}, \PPi_{2} )_{ ( \YYM_{1})^{\rvvee} } } \ar[d] & 
\PPi_{1} \lotimes_{\Pa}  ( \YYM_{1}, \PPi_{1} ) \lotimes_{\Pa}  ( \YYM_{1})^{\rvvee} 
\ar[d]^{{}_{\PPi_{1}}\sfe'_{ \YYM_{1}^{\rvvee} } }
 \\
\PPi_{1}\lotimes_{\Pa} (\YYM_{1})^{\rvvee}\ar@{=}[d]
\ar[r]^{ {}_{\PPi_{1} }  \rrho_{ (\YYM_{1})^{\rvvee}  } }
&
\PPi_{1}\lotimes_{\Pa}  \YYM_{1} \lotimes_{\Pa} (\YYM_{1})^{\rvvee} 
\ar[d]^{(\sfc'_{\PPi_{1}} )^{-1}_{(\YYM_{1})^{\rvvee}  }} \\ 
\PPi_{1}\lotimes_{\Pa} (\YYM_{1})^{\rvvee}\ar[d]_{\sfe' }
\ar[r]^{   \rrho_{ \PPi_{1} \lotimes (\YYM_{1})^{\rvvee} } }
&
  \YYM_{1} \lotimes_{\Pa} \PPi_{1}\lotimes_{\Pa} (\YYM_{1})^{\rvvee} 
\ar[d]^{ {}_{  \YYM_{1}}  \sfe'} \\
 \YYM_{1}\ar[r]_{ \rrho_{\YYM_{1}}} & 
\YYM_{1} \lotimes_{\Pa}  \YYM_{1}
}\end{xymatrix}
\end{equation}
where 
we use the abbreviation  $(-, +) := \RHom_{\Pa^{\op}}(-, +)$. 
We note that the right column is $\sfiso$. 

On the other hand, applying $\RHom_{\Pa^{\op}}(- , \PPi_{2})$ 
to the morphism $\rrho_{\PPi_{1}} :  \PPi_{1} \to  \YYM_{1} \lotimes_{\Pa} \PPi_{1}$, 
we obtain the following commutative diagram. 
\[
\begin{xymatrix}@C=60pt{
\RHom_{\Pa^{\op}}( \YYM_{1} \lotimes_{\Pa} \PPi_{1}, \PPi_{2} ) 
\ar[r]^{\RHom_{\Pa^{\op}}(\rrho_{\PPi_{1}} , \PPi_{2})}\ar[d] 
& 
\RHom_{\Pa^{\op}}(  \PPi_{1}, \PPi_{2} ) \ar[d] \\ 
\RHom_{\Pa^{\op}}(  \YYM_{1} , \PPi_{1} ) \ar[r]^{\RHom_{\Pa^{\op}}(\rrho , \PPi_{1})}\ar[d]^{\sfe'} & 
\RHom_{\Pa^{\op}}(\Pa, \PPi_{1} ) \ar[d] \\
\YYM_{1}\ar[r]_{ - \lambda\ppi_{ 1 }} &
\PPi_{1}.
}\end{xymatrix} 
\]
Observe that the right column coincides with that of the diagram \eqref{2021022216151}.

Thus applying $\RHom_{\Pa^{\op}}(- , \PPi_{2})$ to the commutative diagram 
\[
\begin{xymatrix}@C=60pt{\PPi_{1} \ar[d]_{ \eeta^{*}} \ar[dr]^{-\rrho_{\PPi_{1}}} & \\ 
 \YYM_{1} \lotimes_{\Pa} \YYM_{1} \ar[r]_{  {}_{\YYM_{1} } \ppi_{1} } & 
 \YYM_{1} \lotimes_{\Pa} \PPi_{1}, 
}\end{xymatrix}
\]
we obtain the following commutative diagram 
\[
\begin{xymatrix}@C=60pt{\YYM_{1} \ar[r]^{\rrho_{\YYM_{1} }} \ar[dr]_{\ppi_{1} }& 
\YYM_{1} \lotimes_{\Pa} \YYM_{1}  \ar[d]^{\xxi_{2}}.\\
& \PPi_{1}.
}\end{xymatrix}
\]
This proves the first equality of (1).

Applying  $\RHom_{\Pa^{\op}}(-, \PPi_{2})$ to the morphism 
$\ppi_{1,\YYM_{1} }:  \YYM_{1} \lotimes_{\Pa} \YYM_{1} \to \PPi_{1} \lotimes_{\Pa} \YYM_{1}$, 
we obtain the following commutative diagram 
\begin{equation}\label{2021022216152}
\begin{xymatrix}@C=60pt{
 (   \PPi_{1} \lotimes_{\Pa} \YYM_{1}, \PPi_{2} ) 
\ar[r]^{(\ppi_{1,\YYM_{1} }, \PPi_{2} )} \ar[d] & 
 (   \YYM_{1} \lotimes_{\Pa} \YYM_{1}, \PPi_{2} ) \ar[d]
  \\
\PPi_{1} \lotimes_{\Pa}  ( \YYM_{1}, \PPi_{1} )   \lotimes_{\Pa} ( \PPi_{1})^{\rvvee} 
\ar[r]^{ {}_{\PPi_{1} \lotimes  (\YYM_{1}, \PPi_{2} )} \ppi_{1}^{\rvvee} } \ar[d]_{{}_{\PPi_{1}} \sfe'_{\PPi_{1}^{\rvvee} } } & 
\PPi_{1} \lotimes_{\Pa}  ( \YYM_{1}, \PPi_{1} ) \lotimes_{\Pa}  ( \YYM_{1})^{\rvvee} 
\ar[d]^{{}_{\PPi_{1}}\sfe'_{ \YYM_{1}^{\rvvee} } }
 \\
\PPi_{1}\lotimes_{\Pa} \YYM_{1}  \lotimes_{\Pa} \PPi_{1}^{\rvvee} \ar[d]_{ (\sfc'_{\PPi_{1}})^{-1}_{\PPi_{1}^{\rvvee} }}
\ar[r]^{ {}_{\PPi_{1}  \lotimes \YYM_{1} }  \ppi^{\rvvee}  } 
&
\PPi_{1}\lotimes_{\Pa}  \YYM_{1} \lotimes_{\Pa} (\YYM_{1})^{\rvvee} 
\ar[d]^{(\sfc'_{\PPi_{1}} )^{-1}_{(\YYM_{1})^{\rvvee}  }} 
\\ 
\YYM_{1} \lotimes_{\Pa} \PPi_{1}\lotimes_{\Pa} (\PPi_{1})^{\rvvee}\ar[d] 
\ar[r]^{ {}_{\YYM_{1}  \lotimes \PPi_{1} }  \ppi^{\rvvee}  } 
&
  \YYM_{1} \lotimes_{\Pa} \PPi_{1}\lotimes_{\Pa} (\YYM_{1})^{\rvvee} 
\ar[d]^{ {}_{  \YYM_{1}}  \sfe'} \\
 \YYM_{1}\ar[r]_{ {}_{\YYM_{1}} \rrho} & 
\YYM_{1} \lotimes_{\Pa}  \YYM_{1}
}\end{xymatrix}
\end{equation}
where 
we use the abbreviation  $(-, +) := \RHom_{\Pa^{\op}}(-, +)$. 
We note that the right column is $\sfiso$.

On the other hand, applying $\RHom_{\Pa^{\op}}(- , \PPi_{2})$ 
to the morphism ${}_{\PPi_{1}} \rrho :  \PPi_{1} \to  \PPi_{1} \lotimes_{\Pa} \YYM_{1}$, 
we obtain the following commutative diagram. 
\[ 
\begin{xymatrix}@C=60pt{
 (   \PPi_{1} \lotimes_{\Pa} \YYM_{1}, \PPi_{2} ) 
\ar[r]^{( {}_{\PPi_{1}} \rrho, \PPi_{2} )} \ar[d] & 
 (   \PPi_{1}, \PPi_{2} ) \ar[d]
  \\
\PPi_{1} \lotimes_{\Pa}  ( \YYM_{1}, \PPi_{1} )   \lotimes_{\Pa} ( \PPi_{1})^{\rvvee} 
\ar[r]^{ {}_{\PPi_{1}} (\rrho, \PPi_{1})_{\PPi_{1}^{\rvvee}}  } 
\ar[d]_{{}_{\PPi_{1}} \sfe'_{\PPi_{1}^{\rvvee} } } & 
\PPi_{1} \lotimes_{\Pa}  ( \Pa, \PPi_{1} ) \lotimes_{\Pa}  ( \PPi_{1})^{\rvvee} 
\ar[d]
 \\
\PPi_{1}\lotimes_{\Pa} \YYM_{1}  \lotimes_{\Pa} \PPi_{1}^{\rvvee} \ar[d]_{ (\sfc'_{\PPi_{1}})^{-1}_{\PPi_{1}^{\rvvee} }}
\ar[r]^{ - \lambda {}_{\PPi_{1}}\ppi_{1, \PPi_{1}^{\rvvee} } } 
&
\PPi_{1}\lotimes_{\Pa}  \PPi_{1} \lotimes_{\Pa} (\PPi_{1})^{\rvvee} 
\ar@{=}[d]
\\ 
\YYM_{1} \lotimes_{\Pa} \PPi_{1}\lotimes_{\Pa} (\PPi_{1})^{\rvvee}\ar[d] 
\ar[r]^{ -  \lambda^{2} \ppi_{1, \PPi_{1} \lotimes \PPi_{1}^{\rvvee} } } 
&
  \PPi_{1} \lotimes_{\Pa} \PPi_{1}\lotimes_{\Pa} (\PPi_{1})^{\rvvee} 
\ar[d] \\
 \YYM_{1}\ar[r]_{ - \lambda^{2} \ppi_{1}  } & 
\PPi_{1}.
}\end{xymatrix}
\]
The commutativity of third square is proved in Corollary \ref{20210108160321}.
Therefore, the right column is $\sfcan$.

Thus applying $\RHom_{\Pa^{\op}}(- , \PPi_{2})$ to the commutative diagram 
\[
\begin{xymatrix}@C=60pt{\PPi_{1} \ar[d]_{ \eeta^{*}} \ar[dr]^{{}_{\PPi_{1}} \rrho } & \\ 
 \YYM_{1} \lotimes_{\Pa} \YYM_{1} \ar[r]_{   \ppi_{1, \YYM_{1}} } & 
 \PPi_{1} \lotimes_{\Pa} \YYM_{1}, 
}\end{xymatrix}
\]
we obtain the following commutative diagram 
\[
\begin{xymatrix}@C=60pt{\YYM_{1} \ar[r]^{{}_{\YYM_{1}} \rrho } \ar[dr]_{-\lambda \ppi_{1} }& 
\YYM_{1} \lotimes_{\Pa} \YYM_{1}  \ar[d]^{\xxi_{2}}.\\
& \PPi_{1}.
}\end{xymatrix}
\]
This proves the second equality of (1).

(2) 
Let $i \in Q_{0}$. 
By (1), we have $\xxi_{2,\Pa e_{i}} \rrho_{\YYM_{1} e_{i}} = \ppi_{1, \Pa e_{i}}$. 
It follows from 
Corollary \ref{202102221805} and Corollary \ref{202102211730} that 
the endomorphism 
$( \xxi_{2} \eeta^{*}_{2} + (1 +\lambda) \id_{\PPi_{1}}) \lotimes \Pa e_{i} = 
\xxi_{2, \Pa e_{i}} \eeta^{*}_{2, \Pa e_{i}} + (1 +\lambda) \id_{\PPi_{1}e_{i}}$ of $\PPi_{1}e_{i}$ belongs to  in $\rad \End_{\Pa}(\PPi_{1}e_{i})$. 
Therefore, 
the endomorphism 
$\xxi_{2} \eeta^{*}_{2} + (1 +\lambda) \id_{\PPi_{1}}$ of $\PPi_{1}$ in $\sfD(\Pa)$ belongs to  in $\rad \End_{\Pa}(\PPi_{1})$. 

Since $\xxi_{2} \eeta^{*}_{2} + (1 +\lambda) \id_{\PPi_{1}}$ is an endomorphism of $\PPi_{1}$ in $\sfD(\Pa^{\mre})$, 
it follows from by Lemma \ref{202103040935} that 
the endomorphism $\xxi_{2} \eeta^{*}_{2} + (1 +\lambda) \id_{\PPi_{1}}$ belongs to $\rad \End_{\Pa^{\mre}}(\PPi_{1})$.

Observe that $\End_{\Pa^{\mre}}(\PPi_{1}) \cong \End_{\Pa^{\mre}}(\Pa) = (\textup{the center of } A) = \kk$ and $\rad \End_{\Pa^{\mre}}(\PPi_{1}) = 0$. 
Therefore we conclude that $\xxi_{2} \eeta^{*}_{2} = - (1 +\lambda) \id_{\PPi_{1}}$ as desired.

(3) is a consequence of (2). 
\end{proof} 

\subsection{}

\begin{theorem}\label{202105261457} 
Assume that a sincere weight $v \in \kk^{\times} Q_{0}$  
is a regular (resp. semi-regular) eigenvector of $\Psi$ with the eigenvalue $\lambda$. 

Assume moreover that there is   a natural number  $n \geq 2$ 
satisfies the condition: 
\[ 
V_{m}(\lambda )  \neq 0, \textup{ for } 1 \leq m \leq n.
\]
Then 
there exists a morphism of 
$\oomega_{n -1}: \YYM_{n -1}  \to \YYM_{1} \lotimes_{\Pa} \YYM_{n -2} $ in $\Dbmod{\Pa^{\mre}}$ 
that has the following properties 

\begin{enumerate}[(1)] 
\item $\zzeta_{n -1} \oomega_{n -1} = \id_{\YYM_{n -1}}$. 

\item If we set $\xxi_{n, M} := (\xxi_{2, \YYM_{ n-2}})({}_{\YYM_{1}} \oomega_{n -1})$
\[
\xxi_{n} 
: \YYM_{1} \lotimes_{\Pa} \YYM_{n -1} 
\xrightarrow{{}_{\YYM_{1}} \oomega_{n -1} } 
\YYM_{1} \lotimes_{\Pa} \YYM_{1} \lotimes_{\Pa} \YYM_{n -2}  
\xrightarrow{\xxi_{2, \YYM_{n -2}  } } 
\PPi_{1} \lotimes_{\Pa} \YYM_{n -2},
\]
then the following equality holds  in $\ResEnd_{\Pa^{\mre} } (\PPi_{1} \lotimes_{\Pa} \YYM_{n -2})$ 
\[
\xxi_{n} \eeta^{*}_{n} =-  \frac{V_{n}(\lambda)}{V_{n-1}(\lambda) } \neq 0.
\]
Therefore it is an automorphism of $\PPi_{1} \lotimes_{\Pa} \YYM_{n -2}$ in $\Dbmod{\Pa^{\mre}}$. 

\item We have $
(\xxi_{n})( \rrho_{\YYM_{ n-1}}) =(\ppi_{1, \YYM_{n-2}})(\oomega_{n-1})$. 

\item Setting $\eepsilon_{m} := (\xxi_{m} \eeta^{*}_{m} )^{-1}, \  \varppi_{m}:= ( -1)^{n} (\ppi_{1, \YYM_{m-2}})(\oomega_{m-1})$ for $m = 2,3, \ldots, n $, 
we obtain  the commutative diagram below for $m=2,3,\ldots, n$ 
\[
\begin{xymatrix}@C=60pt{
& \PPi_{1} \lotimes_{\Pa} \YYM_{m -2} 
\ar@<5pt>[d]^{ ( \eeta^{*}_{m})(\eepsilon_{m}) } \ar@/^5pt/[dr]^{({}_{\PPi_{1}} \rrho_{\YYM_{m -1} })(\eepsilon_{m})} &  &\\ 
\YYM_{m-1}  \ar[r]^{\rrho_{\YYM_{m-1} } } 
\ar@/^5pt/[ur]^{(-1)^{n -1} \varppi_{m-1} }  \ar@/_5pt/[dr]_{\brho_{m}} 
& 
\YYM_{1} \lotimes_{\Pa} \YYM_{m-1}   
\ar[r]^{ \ppi_{1, \YYM_{n-1}  } } \ar@<5pt>[u]^{\xxi_{m}} 
\ar@<-5pt>[d]_{\zzeta_{m} } & \PPi_{1} \lotimes_{\Pa} \YYM_{m-1} \\ 
 &  \YYM_{m}   \ar@<-5pt>[u]_{\oomega_{m}} \ar@/_5pt/[ur]_{( -1)^{m}\varppi_{m} } &&
}\end{xymatrix}
\] 
Hence we obtain the following homotopy Cartesian diagram that is folded to a direct sum of AR-triangles
\[
\begin{xymatrix}@C=60pt{
\YYM_{m-1}   \ar[d]_{ \varppi_{m-1}  } \ar[r]^{\brho_{m}} 
& \YYM_{m} \ar[d]^{\varppi_{m}} \\
\PPi_{1} \lotimes_{\Pa} \YYM_{m -2} 
\ar[r]_{ ( {}_{\PPi_{1}} \brho_{m -1} )(\eepsilon_{m })  } & \PPi_{1} \lotimes_{\Pa} \YYM_{m-1}.  
}\end{xymatrix} 
\]
Applying $- \lotimes_{\Pa} M$ where $M$ is an object of $\Dbmod{\Pa}$ (resp. $\cU_{\Pa}[\ZZ]$) 
to the above diagrams, 
we obtain the diagram 
\eqref{202105261807} and \eqref{202105261808}. 

\item Consequently, 
we have an exact triangle 
\[
\Pa \xrightarrow{ \brho^{n} } \YYM_{n} \xrightarrow{ \varppi_{n}} \PPi_{1} \lotimes_{\Pa} \YYM_{n-1} \to 
\]
in $\Dbmod{\Pa^{\mre}}$ 
such that 
applying  $- \lotimes_{\Pa} M$ where  $M$ is an object of $\Dbmod{\Pa}$ (resp. $\cU_{\Pa}[\ZZ]$), 
we obtain an exact triangle 
\[
M \xrightarrow{ \brho^{n}_{M} } \YYM_{n} \lotimes_{\Pa} M \xrightarrow{ \varppi_{n, M}} \PPi_{1} \lotimes_{\Pa} \YYM_{n-1} \lotimes_{\Pa} M \to 
\]
in $\Dbmod{\Pa}$ 
the first arrow $\brho^{n}_{M}$  of which is a minimal left $\rad^{n}$-approximation of $M$.

\end{enumerate}

\end{theorem}

We can prove Theorem \ref{202105261457} in the same way of Theorem \ref{202104091259} by using 
Proposition \ref{202102231445} and Lemma \ref{202103040935}. 
Hence we leave details to the readers.

We note we have a left ladder in $\Dbmod{\Pa^{\mre}}$ of the following form  
such that if we apply $ -\lotimes M$ then we obtain the diagram \eqref{202105261824}.  
{ \tiny 
\[ 
\begin{xymatrix}@C=30pt{
\Pa \ar[r]^{\brho } & \YYM_{1} \ar[r]^{\brho} \ar[d]_{\ppi} & 
\YYM_{2}  \ar[r]^{\brho } \ar[d]_{\varpi} &  \hspace{20pt} \cdots \hspace{-20pt}  &\ar[r] &
\YYM_{n-1}   \ar[d]_{\varpi  } \ar[r]^{\brho} 
& \YYM_{n}  \ar[d]^{ \varpi} \\
& \PPi_{1}  \ar[r]_-{{}_{\PPi_{1}} \brho \epsilon } & \PPi_{1} \lotimes_{\Pa} \YYM_{1}  \ar[r]_-{{}_{\PPi_{1}} \brho \epsilon }  &
\hspace{20pt} \cdots \hspace{-20pt} &\ar[r] 
& 
\PPi_{1} \lotimes_{\Pa} \YYM_{n -2}  
\ar[r]_{  {}_{\PPi_{1}} \brho \epsilon } & \PPi_{1} \lotimes_{\Pa} \YYM_{n-1} .  
}\end{xymatrix} 
\]}

\appendix

\renewcommand{\theequation}{\Alph{section}-\arabic{equation}}

\section{Homotopy Cartesian square}\label{section: homotopy Cartesian square}

Let $\sfD$ be a triangulated category. 
We recall from \cite[Definition 1.4.1]{Neeman} that 
a commutative diagram 
\[
\begin{xymatrix}{ 
W \ar[r]^{k} \ar[d]_{w} & V \ar[d]^{v} \\
X \ar[r]_{f} & Y
}\end{xymatrix}\]
is called \emph{homotopy cartesian} if 
there exists a morphism $c: Y \to W[1]$ that makes the following exact triangle 
\[
W \xrightarrow{(-w, k)^{t} } X \oplus V \xrightarrow{ (f, v) } Y \xrightarrow{ c  } \Sigma W.  
\]

\begin{remark}\label{homotopy Cartesian remark}
We remark that the definition is  modified from that given in \cite{Neeman}. 
In op cit., the morphism $(w, -k)^{t}$ is used in place of $(-w, k)^{t}$. 
But it is easy to check that these two definitions  are equivalent. 
\end{remark}

In other words, if we have a diagram 
\[
\begin{xymatrix}@C=60pt{
& X\ar@<-5pt>[d]_{i} &  &\\ 
W \ar[r]^{a} & U \ar[r]^{b} \ar@<-5pt>[u]_{p} \ar@<-5pt>[d]_{q} & Y \ar[r]^{c} & \Sigma W \\
&  V \ar@<-5pt>[u]_{j} &&
}\end{xymatrix}
\] 
whose row is an exact triangle and the column gives a splitting of 
of $U$, i.e., we have $pi= \id_{X}, qj  = \id_{V}, ip + j q = \id_{U}$, 
then the following diagram is homotopy Cartesian
\[
\begin{xymatrix}{ 
W \ar[r]^{qa} \ar[d]_{-pa} & V \ar[d]^{bj} \\
X \ar[r]_{bi} & Y.
}\end{xymatrix}\]

We use the following lemma that compare two homotopy Cartesian square.

\begin{lemma}\label{2020071917401}
Let $Z[-1] \xrightarrow{h} X \xrightarrow{f} Y \xrightarrow{ g} Z $ be an exact triangle 
and $k: W \to V$ a morphism in $\sfD$. 
Assume two commutative squares are given
\[
\begin{xymatrix}{ 
W \ar[r]^{k} \ar[d]_{w_{i}} & V \ar[d]^{v_{i}} \\
X \ar[r]_{f} & Y
}\end{xymatrix} \ \ ( i =1,2).
\]
Assume moreover that $gv_{1} =g v_{2}$ and 
$\Hom_{\sfD}(W, Z[-1]) = 0$. 
Then there exists 
a morphism $s: V \to X$ such that
the automorphism $\begin{pmatrix} \id  & s \\ 0 & \id \end{pmatrix}$ of $V \oplus X$ completes the following commutative diagram. 
\[
\begin{xymatrix}@C=40pt{
W \ar[r]^{\begin{pmatrix} -w_{1} \\ k  \end{pmatrix} }\ar@{=}[d] & X \oplus V  \ar[r]^{(f , v_{1})} \ar[d]_{\cong}^{\tiny \begin{pmatrix} \id  & s \\ 0 & \id \end{pmatrix}} & Y \ar@{=}[d] \\
W \ar[r]_{\begin{pmatrix}-w_{2} \\ k  \end{pmatrix}  }& X \oplus V  \ar[r]_{(f, v_{2})} & Y  
}\end{xymatrix}
\]
(We note that if the squares are homotopy Cartesian then, the both rows are exact triangles.)
\end{lemma}

\begin{proof}
We have $g(v_{1} -v_{2} ) = 0$. 
Thus there exists a morphism $s: V \to X$ such that $v_{1} - v_{2} = fs$ or equivalently $v_{1} =v_{2} + fs$. 
\[
\begin{xymatrix}@C=50pt{
 & W \ar[d]_{w_{i}} \ar[r]^{k} & V \ar@{-->}[dl]_{s}  \ar[d]^{v_{i}} \ar[r]^{gv_{1} = gv_{2} } & Z \ar@{=}[d] \\
Z[-1] \ar[r]_{h} & X \ar[r]_{f} & Y \ar[r]_{g} & Z
}
\end{xymatrix}
\]
We have $fw_{1} = v_{1} k = ( v_{2} + fs)k=v_{2}k + fsk$ and hence $v_{2} k = fw_{1} - fsk =f(w_{1} -sk)$. 
On the other hand, we have $v_{2}k = fw_{2}$. 
Consequently we see  $f(w_{1} -sk -w_{2} ) = 0$. 
It follows from the assumption $\Hom_{\sfD}(W, Z[-1]) =0$ that 
$w_{1} -sk -w_{2}=0$ 
and hence $-w_{2} = -w_{1} +sk$. 
Now it is straightforward to check that 
the automorphism $ \begin{pmatrix} \id & s \\ 0 & \id \end{pmatrix}$ of $X \oplus V$ satisfies the desired property.
\end{proof}

\section{Inverse of Serre functors and Happel's criterion }\label{section: Happel's criterion}

In this section, we fix notations for Serre duality and recall Happel's criterion for AR-triangles.

Let $R$ be a finite dimensional algebra of finite global dimension. 
We set $\sfD := \Dbmod{R}$.

\subsection{Inverse of  Serre functors}\label{202106091104}

\subsubsection{Exact functors}\label{Section: Exact functors}

To fix a notation,  
we recall that an exact functor $F: \sfD \to \sfD'$ is a pair $(F, \sigma_{F})$ of 
a  functor $F: \sfD \to \sfD'$ between triangulated categories and a natural isomorphism  
$\sigma_{F} : F \Sigma_{\sfD} \to \Sigma_{\sfD'} F $ that satisfies certain conditions. 

We note that the pair $(\Sigma_{\sfD}, - \id_{\Sigma^{2}})$ is an exact autoequivalence of $\sfD$, 
i.e., $\sigma_{\Sigma_{\sfD}} = - \id_{\Sigma^{2}}$.

\subsubsection{Inverse of  Serre functors}\label{Section: Serre functor}

Let $\sfS$ be the Serre functor of $\sfD$. 
By $\isagl{-,+}$ we denote the pairing for Serre duality. 
Namely it is a non-degenerate pairing 
\[
\isagl{-, +} : \Hom_{\sfD}(Y, X) \otimes \Hom_{\sfD}(\sfS^{-1}(X), Y) \to \kk 
\]
for $X, Y \in  \sfD$. 
It follows from functoriality of the pairing that 
for  morphisms given in the diagram below 
\[
\sfS^{-1}(Z) \xrightarrow{ \sfS^{-1}(h)} \sfS^{-1}(X) \xrightarrow{ f} Y \xrightarrow{ g} Z \xrightarrow{h} X,  
\]
we have the following equality 
\begin{equation}\label{202012091955}
\isagl{h, gf} = \isagl{hg, f} = \isagl{g, f\sfS^{-1}(h)}. 
\end{equation}

In particular,  we have  
\begin{equation}\label{202102071308}
\isagl{f, g} = \isagl{\sfS^{-1}(f), \sfS^{-1}(g) }
\end{equation}
for all  $f: Y \to X, \ g : \sfS^{-1} (X) \to Y$.

Let $F$ be an exact autoequivalence of $\sfD$. 
Then there exists a natural isomorphism $\gamma_{F}: \sfS^{-1} F \to F \sfS^{-1}$ that makes the following commutative diagram. 
\[
\begin{xymatrix}@C=50pt{ 
\Hom_{\sfD}(Y, X) \otimes \Hom_{\sfD}(\sfS^{-1}(X), Y)  \ar[r]^-{\isagl{-,+}}  \ar[d]_{F_{Y, X} \otimes F_{\sfS^{-1}(X), Y} }
& \kk\\
\Hom_{\sfD}(F(Y)   F(X) ) \otimes \Hom_{\sfD}(F\sfS^{-1}(X), F(Y)) \ar[d]_{\id_{\Hom(F(Y) , F(X) )} \otimes \Hom(\gamma_{F,X}, F(Y))} &\\
\Hom_{\sfD}(F(Y), F(X)) \otimes \Hom_{\sfD}(\sfS^{-1}F(X), F(Y) )  \ar@/^-30pt/[ruu]_-{\isagl{-,+}}
}\end{xymatrix}
\]
In other words, we have 
\[
\isagl{f, g} = \isagl{F(f), F(g)\gamma_{F, X} }
\]
for all  $f: Y \to X, \ g : \sfS^{-1} (X) \to Y$. 
Comparing it with \eqref{202102071308}, 
we see that 
\begin{equation}\label{202102071309} 
\gamma_{\sfS^{-1} } = \id_{\sfS^{-1}}.  
\end{equation}
It is also easy to check that 
for exact autoequivalences $F, G$ we have 
\[
\gamma_{FG}= F(\gamma_{G}) (\gamma_{F})_{G}: \sfS^{-1}FG \xrightarrow{ (\gamma_{F})_{G} } F\sfS^{-1}G \xrightarrow{ F(\gamma_{G}) } FG\sfS^{-1}. 
\]

It is shown by van den Bergh \cite[Appendix]{Bocklandt} that 
the pair $(\sfS^{-1}, - \gamma_{\Sigma})$ is an exact autoequivalence of $\sfD$. 
Therefore we set 
\[
\sigma_{\sfS^{-1} }:  = - \gamma_{\Sigma}: \sfS^{-1} \Sigma \to \Sigma \sfS^{-1}. 
\]

\subsubsection{A lemma}

We provide a technical lemma.

For simplicity we set $\sfP := \nu_{1}^{-1} := \Sigma \sfS^{-1}$. 
Since there are two natural isomorphism $\gamma_{-}$ and $\sigma_{-}$, 
 it is necessary to care about the way that we exchange, for example, 
 $\sfP^{n}$ with $\Sigma^{-1} $.

For an exact endofunctor  $F$  of $\sfD$,  
we denote by $\sigma'_{ F} :    \Sigma^{-1}F \to F \Sigma^{-1}  $ 
the natural isomorphism induced from $\sigma_{F} : F \Sigma \to \Sigma F$. 
\[
\sigma'_{F}:  \Sigma^{-1}F \xrightarrow{\Sigma^{-1} F(\beta)} 
\Sigma^{-1} F \Sigma \Sigma^{-1} 
\xrightarrow{ \Sigma^{-1} ((\sigma_{F})_{\Sigma^{-1}})} 
\Sigma^{-1}\Sigma   F\Sigma^{-1} \xrightarrow{\alpha_{F\Sigma^{-1}} } F\Sigma^{-1}
\]
where  we denote the canonical natural isomorphisms by  $\alpha: \Sigma^{-1} \Sigma \to \id_{\sfD}, 
\ \beta: \id_{\sfD} \to \Sigma \Sigma^{-1} $.
 
 We note that $\sigma'_{\Sigma} $ may be  identified with $- 1: \id_{\sfD} \to \id_{\sfD}$.

\begin{lemma}\label{202102171838}

\begin{enumerate}[(1)]

\item 
Let $n \geq 1$ be a positive integer. Then the following composition equals to $\id_{ \sfP^{n}}$.
\[
\mathsf{comp}^{(n)}:\sfP^{n} = \Sigma \sfS^{-1} \sfP^{n -1} \xrightarrow{ \Sigma (\gamma_{\sfP^{n -1} }) } \Sigma \sfP^{n -1} \sfS^{-1} 
\xrightarrow{ (\sigma^{-1}_{\sfP^{n -1}})_{\sfS^{-1}} } \sfP^{ n-1} \Sigma \sfS^{-1} =\sfP^{n}.
\]

\item 
Let $n \geq 1$ be a positive integer. 
Then the following diagram is commutative 
\[
\begin{xymatrix}{
\Sigma^{-1}\sfP^{n} \ar[r]^{ (\sigma'_{\sfP^{n -1} } )_{\sfP} } 
\ar@{=}[d] 
& 
\sfP^{n -1} \Sigma^{-1} \sfP \ar@{=}[r] & 
\sfP^{n -1}\Sigma^{-1} \Sigma \sfS^{-1} \ar[d]^{ \sfP^{n -1} (\alpha_{\sfS^{-1} }) }
\\ 
\Sigma^{-1}\Sigma \sfS^{-1} \sfP^{n -1} \ar[r]_{ \alpha_{\sfS^{-1} \sfP^{n -1} } } 
& 
\sfS^{-1} \sfP^{n -1} \ar[r]_{\gamma_{\sfP^{n -1}}} 
& 
\sfP^{n -1} \sfS^{-1} 
}\end{xymatrix} 
\]

\end{enumerate}
\end{lemma}

\begin{proof}
(1) The case $n =1$ is clear. 
The case $ n= 2$ is proved in the following computation:
\[
\begin{split} 
\mathsf{comp}^{(2)} 
&= \Sigma(\sigma^{-1}_{\sfS^{-1}}) \circ  (\sigma^{-1}_{\Sigma})_{\sfS^{-2}} \circ  \Sigma^{2}(\gamma_{\sfS^{-1}}) \circ \Sigma(\gamma_{\Sigma})_{\sfS^{-1}}\\
&=- \Sigma(\sigma^{-1}_{\sfS^{-1}}) \circ \Sigma(\gamma_{\Sigma})_{\sfS^{-1}}\\
& = \id_{P^{2}}
\end{split} 
\]
where for the second equality we use $\sigma_{\Sigma}= -\id, \ \gamma_{\sfS^{-1}} = \id$.

The case $n \geq 3$ is shown by induction on $n$ using  the  equality
$\mathsf{comp}^{(n)} = \sfP^{n -2}(\mathsf{comp}^{(2)} ) \circ \mathsf{comp}^{(n-1)}.$

(2) can be checked by using  the following commutative diagram
\[
\begin{xymatrix}@C=60pt{
\Sigma^{-1} \sfP^{n} \ar[d]_{ \Sigma^{-1}\sfP^{ n-1}(\beta)_{\sfP} } 
\ar[r]^{ (\sigma'_{\sfP^{n -1} })_{\sfP} } 
&
 \sfP^{n -1}\Sigma^{-1} \sfP \ar[d]^{ \alpha^{-1}_{\sfP^{n -1}\Sigma^{-1} \sfP }} 
 \\ 
\Sigma^{-1} \sfP^{n -1} \Sigma \Sigma^{-1} \Sigma \sfS^{-1}
 \ar[r]_{ \Sigma^{-1}(\sigma_{\sfP^{n -1}})_{\Sigma^{-1}\Sigma\sfS^{-1}}} 
\ar[d]_{ \Sigma^{-1}\sfP^{n-1} \Sigma (\alpha)_{\sfS^{-1}} } 
& 
\Sigma^{-1} \Sigma \sfP^{n -1} \Sigma^{-1} \Sigma \sfS^{-1}
\ar[d]^{\Sigma^{-1} \Sigma\sfP^{n-1}  (\alpha)_{\sfS^{-1}}} 
\\
\Sigma^{-1}  \sfP^{n -1}  \Sigma \sfS^{-1} \ar@{=}[d]\ar[r]_{ \Sigma^{-1}(\sigma_{\sfP^{n -1}})_{\sfS^{-1}}} &
\Sigma^{-1} \Sigma \sfP^{n -1}\sfS^{-1} \ar@{=}[d] \\
\Sigma^{-1}\Sigma \sfS^{-1} \sfP^{n -1} \ar[r]_{ \Sigma^{-1}\Sigma( \gamma_{\sfP^{n-1}}) } & 
\Sigma^{-1}\Sigma \sfP^{n -1} \sfS^{-1}  
}\end{xymatrix} 
\]
where the bottom square is commutative by (1) and the left column is the identity morphism.
\end{proof}

\subsection{Happel's criterion}\label{subsection: Happel's criterion}

\subsubsection{For indecomposable objects }
Let  $X \in \ind  \sfD$.  
A morphism $s: \sfS^{-1}(X) \to X$ is called \emph{Auslander-Reiten (AR)-coconnecting} 
if it is a coconnecting morphism of  an AR-triangle starting from $X$. 
In other words, we have an AR-triangle of the following form:
\[
X \to Y \to \sfS^{-1}(X) [1]  \xrightarrow{ \ -s[1] \ } X[1].
\]

We note that if $s, t$ are AR-coconnecting to $X$, 
then there exist automorphisms $\phi, \psi$ of $X$ such that 
$t = \phi s, t= s\sfS^{-1}(\psi)$.

We recall Happel's criterion for AR-coconnecting morphisms. 

\begin{theorem}[{\cite[p37]{Happel Book}}]\label{Happel's criterion}
Let $X \in \ind \sfD$. 
Then a morphism $s: \sfS^{-1}(X) \to X$ is  an AR-coconnecting morphism  
if the following equations hold
\[
\isagl{\id_{X}, s} \neq 0, \ 
\isagl{f, s} = 0
\]
where  $f \in \rad \End_{\sfD}(X)$. 
The converse holds if $\dim \ResEnd_{\sfD}(X) = 1$. 
\end{theorem}

Let $X \in \ind \sfD$. 
If $\dim \ResEnd_{\sfD}(X) = 1$, then the subspace of $\Hom_{\sfD}(\sfS^{-1}(X), X)$ formed by all AR-coconnecting morphisms is one dimensional over $\kk$.  
Thus we have the following lemma. 

\begin{lemma}\label{202102071323}
Let $X$ be an  indecomposable object of $X$ such that $\dim \ResEnd_{\sfD}(X) = 1$. 
Then, the following holds. 
\begin{enumerate}[(1)] 
\item Two AR-coconnecting morphisms $s,t$ to $X$ are proportional over $\kk$ to each other. 
More precisely we have 
\[
t = \frac{\isagl{\id_{X}, t}}{\isagl{\id_{X}, s}} s. 
\]

\item An element $s \in \Hom_{\sfD}(\sfS^{-1}(X), X)$ is $0$ if and only if it satisfies 
\[
\isagl{\id_{X}, s } = 0, \ \ \isagl{f, s} = 0 \textup{ for } f \in \rad\End_{\sfD}(X). 
\]
\end{enumerate}
\end{lemma}

\subsection{Happel's criterion for not necessarily indecomposable objects}\label{202306211510}
The results given in Section \ref{202306211510} and \ref{2023062115101} are not used in the main body of the paper, but may be of independent interest.

We deal with the case where $X$ is not necessarily indecomposable. From now on we assume that $\dim \ResEnd_{\sfD}(X) = 1$ for all indecomposable $X$. Notice that this is satisfied if the base field $\kk$ is algebraically closed. 

\begin{definition}\label{202106040919} 
Let  $X \in \Dbmod{R}$ an object.  
A morphism $s: \sfS^{-1} (X) \to X$ is called \emph{AR-coconnecting} 
if there exists an indecomposable decomposition $\phi: X \xrightarrow{\cong} \bigoplus_{i =1}^{n} X_{i}$ 
and AR-coconnecting morphisms $s_{i}: \sfS^{-1}(X_{i}) \to X_{i}$ for all $i = 1,2, \ldots, n$ 
satisfying that 
the composition 
$\phi s \sfS^{-1}(\phi^{-1}): \sfS^{-1}(\bigoplus_{i =1}^{n} X_{i}) \to \bigoplus_{i =1}^{n} X_{i}$ 
is equal to $\bigoplus_{i =1}^{n} s_{i}$. 
\end{definition}

In other words, for an AR-coconnecting morphism $s: \sfS^{-1} (X) \to X$ there is a direct sum of AR-triangles 
\[
\sfS^{-1}(X) \xrightarrow{ s} X \to Y \to \sfS^{-1}(X) [1].
\]
We provide a Happel's criterion  for a not necessarily indecomposable object $X \in \Dbmod{R}$.

\begin{proposition}\label{202008081702}
Let $X \in \Dbmod{R}$. 
Then for  a morphism $s: \sfS^{-1} (X) \to X$ the following conditions are equivalent:

\begin{enumerate}[(1)] 

\item The morphism $s$ AR-coconnecting. 

\item The following conditions hold. 
\begin{enumerate}[(a)]
\item There exists a complete set $\{ e_{1}, e_{2}, \ldots, e_{n} \}$ of orthogonal primitive idempotent elements 
of $\End_{R}(X)$ 
such that 
$\isagl{e_{a},  s} \neq 0$ for all $a =1,2, \ldots, n$. 

\item $\isagl{ f, s} = 0$ for all $f\in \rad\End(X)$. 

\end{enumerate}
\end{enumerate}
\end{proposition}

\begin{proof}
The implication  (1) $\Rightarrow$ (2) is clear. We prove (2) $\Rightarrow$ (1). 
 For $a =1,2, \ldots, n$ we denote by $p_{a}: X \to X_{a}$ the projection induced from $e_{a}$ and 
by $i_{a}: X_{a} \to X$   the injection induced from $e_{a}$. 
The morphism $\phi:= (p_{1}, p_{2}, \ldots, p_{n})^{t}: X \to \bigoplus_{a=1}^{n}X_{a}$. 
is an isomorphism and $\phi^{-1} = (i_{1}, i_{2}, \ldots, i_{n})$. 
We set $s_{a} := p_{a} s \sfS^{-1}(i_{a})$ for $a =1,2, \ldots, n$. 
Then it is immediate to check the equality $s = \phi^{-1} (\bigoplus_{a=1}^{n} s_{i}) \sfS^{-1}(\phi)$. 
Thus it only remains to show that $s_{a}$ is AR-coconnecting to $X_{a}$. 

We use Happel's criterion for indecomposable objects. 
For $f \in \End_{R}(X_{a})$ we have 
\[
\isagl{f, s_{a}} = \isagl{i_{a} f p_{a}, s} \neq 0.
\]
It follows from the assumption (a) that 
$\isagl{\id_{X_{a}}, s_{a}} = \isagl{e_{a}, s} \neq 0$. 
On the other hand, 
if $f \in \rad\End_{R}(X_{a})$, then $i_{a} f p_{a} \in \rad\End_{R}(X)$. 
Thus it follows from the assumption (b) 
that 
$\isagl{f, s_{a}} = \isagl{i_{a} f p_{a}, s}= 0$. 
\end{proof}

We point out the following property of AR-coconnecting morphisms. 

\begin{lemma}\label{202106041023}
Let $X \in \sfD$ and $s, t$ be AR-coconnecting to $X$. 
Then there exist automorphisms $\phi, \psi$ of $X$ such that 
$t= \phi s, \ t = s\sfS^{-1}(\psi)$. 
\end{lemma}

\subsection{Happel's criterion for left $\rad^{n}$-approximations}\label{2023062115101}

\begin{theorem}\label{202106041030}
Let $n \geq 2$ be a positive integer, $M \in \sfD$ 
and $\lambda_{n -1}: M \to L_{n-1}$ a minimal left $\rad^{n -1}$-approximation, which fits an exact triangle 
$M \xrightarrow{ \lambda_{n-1}} L_{n -1} \xrightarrow{ \lambda'_{n -1} } C \xrightarrow{\lambda''_{n -1}} M[1]$. 
Assume that  $\lambda'_{n -1}: L_{n -1} \to C_{n -1}$ satisfies  the left $\rad$-fitting condition. 
Then for a morphism $s: \sfS^{-1}(L_{n -1} ) \to M$  the following conditions are equivalent. 
\begin{enumerate}[(1)] 
\item A cone morphism $t: M \to C$ of $s$ is a minimal $\rad^{n}$-approximation. 

\item The following conditions are satisfied. 
\begin{enumerate}[(a)]
\item There exists a complete set $\{e_{a}\}$ of orthogonal primitive idempotent elements of $\End(L_{n -1})$ 
such that $\isagl{ e_{a} \lambda_{n -1}, s} \neq 0$. 

\item We have $\isagl{ f \lambda_{ n-1}, s} = 0$ for all $ f \in \rad\End_{R}(L_{n -1})$. 
\end{enumerate}

\end{enumerate}

\end{theorem}

\begin{proof}
The implication (1) $\Rightarrow$ (2) follows from Lemma \ref{202105091801} and Proposition \ref{202008081702}. 
We prove the implication (2) $\Rightarrow$ (1). 
By the octahedral axiom, we obtain the following diagram 
\[
\begin{xymatrix}{ 
& C[-1] \ar@{=}[r] \ar[d] & C[-1] \ar[d]& \\
\sfS^{-1}(L_{n -1}) \ar@{=}[d] \ar[r]^{s} & M  \ar[d]_{\lambda_{n -1}} \ar[r] &\cone(s)  \ar[r] \ar[d] & \nu_{1}^{-1}(L_{n -1}) \ar@{=}[d] \\
\sfS^{-1}(L_{n -1})  \ar[r]_{\lambda_{n -1}s}  & L_{n -1} \ar[d]_{\lambda'_{n -1}} \ar[r] & N \ar[r] \ar[d] & \nu_{1}^{-1}(L_{n -1}) \\ 
& C \ar@{=}[r]  & C  & 
}\end{xymatrix} 
\]
whose middle lows and columns are exact. 
It follows from Proposition \ref{202008081702} that 
the condition  (a) and (b) implies that the composition $\lambda_{n -1} s$ is AR-coconnecting to $L_{n -1}$. 
 Thus the third row is a direct sum of AR-triangles. 
Since we assume that the morphism $\lambda'_{n -1}$ satisfies the left $\rad$-fitting condition, 
the morphism $N \to C_{n-1}$ is an split-epimorphism by  Lemma \ref{202105171508}. 
Thus, the morphism $\cone(s) \to N$ is a split-monomorphism. 
Thus by Lemma \ref{202008081702} we conclude that the morphism $M \to \cone(s)$ is a minimal left $\rad^{n}$-approximation of $M$. 
\end{proof}

\section{Natural isomorphisms}\label{section: natural isomorphisms}

In this section, $A$ denotes a finite-dimensional algebra with 
$\gldim A^{\mre} < \infty$ (or more generally, a proper and smooth dg-algebra).
We collect natural isomorphisms used in the main body of the paper. 
For the reader's convenience, we give proofs by direct computations. 
More formal arguments can be found in \cite{Minamoto Mukai}.

\subsection{}
 First note that for $M, N \in \Dbmod{A}$, we have a natural isomorphism 
 \begin{equation}\label{202106091014}
\tuD( \tuD(N) \lotimes_{\Pa} M ) \cong \RHom_{\Pa}(M, \tuD\tuD(N) ) \cong \RHom_{\Pa}(M, N).  
 \end{equation}
 
 Recall that 
  for $X \in \Dbmod{\Pa^{\mre}}$, 
 we set $X^{\vvee} := \RHom_{\Pa^{\mre}}(X, \Pa^{\mre})$. 
 Replacing $M$ with $X^{\vvee} \lotimes_{\Pa} M $ in \eqref{202106091014} and taking $\tuD$, we obtain the following natural isomorphism
\begin{equation}\label{202106091015}
\tuD(N) \lotimes_{\Pa}X^{\vvee} \lotimes_{\Pa} M 
\cong \tuD\RHom_{\Pa}(X^{\vvee}\lotimes_{\Pa} M ,N).
\end{equation}

On the other hand, 
we have the following natural  isomorphism
\begin{equation}\label{202106020808}
\begin{split}
\tuD(N) \lotimes_{\Pa}X^{\vvee} \lotimes_{\Pa} M
& \cong (\tuD(N) \otimes_{\kk} M) \lotimes_{\Pa^{\mre}} \RHom_{\Pa^{\mre}}(X, \Pa^{\mre}) \\ 
&\cong \RHom_{\Pa^{\mre}}(X, M \otimes_{\kk} \tuD(N)) \\ 
&\cong \RHom_{\Pa^{\mre}}(X, \Hom_{\kk}(N, M)) \\ 
&\cong \RHom_{\Pa }(X\lotimes_{\Pa}N , M). \\
\end{split} 
\end{equation}

Combining \eqref{202106091015} and \eqref{202106020808} in the case $X = A$, we obtain 
a natural isomorphism 
\begin{equation}\label{202106020807}
\tuD\RHom_{\Pa}(\Pa^{\vvee} \lotimes_{\Pa} M ,N) \cong \RHom_{\Pa }(N , M).  
\end{equation}
which shows that the functor $A^{\vvee} \lotimes_{\Pa} -$ is the inverse of a Serre functor.

Recall that  $X^{\lvvee}:= \RHom_{\Pa}(X, \Pa), \ X^{\rvvee} := \RHom_{\Pa^{\op}}(X, A^{\op})$.
Setting  $M = N =A$ in  \eqref{202106020808}, we obtain an isomorphism below in $\Dbmod{\Pa^{\mre}}$. 
\[
\tuD(A) \lotimes_{\Pa} X^{\vvee} \cong X^{\lvvee}.
\]
Repeating the same argument with right modules, 
we obtain an isomorphism $X^{\vvee} \lotimes_{\Pa} \tuD(A) \cong X^{\rvvee}$ 
in $\Dbmod{\Pa^{\mre}}$.

We remark that in the case where $M, N \in \Dbmod{\Pa^{\mre}}$, 
the above isomorphisms are isomorphisms in $\Dbmod{\Pa^{\mre}}$.

\subsection{} 

We write $\sfS^{-1} := \Pa^{\vvee} \lotimes_{\Pa} -$, since, as is explained above,  it is the inverse of a Serre functor. 

Let  $T \in \Dbmod{\Pa^{\mre}}$ be  a two-sided tilting complex over $A$. 
We denote by  $F: = T \lotimes_{\Pa}-$ the associated autoequivalence  of $\Dbmod{\Pa}$. 
Then, there exists a natural isomorphism $\gamma_{F}: \sfS^{-1} F \to F\sfS^{-1}$ 
induced from the defining property of a Serre functor (see Section \ref{202106091104}). 
Since $\sfS^{-1} F = \Pa^{\vvee} \lotimes_{\Pa} T \lotimes_{\Pa} -$ and 
$F \sfS^{-1} = T\lotimes_{\Pa} \Pa^{\vvee} \lotimes_{\Pa} -$,
it is natural to expect that the natural isomorphism $\gamma_{F}$ is induced from 
an isomorphism in $\Dbmod{\Pa^{\mre}}$. 
 In the following lemma, we prove that  it is the case.

\begin{lemma}[{\cite[Corollary 3.7]{Minamoto Mukai}}]\label{202106091106} 
There exists an isomorphism $\gamma_{T} : \Pa^{\vvee} \lotimes_{\Pa} T \to T \lotimes_{\Pa} \Pa^{\vvee}$ 
in $\Dbmod{\Pa^{\mre}}$ 
such that $\gamma_{F} = \gamma_{T}\lotimes_{\Pa} -$. 
\end{lemma}

\begin{proof}
By Rickerd \cite{Rickard}, 
there exists an object  $T' \in \Dbmod{\Pa^{\mre}}$  such that the endofunctor $T'\lotimes -$ of $\Dbmod{\Pa}$ is a quasi-inverse of $T\lotimes -$

We define an isomorphism 
$
\delta: T' \lotimes_{\Pa} \Pa^{\vvee} \lotimes_{\Pa} T \xrightarrow{ \cong } \Pa^{\vvee}
$ in $\Dbmod{\Pa^{\mre}}$ to be the following composition of isomorphisms 
\begin{equation}\label{202106091125} 
\begin{split}
 T' \lotimes_{\Pa} \Pa^{\vvee} \lotimes_{\Pa} T &\cong 
\tuD\RHom_{\Pa}(T' \lotimes_{\Pa} \Pa^{\vvee} \lotimes_{\Pa} T, \tuD(\Pa) ) \\
& \cong \tuD\RHom_{\Pa}(\Pa^{\vvee} \lotimes_{\Pa} T,  T \lotimes_{\Pa} \tuD(\Pa) ) \\
& \cong  \RHom_{\Pa}( T \lotimes_{\Pa} \tuD(\Pa), T ) \\
& \cong  \RHom_{\Pa}( \tuD(\Pa), \Pa) \\
& \cong \tuD\RHom_{\Pa}( \Pa^{\vvee}, \tuD(\Pa)) \\ & 
\cong \Pa^{\vvee}
\end{split}
\end{equation}
where  the third  isomorphism  and the fifth isomorphism are 
Serre dualities.

There exists an isomorphism $\epsilon: A \to  T\lotimes_{\Pa} T'$ in $\Dbmod{\Pa^{\mre}}$. 
We define an isomorphism $\gamma_{T}$ to  be $\gamma_{T} := ({}_{T}\delta)(\epsilon_{\Pa^{\vvee}\lotimes T})$. 
\[
\gamma_{T}: \Pa^{\vvee}\lotimes_{\Pa} T  \xrightarrow{ \ \epsilon_{\Pa^{\vvee}\lotimes T} \ } 
T \lotimes_{\Pa} T' \lotimes_{\Pa} \Pa^{\vvee} \lotimes_{\Pa}  T \xrightarrow{ {}_{T} \delta} T \lotimes_{\Pa} \Pa^{\vvee}  
\]

Let $M, N \in \Dbmod{\Pa}$. 
By Serre duality we have $\tuD(N) \cong (\Pa^{\vvee} \lotimes_{\Pa}N)^{\lvvee}$. 
Applying $\tuD(N) \lotimes - \lotimes M = ( \Pa^{\vvee} \lotimes_{\Pa} N)^{\lvvee} \lotimes_{\Pa} - \lotimes_{\Pa} M$ to \eqref{202106091125},  
we obtain the right square of the following commutative diagram
\[
\begin{xymatrix}{ 
\tuD(\Pa^{\vvee}\lotimes_{\Pa} T \lotimes_{\Pa} M, T\lotimes_{\Pa} N) 
\ar@{=}[r] 
\ar[d]_{\tuD(\epsilon_{\Pa^{\vvee} \lotimes T \lotimes M}^{*})} &
\tuD(\Pa^{\vvee}\lotimes_{\Pa} T \lotimes_{\Pa} M, T\lotimes_{\Pa} N) 
\ar[r]^-{\cong} \ar[d]_{\tuD(\mathsf{adj})}^-{\cong} & 
(T\lotimes_{\Pa} N, T\lotimes_{\Pa} M) \\
\tuD(T \lotimes_{\Pa} T' \lotimes_{\Pa} \Pa^{\vvee}\lotimes_{\Pa} T \lotimes_{\Pa} M, T\lotimes_{\Pa} N) 
\ar[r]^-{T} \ar[d]_{\tuD({}_{T}\delta_{ M}^{*})}
&
\tuD(T' \lotimes_{\Pa} \Pa^{\vvee}\lotimes_{\Pa} T \lotimes_{\Pa} M,  N) 
\ar[d]_{\tuD(\delta_{ M}^{*})} & \\
 \tuD(T \lotimes_{\Pa} \Pa^{\vvee}\lotimes_{\Pa} M, T\lotimes_{\Pa} N) \ar[r]_-{T} & 
 \tuD(\Pa^{\vvee}\lotimes_{\Pa} M,  N)  \ar[r] &
 (N, M) \ar[uu]_{T}
}\end{xymatrix}
\]
where we use the abbreviation $(-,+) := \RHom_{\Pa}(-, +)$, 
the arrows labeled by $T$ are induced from the functor $T\lotimes_{\Pa}-$ 
and $\mathsf{adj}$ denotes the adjoint isomorphism. 

Since the left column is induced from $\gamma_{T}$, 
this proves that $\gamma_{T}$ has the desired property.
\end{proof}

 \end{document}